\documentclass[11pt,a4paper]{article}  

\usepackage[usenames,dvipsnames]{xcolor}
\usepackage{amsmath}    
\usepackage{amssymb}    
\usepackage{amsthm}     
\usepackage{amscd}      
\usepackage{tikz}       
\usepackage{tikz-cd}    
\usepackage{etoolbox}   
\usepackage{mathrsfs}   
\usepackage{fixme}      
\usepackage{euscript}   
\usepackage{mathtools}
\usepackage{stmaryrd}
\usepackage[shortlabels]{enumitem}
\usepackage[T1]{fontenc}
\usepackage{lmodern}    
\usepackage{microtype}
\usepackage{enumitem}   
\setlist{topsep=4pt,itemsep=2pt,parsep=2pt,partopsep=0pt}

\newcommand*\mathcircled[1]{\tikz[baseline=(char.base)]{\node[shape=circle,draw,inner sep=2pt] (char) {#1};}}

\usepackage[colorlinks=true, linkcolor=blue, citecolor=blue, urlcolor=blue]{hyperref}

\usepackage[margin=3cm]{geometry}

\newcommand{\DeclareMathAlphabets}[1]{
  \expandafter\newcommand\csname c#1\endcsname{\mathcal{#1}}
  \expandafter\newcommand\csname f#1\endcsname{\mathfrak{#1}}
  \expandafter\newcommand\csname b#1\endcsname{\mathbb{#1}}
  \expandafter\newcommand\csname bf#1\endcsname{\mathbf{#1}}
  \expandafter\newcommand\csname s#1\endcsname{\mathscr{#1}}
}

\forcsvlist{\DeclareMathAlphabets}{A,B,C,D,E,F,G,H,I,J,K,L,M,N,O,P,Q,R,S,T,U,V,W,X,Y,Z}

\def\mk {\mathfrak}
\newcommand\Hom{\operatorname{Hom}}

\newcommand\Ext{\operatorname{Ext}}

\newcommand\Imm{\operatorname{Im}}
\newcommand\Ker{\operatorname{Ker}}
\newcommand\fib{\operatorname{fib}}
\newcommand\tr{\operatorname{tr}}

\numberwithin{equation}{section}

\newtheorem{theorem}{Theorem}[section]
\newtheorem{proposition}[theorem]{Proposition}
\newtheorem{lemma}[theorem]{Lemma}
\newtheorem{corollary}[theorem]{Corollary}

\theoremstyle{definition}
\newtheorem{definition}[theorem]{Definition}
\newtheorem{remark}[theorem]{Remark}
\newtheorem{example}[theorem]{Example}
\newtheorem{question}[theorem]{Question}
\newtheorem{conjecture}[theorem]{Conjecture}

\newtheorem{propdef}[theorem]{Proposition/Definition}

\begin{document}

\title{Towards Categorical K\"ahler Geometry}


\author{Fabian Haiden, Ludmil Katzarkov, Maxim Kontsevich, Pranav Pandit}

\newcommand{\Addresses}{{
  \setlength{\parindent}{0pt}
  \bigskip
  \footnotesize
    FH: \textsc{Centre for Quantum Mathematics, Department of Mathematics and Computer Science, University of Southern Denmark, Campusvej 55, 5230 Odense, Denmark} \par\nopagebreak
	\textit{E-mail:} \texttt{fab@sdu.dk}
    \medskip

    LK: \textsc{Simons Center for Geometry and Physics, Stony Brook, NY 11794, USA and International Center for Mathematical Sciences (ICMS), Sofia, Bulgaria.}
    \par\nopagebreak
	\textit{E-mail:} \texttt{lkatzarkov@gmail.com} 
    \medskip

    MK: \textsc{Institut des Hautes \'Etudes Scientifiques, Bures-sur-Yvette, France}
    \par\nopagebreak
	\textit{E-mail:} \texttt{maxim@ihes.fr} 
    \medskip

    PP: \textsc{International Centre for Theoretical Sciences, Bengaluru, India}
    \par\nopagebreak
	\textit{E-mail:} \texttt{pranav.pandit@icts.res.in} 
}}

\date{\today} 

\maketitle

\begin{abstract}
We outline the contours of an emerging theory of K\"ahler metrics in derived noncommutative geometry.
This is a refinement of the theory of Bridgeland stability conditions encoding underlying differential-geometric structures.
We propose elements of such a structure in both Archimedean and non-Archimedean settings, including metrized objects, mass measures satisfying a BPS inequality, harmonic metrics, minimizing flows, and complexified K\"ahler potentials. 
We develop the framework through examples and constructions involving Fukaya categories, quiver representations and associated C$^*$-algebras, spectral networks, and comonadic adjunctions of stable $\infty$-categories.
\end{abstract}

\setcounter{tocdepth}{2}
\tableofcontents

\section{Introduction}
\label{sec:intro}

\subsection{Background: From K\"ahler geometry to stability conditions}\label{subsec:background}

K\"ahler manifolds lie at the fascinating interface of algebraic geometry, differential geometry, and topology.
A central theme in this subject is that classes of algebraic or topological objects often have an essentially unique representative which minimizes some type of energy functional.
A first example is provided by Hodge theory: for compact $X$, each class in $H^n(X;\bC)$ has a unique harmonic representative minimizing the $L^2$-norm.

A more sophisticated example is provided by the Kobayashi--Hitchin correspondence for compact K\"ahler manifolds $(X,\omega)$, known as the Donaldson--Uhlenbeck--Yau (DUY) theorem~\cite{donaldson85,uhlenbeck_yau}. This states that among the Hermitian metrics on an indecomposable holomorphic vector bundle $E$ over $X$, there exists one solving the Hermitian Yang--Mills equations if and only if the bundle is slope-stable. 
Slope stability of $E$ means, by definition, that there are no proper coherent subsheaves with equal or larger slope $\mu(E)\coloneqq \mathrm{deg}(E)/\mathrm{rk}(E)$, $\deg(E)\coloneqq 
\int_Xc_1(E)\wedge\omega^{n-1}$, a purely algebraic notion.

As highlighted by Donaldson~\cite[Section 4]{donaldson85}, the Kobayashi--Hitchin correspondence can be conceptually interpreted as an infinite-dimensional instance of the Kempf--Ness principle, which, in finite dimensions, states that the geometric invariant theory quotient by a reductive group, and the symplectic reduction by its maximal compact subgroup are identified. The Hermitian Yang--Mills equation is then seen to be a moment map-type equation.

Inspired by mirror symmetry, Thomas and Yau~\cite{thomas01,thomas_yau02} initiated the search for a symplectic analogue of the DUY theorem, in which Hermitian vector bundles are replaced by Lagrangian submanifolds  $L$ (of a Calabi--Yau manifold) and the Hermitian Yang--Mills condition by the special Lagrangian condition $\arg(\Omega|_L)=\text{const}$.
Many new difficulties arise, one of which is that the relevant category, the Fukaya category of the underlying symplectic manifold, is not abelian but an (enhanced) triangulated category.
In particular, there is no intrinsic notion of subobject and the definition of slope stability breaks down.

An axiomatic notion of stability for triangulated categories was developed by Bridgeland in his influential paper~\cite{bridgeland07}, which was inspired by the work of Douglas on $\Pi$-stability of D-branes in string theory~\cite{DouglasFiolRomelsberger2005,Douglas2001Dbranes,douglas_icm}.
Bridgeland stability subsumes slope stability in abelian categories and leads to a much richer theory with interesting moduli spaces of stability conditions.
Joyce has formulated a version of the Thomas--Yau conjectures in terms of Bridgeland stability~\cite{joyce_conj} and highlighted subtleties regarding finite-time blow-up and obstructedness.
For more recent developments, see~\cite{lotay_oliveira,li22,stoppa}.

\subsection{Origins of the project}\label{subsec:origins}

We briefly explain the original motivations for this project.\footnote{This section can be skipped.}
The results of the paper~\cite{hkk} of the first three named authors relate stability conditions on Fukaya categories of surfaces to geometric structures on these surfaces (flat metrics, finite-length geodesics).
Our initial goal was to substantially generalize these results to the setting of Fukaya categories of surfaces \textit{with coefficients} in a perverse schober (see also~\cite{HKS}). 
Heuristically, this is related to the Thomas--Yau--Joyce proposal in the following way: given a symplectic fibration $p\colon X\to S$ over a surface $S$, possibly with singular fibers, one can hope to compute the Fukaya category of $X$, $\cF(X)$, by a two-step process: first, construct a system $\cE$ of enhanced triangulated categories over $S$, a perverse schober~\cite{KS_schobers}, assembled from the Fukaya categories of the fibers of $p$; 
second, compute the Fukaya category of $S$ with coefficients in $\cE$, $\cF(S;\cE)$.
If $X$ is moreover K\"ahler with holomorphic volume form $\Omega$ and $p$ is complex analytic, one expects to obtain a holomorphically varying family of stability conditions on $\cE$.
The goal is then to integrate this family to a stability condition on $\cF(S;\cE)$.
The semistable objects of $\cF(S;\cE)$ should then admit geometric representatives given by special graphs on $S$ labeled by algebraic data from $\cE$.
The underlying graphs can be thought of as degenerate limits of the projections to $S$ of special Lagrangian submanifolds in $X$, as the volumes of the fibers of $p$ tend to zero.
They are generalizations of the \textit{finite webs} of Gaiotto--Moore--Neitzke~\cite{GMN-spectral}, closely related to their~\textit{spectral networks}.

During this project, we realized the importance of various additional geometric structures on a triangulated category with a stability condition which encode its geometric origin. This would include, in particular, non-Archimedean K\"ahler metrics on the moduli spaces of polystable objects.
In the search for simpler toy models of our problem, we turned to quiver representations and complexes of quiver representations, over both $\bC$ and non-Archimedean fields.
This transformed the project into an ambitious program aimed at finding a categorical notion of K\"ahler geometry and drawing on many areas of mathematics, including higher category theory, homological mirror symmetry, C$^*$-algebras, scattering diagrams, and buildings.

\subsection{Goals}\label{subsec:goals}

In its most concise form, our goal can be stated as completing the following analogy between classical differential geometry and categorical geometry.
Recall that the slope of a holomorphic vector bundle, and hence its slope-stability condition, depends only on the K\"ahler class, whereas the Hermitian Yang--Mills equation depends on the K\"ahler metric itself.

\begin{center}
\renewcommand{\arraystretch}{1.8}
\begin{tabular}{c|c}
    complex manifold $X$ & enhanced triangulated category $\cC$ \\
    \hline
     K\"ahler form $\omega\in\Omega^{1,1}(X)$ & \textit{categorical K\"ahler metric} on $\cC$ \\
     \hline
     K\"ahler class $[\omega]\in H^{1,1}(X)$ & stability condition on $\cC$
\end{tabular}
\end{center}

We are thus seeking a definition of ``categorical K\"ahler metric'', both in Archimedean and non-Archimedean variants, which should arise naturally on categories $\cC$ of the form $D^b\mathrm{Coh}(X)$ or $\cF(X)$ for some Calabi--Yau manifold $X$, as well as other examples, e.g.\ the ones considered in this paper.
Such a structure should include:
\begin{itemize}
    \item a notion of a \textit{metric} on an object of $\cC$,
    \item \textit{masses} of metrized objects satisfying the BPS inequality with respect to the central charge,
    \item a notion of harmonic metric such that existence of a harmonic metric implies semistability of the underlying object, and polystability guarantees its existence,
    \item a \textit{flow} on the space of metrizations of each object which decreases mass and converges to a harmonic metric if the underlying objects admits one,
    \item a \textit{complexified K\"ahler potential} on the space of metrized objects which provides K\"ahler metrics on (smooth loci of) moduli spaces of polystable objects.
\end{itemize}
See Section~\ref{section:axiomatic} for a more detailed discussion.
We also hope that a future, more complete framework will strengthen the bridge between (categorical, homological) noncommutative algebraic geometry and operator-algebraic noncommutative geometry. Here we only touch on this topic briefly  in Section~\ref{sec:quivers_arch} in the context of quiver representations.

While we have not yet achieved this goal, our work on this program has led us to several exciting subprojects with interesting new constructions, results, and conjectures.
These include our work on the asymptotics of minimizing flows and iterated logarithms~~\cite{HKKPitlogs1,HKKPitlogs2}, as well as the constructions and conjectures discussed in later sections of this paper, related to quiver representations, spectral networks, and comonads on stable $\infty$-categories.

At this point, we have worked on the project on and off for over a decade and given talks and lecture series on much of the material. 
We feel that the time has come to record what we have established thus far, both to create a reference point and to inspire further development of the ideas.

\subsection{Related Work}\label{subsect:related}
Beyond the Thomas--Yau picture discussed above, and its subsequent development by Joyce, several developments are closely related to the perspective pursued here. On the symplectic side, Smith has discussed potential applications of Bridgeland stability conditions on Fukaya categories to symplectic topology, in particular to symplectic mapping class groups \cite{smith_stability_symplectic}. The third named author, together with Yan Soibelman, has recently proposed a generalized Riemann--Hilbert correspondence involving analytic families of Fukaya categories and limiting Bridgeland stability conditions; in the same article, he also discusses a related twistor picture involving harmonic objects \cite{Maxim-ECM24}.

In complex geometry, the deformed Hermitian--Yang--Mills equation, which is mirror to the special Lagrangian equation, has been studied as a differential-geometric counterpart to stability~\cite{CollinsJacob2013,CollinsShi2019}; Collins, Xie, and Yau discuss its relation to algebraic stability conditions \cite{Collins-Xie-Yau}, while Collins and Yau develop its variational and moment-map interpretation as an infinite-dimensional GIT problem \cite{Collins-Yau}. More generally, Dervan, McCarthy, and Sektnan associate geometric PDEs on holomorphic vector bundles to central charges arising in Bridgeland stability, whose solutions they call \(Z\)-critical connections, and establish a correspondence between asymptotic \(Z\)-stability and the existence of such connections in the large-volume limit \cite{Dervan-McCarthy-Sektnan}. Dervan has developed a parallel theory for polarized varieties \cite{Dervan-K-Stability}. These works provide differential-geometric manifestations, in classical settings, of a principle closely related to the one motivating the present paper: that stability data encoded by a central charge should admit a refinement by metric and differential-geometric structures.

A complementary perspective comes from extensions of geometric invariant theory itself. Halpern-Leistner's theory of \(\Theta\)-stratifications develops intrinsic notions of instability and Harder--Narasimhan filtrations for general moduli problems, generalizing both the Kempf--Ness stratification in GIT and Harder--Narasimhan stratifications, and applies in particular to moduli stacks of objects in hearts of Bridgeland stability conditions \cite{Halpern-Leistner-instability}. More recently, Dervan, with an appendix by Ib\'a\~nez N\'u\~nez, has developed structures in GIT analogous to stability conditions on abelian and triangulated categories, including notions of central charge on schemes with group actions and on stacks, together with an analytic counterpart to stability and a corresponding Kempf--Ness picture \cite{Dervan-Stability-GIT}. This provides a finite-dimensional framework connecting central charges, stability, moment maps, and canonical geometric structures which is closely related to, and complementary to, the categorical metric structures considered here.

A different point of contact comes from differential-geometric approaches to categorical noncommutative geometry. Block associates differential graded algebra data to dg categories of modules in a framework which, in the complex-geometric case, recovers the derived category of coherent sheaves \cite{Block-cohesive}. Closely related constructions using \(\bar\partial\)-superconnections were developed by Bondal and Rosly, who obtain a dg enhancement of the derived category of coherent sheaves and use superconnections to define Chern and Bott--Chern classes of its objects \cite{Bondal-Rosly}. These perspectives are closely related in spirit to our use of differential graded and curved differential graded structures in categorical K\"ahler geometry.

\subsection{Contents of the paper}\label{subsect:organization}

\subsubsection{Stability conditions}

Section~\ref{section:stability} is mainly preparatory and discusses abstract aspects of stability in abelian and triangulated categories, i.e.\ without any additional geometric structures.
After a warm-up on slope stability, we begin with a review of the definition of Bridgeland stability.
The remaining three subsections contain partially new material.
In the first (Section~\ref{subsec:mass}), we establish conditions under which the triangle inequality for mass, $m(B)\leq m(A)+m(C)$ for an exact triangle $A\to B\to C\to A[1]$, is either strict or an equality (Proposition~\ref{prop:mass}).
This is used later, in Sections~\ref{section:spectral} and~\ref{section:comonads}, in proofs based on induction on the mass.
In Section~\ref{subsec:effective}, we relate stability conditions to positivity in algebraic geometry, showing that the degree of any curve in the moduli space of semistable objects, defined in terms of the central charge, is non-negative (Theorem~\ref{thm:ample}).
This motivates the idea that a Bridgeland stability condition can fruitfully be thought of as a ``noncommutative K\"ahler class''.
We finish the section with a discussion of ascending-phase filtrations of objects in hearts of stability conditions.
The main result is a $\widetilde{\mathrm{GL}}^+(2,\bR)$-invariant construction of a finite-length abelian category of pairs of such filtrations, the \textit{vertex category}.
This category describes the algebraic data attached to a vertex of a spectral network and plays a central role in Section~\ref{section:spectral}.

\subsubsection{Heuristics and framework}

Section~\ref{section:heuristics} is devoted to the heuristics that inform
our axiomatic approach to K\"ahler metrics on noncommutative spaces.  These
heuristics derive largely from Fukaya categories and homological mirror
symmetry.  We first recall the relation between stability conditions,
special Lagrangians, and the large-volume limit, and discuss a conjectural
 Fukaya category  over the positive Novikov ring  $Nov_+$, built from singular Lagrangian supports and  canonical
deformations of their local Fukaya categories determined by
pseudo-holomorphic discs.

 In the final part of the section we give a more concrete sheaf-theoretic
model in the case of a split symplectic torus.  The construction does not
involve a choice of almost-complex structure, although it uses the
Lagrangian torus fibration. On the mirror side this
leads to a large $Nov_+$-linear category associated with a power of the
Tate elliptic curve.  Using a Rees construction and microlocal sheaf
theory, we associate to its objects a reduced singular support in the
symplectic torus and describe its covariance under symplectomorphisms.
For a natural class of metrized objects these supports are rational
piecewise-linear Lagrangian subsets carrying microlocal coefficient
systems.

This allows us to formulate a precise conjectural characterization of a
Bridgeland stability condition: semistable objects of a given phase should
be exactly those admitting metrized representatives whose Lagrangian
supports have arbitrarily small phase amplitude.  We finally extend this
picture heuristically to Fukaya categories with coefficients in a local
system of categories equipped with varying stability conditions.

\ 

While Section~\ref{section:axiomatic} does not contain a definitive definition of \textit{categorical K\"ahler structure} adapted to a given stability condition, we identify various desired elements of such a structure.
This includes a notion of \textit{metrized object}, and these should, at least in the non-Archimedean case, form a category (see Section~\ref{subsec:frameworknonarch}) which comes with a rescaling action. 
Each metrized object $(X,h)$ has not only a mass, $m(h)$, satisfying the BPS inequality $|Z(X)|\leq m(h)$, but also a \textit{mass measure}, $\mu_h$, which records the distribution of mass across a bounded interval of phases (see Section~\ref{subsec:generalframework}).
The total mass $\mu_h(\bR)$ is then $m(h)$ while the central charge is the total weighted measure
\[ 
Z(X)=\int_{\bR}e^{\pi\sqrt{-1}\phi}d\mu_h(\phi). 
\]
In the Archimedean case, these mass measures, as well as the K\"ahler potential and flow, are all determined by a $*$-algebra $\cA_{X,h}$ with linear functional $\Omega\colon\cA_{X,h}\to\bC$ depending on the metric $h$ on an object, as we propose in Section~\ref{subsec:frameworkarch}.
The heuristic interpretation of $\cA_{X,h}$ is as the complexified tangent space at $h$ to the space of metrizations, while $\Omega$ is the variation of the complexified K\"ahler potential $S_{\bC}$.

A more complete, though still approximate, local picture of ``categorical K\"ahler geometry'' over $\bC$ is discussed in Section~\ref{subsec:lozenge}.
In fact, much of this framework is already present in our previous work~\cite{HKKPitlogs2} and in the paper~\cite{BK21} by Bhattacharya and the third named author, where a wide range of examples, including Nekrasov's noncommutative instantons, are considered. 
The main new contribution in this section is to show that both frameworks are essentially equivalent.

\

The remaining sections of the paper describe examples of categorical K\"ahler structures arising from different areas of mathematics.

\subsubsection{Quiver representations}

In Sections~\ref{sec:quivers_arch} and \ref{sec:quivers_nonarch} we study categorical K\"ahler geometry on abelian and derived categories of representations of quivers over $\bC$ and non-Archimedean fields, respectively. 
We begin by explaining how this example fits into our general framework by employing a theorem of A.~King on the ``GIT quotient = symplectic reduction'' isomorphism in this setting (Theorem~\ref{thm:King}).
We then discuss a reformulation of this result in terms of unitary representations of certain $*$-algebras, which are generated by arrows $a\in Q_1$ of a quiver $Q$ subject to the relation
\[
\sum_{a\in Q_1}\left[a,a^*\right]=\sum_{i\in Q_0}\theta_ie_i
\]
where $\theta_i\in \bR$ are parameters for each vertex $i\in Q_0$.
These $*$-algebras may be completed to C$^{*}$-algebras in the acyclic case and pro-C$^{*}$-algebras in the general case (see Section~\ref{subsec:staralg}). 
They can be thought of as representing the underlying topological  space of the hypothetical noncommutative K\"ahler spaces arising from quivers.

A first step from the abelian to triangulated world is taken by considering Bernstein--Gel'fand--Ponomarev reflection functors, which induce derived equivalences between a quiver and its mutation obtained by reversing the direction of all arrows at a source/sink.
We show that the C$^*$-algebras and the induced K\"ahler metrics on the moduli spaces of polystable objects, are invariant under reflection (Theorem~\ref{thm:C-star} and Theorem~\ref{thm:mutation-invariance}).
We then continue with a discussion of general complexes of quiver representations and how they fit into the general framework.
This can be seen as a quiver analogue of (Chern-)superconnections and the corresponding representatives of the Chern character. 
The main novelty is the study of mass and flow in this context, which is illustrated in simple examples.
Several important questions about the flow in this setting remain open.
In Section~\ref{subsec:rampfunctions} we propose a more general class of categorical K\"ahler metrics for quivers which depend on a choice of increasing convex function (\textit{ramp function}) for each arrow.

In Section~\ref{sec:quivers_nonarch} we transition from $\bC$ as ground field to non-Archimedean fields.
Our main goal is to find the non-Archimedean analogue of King's theorem (Theorem~\ref{thm:King}). 
This involves (1) finding an analogue of the harmonicity (moment map) condition on the metric, and (2) proving the existence of non-Archimedean harmonic norms on polystable quiver representations.
To achieve (1), we consider families of K\"ahler metrics depending on a small parameter, then pass to the formal non-Archimedean limit, see  Section~\ref{subsec:nonarchlimit}. 
The condition we arrive at involves stability of a representation of an auxiliary ``star-shaped'' quiver over the residue field.
The purpose of this section is purely to illustrate the general technique and motivate later definitions; it is not logically necessary for what follows.
We then recall definitions and results around non-Archimedean norms on finite-dimensional vector spaces, using some of the basic terminology from the theory of buildings, which provides a useful conceptual guide.
Goal (2) is achieved in Section~\ref{subsect:non-arch-quiver}, at least under a spherical completeness assumption on the ground field.
This assumption ensures that the spaces of diagonalizable non-Archimedean norms (which are affine buildings) are complete as metric spaces.
The main result, Theorem~\ref{thm:nonarch_harm}, shows that polystability implies existence of a harmonic norm, which in turn implies semistability. 
The situation is thus somewhat more complicated than in the Archimedean case, where polystability is equivalent to existence of a harmonic metric.
The main tool in the proof of Theorem~\ref{thm:nonarch_harm} is the K\"ahler potential on the space of norms on the representation, a building-like space.

\subsubsection{Spectral networks}

In Section~\ref{section:spectral} we formulate a conjectural correspondence between spectral networks in the plane and quiver representations over a non-Archimedean field equipped with a pair of ascending-phase filtrations with metrized subquotients.
As discussed in Section~\ref{subsec:antihn}, a pair of ascending-phase filtrations corresponds to a diagram with a single vertex, which we think of as a vertex of a spectral network.
Choices of metrics on the subquotients conjecturally prescribe a deformation of the one-vertex spectral network to an arbitrary spectral network in the plane with finite topology.
In the other direction, from a very ``zoomed out'' view, the deformed spectral network collapses to the original one-vertex network.
The correspondence depends on a choice of area form on $\bR^2$ such that (1) the area of the strip between any pair of parallel lines is finite, and (2) the area of any conical sector between intersecting lines is infinite.
The choice of this (K\"ahler) area form can be viewed as providing categorical K\"ahler structure in this context.

\begin{figure}[ht]
    \centering
    \includegraphics[width=0.9\linewidth]{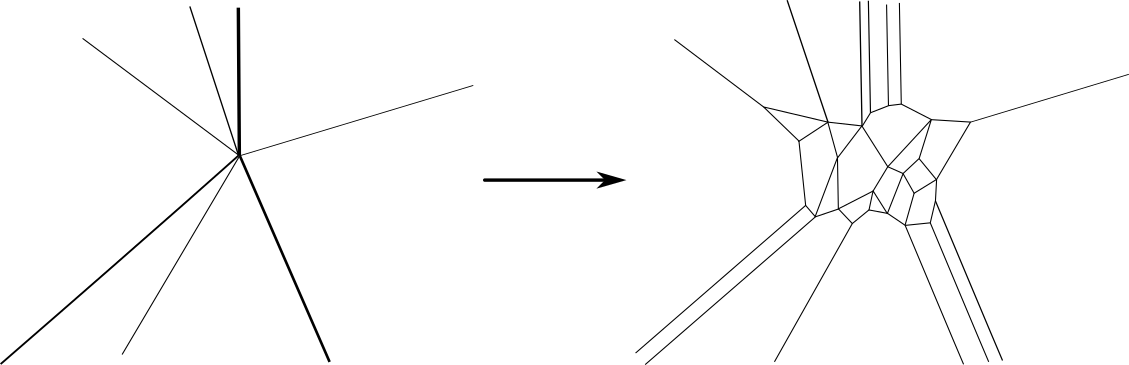}
    \caption{Deformation of a spectral network near a single vertex (``bubbling'').}
    \label{fig:specnetdeform}
\end{figure}

Our motivation for this conjecture is the problem, mentioned in Section~\ref{subsec:origins} above, of showing existence of finite web representatives of semistable objects in Fukaya categories of surfaces with coefficients in a local system of categories.
Heuristically, we expect the flow minimizing total mass of the graphs on the surface to deform a vertex of the graph, described by the pair of ascending-phase filtrations, to a more complicated graph where the vertex is replaced by the \textit{harmonic path} (see Figure~\ref{fig:specnetdeform}).

\subsubsection{Comonads}

Fukaya categories of symplectic manifolds $(M,\omega)$ come in (formal) one-parameter families by rescaling the symplectic form: $\omega\mapsto t\cdot\omega$, $t>0$. 
This family extends to $t=+\infty$, the \textit{large volume limit}, and can thus be considered as a kind of non-Archimedean metric on the Fukaya category, viewed as the general fiber of this family.
Similarly, in algebraic geometry, we can consider one-parameter families of varieties, or more generally, formal neighborhoods of effective divisors.
Abstracting these examples, one arrives at the notion of a comonadic adjunction of stable $\infty$-categories 
\[
G:\cD\leftrightarrows\cC:F,
\]
where $\cC$ is interpreted as representing the total space of the family and $\cD$ the central fiber.
The general fiber is then $\cE\coloneqq \cC/G(\cD)$.
From the point of view promoted here, the adjunction provides a kind of categorical K\"ahler metric on $\cE$.
The category of metrized objects is $\cC$.
Fix a stability condition on $\cD$ which is invariant under $T\coloneqq\mathrm{fib}(FG\xrightarrow{\varepsilon}1_{\cD})\circ [-1]$, which we assume is an autoequivalence.
An object in $\cC$ is, by definition, \textit{harmonic} if its image in $\cD$ is semistable.
More generally, we get mass measures from Harder--Narasimhan filtrations in $\cD$.
Our main result (Theorem \ref{thm:comonads}) is then the construction of a stability condition on $\cE$, depending only on the initial stability condition on $\cD$, whose semistable objects are precisely those admitting a harmonic metric.
We discuss examples arising from formal completions of smooth schemes along effective divisors (Section~\ref{subsec:nbhdiv}) and curved deformations of $A_\infty$-categories (Section~\ref{subsec:defAinf}).

\subsubsection{List of main new contributions}

Since the paper is quite long, and parts of it are reviews of the relevant background or technical preliminaries, let us briefly highlight what we feel are the most important new contributions.

\begin{itemize}
    \item The study of the \textit{vertex category} of pairs of ascending-phase filtrations (Section~\ref{subsec:antihn}).
    \item The sheaf-theoretic $Nov_+$-linear model for Fukaya categories of
split symplectic tori, its reduced singular support, and the conjectural
characterization of stability in terms of approximate special-Lagrangian
representatives (Section~\ref{sec:sheaf}).
    \item The proposed framework for categorical K\"ahler geometry (Sections~\ref{subsec:generalframework}, \ref{subsec:frameworkarch}, \ref{subsec:frameworknonarch}).
    \item The construction of C$^*$-algebras from quivers with stability parameters (Section~\ref{subsec:staralg}) and invariance under reflection functors (Theorem~\ref{thm:mutation-invariance}, Theorem~\ref{thm:C-star}).
    \item The notion of \textit{harmonic norms} and their existence in the non-Archimedean setting (Theorem~\ref{thm:nonarch_harm}).
    \item The conjecture on existence and uniqueness of spectral networks in the plane with given boundary conditions (Conjecture~\ref{conj:bubbling}).
    \item The transport of stability conditions for (co)monadic adjunctions of stable $\infty$-categories (Theorem~\ref{thm:monadtransport}).
\end{itemize}

\subsubsection{Guide to the reader}

The various sections in the body of this paper are largely logically independent, with the main exception being the first three subsections of Section~\ref{section:stability}, which provide the basic definitions and results on stability conditions used throughout the rest of the text.
The kinds of mathematics appearing in the sections also vary significantly. For instance, Section~\ref{sec:quivers_arch} contains quiver representations, geometric invariant theory quotients, and $*$-algebras, while Section~\ref{section:comonads} deals with stable $\infty$-categories and monads.
Consequently, most readers will not want to tackle the entire paper in sequential order, but instead approach different sections according to their own interests.

\subsection{Conventions, notation, and terminology}\label{subsect:conventions}

\begin{itemize}
\item We typically write $\sqrt{-1}$ for the imaginary unit, as $i$ is often used for indices, in particular vertices of a quiver.
\item An \textit{(increasing) $\bR$-filtration} $\cF_\bullet$ of an object $E$ in an abelian category is a collection of subobjects $\cF_\alpha\subseteq E$ with $\cF_\alpha\subseteq \cF_\beta$ for $\alpha\leq\beta$.
We always require the right-continuity property $\cF_\alpha=\bigcap_{\beta>\alpha}\cF_\beta$.
The filtration is \textit{complete} if it is separated and exhaustive, meaning that $\bigcap_\alpha\cF_\alpha=0$, $\bigcup_\alpha\cF_\alpha=E$.
Write $\cF_{<\alpha}\coloneqq\bigcup_{\beta<\alpha}\cF_\beta$.
The filtration is \textit{finite} if the set of $\alpha$ with $\cF_\alpha/\cF_{<\alpha}\neq 0$, the \textit{steps}/\textit{jumps} of the filtration, is finite.
While $\bR$-filtrations are by default increasing, we will sometimes consider decreasing $\bR$-filtrations as well.

\item Phase is related to central charge by the normalization $\phi(E)=\arg(Z(E))/\pi$.
\item Stability conditions are assumed to satisfy the support property.
\item We write $\mathbf k$ for the ground field. (In a non-Archimedean context, $\mathbf k$ will usually denote the residue field instead.)
\item For objects $A,B$ in a triangulated category $\cC$, we write $\mathrm{Ext}^n(A,B)\coloneqq\mathrm{Hom}_{\cC}(A,B[n])$.
\item Following Lurie~\cite{LurieHTT}, \textit{$\infty$-category} means a quasi-category. 
The \textit{homotopy category} of an $\infty$-category $\cC$ is denoted by $\mathrm{h}\cC$. 
\item A morphism in an $\infty$-category $\cC$ is an \textit{isomorphism} iff it has an inverse up to homotopy, or equivalently it maps to an isomorphism in $\mathrm h\cC$.
\item If $\cC$ is a pointed $\infty$-category, we write $X=0$ when $X$ is a zero object in $\cC$ and $f=0$ when $f$ is a zero morphism in $\cC$, i.e.\ factors through a zero object, up to homotopy.
\item Functors between stable $\infty$-categories are always assumed to be exact.
\item A $*$-structure on an associative $\bC$-algebra $A$ is a conjugate-linear involution $A\to A,a\mapsto a^*$, such that $(ab)^*=b^*a^*$. If $A$ is graded, we require $|a^*|=-|a|$ and $(ab)^*=(-1)^{|a||b|}b^*a^*$, unless otherwise stated.
\item $\mathbf{Rep}(Q,\mathbf k)$ is the abelian category of finite-dimensional representations of a quiver $Q$ over a field $\mathbf k$.
\item Given a non-Archimedean field $K$, we write $\cO$ for its ring of integers (closed unit ball), $\mathfrak m\subset \cO$ for the maximal ideal (open unit ball), and $\mathbf k=\cO/\mathfrak m$ for the residue field. We allow the trivial valuation for which $K=\cO=\mathbf k$.
\end{itemize}

\subsection*{Acknowledgements}

F.H.\ thanks Merlin Christ, Emanuele Macr\`i, Carlos Simpson, and Wojciech Szyma\'nski for stimulating discussions related to this paper.
P.P. is especially grateful to Anthony Blanc, Alexander Noll, Carlos Simpson, and Ted Spaide for collaborations that have intersected with the ideas of this project in various ways, and for many illuminating discussions over the years. P.P. thanks Indranil Biswas, Tristan Collins, Ruadha\'i Dervan, Mario Garc\'ia-Fern\'andez, Oscar Garc\'ia-Prada, Daniel Halpern-Leistner, Andr\'es Ib\'a\~nez-N\'u\~nez, Dmitry Kaledin, Michael Kapovich, Mikhail Kapranov, Jacob Kryczka, Yu-Shen Lin, R. Loganayagam, Jacob Lurie, John McCarthy, Andrew Neitzke, Nikita Nekrasov, Tony Pantev, Ashoke Sen, Artan Sheshmani, Yan Soibelman, Yuuji Tanaka, Richard Thomas, Bertrand To\"en, Spenta Wadia, Johannes Walcher and Shing-Tung Yau for stimulating conversations related to the subject matter of this paper. Much of the development of this work took place during extended visits to the Institut des Hautes \'Etudes Scientifiques, whose hospitality and stimulating environment he gratefully acknowledges. The project began while he was a postdoc at the University of Vienna, and he thanks the University for its support. Meetings organized through the Simons Collaboration on Homological Mirror Symmetry helped facilitate the early development of some of the ideas in this work.

F.H. was supported by the Sapere Aude grant 3120-00076B from the Independent Research Fund Denmark (DFF) and the VILLUM FONDEN, VILLUM Investigator grant 37814.
L.K. was supported by the Simons Foundation, grant
SFI-MPS-T-Institutes-00007697, a Simons Investigators Award (no.\
003136), the Simons Collaboration on Homological Mirror Symmetry
(award no.\ 003093), NSF FRG grant DMS-2245171, NSF grant DMS-2401495, and the Ministry of Education and Science of the Republic of Bulgaria, grant DO1-239/10.12.2024.
P.P. acknowledges support from the Department of Atomic Energy, Government of India, under project no.\ RTI4019.

\section{Categorical notions of stability}
\label{section:stability}

In this section we discuss general features of the theory of stability conditions on triangulated categories, including mass (Section~\ref{subsec:mass}), positivity (Section~\ref{subsec:effective}), and ascending-phase filtrations (Section~\ref{subsec:antihn}), which are foundational to our considerations in later sections.
The first two subsections are mostly expository, fixing conventions and recalling slope and Bridgeland stability.

\subsection{Stability in abelian categories}\label{subsec:stabab}

We recall slope stability and Harder--Narasimhan filtrations in abelian categories in preparation for the discussion of Bridgeland stability in the next subsection.

\begin{definition}\label{def:slopestability}
A \textbf{central charge}, also known as a \textit{stability function}, on an abelian category $\cA$ is an additive map $Z\colon K_0(\cA)\to\bC$ such that $Z(E)\in\bR_{>0}e^{\sqrt{-1}(0,\pi]}$ for $0\neq E\in\cA$, i.e.\ $Z(E)$ lies in the upper half-plane or negative real axis.
The \textbf{phase} of a non-zero object $E$ is defined to be $\phi(E)=\arg(Z(E))/\pi\in(0,1]$.
An object $E\in\cA$ is said to be
\begin{itemize}
    \item \textbf{semistable} if $0\neq A\subseteq E\implies \phi(A)\leq \phi(E)$,
    \item \textbf{stable} if $E\neq 0$ and $0\neq A\subsetneq E\implies \phi(A)< \phi(E)$,
    \item \textbf{polystable} if $E$ is a finite direct sum of stable objects of equal phase.
\end{itemize}
\end{definition}

\begin{example}
If $X$ is a smooth projective curve, then $Z(E)\coloneqq -\deg(E)+\sqrt{-1}\mathrm{rk}(E)$ is a central charge on the category $\mathrm{Coh}(X)$ of coherent sheaves on $X$.
The notion of stability with respect to this central charge goes back to Mumford~\cite{mumford_icm}.
\end{example}

\begin{remark}
Instead of the phase $\phi(E)=\arg(Z(E))/\pi\in(0,1]$ one can equivalently consider the \textit{slope} $\mu(E)\coloneqq -\operatorname{Re}(Z(E))/\operatorname{Im}(Z(E))\in (-\infty,+\infty]$, which is just $\mathrm{deg}(E)/\mathrm{rk}(E)$ in the case of coherent sheaves on a curve.
In the setting of triangulated categories, the phase is more convenient.
\end{remark}

\begin{definition}\label{def:HNabelian}
Let $\cA$ be an abelian category with central charge $Z$.
A filtration
\[
0=E_0\subset E_1\subset E_2\subset\ldots\subset E_n=E
\]
of an object $E\in\cA$ is a \textbf{Harder--Narasimhan (HN) filtration} if $A_i\coloneqq E_i/E_{i-1}$ is non-zero and semistable for $i=1,\ldots,n$ with $\phi(A_1)>\phi(A_2)>\ldots>\phi(A_n)$.

The central charge $Z$ has the \textbf{Harder--Narasimhan (HN) property} if every object in $\cA$ has a Harder--Narasimhan filtration.
\end{definition}

The terminology originates from the case $\cA=\mathrm{Coh}(X)$ as above, where the HN property of $Z=-\deg+\sqrt{-1}\mathrm{rk}$ was proven in~\cite{harder_narasimhan75}.
The HN property automatically holds whenever $\cA$ is finite-length, i.e.\ any filtration of an object in $\cA$ is essentially finite. This is a special case of \cite[Proposition 2.4]{bridgeland07}.

Let $A$ be a $\mathbf k$-algebra and $\mathcal A=\mathrm{Mod}^{fd}(A)$ the category of finite-dimensional modules over $A$. Then $\mathcal A$ is a typical example of a finite-length category.
A $\mathbb Z$-basis of $K_0(\cA)$ is given by isomorphism classes of simple modules, hence central charges correspond to a choice of complex number $Z(S)\in\bR_{>0}e^{\sqrt{-1}(0,\pi]}$ for every simple $S$, up to isomorphism, and they all satisfy the HN property.

The particular case when $A$ is the path algebra of a quiver will be the main subject of Sections~\ref{sec:quivers_arch} and \ref{sec:quivers_nonarch}.
In this case, the relation between slope stability and stability in the sense of geometric invariant theory~\cite{git} was established by King~\cite{king_reps}.

\begin{example}\label{example:A_2}
Let $Q=\bullet\to\bullet$ be the ``$A_2$ quiver''.
Gaussian elimination shows that there are, up to isomorphism, three indecomposable objects in the category $\mathbf{Rep}(Q)$ of finite-dimensional modules over the path algebra $\mathbf kQ$ of $Q$:
\[
A\coloneqq(0\to\mathbf k),\qquad B\coloneqq(\mathbf k\xrightarrow{1}\mathbf k),\qquad C\coloneqq(\mathbf k\to 0).
\]
Moreover, they fit into a non-split short exact sequence $0\to A\to B\to C\to 0$.

The central charge $Z$ is fixed by the choice of $z_1=Z(C)$ and $z_2=Z(A)$.
The space $\{z\mid\operatorname{Im}(z)>0\}^2$ of central charges splits into two \textit{chambers} where $\phi(A)<\phi(C)$ or $\phi(A)>\phi(C)$, respectively, separated by the codimension 1 \textit{wall} where $\phi(A)=\phi(C)$.
In the chamber where $\phi(A)<\phi(C)$, $B$ is stable, while in the chamber where $\phi(A)>\phi(C)$, $B$ is unstable (not semistable).
On the wall, $B$ is semistable but not polystable (sometimes called \textit{strictly semistable}).
Note that, as they are simple, $A$ and $C$ are stable for any choice of $Z$.
\end{example}

\begin{remark}
Slope stability in the sense discussed here is certainly not the most general notion of stability in abelian categories. 
Generalizations were developed by Rudakov~\cite{rudakov97} and Joyce~\cite{joyce_conf3}. 
The main feature one retains is the existence and uniqueness of HN filtrations.
\end{remark}

We conclude this subsection with a brief discussion on $\theta$-stability as introduced by King~\cite{king_reps}.
Let $\cA$ be an abelian category and $\theta\colon K_0(\cA)\to\bR$ additive.
An object $E\in\cA$ is \textbf{$\theta$-semistable} if $\theta(E)=0$ and $\theta(A)\geq 0$ for any subobject $A\subseteq E$.\footnote{This is the sign convention in~\cite{king_reps}. More recent papers often use the opposite sign convention for $\theta$, i.e.\ define $E$ to be semistable iff $\theta(A)\leq 0$ for subobjects $A\subseteq E$.}
If, moreover, $E\neq 0$ and $0\neq A\subsetneq E$ implies $\theta(A)>0$, then $E$ is said to be \textbf{$\theta$-stable}.
The relation to stability with respect to a central charge $Z$ is the following.

\begin{lemma}
Let $\cA$ be an abelian category, $Z\colon K_0(\cA)\to\bC$ a central charge, $\phi\in(0,1]$, and $\theta\coloneqq-\operatorname{Im}(e^{-\pi \sqrt{-1}\phi}Z)$.
Then $E$ is $Z$-(semi)stable of phase $\phi$ iff $E$ is $\theta$-(semi)stable.
\end{lemma}

\begin{proof}
Let $0\neq A\in\cA$ and write $Z(A)=re^{\pi \sqrt{-1}\psi}$ for some $r>0$, $\psi\in (0,1]$.
Then
\[
\theta(A)=-r\operatorname{Im}e^{\pi \sqrt{-1}(\psi-\phi)}=-r\sin(\pi(\psi-\phi))
\]
and thus $\operatorname{sgn}(\theta(A))=-\operatorname{sgn}(\psi-\phi)$, using that $|\psi-\phi|<1$.
Hence $\theta(A)\geq 0$ iff $\phi(A)\leq \phi$, proving the equivalence.
\end{proof}

The following characterization of $\theta$-stability in terms of filtrations will be useful in Section~\ref{subsect:non-arch-quiver}.
This is similar to the discussion in~\cite[Section 3]{king_reps}.
See Section~\ref{subsect:conventions} in the introduction for definitions and conventions on $\bR$-filtrations.

\begin{lemma}
\label{lem:thetastablefilt}
Let $\cA$ be an abelian category and $\theta\colon K_0(\cA)\to\bR$. 
Then $E\in\cA$ is $\theta$-semistable iff 
\begin{equation}\label{eq:thetafilt}
\langle\theta,\cF_\bullet\rangle\coloneqq \sum_{\alpha\in\bR}\alpha\theta\left(\cF_{\alpha}/\cF_{<\alpha}\right)\leq 0
\end{equation}
for all complete, finite $\bR$-filtrations $\cF_{\bullet}$ of $E$.
Moreover, $E$ is $\theta$-stable iff $E\neq 0$ and strict inequality holds in \eqref{eq:thetafilt} for all filtrations with at least two steps.
\end{lemma}

\begin{proof}
Suppose $\theta(E)=0$ and $\cF_\bullet$ is a complete $\bR$-filtration of $E$ that jumps precisely at $\alpha_1<\ldots<\alpha_n$.
Then
\begin{equation}\label{eq:thetafilt_pf}
\langle\theta,\cF_\bullet\rangle=\sum_{i=1}^n\alpha_i\left(\theta(\cF_{\alpha_i})-\theta(\cF_{<\alpha_{i}})\right)=\sum_{i=1}^{n-1}(\alpha_i-\alpha_{i+1})\theta(\cF_{\alpha_i}).
\end{equation}
Thus, if $E$ is $\theta$-semistable, then \eqref{eq:thetafilt_pf} is non-positive, and moreover negative if $E$ is $\theta$-stable and $n\geq 2$.

For the converse note that given a proper subobject $0\neq A\subsetneq E$ we can construct a 2-step filtration $\cF_\bullet$ with
\[
\cF_\alpha\coloneqq\begin{cases}0 & \text{for }\alpha<0 \\ A & \text{for }0\leq \alpha<1 \\ E & \text{for }\alpha\geq 1\end{cases}
\]
for which $\langle\theta,\cF_\bullet\rangle=-\theta(A)$ under the assumption $\theta(E)=0$.
If $\theta(E)\neq 0$, then~\eqref{eq:thetafilt} is violated already for suitable 1-step filtrations.
\end{proof}

Much of the above discussion, in particular Lemma~\ref{lem:thetastablefilt}, generalizes to (Quillen) exact categories, if we take ``subobject'' to mean \textit{equivalence class of admissible mono}, as usual.

\subsection{Stability in triangulated categories}\label{subsec:bridgeland}

The definition of a stability condition on a triangulated category is due to Bridgeland~\cite{bridgeland07}, inspired by ideas of Douglas on the stability of D-branes in string theory~\cite{douglas_icm} and properties of semistable bundles on curves~\cite{harder_narasimhan75}, already mentioned in the previous subsection.

Let $\cC$ be an essentially small triangulated category. 
Fix a homomorphism $\mathrm{cl}\colon K_0(\cC)\to\Gamma$ to a finite rank abelian group $\Gamma$. This will play the role of a ``Chern character''. 

\begin{definition}\label{def:stability}
A \textbf{stability condition} on $(\cC,\Gamma,\mathrm{cl})$ is given by 
\begin{itemize}
\item an additive map $Z\colon \Gamma\to\bC$, the \textit{central charge},
\item full additive subcategories $\cC_\phi\subset\cC$, $\phi\in\bR$, the \textit{semistable objects of phase $\phi$},
\end{itemize}
such that:
\begin{enumerate}[(1)]
\item $\cC_{\phi+1}=\cC_\phi[1]$.
\item $\Hom(E_2,E_1)=0$ for $E_i\in\cC_{\phi_i}$, $\phi_1<\phi_2$.
\item Any $E\in\cC$ has a \textit{Harder--Narasimhan (HN) filtration}: A tower of triangles 
\[
\begin{tikzcd}[column sep=tiny]
0=E_0 \arrow{rr} & & E_1\arrow{rr}\arrow{dl} & & E_2\arrow{rr}\arrow{dl} & &\cdots \arrow{rr} &  & E_{n-1} \arrow{rr} & & E_n\cong E \arrow{dl} \\
& A_1 \arrow[dashed]{ul} & & A_2 \arrow[dashed]{ul} & & & &  & & A_n \arrow[dashed]{ul}
\end{tikzcd}
\]
with \textit{semistable factors} $0\neq A_i\in\cC_{\phi_i}$, $\phi_1>\phi_2>\ldots>\phi_n$;
\item $Z(E)\coloneqq Z(\mathrm{cl}(E))\in\bR_{>0}e^{\pi \sqrt{-1}\phi}$ for $E\in\cC_{\phi}$, $E\neq 0$.
\item The \textit{support property} (from~\cite{KS_motivicdt}): There is a norm $\|\cdot\|$ on $\Gamma\otimes_{\bZ}\bR$ and $C>0$ such that
\[
\|\gamma\|\leq C|Z(\gamma)|
\]
for any $\gamma\in\Gamma$ which is the class of a semistable object. 
\end{enumerate}
\end{definition}

Given an object $0\neq E\in\cC$ with HN filtration with semistable factors $0\neq A_i\in\cC_{\phi_i}$, $\phi_1>\phi_2>\ldots>\phi_n$, one defines 
\[
\phi_+(E)\coloneqq\phi_1,\qquad \phi_-(E)\coloneqq\phi_n, \qquad m(E)\coloneqq\sum_{i=1}^n|Z(A_i)|,
\]
where $m(E)$ is referred to as the \textit{mass} of $E$.
Given an interval $I\subset\bR$, $\cC_I$ is the full subcategory of those objects $E$ with $\phi_{\pm}(E)\in I$, together with the zero objects. These are closed under extensions. If the length, $\ell(I)$ of $I$ is strictly less than $1$, we have the following elementary two-sided bound on the mass:
\begin{equation}\label{eq:two-sided}
  \ell(I)<1 \implies \forall\, E\in\cC_I:\quad |Z(E)| \leq m(E) \leq \frac{1}{\cos\frac{\pi \ell(I)}{2}}\,|Z(E)|
\end{equation}

There is an equivalent definition of stability conditions via t-structures.

\begin{proposition}[\cite{bridgeland07}]
    Let $\sigma=(Z,(\cC_{\phi})_\phi)$ be a stability condition, then
    \begin{enumerate}[(1)]
        \item $\cA\coloneqq \cC_{(0,1]}$ is the heart of a bounded t-structure on $\cC$, in particular an abelian category. It follows that $K_0(\cA)\cong K_0(\cC)$.
        \item $Z\colon \Gamma\to\mathbb C$ is a central charge on $\cA$ satisfying the HN property (see Definition~\ref{def:slopestability}, Definition~\ref{def:HNabelian}) and the support property.
        Moreover, an object $E\in\cA$ is $Z$-semistable iff $E\in\cC_{\phi(E)}$.
    \end{enumerate}
    Conversely, given $\cA\subset\cC$ and $Z$ satisfying the above conditions, there is a unique corresponding stability condition where for $\phi\in(0,1]$, $\cC_{\phi}$ is given by slope-semistable objects $E\in\cA$ of phase $\phi(E)=\phi$.
\end{proposition}

In the special case where $\cA$ is finite-length with finitely many isomorphism classes of simple objects $S_1,\ldots,S_n$, the central charge $Z$ is uniquely determined by the numbers $z_i\coloneqq Z(S_i)$, and any choice of $z_i\in\bR_{>0}e^{\sqrt{-1}(0,\pi]}$ yields a central charge automatically satisfying the HN and support properties.

The set of stability conditions on $(\cC,\Gamma,\mathrm{cl})$ is denoted $\mathrm{Stab}(\cC,\Gamma,\mathrm{cl})$ or just $\mathrm{Stab}(\cC)$.
There is a natural topology on this space, characterized by the requirement that the maps $\sigma=(Z,(\mathcal C_\phi)_\phi)\mapsto Z\in\Hom(\Gamma,\bC)$, $\sigma\mapsto \phi_\pm(E)\in \bR$, $0\neq E\in\cC$, are continuous.
With respect to this topology, the map
\[
\mathrm{Stab}(\cC,\Gamma,\mathrm{cl})\longrightarrow \Hom(\Gamma,\bC),
\]
which sends a stability condition to its central charge and forgets which objects are semistable, is a local homeomorphism~\cite[Theorem 1.2]{bridgeland07}.
In particular, $\mathrm{Stab}(\cC)$ naturally has the structure of a complex (in fact complex affine) manifold.

The group $\mathrm{GL}^+(2,\bR)$ acts on $\mathrm{Hom}(\Gamma,\bC)$, and its universal cover $\widetilde{\mathrm{GL}}^+(2,\bR)$ naturally acts on the space of stability conditions by
\[
g\cdot (Z,(\cC_{\phi})_\phi)\coloneqq (g\cdot Z,(\cC_{g^{-1}\cdot\phi})_\phi)
\]
where $\widetilde{\mathrm{GL}}^+(2,\bR)$ acts on phases via the natural action on $\widetilde{\bC^{\times}/\bR_{>0}}\cong\bR$.

There is an identification $\mathrm{Stab}(\cC^{\mathrm{op}})\cong \overline{\mathrm{Stab}(\cC)}$ induced by the following construction:

\begin{propdef}\label{propdef:conjstab}
Given a stability condition $\sigma=(Z,(\cC_{\phi})_\phi)$ on a triangulated category $\cC$, there is a corresponding stability condition on the opposite triangulated category $\cC^{\mathrm{op}}$, the \textbf{conjugate stability condition}, given by $\overline{\sigma}=(\overline{Z},(\cC_{-\phi})_\phi)$.
(Thus, the semistable objects of phase $\phi$ in $\cC^{\mathrm{op}}$ are the semistable objects of phase $-\phi$ in $\cC$.)
\end{propdef}

\begin{proof}
Note that $\cC^{\mathrm{op}}$ has the triangulated structure with shift functor $[-1]^{\mathrm{op}}$, where $[-1]$ is the inverse of the shift functor on $\cC$.
The only axiom to check is the Harder--Narasimhan property.
Let $E\in\cC$ and let 
\[
\begin{tikzcd}[column sep=tiny]
0=E_0 \arrow{rr} & & E_1\arrow{rr}\arrow{dl} & & E_2\arrow{rr}\arrow{dl} & &\cdots \arrow{rr} &  & E_{n-1} \arrow{rr} & & E_n\cong E \arrow{dl} \\
& A_1 \arrow[dashed]{ul} & & A_2 \arrow[dashed]{ul} & & & &  & & A_n \arrow[dashed]{ul}
\end{tikzcd}
\]
be its HN filtration, i.e.\ $A_1,\ldots,A_n$ are $\sigma$-semistable of decreasing phase.
For $i=0,\ldots,n$, the composite morphism $E_{n-i}\to E$ in $\cC$ corresponds to $E\to E_{n-i}$ in $\cC^{\mathrm{op}}$ which we complete to a triangle
\[
F_i\to E\to E_{n-i}\to F_i[1].
\]
The octahedral axiom then yields triangles $F_i\to F_{i+1}\to A_{n-i}\to F_i[1]$. Moreover, $A_n,\ldots,A_1$ have decreasing phases $-\phi_n>\cdots>-\phi_1$ with respect to $\overline{\sigma}$.
The $F_i$ and the morphisms between them thus constitute the HN filtration of $E$ in $\cC^{\mathrm{op}}$. 
\end{proof}

\subsection{Mass and triangle inequality}\label{subsec:mass}

Throughout this subsection we fix a triangulated category $\cC$ with a stability condition $\sigma$.
Recall that the mass of a semistable object $E\in\cC_\phi$ is $m(E)\coloneqq |Z(E)|$, and the mass of a general object of $\cC$ is the sum of the masses of its semistable HN factors.

We first observe a well-known consequence of the support property, which is useful in inductive arguments (cf.\ Section~\ref{section:comonads}).
The first assertion is already found in~\cite[Remark~1 in Section~1.2]{KS_motivicdt}, while the second uses $m(E\oplus F)=m(E)+m(F)$.

\begin{proposition}\label{prop:massdiscrete}
The set of central charges of semistable objects is a discrete subset of $\bC$.
The set of masses of objects, $M\coloneqq \left\{m(E)\mid E\in\cC\right\}$, is a discrete submonoid of $(\mathbb R_{\geq 0},+)$.
\end{proposition}

Another important fact is the triangle inequality for exact triangles.

\begin{proposition}[Triangle inequality for mass]
\label{prop:mass}
Let $A\to B\to C\xrightarrow{\delta}A[1]$ be an exact triangle, then
\begin{enumerate}[(1)]
    \item $m(B)\leq m(A)+m(C)$,
    \item if $\phi_-(A)\geq\phi_+(C)$, then $m(B)=m(A)+m(C)$,
    \item if $\phi_+(A)<\phi_-(C)$ and $0\neq\delta\in\Ext^1(C,A)$, then $m(B)<m(A)+m(C)$.
\end{enumerate}
\end{proposition}

\begin{proof}
The first part was stated in~\cite[Section 4.5]{DHKK} and proven in~\cite[Proposition 3.3]{Ikeda21}.
The equality in (2) holds because under the assumption we can concatenate the HN filtration of $A$ with the HN filtration of $C$, possibly merging adjacent factors of equal phase, to obtain the HN filtration of $B$.
It remains to show (3).

\textbf{Step 1:} \textit{Case of semistable $A$ and $C$.}
By assumption, there is a non-zero morphism $\delta\colon C\to A[1]$, so $\phi(C)\leq \phi(A[1])=\phi(A)+1$.
Suppose first that $\phi(C)<\phi(A)+1$. 
We also have $\phi(A)=\phi_+(A)<\phi_-(C)=\phi(C)$, thus $Z(A)$ and $Z(C)$ are contained in some conical sector of angle less than $\pi$, and $A,B,C$ are contained in the heart of a bounded t-structure on $\cC$.
Let $B_1,\ldots,B_n$ be the semistable factors of $B$.
Then $\phi(A)\leq \phi(B_i)\leq \phi(C)$, $i=1,\ldots,n$, which implies that the piecewise linear path from $0\in\mathbb C$ to $Z(B)$ formed by the vectors $Z(B_i)$ is no longer than the path formed by the vectors $Z(A),Z(C)$, see Figure~\ref{fig:extHNfilt}.
To show strict inequality, we need to exclude the possibility that $B$ has a 2-step HN filtration with semistable factors $B_1$ and $B_2$ with $Z(B_1)=Z(C)$ and $Z(B_2)=Z(A)$.
For contradiction, suppose this is the case.
If $A\cap B_1=0$, then 
\[
Z(B/(A+ B_1))=Z(B)-Z(A)-Z(B_1)=0,
\]
thus $B=A\oplus B_1$, contradicting $\delta\neq 0$.
If $A\cap B_1\neq 0$, then the mono $A/(A\cap B_1)\hookrightarrow B/B_1$, together with semistability of $A$ and $B/B_1\cong B_2$ imply
\[
\phi(A)\leq \phi(A/(A\cap B_1))\leq\phi(B/B_1)=\phi(A)
\]
thus $\phi(A\cap B_1)=\phi(A)$.
Similarly, using $B_1/(A\cap B_1)\hookrightarrow B/A=C$:
\[
\phi(B_1)\leq \phi(B_1/(A\cap B_1))\leq \phi(B/A)=\phi(B_1)
\]
thus $\phi(A\cap B_1)=\phi(B_1)\neq \phi(A)$, a contradiction.

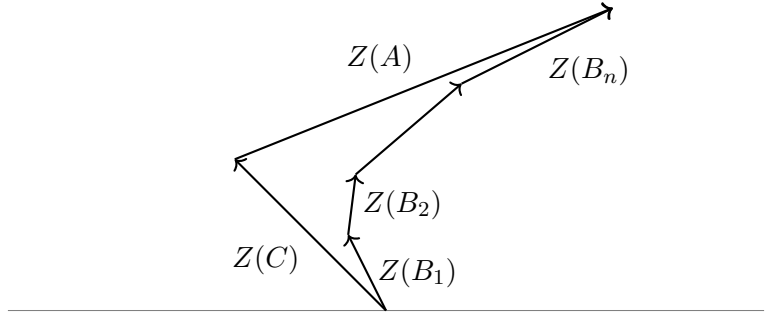
\begin{figure}
    \centering
    \begin{tikzpicture}
        \draw[gray] (-5,0) -- (5,0);
        \draw[thick,->] (0,0) -- (-2,2) node[midway,below left]{$Z(C)$};
        \draw[thick,->] (-2,2) -- (3,4) node[midway,above left]{$Z(A)$};
        \draw[thick,->] (0,0) -- (-.5,1) node[midway,right]{$Z(B_1)$};
        \draw[thick,->] (-.5,1) -- (-.4,1.8) node[midway,right]{$Z(B_2)$};
        \draw[thick,->] (-.4,1.8) -- (1,3);
        \draw[thick,->] (1,3) -- (3,4) node[midway,below right]{$Z(B_n)$};
    \end{tikzpicture}
    \caption{Showing $m(B)=m(B_1)+\ldots+m(B_n)< m(A)+m(C)$.}
    \label{fig:extHNfilt}
\end{figure}

Suppose now that $\phi(C)=\phi(A)+1\eqqcolon \phi$. 
Then $\delta\colon C\to A[1]$ is a morphism in the abelian category $\cC_\phi$ of semistable objects of phase $\phi$.
Since $\delta\neq 0$, $0\neq \operatorname{Im}(\delta)\subseteq A[1]$, thus $m(\mathrm{Coker}(\delta))< m(A)$.
Similarly, $\mathrm{Ker}(\delta)\subsetneq C$, thus $m(\mathrm{Ker}(\delta))<m(C)$.
But since $B\cong \mathrm{Cocone}(\delta)$ is an extension of $\mathrm{Ker}(\delta)$ and $\mathrm{Coker}(\delta)[-1]$ we conclude that
\[
m(B)\leq m(\mathrm{Coker}(\delta))+m(\mathrm{Ker}(\delta))<m(A)+m(C).
\]

\textbf{Step 2:} \textit{Case of semistable $C$, general $A$.}
Let 
\[
\begin{tikzcd}[column sep=tiny]
0=A_0 \arrow{rr} & & A_1\arrow{rr}\arrow{dl} & & A_2\arrow{rr}\arrow{dl} & &\cdots \arrow{rr} &  & A_{n-1} \arrow{rr} & & A_n\cong A \arrow{dl} \\
& S_1 \arrow[dashed]{ul} & & S_2 \arrow[dashed]{ul} & & & &  & & S_n \arrow[dashed]{ul}
\end{tikzcd}
\]
be the HN filtration of $A$. 
We prove the strict mass inequality by induction on $n$, the base case $n=1$ having been completed in the previous step.

Let $\varepsilon\colon C\xrightarrow{\delta} A[1]\to S_n[1]$ be the composition.
If $\varepsilon\neq 0$, then by part (1) of the proposition and the previous step:
\[
\begin{aligned}
m(B) &\leq m(A_{n-1})+m(\mathrm{Cocone}(C\xrightarrow{\varepsilon} S_n[1])) \\
&<m(A_{n-1})+m(S_n)+m(C)=m(A)+m(C).
\end{aligned}
\]
On the other hand, if $\varepsilon = 0$, then we can switch the order of $S_n$ and $C$ in the filtration of $B$ to obtain, using the octahedral axiom, a filtration of the following form:
\[
\begin{tikzcd}[column sep=tiny]
0=A_0 \arrow{rr} & & A_1\arrow{rr}\arrow{dl} & &\cdots \arrow{rr} &  & A_{n-2} \arrow{rr} & & A_{n-1} \arrow{rr}\arrow{dl} & & Y  \arrow{rr}\arrow{dl} & & B \arrow{dl} \\
& S_1 \arrow[dashed]{ul} & & & &  & & S_{n-1} \arrow[dashed]{ul} & & C \arrow[dashed]{ul} & & S_n \arrow[dashed]{ul}
\end{tikzcd}
\]
The morphism $C\to A_{n-1}[1]$ is necessarily non-zero, since its composition with $A_{n-1}[1]\to A[1]$ is $\delta$.
Thus, by induction, 
\[
m(Y)<m(C)+m(A_{n-1})=m(C)+m(S_1)+\ldots+m(S_{n-1})
\]
hence $m(B)\leq m(Y)+m(S_n)<m(C)+m(A)$.

\textbf{Step 3:} \textit{Case of general $A$ and $C$.}
Passing to $\cC^{\mathrm{op}}$, we obtain from the previous step also the case where $A$ is semistable and $C$ is general.
We can then repeat the argument as in Step 2, but with $C$ not necessarily semistable, and with appeals to Step~1 replaced by the dual of Step~2.
\end{proof}

\subsection{Effective curves and central charge}\label{subsec:effective}

Throughout this subsection, $\mathbf k$ is a field and $\cC$ is a $\mathbf k$-linear triangulated category. 

\begin{definition}\label{def:curveinob}
A \textbf{curve in $\mathrm{Ob}(\cC)$} is an exact $\mathbf k$-linear functor $F\colon \mathrm{Perf}(X)\to \cC$, where $X$ is a connected smooth projective curve over $\mathbf k$, and satisfying the following representability conditions:
\begin{enumerate}[(1)]
    \item There is a functor $R_F\colon\cC\to D^-_{\mathrm{coh}}(X)$ such that
    \[
    \Hom_{\cC}(F(E),A)\cong \Hom_{D^-_{\mathrm{coh}}(X)}(E,R_F(A))
    \]
    for $E\in \mathrm{Perf}(X)$, $A\in\cC$.
    \item There is a functor $K_F\colon\cC^{\mathrm{op}}\to D^+_{\mathrm{coh}}(X)$ such that
    \[
    \Hom_{\cC}(A,F(E))\cong H^0R\Gamma(X;E\otimes^LK_F(A))
    \]
    for $E\in \mathrm{Perf}(X)$, $A\in\cC$.
\end{enumerate}
\end{definition}

In particular, the representability assumptions are satisfied if $F$ has both a right and a left adjoint. This is a reasonable assumption if $\cC$ is the homotopy category of a smooth and proper dg category over $\mathbf k$.

Suppose that $\mathrm{cl}\colon K_0(\cC)\to\Gamma$ and a stability condition $\sigma=(Z,(\cC_\phi)_\phi)$ on $(\cC,\Gamma,\mathrm{cl})$ have been fixed.
Given $\phi\in\bR$ and $\gamma\in\Gamma$ such that $Z(\gamma)\in \bR_{>0}e^{\pi \sqrt{-1}\phi}$, we let $\cC_{\phi,\gamma}$ be the full subcategory of $\cC_\phi$ of those objects $E$ which have class $\mathrm{cl}(E)=\gamma$.
We can define the group of 1-cycles (modulo algebraic equivalence) and its submonoid of effective cycles on $\mathrm{Ob}(\cC_{\phi,\gamma})$.
Effective cycles are defined as follows:

\begin{definition}
A \textbf{curve in $\mathrm{Ob}(\cC_{\phi,\gamma})$} is a curve $F\colon \mathrm{Perf}(X)\to \cC$ in $\mathrm{Ob}(\cC)$ such that if $E$ is any coherent sheaf on $X$ with 0-dimensional support and $n=\dim_{\mathbf k} \Gamma(X,E)$, then $F(E)\in \mathrm{Ob}(\cC_{\phi,n\gamma})$.
\end{definition}

For any curve in $\mathrm{Ob}(\cC_{\phi,\gamma})$ we define its \textit{degree} with respect to $Z$ to be
\[
-\operatorname{Im}\left(e^{-\pi\sqrt{-1}\phi} Z(F(\cO_X))\right)\in \mathbb R.
\]

\begin{theorem}\label{thm:ample} 
The degree of any curve in $\mathrm{Ob}(\cC_{\phi,\gamma})$ is non-negative. 
\end{theorem}

In the context of derived categories of coherent sheaves, this is part of \cite[Positivity Lemma 3.3]{bayer_macri_projectivity}.
The proof of Theorem~\ref{thm:ample} will be completed after the following lemma.

\begin{lemma}\label{lem:pointwise_to_global}
Let $X$ be a smooth projective curve over $\mathbf k$. 
\begin{enumerate}[(1)]
    \item Suppose $E\in D^+_\mathrm{coh}(X)$ with $H^{\leq 0}(\cO_p\otimes^L E)=0$ for all closed points $p\in X$, then $E\in D^{\geq 1}_\mathrm{coh}(X)$, in particular $H^{\leq 0}R\Gamma(X;E)=0$.
    \item Suppose $E\in D^-_\mathrm{coh}(X)$ with $\mathrm{Ext}^{\leq 0}(\cO_p,E)=0$ for all closed points $p\in X$, then $E\in D^{\geq 0}_\mathrm{coh}(X)$, in particular $\mathrm{Ext}^{\leq -1}(\cO_X,E)=0$.
\end{enumerate}
\end{lemma}

\begin{proof}
(1) We claim $E\in D^{\geq 1}_\mathrm{coh}(X)$. For contradiction, suppose that $n\leq 0$ is minimal such that $H^nE\neq 0$.
Choose a closed point $p\in X$ in the support of $H^nE$, then
\[
H^n(\cO_p\otimes^L E)\cong H^nE\otimes_{\cO_X}k(p)
\]
which vanishes by assumption, contradicting Nakayama's lemma.

(2) By Serre duality, and using $\cO_p\otimes\omega_X\cong\cO_p$, the hypothesis is equivalent to $\mathrm{Ext}^{\geq 1}(E,\cO_p)=0$ for all closed points $p\in X$.
Because $\mathrm{Coh}(X)$ is hereditary, each $(H^nE)[-n]$, $n\in\bZ$, is a direct summand of $E$.
Moreover, $\Hom(H^nE,\cO_p)$ is a direct summand of $\mathrm{Ext}^{-n}(E,\cO_p)$.
Thus, if $n<0$ then the assumption implies $\Hom(H^nE,\cO_p)=0$ for all closed points $p\in X$, thus $H^nE=0$ by coherence of $H^nE$.
This shows $E\in D^{\geq 0}_\mathrm{coh}(X)$.
\end{proof}

\begin{proof}[Proof of Theorem~\ref{thm:ample}]
We will show the stronger statement that $F(\cO_X)\in\cC_{[\phi-1,\phi]}$.
This is equivalent to the combination of:
\begin{enumerate}[(1)]
    \item $\mathrm{Hom}(A,F(\cO_X))=0$ for $A\in\cC_{(\phi,\infty)}$ and
    \item $\mathrm{Ext}^{-1}(F(\cO_X),A)=0$ for $A\in\cC_{(-\infty,\phi)}$.
\end{enumerate}

Let us first show (1). Let $K_F$ be as in Definition~\ref{def:curveinob}.
Suppose $A\in\cC_{(\phi,\infty)}$, then since $F(\cO_p)\in\cC_{\phi}$ we have for any $p\in X$:
\[
H^{\leq 0}R\Gamma(X;\cO_p\otimes^LK_F(A))=\mathrm{Ext}^{\leq 0}(A,F(\cO_p))=0.
\]
By the first part of Lemma~\ref{lem:pointwise_to_global}, 
\[
\Hom(A,F(\cO_X))\cong H^{0}R\Gamma(X;K_F(A))=0.
\]

For the proof of (2) we use a similar argument.
Let $R_F$ be as in Definition~\ref{def:curveinob}.
Suppose $A\in\cC_{(-\infty,\phi)}$, then for any $p\in X$:
\[
\mathrm{Ext}^{\leq 0}(\cO_p,R_F(A))=\mathrm{Ext}^{\leq 0}(F(\cO_p),A)=0.
\]
By the second part of Lemma~\ref{lem:pointwise_to_global}, 
\[
\Ext^{-1}(F(\cO_X),A)\cong \Ext^{-1}(\cO_X,R_F(A))=0.
\]
\end{proof}

Assume that the finite rank abelian group $\Gamma$ in Definition~\ref{def:stability} is free\footnote{One can always replace the finitely generated abelian group $\Gamma$ by $\Gamma/\Gamma_{\mathrm{torsion}}$.} of rank $r\ge 0$ and the homomorphism $\mathrm{cl}\colon K_0(\cC)\to \Gamma\cong \bZ^r$ is given by the $r$-tuple of Euler characteristics
\[
\mathrm{cl}(A)=(\chi( G_i(A)))_{i=1,\dots,r}
\]
where $G_1,\dots, G_r$ are functors from $\cC$ to $\mathrm{Perf}(\mathbf k)$. Then for each object $A\in \cC$ we have an $r$-tuple of determinant lines (1-dimensional spaces over $\mathbf k$)
\[
\cL_i(A)\coloneqq\operatorname{Det} G_i(A),\,\,i=1,\dots r\,.
\]
So, in a sense, we get $r$ line bundles on the ``stack of objects'' $\mathrm{Ob}(\cC)$. For example, in the case of representations of quivers (see Section~\ref{sec:quivers_arch}) the group $\Gamma$ is freely generated by the set $Q_0\cong \{1,\dots,r\}$ of vertices and the corresponding functor $G_i$ associates with each representation $E$ its component $E_i$ at vertex $i$. Then we obtain actual line bundles on the derived Artin stack $\mathrm{Ob}(\cC)$. In general, for any curve $F\colon\mathrm{Perf}(X)\to \cC$ we obtain ordinary line bundles on $X$. The fiber of the $i$-th bundle $F^*\cL_i$ at point $p\in X$ is defined as the determinant line $\operatorname{Det}\big( (G_i\circ F)(\cO_p)\big)$. Theorem \ref{thm:ample} means that for any curve in $\mathrm{Ob}(\cC_{\phi,\gamma})$ we have
\[
-\sum_{i=1}^r \operatorname{Im}\left(e^{-\pi\sqrt{-1}\phi } z_i\right)\cdot \deg F^*\cL_i\ge 0.
\]
Here $z_1,\dots ,z_r\in \bC$ are parameters of the central charge:
\[
Z\colon (d_1,\dots,d_r)\in \bZ^r\cong \Gamma\mapsto \sum_{i=1}^r d_i z_i \in \bC.
\]
The conclusion is that we have a real linear combination of first Chern classes of line bundles on $\mathrm{Ob}(\cC_{\phi,\gamma})$ which is numerically effective, i.e.\ informally this stack (or better, the coarse moduli space parametrizing polystable objects) ``wants'' to be a quasi-projective variety, which is the essence of Mumford's theory of stability. The main theme of our paper is to upgrade this purely algebraic non-negativity to some kind of  K\"ahler metric (or at least non-negative $(1,1)$-current) on the moduli of polystable objects, whose integral over effective curves is the degree determined by the central charge.

\

The group $\bG_{m,\mathbf{k}}$ of rescalings of the object $A\in \cC$ acts on the line $\cL_i(A)$ by weight $\chi(G_i(A))$. Therefore, for any object $A\in \cC_\phi$, on the formal tensor product of ``real tensor powers''
\[ 
\bigotimes_{i=1}^r \cL_i(A)^{\otimes -\operatorname{Im}\left(e^{-\pi\sqrt{-1}\phi } z_i\right)}
\]
the group of rescalings acts trivially.

To make rigorous sense of the tensor product of real tensor powers, suppose that the ground field $\mathbf k$ is $\bR$, $\bC$, or a non-Archimedean field, so we have a norm homomorphism
\[|\cdot|:\mathbf k^\times\to \bR_{>0}\,.\]
For a one-dimensional $\mathbf k$-vector space $L$, denote by $|L|$ the
oriented real line obtained from the $\mathbf k^\times$-torsor
$L^\times$ by extension of structure group along $|\cdot|$.  Thus, with
$A\in\cC$ we associate the oriented real lines $|\cL_i(A)|$.

For an oriented real line $V$ and $a\in\bR$, we denote by
$V^{\otimes a}$ the oriented real line obtained from the
$\bR_{>0}^{\times}$-torsor $V_{>0}$ of positive nonzero vectors by
extension of structure group along the homomorphism
$\lambda\mapsto\lambda^a$. With this convention,   the tensor product 
\[
|\cL_{Z,\phi}(A)|\coloneqq\bigotimes_{i=1}^r |\cL_i(A)|^{\otimes -\operatorname{Im}\left(e^{-\pi\sqrt{-1}\phi }z_i \right)}
\]
is well defined. The rescaling group acts \textit{trivially} on this real line. Moreover, one can easily see that the whole group $\mathrm{Aut}(A)$ acts trivially on the oriented real line  $|\cL_{Z,\phi}(A)|$, and that this line depends only on the isomorphism class of the semisimplification of $A$ considered as an object of the finite-length category $\cC_\phi$. Therefore, we obtain, at least set-theoretically, an oriented real line bundle on the space of isomorphism classes of polystable objects. 

Presumably, in reasonable situations, this oriented real line bundle is independent of the choice of functors $G_i$ which categorify the components of the map $\mathrm{cl}\coloneqq K_0(\cC)\to \bZ^r$, and even defined when there is no natural categorification at all. The central charge should be replaced from this perspective by \textit{two} oriented real line bundles on $\mathrm{Ob}(\cC)$ with fibers  
\[
|\cL_{\mathrm{Re}}|(A)\coloneqq\bigotimes_{i=1}^r |\cL_i(A)|^{\otimes \operatorname{Re}z_i },\quad |\cL_{\mathrm{Im}}|(A)\coloneqq\bigotimes_{i=1}^r |\cL_i(A)|^{\otimes \operatorname{Im}z_i }
\]
which implies that 
\[ 
|\cL_{Z,\phi}|=|\cL_{\mathrm{Re}}|^{\sin \pi\phi}\otimes |\cL_{\mathrm{Im}}|^{-\cos \pi\phi}.
\]

\subsection{Ascending-phase filtrations}\label{subsec:antihn}

The constructions in this subsection will only be used in Section~\ref{section:spectral}, but are also of interest in their own right.
Fix throughout a triangulated category $\cC$ with stability condition $\sigma=(Z,(\cC_\phi)_\phi)$ and heart $\cA\coloneqq \cC_{(0,1]}$.

\begin{definition}\label{def:antiHN}
    An \textbf{ascending-phase filtration} of an object $E\in\cA$ is a filtration 
    \[
    0=E_0\subset E_1\subset E_2\subset\cdots\subset E_n=E
    \]
    such that $0\neq A_i\coloneqq E_i/E_{i-1}$ are semistable for $i=1,\ldots,n$ with strictly ascending phases $\phi(A_1)<\phi(A_2)<\cdots<\phi(A_n)$. We write $E_{\leq\phi}\coloneqq E_k$, where $k=\max\left(\{0\}\cup\{i\mid\phi(A_i)\leq\phi\}\right)$ and regard the collection $(E_{\leq\phi})_\phi$ as an ascending $[0,1]$-filtration of $E$.

    If $E,F\in\cA$ are equipped with ascending-phase filtrations $(E_{\leq\phi})_\phi$, $(F_{\leq\phi})_\phi$, then a morphism $f\colon E\to F$ is said to be \textit{compatible} if it is filtration preserving, i.e.\ $f(E_{\leq\phi})\subseteq F_{\leq\phi}$.
\end{definition}

The following is~\cite[Lemma 4.13]{HKS}, where ascending-phase filtrations were called \textit{anti-HN filtrations}.

\begin{lemma}\label{lem:strictness}
    Let $f\colon E\to F$ be a compatible morphism between objects equipped with ascending-phase filtrations.
    Then $f$ is strict, i.e.\ $f(E_{\leq\phi})=F_{\leq\phi}\cap\operatorname{Im}(f)$.
\end{lemma}

\begin{proposition}
    Fix $n\geq 1$ and consider the category whose objects are objects $E\in\cA$ equipped with $n$  ascending-phase filtrations $\left(E^{i}_{\leq\phi}\right)_\phi$ of $E$, $i=1,\ldots,n$, and whose morphisms are those morphisms $E\to F$ which send the $i$-th filtration on $E$ to the $i$-th filtration on $F$. Then this category is abelian and finite-length.
\end{proposition}

\begin{proof}
    The proof that the category is abelian is identical to the argument in~\cite[Proposition 4.12]{HKS}. 
    To see that it is of finite length, note that by the support property each category $\cC_\phi$ is of finite lengthh, and any short exact sequence in the category of objects with ascending-phase filtrations induces a short exact sequence of each of the semistable factors.
\end{proof}

The case $n=2$, i.e.\ objects with a pair of ascending-phase filtrations, plays a special role.

\begin{definition}
    The \textbf{vertex category}\footnote{Following catchy terminology of Bousseau--Bridgeland--Giovenzana, one might also call this the \textit{boomerang category}.}, $\mathcal V(\sigma)$, of a stability condition $\sigma$ with heart $\cA$, is the finite-length abelian category of objects $E\in\cA$, together with a pair of ascending-phase filtrations $\left(E^\pm_{\leq\phi}\right)_\phi$.
\end{definition}

Given an object of $\mathcal V(\sigma)$ we can draw a diagram as in Figure~\ref{vertexdiagram}, cf.~\cite[Figure 4.9]{HKS}, which we interpret as a vertex of a spectral network with coefficients in $\cC$, see Section~\ref{section:spectral}.

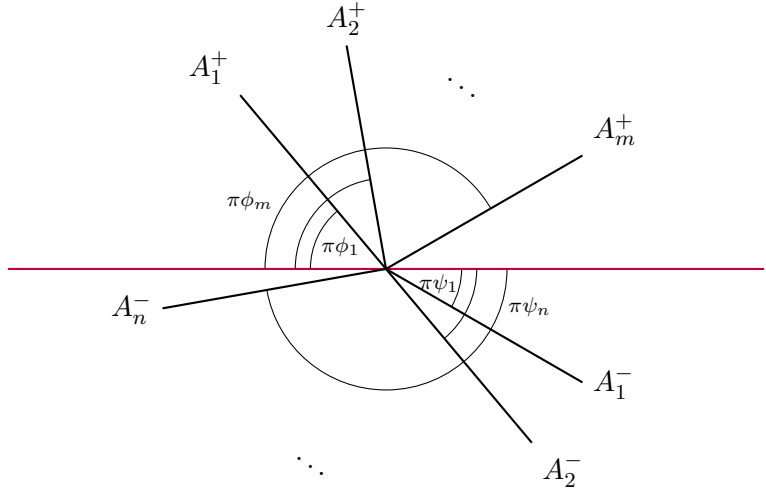
\begin{figure}[ht]
\centering
\begin{tikzpicture}
\draw[purple,thick] (-5,0) -- (5,0);
\draw[thick] (0,0) -- (130:3) node[anchor=south east] {$A_1^+$};
\draw (-1,0) arc (180:130:1);
\node at (-.6,.3) {$\scriptstyle{\pi\phi_1}$};
\draw[thick] (0,0) -- (100:3) node[anchor=south] {$A_2^+$};
\draw (-1.2,0) arc (180:100:1.2);
\draw[thick] (0,0) -- (30:3) node[anchor=south west] {$A_m^+$};
\draw (-1.6,0) arc (180:30:1.6);
\node at (-1.8,.9) {$\scriptstyle{\pi\phi_m}$};
\node at (1,2.5) {$\ddots$};
\draw[thick] (0,0) -- (-30:3) node[anchor=west] {$A_1^-$};
\draw (1,0) arc (0:-30:1);
\node at (.7,-.2) {$\scriptstyle{\pi\psi_1}$};
\draw[thick] (0,0) -- (-50:3) node[anchor=north west] {$A_2^-$};
\draw (1.2,0) arc (0:-50:1.2);
\draw[thick] (0,0) -- (-170:3) node[anchor=east] {$A_n^-$};
\draw (1.6,0) arc (0:-170:1.6);
\node at (1.9,-.5) {$\scriptstyle{\pi\psi_n}$};
\node at (-1,-2.5) {$\ddots$};
\end{tikzpicture}
\caption{Diagram of a pair of ascending-phase filtrations with semistable subquotients $A^+_1,\ldots,A_m^+$ of the first filtration and $A^-_1,\ldots,A_n^-$ of the second filtration. The phases are $\phi_i=\phi(A^+_i)$ and $\psi_i=\phi(A^-_i)$.}
\label{vertexdiagram}
\end{figure}

\begin{theorem}\label{thm:vcatgl2}
    Assume that $\cC$ admits an $\infty$-categorical enhancement, i.e.\ $\cC$ is equivalent to the homotopy category, $\mathrm h\cC'$, of a stable $\infty$-category $\cC'$ in the sense of~\cite{LurieHA}. 
    Then the construction $\sigma\mapsto\cV(\sigma)$ is $\widetilde{\mathrm{GL}}^+(2,\bR)$-invariant, i.e.\ for $g\in\widetilde{\mathrm{GL}}^+(2,\bR)$ there is an equivalence $\cV(\sigma)\to\cV(g\cdot \sigma)$.
\end{theorem}

\begin{proof}
The group $\widetilde{\mathrm{GL}}^+(2,\bR)$ acts on the set $\bR$ of phases. If an element $g\in\widetilde{\mathrm{GL}}^+(2,\bR)$ fixes the phase $0\in \bR$, and thus $\bZ\subset\bR$, then, since $g$ fixes the heart and the order of the phases, it is clear that there is a canonical equivalence $\cV(\sigma)\simeq\cV(g\cdot \sigma)$.
Thus, we can restrict to the action of the subgroup $\bR\subset \widetilde{\mathrm{GL}}^+(2,\bR)$ which covers $U(1)\subset \mathrm{GL}^+(2,\bR)$.
In fact, it suffices to consider the following situation: Let $t\in (0,1)$, then we want an equivalence, $\Phi$, from $\cV(\sigma)$, the category of objects in $\cC_{(0,1]}$ with a pair of ascending-phase filtrations, to $\cV_t(\sigma)$, the category of objects in $\cC_{(t,t+1]}$ with a pair of ascending-phase filtrations (indexed by $(t,t+1]$).
Intuitively, we are tilting the horizontal line in~Figure~\ref{vertexdiagram} by an angle $\pi t$ in the clockwise direction.
For other values of $t\in \bR$, the equivalence follows by transitivity. 

We use the stable $\infty$-categorical enhancement $\cC'$. 
The heart $\cA=\cC_{(0,1]}$ embeds in $\cC'$, and for $X,Y\in\cA$ we have 
\[ 
\pi_n\operatorname{Map}_{\cC'}(X,Y) \cong \Hom_{\cC}(X[n],Y)=0, \qquad n>0. 
\] 
Thus the full $\infty$-subcategory of $\cC'$ spanned by the objects of $\cA$ has discrete mapping spaces and is canonically equivalent to the nerve of the ordinary category $\cA$. 
The same applies to finite diagram categories of filtrations in $\cA$. Hence the vertex category may be regarded as an ordinary category of coherent diagrams in $\cC'$. 

Let 
\[ 
(E,(E^\pm_{\leq\phi})_\phi)\in\cV(\sigma) 
\] 
and define 
\begin{equation}\label{eq:Phi-total-cofiber} 
\Phi(E)\coloneqq \operatorname{cofib}\left(E^-_{\leq t}\oplus E^+_{\leq t}\longrightarrow E\right), 
\end{equation} 
where the two components of the displayed morphism are the inclusions. 
Since $E\mapsto E^-_{\leq t}$ and $E\mapsto E^+_{\leq t}$ are functorial and cofibers are functorial in a stable $\infty$-category, this defines a functor on the underlying coherent diagrams. 

The canonical total-cofiber equivalence for a composable pair gives natural equivalences 
\[ 
\Phi(E) \simeq \operatorname{cofib}\left( E^-_{\leq t}\longrightarrow E/E^+_{\leq t} \right) 
\] 
and 
\[ \Phi(E) \simeq \operatorname{cofib}\left( E^+_{\leq t}\longrightarrow E/E^-_{\leq t} \right). 
\] 
Equivalently, there are natural exact triangles \begin{equation}\label{eq:tiltingtri1} 
E^-_{\leq t} \longrightarrow E/E^+_{\leq t} \longrightarrow \Phi(E) \longrightarrow E^-_{\leq t}[1] 
\end{equation} 
and 
\begin{equation}\label{eq:tiltingtri2} 
E^+_{\leq t} \longrightarrow E/E^-_{\leq t} \longrightarrow \Phi(E) \longrightarrow E^+_{\leq t}[1]. 
\end{equation} 

Since 
\[ 
E/E^+_{\leq t}\in\cC_{(t,1]}, \qquad E^-_{\leq t}[1]\in\cC_{(1,t+1]}, 
\] 
the first triangle shows that $\Phi(E)\in\cC_{(t,t+1]}$. 
The ascending-phase filtration on $E/E^+_{\leq t}$ induced by the first filtration on $E$, followed by the shifted ascending-phase filtration on $E^-_{\leq t}$ induced by the second filtration on $E$, gives the first filtration, $\Phi(E)^+$, on $\Phi(E)$. 

Similarly, the second triangle combines the filtration on $E/E^-_{\leq t}$ induced by the second filtration on $E$ with the shifted filtration on $E^+_{\leq t}$ induced by the first filtration on $E$, and gives the second filtration, $\Phi(E)^-$, on $\Phi(E)$. 

All these constructions are natural in $E$. Hence a morphism in $\cV(\sigma)$ induces a morphism in $\cV_t(\sigma)$ preserving both filtrations, and $\Phi$ is a well-defined functor. 

Conversely, let 
\[ 
(F,(F^\pm_{\leq\phi})_\phi)\in\cV_t(\sigma). 
\] 
Cut both filtrations at phase $1$, and define 
\[ 
\Psi(F)\coloneqq \operatorname{fib}\left( F^-_{\leq1}\oplus F^+_{\leq1}\longrightarrow F \right). 
\] 
The total-fiber identities give natural exact triangles \[ 
\bigl(F/F^+_{\leq 1}\bigr)[-1] \longrightarrow \Psi(F) \longrightarrow F^-_{\leq 1} \longrightarrow F/F^+_{\leq 1} 
\] 
and 
\[ \bigl(F/F^-_{\leq 1}\bigr)[-1] \longrightarrow \Psi(F) \longrightarrow F^+_{\leq 1} \longrightarrow F/F^-_{\leq 1}. 
\]
Since 
\[ 
F^\pm_{\leq1}\in\cC_{(t,1]},\quad (F/F^\pm_{\leq1})[-1]\in\cC_{(0,t]}, 
\] 
after the evident reindexing of phases these triangles define two ascending-phase filtrations on $\Psi(F)$, and $\Psi(F)\in\cC_{(0,1]}$.
We obtain a well-defined functor 
\[ 
\Psi\colon\cV_t(\sigma)\longrightarrow\cV(\sigma). 
\] 

Lastly, we check that $\Phi$ and $\Psi$ are mutually inverse. We use the following standard identity in a stable $\infty$-category. Given morphisms $A\to E$ and $B\to E$, set \[ T\coloneqq\operatorname{cofib}(A\oplus B\to E). \] Then the canonical maps $E/A\to T$ and $E/B\to T$ fit into a natural fiber sequence \[ E\longrightarrow E/A\oplus E/B\longrightarrow T, \] where one of the two components of the first map is taken with the usual sign convention. Equivalently, \[ E\simeq \operatorname{fib}\bigl(E/A\oplus E/B\to T\bigr). \] This follows from the standard $3\times3$ bicartesian diagram 
associated with the pair of morphisms $A,B\to E$. 

For $E\in\cV(\sigma)$, the definitions of the filtrations on $\Phi(E)$ give \[ \Phi(E)^-_{\leq1}\simeq E/E^-_{\leq t}, \qquad \Phi(E)^+_{\leq1}\simeq E/E^+_{\leq t}. \] Applying the preceding identity with $A=E^-_{\leq t}$ and $B=E^+_{\leq t}$ therefore gives a natural equivalence \[ \Psi\Phi(E)\simeq E. \] Under this equivalence, the two filtrations reconstructed by $\Psi$ identify with the original two filtrations on $E$, as follows from the same bicartesian diagram. 

Dually, for $F\in\cV_t(\sigma)$, the total-fiber version of the same identity gives a natural equivalence \[ \Phi\Psi(F)\simeq F, \] again compatible with both filtrations. Hence $\Phi$ and $\Psi$ are mutually inverse equivalences.
\end{proof}

\begin{remark}
Call two slicings of a triangulated category \textit{equivalent} if they differ by a reparameterization (orientation preserving homeomorphism) of the set $\bR$ of phases.
The proof of Theorem~\ref{thm:vcatgl2} shows, more generally, that $\cV(\sigma)$ only depends on the equivalence class of the slicing.
In the terminology of \textit{walls of the first/second kind} of~\cite{KS_motivicdt}, we can say that to any stability condition $\sigma$ we can attach two abelian categories:
\begin{enumerate}[(1)]
    \item The vertex category $\cV(\sigma)$, which is constant along a path of stability conditions not crossing any walls of the first kind.
    \item The heart $\cA=\cC_{(0,1]}$, which is constant along a path of stability conditions not crossing any walls of the second kind.
\end{enumerate}
\end{remark}

\begin{example}
    Fix a coefficient field $\mathbf k$ and let $\cA=\mathbf{Rep}(\bullet\to\bullet)$ be the category of representations of the $A_2$ quiver over  $\mathbf k$ (see also Section~\ref{subsec:quivers}). This category contains three indecomposable objects $A,B,C$, up to isomorphism, which fit into a non-split short exact sequence $0\to A\to B\to C\to 0$.
    We consider the triangulated category $\cC\coloneqq D^b(\cA)$ with any stability condition $\sigma$ with heart $\cA$.

\begin{itemize}
    \item     Suppose first that $\phi(A)>\phi(C)$, i.e.\ $B$ is not semistable. Then $B$ does not admit an ascending-phase filtration and the category $\cV(\sigma)$ is semisimple with two simples,  up to isomorphism.
    \item In the case $\phi(A)=\phi(C)$ there are no non-trivial ascending-phase filtrations, so $\cV(\sigma)=\cA$.
    \item   The case $\phi(A)<\phi(C)$ is more interesting. In this case, $B$ is stable and admits two distinct ascending-phase filtrations: the trivial one and $0\subset A\subset B$.
    It can be deduced from the discussion in~\cite[Section 6]{HKS} that $\cV(\sigma)$ is the category of coherent sheaves with 0-dimensional support (i.e.\ torsion sheaves) on a stacky $\bA^1$ with two orbifold points with stabilizers $\bZ/3$ and $\bZ/2$. This is the stacky $\mu_6$-quotient of a plane elliptic curve $[\{y^2=x^3+1\}/\mu_6]$. (Assume $\operatorname{char}\mathbf k\neq 2,3$.)
\end{itemize}
    It would be interesting to generalize this analysis to $A_n$-type quivers with larger $n$.
    This could shed some light on the complicated chamber structure of $\mathrm{Stab}(D^b(A_n))$.
\end{example}

\begin{example}
    Let $E=\bC^{\times}/q^{\bZ }$ be an elliptic curve over $\bC$ and $\cC=D^b(\mathrm{Coh}(E))$ with arbitrary numerical stability condition $\sigma$. (There is a single orbit under the action of $\widetilde{\mathrm{GL}}^+(2,\bR)$.)
    Then $\cV(\sigma)$ is equivalent to the category of holonomic modules over the algebra $$\left.\bC\langle X^{\pm 1},Y^{\pm 1}\rangle\middle/YX=qXY\right..$$
    This can be deduced, with some work, from the results of~\cite{sauloy09}, see also~\cite{RSZ}. 

    Bousseau--Bridgeland--Giovenzana have announced upcoming work in which they identify certain moduli spaces of objects in $\cV(\sigma)$ with complements of anticanonical divisors in del Pezzo surfaces.
\end{example}

\section{Heuristics from Fukaya categories}
\label{section:heuristics}

\subsection{Stability conditions from SCFT}\label{subsec:scft}

Recall that the very notion of stability condition by Bridgeland~\cite{bridgeland07} came as a formalization of ideas of Douglas~\cite{douglas_icm} who tried to reconcile the notion of $D$-brane with the $A_\infty$-category viewpoint on mirror symmetry proposed in~\cite{kontsevich_hms}.

The modern perspective is the following:
\begin{itemize}
    \item There is a well-defined notion of a $\cN=(2,2)$ supersymmetric 2-dimensional conformal field theory (SCFT in short), as a Hilbert space valued symmetric monoidal functor on a certain category of 2-dimensional bordisms with conformal structures. 
    
    Roughly, it is given by a countable-dimensional super representation of
the product of two $N=2$ super-Virasoro algebras, with a pre-Hilbert norm,
together with correlation functions which are real-analytic sections of
appropriate determinant (Berezinian) line bundles on moduli spaces of
$N=2$ superconformal curves (equivalently, complex $(1|1)$-dimensional
supermanifolds), the powers being determined by the central charges.
After either of the two topological twists, these pull back to closed
differential forms on the ordinary moduli spaces of complex curves with
marked points and formal coordinates, compatibly with sewing; see
\cite{Schwarz_superanalogs}.
It is notoriously difficult to work with this definition and to deduce
rigorously useful consequences from it.
    \item A typical connected component of the moduli space of SCFTs with central charge $(c,\bar c)=(3n,3n)$ for some integer $n\ge  0$ is expected to be a product $\cM_A\times \cM_B$ where each factor is the moduli space of complex structures of some complex projective Calabi--Yau varieties $X,X^\vee$ of complex dimension $n$. 
    The two isomorphism classes of CY varieties up to deformation are \textit{not} canonically defined. The actual mirror pair of varieties arises only near the product of two cusps (points at infinity of maximal degeneration) in the product $\cM_A\times \cM_B$. Complexified K\"{a}hler parameters for $X$ correspond to complex parameters for $X^\vee$ via so-called \textit{flat coordinates}, and vice versa.
    \item If we approach in the moduli space $\cM_A\times \cM_B$ as above a limiting point at infinity which is the product of a cusp in $\cM_A$ and some point in the interior of $\cM_B$ (so-called \textit{large volume limit}), we will have an approximate Lagrangian description of the SCFT as a $\sigma$-model whose target space $X$ is a Calabi--Yau variety with the metric rescaled by a large constant. The latter is a compact K\"{a}hler manifold $X$ endowed with a holomorphic volume form $\Omega^{n,0}$ such that 
  \begin{equation}
     \Omega^{n,0}\wedge \overline{\Omega^{n,0}}= \lambda\cdot (\omega^{1,1})^n \text { for some non-zero constant }\lambda\,.
  \end{equation}
  implying that the Ricci curvature of the underlying Riemannian metric vanishes. Another geometric piece of data is a $U(1)$-gerbe with flat connection (called \textit{$B$-field}). The cohomology class of the $B$-field,  $[i B]\in H^2(X,i\bR/2\pi i \bZ)$, is the imaginary counterpart of the K\"{a}hler class $[\omega^{1,1}]\in H^2(X,\bR)$. The most basic example of such a class is the one coming from a purely imaginary cohomology class in $H^2(X,i\bR)$ via the long exact sequence
  $$\dots \to H^2(X,2\pi i \bZ)\to H^2(X,i \bR)\to H^2(X,i\bR/2\pi i \bZ)\to H^3(X,\bZ)_{\text{tors}}\to 0\,.$$
  In this case the flat gerbe can be represented by the closed purely imaginary 2-form $iB$. The complexified K\"{a}hler class of a K\"{a}hler manifold with $B$-field is defined as 
  $$  [\omega^{1,1}]+[iB]\in H^2(X,\bR)\oplus H^2(X,i\bR/2\pi i \bZ)\simeq H^2(X,\bC^\times)\,.$$

    \item For each SCFT one has two canonically defined triangulated categories $\cC_A,\cC_B$ over $\bC$, corresponding to the two topological twists of the SCFT. The  category $\cC_A$ depends holomorphically on the $\cM_A$-component, and is locally constant in the $\cM_B$-component, whereas the category $\cC_B$ is holomorphic along $\cM_B$ and locally constant along $\cM_A$. In the case when the moduli space is a product $\cM_A\times \cM_B$ as above, the category $\cC_A$ is the Fukaya category of $X$ whose objects are roughly $B$-field twisted local systems  on Lagrangian submanifolds $L\subset X$ or, equivalently, the bounded derived category of coherent sheaves $D^b(\mathrm{Coh}\,X^\vee)$ on the mirror. Similarly, the category $\cC_B$ is the Fukaya category of $X^\vee$, or, equivalently, $D^b(\mathrm{Coh}\,X)$.
    \item There is a local system of $\bZ/2$-graded finitely-generated $\bZ$-modules $K_{top}$ on the moduli space of SCFTs. This local system is simultaneously identified with the ``topological K-theory'' of the categories $\cC_A$ and $\cC_B$. There is a natural map $\mathrm{cl}\colon K_0(\cC)\to K_{top}$ where $\cC$ is either $\cC_A$ or $\cC_B$.  In the case of a mirror pair as above, the fiber of the local system is the topological $K$-theory of the variety $X$ (the sum of the even and of the odd part). The holonomy along the factor $\cM_A$ comes from autoequivalences of Fourier-Mukai type of $D^b(\mathrm{Coh}\,X)$ corresponding to   symplectomorphisms of $X^\vee$, whereas the holonomy along  $\cM_B$ comes from symplectomorphisms of $X$.
    \item Each of the categories $\cC_A,\cC_B$ carries a stability condition, where the central charge $Z$ (up to scalar in $\bC^\times$) is controlled by the functional ``pairing with the holomorphic volume form''.
    \item Stable objects in $\cC_A$, $\cC_B$ are exactly $D$-branes, i.e.\ extensions of the SCFT to a so-called boundary SCFT, preserving a certain part of the supersymmetry.
\end{itemize}
  Heuristically, $D$-branes should be thought of as ``stable objects together with a choice of harmonic metrization''. 
  At the large volume limit, $D$-branes of type $A$ have a nice geometric description\footnote{One expects that at a possibly everywhere dense collection of proper complex subvarieties in $\cM_A$ the corresponding categories $\cC_A$ can have other stable objects called \textit{coisotropic branes} which are roughly ``supported'' on coisotropic subvarieties of dimensions $n+2,n+4,\dots$. However, by Hodge-theoretic reasons, for any given class $\gamma$ in $K_{top}$ which does not correspond to a middle-dimensional Lagrangian cycle, for any point of $\cM_A$ sufficiently close to the large volume limit there are no objects $\cE$ in the corresponding category $\cC_A$ such that $\operatorname{cl}(\cE)=\gamma$.}: Away from possible singularities of the support, it is approximately a \textit{special Lagrangian} submanifold
  $$ \exists \theta\in \bR /2\pi\bZ: \quad \arg {\Omega^{n,0}}|_{L}=\theta$$
  where the angle $\theta$ can be determined as the phase of the integral of $\Omega^{n,0}$ over $L$, i.e.\ the argument of the central charge. More precisely, in the case of a $B$-field coming from a closed purely imaginary 2-form $i B$ on $X$, $L$ carries  a Hermitian vector bundle $\mathcal E$ with unitary connection $\nabla$ whose curvature $\nabla^2$ is the identity endomorphism multiplied by $i B$ restricted to $L$.
  
  \
  
  As noted above, it is almost impossible to work with the definition of SCFT as a symmetric monoidal functor. What we will try below is an attempt to strip off some unnecessary structures (in particular, we will forget about interaction of the two categories, and will deal exclusively with $\cC_A$) and try to extract the geometric origin of stability in the framework of Fukaya categories.

\subsection{Fukaya categories: general idea}\label{subsec:fukayageneralidea}

Let $(X,\omega)$ be a symplectic manifold of dimension $2n$, not necessarily compact.
The Fukaya category associated with $(X,\omega)$ (possibly with an auxiliary $B$-field) is linear over the so-called \textit{Novikov field}
$$Nov=\left\{\sum_{i=1}^\infty c_i T^{A_i}\mid c_i\in \bC, A_i\in \bR, \lim_i A_i=+\infty\right\}$$
which is a complete  non-Archimedean field\footnote{In the case when $[\omega]\in H^2(X,\bR)$ belongs to the image of $H^2(X,\bZ)$, one can define a version of the Fukaya category over the smaller and more common Laurent series field $\bC(\!(T)\!)$.} with valuation group $\bR$.

From the SCFT perspective, the role of the formal parameter $T$ (which can be thought of as a ``small positive number'' $0<T\ll 1$) is that we consider not an individual $\bC$-linear $A$-model category $\cC_A$ but rather a one-parameter family of such categories corresponding to the path in the complexified 
K\"{a}hler cone (identified with a domain in the moduli space $\cM_A$)
\[ 
0<T\ll 1 \mapsto \log (T^{-1})\cdot [\omega^{1,1}]+[iB]\in H^2(X,\bC^\times).
\]
In what follows, we will assume for simplicity that the $B$-field vanishes and omit it from the notation.

Recall that an almost complex structure $J$ on $X$ is called $\omega$-\textit{compatible} iff for any point $x\in X$ and any non-zero tangent vector $v\in T_x X, v\ne 0$, we have
  $$\omega(v, Jv)>0\,.$$
The space of $\omega$-compatible almost complex structures is contractible, being the space of sections of a bundle with contractible fibers.
Let us choose an $\omega$-compatible $J$. This choice is used in the traditional definition of the Fukaya category; the equivalence class does not depend on $J$.
  
To first approximation, the simplest class of  objects of the \textit{compactly supported} Fukaya category $\cF_c(X,\omega)$ consists of compact Lagrangian submanifolds $L\subset X$, with morphisms from $L_1$ to $L_2$ given by linear combinations of intersection points $L_1\cap L_2$ (if the intersection is transverse), and cyclic $A_\infty$ structure given by counting with signs $J$-pseudo-holomorphic discs with boundary on a cyclic chain of Lagrangians $L_0,L_1,\dots, L_k=L_0$ in $X$, with weight $T^{\int_D\omega}$ ensuring adic convergence by Gromov compactness (generically there will be only finitely many discs with symplectic area bounded from above).

  \ 

There are well-known problems with this definition concerning grading and orientation, potential pathological behavior of the symplectic structure at infinity (if $X$ is non-compact), genericity and transversality issues, as well as the fact that some Lagrangian submanifolds are ``obstructed'' and not actual objects of the category (this can happen when there exists a non-trivial $J$-holomorphic disc with boundary on $L$). We will pretend that the problem of obstructedness does not exist, until it will be addressed later, in Section~\ref{sec:singular}.

  \
  
Let us first discuss the grading/orientation problem.
Recall that the symplectic structure leads to a reduction of the structure group $\mathrm{GL}(2n,\bR)$ on the tangent bundle to the subgroup $\mathrm{Sp}(2n,\bR)$ which is homotopy equivalent to $\mathrm U(n)$. Hence one has well-defined Chern classes $c_i\in H^{2i}(X,\bZ)$. In order to have a well-defined $\bZ$-grading on the Fukaya category it is enough to have $2c_1=0\in H^2(X,\bZ)$, together with the choice of ``trivialization'' on the cochain level. In order to alleviate the exposition it will be more convenient for us to assume the stronger condition $c_1=0$.

\begin{definition}
For $n\ge 1$, an \textit{almost Calabi--Yau structure} on $(X,\omega)$ is a choice of reduction of the structure group on  the tangent bundle from $\mathrm{Sp}(2n,\bR)$ to $\mathrm{SU}(n)$.
\end{definition}

Equivalently, such a structure on $(X,\omega)$ is the same as a choice of $\bC$-valued $n$-form $\Omega$ such that for each point $x\in X$ one can identify the tangent space $T_x X$ with standard $\bR^{2n}$ with coordinates $x_1,y_1,x_2,y_2,
\dots,x_n,y_n$ in such a way that
$$\omega|_{T_x X}=dx_1\wedge dy_1+\dots +dx_n\wedge dy_n,\quad \Omega|_{T_x X}=(dx_1+i dy_1)\wedge\dots\wedge(dx_n+idy_n).$$
Notice that an almost Calabi--Yau structure automatically gives an almost complex structure $J$ compatible with $\omega$. The restriction of $\Omega$ to any Lagrangian subspace in $T_x X$ is a non-vanishing top degree complex-valued form.

An almost Calabi--Yau structure exists iff $c_1=0$. The special case when $\Omega$ is closed  corresponds exactly to Calabi--Yau structures in the usual sense (natural from the perspective of large volume limits of SCFTs), i.e.\ $X$ carries a complex structure so that $\omega=\omega^{1,1}$ is  a K\"{a}hler metric with vanishing Ricci curvature, and $\Omega=\Omega^{n,0}$ is a everywhere non-vanishing holomorphic volume  form  normalized by a multiplicative constant so that
\begin{equation}\label{eq:normalized}{\frac{1}{n!}}\omega^n={\frac{1}{ (-2i)^n}}\Omega\wedge\overline{\Omega}\,.\end{equation}

\ 

For any oriented Lagrangian submanifold $L\subset X$ and an almost Calabi--Yau structure we have a well-defined map
$$\arg\Omega|_{L}\colon L\to \bR/2\pi \bZ\,.$$
Namely, at each point $x\in L$ consider the pairing between $\Omega$ and the wedge product of the oriented real basis of $T_x L$. The result is a non-zero complex number whose argument modulo $2\pi \bZ$ depends only on the orientation of $L$.
 The homotopy class of the map $\arg\Omega|_{L}$ from $L$ to the circle $S^1\simeq \bR/2\pi \bZ$ is an element of $H^1(L,\bZ)$ called the \emph{Maslov class} of $L$ in the presence of an almost Calabi-Yau structure. There exists (in the semi-classical approximation) an actual object of $\cF_c(X,\omega)$ (defined up to shift) corresponding to $L$ iff the Maslov class vanishes and also $L$ is endowed with a spin structure\footnote{The role of spin structures is to fix an orientation on the space of pseudo-holomorphic discs, in order to have a well-defined counting in the  definition of the structure constants in Fukaya category.}. Moreover, the choice of a shift of the actual object corresponds to the choice of a continuous  lift (called \textbf{grading})
 $$\arg^*\Omega|_{L}\colon L\to \bR$$
 of the $S^1$-valued map $\arg\Omega|_{L}$. For two Lagrangian spin submanifolds $L_1,L_2\subset X$ together with gradings $\arg^*\Omega|_{L_i},i=1,2$, for any transverse intersection point $p\in L_1\cap L_2$ one can define its index  in $\bZ$ using the lifts and the ``shortest clockwise'' path from $T_p L_1$ to $T_p L_2$ in the  Lagrangian Grassmannian $\mathrm{LGr}(T_p X)$. This index is the cohomological degree of the summand corresponding to $p$ in the Hom-complex of morphisms from $L_1$ to $L_2$ (the Floer complex). In the case of cohomological degree $+1$, one can resolve the intersection point by a little neck, the resulting object is an extension of $L_2$ by $L_1$.

Notice that the normal bundle to Lagrangian submanifold $L$ in $X$ is canonically isomorphic to the cotangent bundle $T^*L$, and the first-order infinitesimal deformations of $L$ as a \textit{Lagrangian} submanifold  correspond to \textit{closed} $1$-forms on $L$. It is known that exact deformations (corresponding to exact $1$-forms) give rise to an isomorphic object of $\cF_c(X,\omega)$ (ignoring temporarily other issues, like potential obstructedness of the object etc.). Therefore, we get a family of isomorphism classes of objects parametrized locally by an open domain in the real vector space $H^1(L,\bR)$. 

 \

In the case of Calabi--Yau metrics, the expected simplest stable objects of category $\cC$ in the classical limit of SCFT are (roughly) oriented graded compact Lagrangian submanifolds $L\subset X$ with spin structure, which are also \textit{special Lagrangian}:
$$ \arg^* \Omega|_{L}\text { is constant}.$$
Such submanifolds are known to be \textit{minimal} submanifolds, i.e.\ minimizing the volume with respect to the induced Riemannian metric among all oriented $n$-dimensional submanifolds (not necessarily Lagrangian or spin) in the given homology class. This follows from the classical calibration inequality
\[
\operatorname{vol}(Y)\ge \int_Y \operatorname{Re} (e^{-i \theta}\Omega)\quad \forall \theta\in \bR/2\pi\bZ,\quad \forall\,Y\subset X,\dim Y=n.
\]

The central charge is given by the composition of the ``Chern class'' map
$$\mathrm{cl}\colon K_0(\cF(X,\omega))\to \Gamma\coloneqq H_n(X,\bZ)$$
which associates to the object corresponding to Lagrangian $L$ its fundamental class $[L]\in \Gamma$,
 and of the homomorphism
$$  Z\colon \Gamma\to \bC,\qquad \gamma\mapsto \int_\gamma \Omega$$
where we use the fact that form $\Omega$ is closed in the Calabi--Yau case.

\subsection{Minimizing flow and  heuristic for Bridgeland stability}\label{subsec:minflow}

  Let us  assume that we have an almost Calabi--Yau structure $(\omega=\omega^{1,1},\Omega=\Omega^{n,0})$ on $X$ where $\Omega$ is closed and  normalized as in \eqref{eq:normalized}. Then on the space of oriented graded spin Lagrangian submanifolds we get a canonical flow
  \begin{equation}\label{eq:flow}\dot L=-d\arg^* \Omega|_{L}\end{equation}
by exact deformations\footnote{One can define an analogous flow (by non-exact deformations) on the space of Lagrangian submanifolds without a choice of orientation, spin structure, or trivialization of the Maslov class. It is a symplectic analog of the classical mean curvature flow for hypersurfaces in a Riemannian manifold.} of $L$. By definition, special Lagrangian submanifolds are exactly the fixed points of this flow.

The basic heuristic picture of Bridgeland's axiomatics is the following: For a given oriented graded spin Lagrangian submanifold $L$, run the flow \eqref{eq:flow} in positive time, obtaining a family $(L_t)_{t\ge 0}$ of submanifolds corresponding to \textit{isomorphic} objects of $\cF_c(X,\omega)$. The expectation is that this object is semistable iff the flow converges as $t\to+\infty$ to a special Lagrangian with phase $\phi$ where
\[
\pi\cdot \phi\coloneqq\arg^* \Omega|_{L_\infty}=\arg^* Z(cl(L_\infty))=\arg^* Z(cl(L)),
\]
here $\arg^*$ denotes also a choice of the argument of the central charge. More generally, 
  for large $t$ we expect that $L_t$ is close to the union of several special Lagrangians $L_1,\dots,L_k$  with different phases $\phi_1>\dots>\phi_k$. This decomposition  corresponds to the canonical Harder--Narasimhan filtration.

  There are several properties of \eqref{eq:flow} which give reasons to believe that the trajectory  $(L_t)_{t\ge 0}$ does not diverge at finite time and stays in a sense bounded:
\begin{itemize}
    \item the total Riemannian volume $\operatorname{vol}(L_t)=\int_{L_t}\left|\Omega|_{L_t}\right|$ decreases in $t$,
    \item the maximum of the value of $\arg^* \Omega|_{L_t}$ decreases in $t$,
    \item the minimum of the value of $\arg^* \Omega|_{L_t}$ increases in $t$.
\end{itemize}

\subsubsection{Complex-valued \texorpdfstring{$n$}{n}-forms non-vanishing on Lagrangians} \label{subsec:complex-forms}

One can extend the above heuristics to a more general class of closed complex-valued forms $\Omega$ on a symplectic manifold $(X,\omega)$, generalizing the Calabi--Yau case. Recall the generalized Siegel domain introduced by the first named author in \cite{H20}:
For a real symplectic vector space $V$ of dimension $2n$ with symplectic pairing $\omega$, we define the complex manifold $\cU(V)$ as the open domain in the space of primitive complex-valued $n$-linear skew-symmetric functionals\footnote{For a functional $\Omega\colon\bigwedge^n V\to \bC$ the condition of being primitive, $\omega^{-1}\llcorner \Omega=0$, is equivalent to the condition of being coprimitive, $\omega\wedge \Omega=0$.} $\Omega$ on $V$ such that $\Omega|_{L}\ne 0\in (\bigwedge^n L^\vee)\otimes \bC\simeq \bC$ for any Lagrangian subspace $L\subset V$. If $n\ge 1$ the space  $\cU(V)$ has two connected components; we denote by $\cU^+(V)$ the one containing translationally-invariant almost Calabi-Yau structures, see \eqref{eq:normalized}.

Any $\Omega\in \cU(V)$ gives a circle-valued function
\[
\phi_\Omega\colon\mathrm{LGr}(V)\to \bR/\pi\bZ,\quad L\mapsto \arg{\Omega|_{L}}
\]
on the Lagrangian Grassmannian of $V$.
\begin{lemma}\label{lem:positive}
     For $\Omega\in \cU^+(V)$ the derivative of $\phi_{\Omega}$ at any point $L\in \mathrm{LGr}(V)$ is a positive-definite symmetric bilinear form on $L^\vee$ after the natural identification $T^*_L\mathrm{LGr}(V)\simeq \mathrm{Sym}^2 L\simeq (\mathrm{Sym}^2 L^\vee)^\vee$.
\end{lemma}
 
\begin{proof}
Let $v_1,\dots, v_n$ be a basis of $L$, and choose a vector $u_1\in V$ such that $\omega(v_1, u_1)=1$ and $\omega(v_i, u_1)=0$, $i\ge 2$. We have two non-zero complex numbers
\[
a\coloneqq\Omega(v_1\wedge v_2\wedge \dots \wedge v_n),\quad b\coloneqq\Omega(u_1\wedge v_2\wedge \dots \wedge v_n)
\]
such that for any $(x,y)\in \bR^2\setminus \{(0,0)\}$ we have
\[
a\cdot x+b\cdot y=\Omega((v_1\cdot x +u_1\cdot y)\wedge v_2\wedge\dots\wedge v_n)\ne 0\in  \bC.
\]
Therefore, the map 
\[
F\colon \bR^2\to \bC,\quad (x,y)\mapsto ax+by
\]
is a real-linear isomorphism.  The condition $\Omega\in \cU^+(V)$ implies that $F$ is orientation-preserving, thus $\implies\operatorname{Im}\bar a b>0$. Consequently, the derivative of $\phi_\Omega$ for the tangent vector at $L$ corresponding to infinitesimal deformation 
\[
L_\epsilon\coloneqq\bR\cdot (v_1+\epsilon u_1)+\sum_{i=2}^n \bR \cdot v_i
\]
is strictly positive.
\end{proof}
For an arbitrary symplectic manifold $(X,\omega)$ of dimension $2n$ we can consider closed complex-valued $n$-forms $\Omega$ such that for any point $x\in X$ the primitive part of $\Omega|_{T_x X}$ belongs to $\cU^+(T_x X)$. For such a form we have a linear functional
\[
Z\colon H_n(X,\bZ)\to \bC,\quad \gamma\mapsto \int_\gamma \Omega.
\]
Like in the (almost) Calabi--Yau case, the form $\Omega$ gives a trivialization of $c_1$ on the cochain level, and we can define the flow as in \eqref{eq:flow}.  The only difference is that there is no obvious almost complex structure canonically associated with $\Omega$, nor is there an obvious canonical Riemannian metric. 
Nevertheless one can define the total volume by
\[
\operatorname{vol}(L)\coloneqq\int_L|\Omega|_{L}|\in \bR_{\ge 0}
\]
as well as two real numbers 
\[
\phi_+(L)\coloneqq\pi^{-1}\cdot \max \arg^*\Omega|_{L}, \quad \phi_-(L)\coloneqq\pi^{-1}\cdot\min \arg^*\Omega|_{L}.
\]
Lemma~\ref{lem:positive} implies that under the minimizing flow we have the same monotonicity properties as in the Calabi--Yau case:
\begin{itemize}
      \item total volume decreases, 
      \item $\phi_+(L)$ decreases,
      \item $\phi_-(L)$ increases.
\end{itemize}

Notice that we have a direct analogue of two-sided bound \eqref{eq:two-sided}: if $\phi_+(L)<\phi_-(L)+1$ then
\begin{equation}\label{eq:phase-volume-bound}
|Z([L])| \leq vol(L) \leq
\frac{1}{\cos\frac{\pi(\phi_+(L)-\phi_-(L))}{2}}\,|Z([L])|.
\end{equation}

\subsection{Singular Lagrangians and a conjectural global Fukaya category}\label{sec:singular}

Here we return to technical issues in the would-be definition of Fukaya category. 

First, in order to control behavior at infinity for non-compact $X$, at least for the definition of the compactly supported version $\cF_c(X,\omega)$, there are several sufficient criteria. One of them says that there exists  a complete Riemannian metric on $X$ such that tensors $\omega,\omega^{-1}$ and the Riemannian curvature tensors are uniformly bounded from above, and the injectivity radius is uniformly bounded from below. Nevertheless, some natural examples  related to spectral networks (see e.g. Section \ref{section:spectral} and \cite{HKS}) deal with a \textit{non-compact} version of the Fukaya category (so-called partially wrapped one) in which Lagrangian submanifolds go to infinity in a certain controlled manner.

\ 

The second issue concerns \textit{smoothness} of Lagrangian submanifolds representing objects of the Fukaya category. The current state-of-the-art situation where one deals only with smooth or at most immersed Lagrangian submanifolds (and possibly endowed with finite rank $\bC$-linear local systems) looks to us far from ideal, especially from the perspective of Bridgeland stability and minimizing flow. It looks very plausible that potential polystable objects are associated with special Lagrangian subsets which could have singularities (at least in real codimension $2$), like e.g.\ holomorphic curves in complex symplectic surfaces.

In what follows we will describe a hypothetical picture proposed some time ago by the third named author (see the initial proposal \cite{Kontsevich-symp-hom} and more recent update~\cite[Sections 3.3, 3.4]{Kontsevich-9ECM}). Also, in the next Section~\ref{sec:sheaf} we will describe an approach via sheaf theory which seems to work in simple toric situation. 

\ 

We assume that $c_1(X)=0$ and the first Chern class of $X$ is trivialized on the cochain level either by an almost Calabi--Yau structure, or more generally, by a complex-valued middle degree form as in Section~\ref{subsec:complex-forms}. 
The basic idea is that with any locally closed ``Lagrangian'' set $L\subset X$ with ``reasonable'' singularities (e.g.\ subanalytic subset which is Lagrangian at smooth points, although we do not want to be too precise here), we would like to associate a certain canonical small $A_\infty$-category $\cW \cF ^{\mathrm{local}}(L)$ defined over $\bZ$ and depending only on the germ\footnote{If $L$ is a retract of its Liouville tubular neighborhood $U$, the category $\cW \cF ^{\mathrm{local}}(L)$ is expected to be a part of an  appropriate wrapped Fukaya category of $U$ with stops, consisting of Lagrangian discs in $U$ transverse to smooth points of $L$.} of $(X,\omega)$ along $L$ such that
\begin{itemize}
  \item the set $\operatorname{Ob}\bigl(\mathcal{WF}^{\mathrm{local}}(L)\bigr)$
        consists of smooth points $x\in L$ together with a local orientation,
        local spin structure near $x$ and a local lift of
        $\arg(\Omega|_{T_xL})$ from $\mathbb{R}/2\pi\mathbb{Z}$ to
        $\mathbb{R}$,
  \item if $L$ is an embedded open $d$-dimensional disc, then all objects are
        isomorphic up to cohomological shift, and the endomorphism
        $A_\infty$-algebra of each object is quasi-isomorphic to $\mathbb{Z}$,
  \item every open embedding $L_1\hookrightarrow L_2$ induces a functor
        (tautological on objects)
        $\mathcal{WF}^{\mathrm{local}}(L_1)\to\mathcal{WF}^{\mathrm{local}}(L_2)$,
        so we get a co-presheaf of categories on every locally closed $L$,
  \item this co-presheaf is a cosheaf for every locally closed $L$,
  \item every closed embedding $L_1\hookrightarrow L_2$ induces a functor
        (obvious on objects)
        $\mathcal{WF}^{\mathrm{local}}(L_2)\to\mathcal{WF}^{\mathrm{local}}(L_1)\sqcup\{0\}$
        which identifies $\mathcal{WF}^{\mathrm{local}}(L_1)$ with the
        localization of $\mathcal{WF}^{\mathrm{local}}(L_2)$ by objects
        corresponding to $x\in L_2^{\mathrm{smooth}}\setminus L_1$.
\end{itemize}
The cosheaf of small categories on $L$ gives a sheaf
$\mathcal{WF}^{\mathrm{local},\vee}_{L}$ of large categories by associating with
each open $L'\subset L$ the category of triangulated functors from
$\mathcal{WF}^{\mathrm{local}}(L')$ to complexes of $\mathbb{Z}$-modules.

\ 

These (hypothetical) properties imply that when $L$ is a oriented graded connected Lagrangian submanifold with spin structure, the category $\cW \cF ^{\mathrm{local}}(L)$ is equivalent to the dg category  whose objects are points of $L$, and the complex of morphisms $\Hom(x_1,x_2)$ is the chain complex of the space of paths in $L$ from $x_1$ to $x_2$. The dual category $\Gamma(\mathcal{WF}^{\mathrm{local},\vee}_{L})$ is the category of complexes of sheaves whose cohomology sheaves are local systems on $L$, possibly of infinite rank.

 \

 The next step (largely undeveloped at the moment, except in the cases when $L$ is the image of a smooth immersion with transverse self-intersections, or when $\dim X=2$) is that pseudo-holomorphic discs in $X$ with boundary on a closed singular Lagrangian $L$ give rise to a canonical deformation of $\cW \cF ^{\mathrm{local}}(L)$ over the Novikov ring 
 \[Nov_+\coloneqq\{a\in Nov\,\mid\,|a|\le 1\}=\left\{\sum_{i=1}^\infty c_i T^{A_i}\mid c_i\in \bC, A_i\in \bR_{\ge 0}, \lim_i A_i=+\infty\right\}.\]

Assuming the existence of such a canonical deformation, we define the $Nov_+$-linear global Fukaya category \textit{supported} on $L$ (denoted by $\cF^\mathrm{global}_{L,+}(X,\omega)$) as the category  of  $\mathrm{Perf}(Nov_+)$-representations of the deformed $\cW \cF ^{\mathrm{local}}(L)$. The localization property implies that for $L_1\subset L_2$ we have a fully faithful embedding $\cF^\mathrm{global}_{L_1,+}(X,\omega)\hookrightarrow \cF^\mathrm{global}_{L_2,+}(X,\omega)$. The universal $Nov_+$-linear Fukaya category $\cF^\mathrm{global}_{+}(X,\omega)$ can  be defined as the inductive limit of $\cF^\mathrm{global}_{L,+}(X,\omega)$ over the poset of all compact singular Lagrangians $L\subset X$.
Finally, the Fukaya category over $Nov$ is defined as the localization
 $Nov\otimes_{Nov_+}\cF^\mathrm{global}_{+}(X,\omega)$.

 We expect that the hypothetical picture of the minimizing flow extends to our putative ``correct'' Fukaya category, and we will get a stability condition on it.

\subsection{A sheaf-theoretic model for Fukaya categories and stability}
\label{sec:sheaf}

Here we propose a  rigorous definition, motivated by mirror symmetry, of a  large $Nov_+$-linear triangulated category which seems to  contain as a full subcategory $\cF^\mathrm{global}_{+}(X,\omega)$ for the split symplectic torus $X\simeq (S^1)^{2n}$, the product of $n$ copies of the standard $2$-dimensional symplectic torus $(\bR/\bZ)^2$ of area $1$. Moreover, one can formulate a precise conjectural characterization of a Bridgeland stability condition on the localized $Nov$-linear category in terms of an arbitrary closed complex-valued $n$-form $\Omega$ on $X$ satisfying the non-vanishing condition from Section~\ref{subsec:complex-forms}. Finally, we propose a generalization to the case of ``Fukaya category with coefficients''.

\subsubsection{A large $Nov_+$-linear category associated to a Tate abelian variety}

Let us consider the following sheaf $\cO_{+,\bR^n}$ of $Nov_+$-algebras on $\bR^n$: Its space of sections over any open bounded convex domain $U\subset \bR^n$ is the set of Laurent series
$$  
\sum_{k=(k_1,\dots,k_n)\in\mathbb Z^n}  
c_k z^k,  
\qquad  
c_k\in Nov,  
\qquad  
z^k=z_1^{k_1}\cdots z_n^{k_n},  
$$
such that for every $k\in \bZ^n$ one has
\[ \log|c_k|+ \sup_{x\in U}\sum_{i=1}^n k_i x_i\le 0.\]
This sheaf is equivariant under the group of translations by $\bR^n$:
\[x_i\mapsto x_i+\delta x_i, i=1,\dots,n,\quad c_k\mapsto c_k\cdot T^{\sum_{i=1}^n k_i \delta x_i}.\]
By taking the quotient by the subgroup $\bZ^n\subset \bR^n$, we obtain a sheaf $\cO_{+,B}$ of $Nov_+$-algebras on $n$-dimensional torus $$B\coloneqq\bT^n=\bR^n/\bZ^n.$$ Extending scalars we get a sheaf \[\cO_{B}\coloneqq Nov\otimes_{Nov_+}\cO_{+,B}\]
of $Nov$-algebras on $B$. 
\begin{proposition}
The category of perfect complexes over $\cO_{B}$ is naturally equivalent to the category of perfect complexes on the abelian variety  over $Nov$ which is the $n$-th power of the Tate elliptic curve $\bE\coloneqq\bG_{m,Nov}^{an}/T^\bZ$. 
\end{proposition}

\begin{proof}(Sketch). The analytic variety $\bE^n$ (with the Berkovich topology) maps  properly  to $B$ by the logarithm-of-norm map. The direct image of the sheaf of analytic functions on  $\bE^n$ coincides with $\cO_{B}$, and higher direct images vanish.
\end{proof}

Let us consider $Nov_+$-linear triangulated category
$$  
D_{B}\coloneqq D\big(\cO_{+,B}\text{-}\mathrm{Mod}\big).  
$$

Consider the endofunctor (cf.\ almost mathematics of Gabber and Ramero~\cite{GabberRamero2018}) 
\[
E\longmapsto \mathfrak m\otimes^L_{Nov_+}E,
\]
where
\[
\mathfrak m\coloneqq\{a\in Nov\mid |a|<1\}
\]
is the maximal ideal of $Nov_+$. Since $\mathfrak m$ is flat over $Nov_+$ and
$\mathfrak m^2=\mathfrak m$, the multiplication map
\[
\mathfrak m\otimes^L_{Nov_+}\mathfrak m\longrightarrow \mathfrak m
\]
is an isomorphism; hence this endofunctor is idempotent.
 We define
$$   
\mathsf C_{B}\subset D_{B}  
$$
to be its essential image.

\subsubsection{Reduced singular support, symplectic covariance}

The torus  $B$ carries a natural \textit{integral affine structure}, i.e.\ a family of commuting vector fields which form a lattice
  in the tangent space at any point $b\in B$.
  In this case we have a natural symplectic manifold
  \[X\coloneqq T^*B/T^*_\bZ B\simeq \bT^{2n}\]
  which is the quotient of the cotangent bundle by the dual lattice and is a Lagrangian torus fibration over $B$.

  For $t\in\bR$, set
\[
\mathfrak m_t\coloneqq\{a\in Nov\mid |a|<e^t\}.
\]
Let $\bR^\to$ be the real line with the topology generated by the right rays
$(a,+\infty)$. For $E\in D_B$, let $\cR^\to(E)$ be the sheaf on
$B\times\bR^\to$ defined on basic open sets by
\[
R\Gamma\bigl(U\times(a,+\infty),\cR^\to(E)\bigr)
=
R\Gamma\bigl(U,\mathfrak m_a\otimes^L_{Nov_+}E\bigr).
\]
The \emph{Rees sheaf} $\cR(E)$ is the pullback of $\cR^\to(E)$ to
$B\times\bR$ with the usual topology. Its stalk at $(b,t)$ is
\[
\cR(E)_{(b,t)}
\simeq
\mathfrak m_t\otimes^L_{Nov_+}E_b.
\]

For a sheaf $\cF$ on a real manifold $Y$, denote by
$SS(\cF)\subset T^*Y$ its singular support in the sense of
Kashiwara--Schapira; it is a closed conic coisotropic subset.
After pulling $\cR(E)$ back to $\bR^n\times\bR_t$, multiplication by powers
of $T$ and by Laurent monomials gives natural equivariances of the Rees
sheaf. 
 Consequently, $SS(\cR(E))$ is invariant under translations in the
$t$-direction and, after restriction to $\tau=1$, under translations of
the cotangent variable by the lattice $T^*_{\bZ}B$.

Write $(b,t;\xi,\tau)$ for coordinates on $T^*(B\times\bR)$.  Taking the
Hamiltonian reduction at $\tau=1$ and then dividing by the integral
cotangent lattice, we obtain a closed coisotropic subset
\[
SS_X(E)\coloneqq
\bigl(SS(\cR(E))\cap\{\tau=1\}\bigr)/
\bigl(\bR_t\times T^*_{\bZ}B\bigr)
\subset
X=T^*B/T^*_{\bZ}B,
\]
which we call the \emph{reduced singular support} of $E$.

For example, let the trivial line bundle on $\bE^n$ be endowed with a
metrization given by a smooth function $f\colon B\to\bR$, and let
$E_f\in\mathsf C_B$ be the corresponding object (locally consisting of
analytic functions $g$ such that $\log|g|<f$).  Then
\[
SS_X(E_f)=\operatorname{graph}(df)\pmod{T^*_{\bZ}B}\subset X.
\]
Thus varying the metrization of a fixed line bundle corresponds to exact
deformations of the associated Lagrangian section.

This picture is closely related to the non-Archimedean approach to mirror
symmetry for torus fibrations of Kontsevich--Soibelman
\cite{KS_torus}.  In particular, for Lagrangian sections the
construction above suggests a $Nov_+$-linear lift of the usual mirror
correspondence: the section
\[
\operatorname{graph}(df)\pmod{T^*_{\bZ}B}\subset X
\]
corresponds to the metrized line bundle $E_f\in\mathsf C_B$.
After extension of scalars from $Nov_+$ to $Nov$, this is compatible with
the usual mirror-symmetry picture.  Abouzaid's family Floer construction
extends the latter, on the $Nov$-linear level, to much more general smooth
Lagrangian submanifolds, associating to them twisted coherent sheaves on
the rigid analytic mirror \cite{Abouzaid_family}; it was subsequently
upgraded to a faithful functor from the Fukaya category to twisted perfect
complexes \cite{Abouzaid_family_functor}.

As further evidence for this picture, one can show that symplectomorphisms
of $X$ sufficiently close to the identity\footnote{Via a Weinstein neighborhood of the diagonal, the graph of a
sufficiently small symplectomorphism is encoded by a small closed
$1$-form, with its exact part described by a generating function.  This
data determines a kernel, and hence a Fourier--Mukai-type integral
transform on $\mathsf C_B$.}, as well as a natural cover of
$\mathrm{Sp}(2n,\bZ)$, act on $\mathsf C_B$ by functors represented by kernels in
$\mathsf C_{B\times B}$.  Moreover, the reduced singular support
construction is equivariant with respect to these symplectomorphisms.

A notable feature of the sheaf-theoretic construction proposed here is
that it involves no choice of a compatible almost-complex structure on
$X$.  Its present formulation does, however, use the choice of the
Lagrangian torus fibration $X\to B$, or equivalently, in the split case
considered here, the integral affine structure on $B$.

\subsubsection{Conjectural stability condition}

The discussion above suggests that, in the present toric situation, the
still conjectural category $\cF^\mathrm{global}_{+}(X,\omega)$ should admit
an alternative description as a suitable full subcategory of $\mathsf C_B$.
At present it is not clear what the optimal ``regularity'' condition defining
this subcategory should be.  For the purposes of constructing a stability
condition, however, it may be enough to work with a sufficiently rich class
of reasonable objects, for instance objects whose reduced singular support
is a piecewise-linear Lagrangian skeleton.

\ 

Fix a closed complex-valued $n$-form $\Omega$ on $X$, not necessarily
translation-invariant, satisfying the
non-vanishing condition of Section~\ref{subsec:complex-forms}.  Since in
our class one should not expect genuine special Lagrangian representatives,
we replace this condition by the existence of representatives with
arbitrarily small phase amplitude
\[
\phi_+(\Lambda)-\phi_-(\Lambda),
\]
where $\Lambda$ is the reduced singular support.  In view of
\eqref{eq:phase-volume-bound}, such representatives have volume arbitrarily
close to the lower bound $|Z([\Lambda])|$, and therefore provide a natural
notion of approximate special-Lagrangian representative.

\ 

We now make this approximation picture concrete by describing an elementary
class of objects in $\mathsf C_B$ for which the reduced singular support can
be read off explicitly.
Choose a finite rational polyhedral covering $\{P_i\}$ of $B$.  On a
neighborhood of each non-empty multiple intersection
\[
P_I\coloneqq\bigcap_{i\in I}P_i
\]
choose a bounded complex $E_I$ whose terms are finite direct sums of line
bundles endowed with piecewise-quadratic metrizations with rational purely
quadratic parts, and whose differential is non-expanding.  For
$P_J\subset P_I$ choose compatible non-expanding chain maps
\[
E_J\longrightarrow E_I|_{P_J}
\]
which become quasi-isomorphisms after applying
$\mathfrak m\otimes^L_{Nov_+}(-)$.  The finite \v{C}ech-type totalization of
this diagram, followed by $\mathfrak m\otimes^L_{Nov_+}(-)$, defines an
object of $\mathsf C_B$.  We denote its reduced singular support by
\[
\Lambda_E\coloneqq SS_X(E)\subset X.\] For the objects constructed above, $\Lambda_E$ is a piecewise-linear
Lagrangian subset of $X$, rational with respect to the integral affine
structure on $B$, and its projection $\Lambda_E\to B$ is finite.

There is one further piece of structure which will be important below.
On the smooth locus $\Lambda_E^{\mathrm{sm}}$ there is a naturally
associated local system of perfect complexes over the residue field
\[
\mathbf k\coloneqq Nov_+/\mathfrak m\simeq\bC.
\]
Indeed, for $\lambda\in\Lambda_E^{\mathrm{sm}}$, choose a lift
$\widetilde\lambda$ to $SS(\cR(E))$ at $\tau=1$.  The microlocal stalk
\[
\mu_{\widetilde\lambda}\cR(E)
\]
is naturally a complex over $\mathbf k$: microlocalization in the
$\tau=1$ direction extracts the associated graded of the strict norm
filtration, on which $\mathfrak m$ acts trivially.  For constructible
metrized perfect complexes these microlocal stalks are perfect and vary
locally constantly along the smooth sheets of $\Lambda_E$.  Thus one
obtains a local system
\[
\cV_E\in
\operatorname{Loc}\bigl(\Lambda_E^{\mathrm{sm}},
\operatorname{Perf}(\mathbf k)\bigr).
\]
The equivariance of the Rees construction under multiplication by powers
of $T$ and Laurent monomials identifies the microlocal stalks associated
with different lifts of the same point of $\Lambda_E$.

The cohomological grading of $\cV_E$ is part of the brane data: a shift by
$[1]$ shifts the normalized lift of the phase by $1$.  Thus the pair
\[
\bigl(\Lambda_E^{\mathrm{sm}},\cV_E\bigr)
\]
rather than the underlying Lagrangian subset alone is the appropriate
object on which to normalize $\arg^*\Omega$ and define $\phi_\pm$.

We can now formulate the proposed characterization of the stability
condition.  Recall that $\Omega$ determines the central charge
\[
Z_\Omega(\gamma)\coloneqq\int_\gamma\Omega .
\]

\begin{conjecture}
For $\phi\in\bR$, let $\cC_\phi$ be the full subcategory of the localized
category
\[
Nov\otimes_{Nov_+}\mathsf C_B
\]
consisting of objects $E$ with the following property: for every
$\varepsilon>0$, $E$ admits a constructible metrized model
$\widetilde E\in\mathsf C_B$ such that
\[
\phi-\varepsilon<
\phi_-(\Lambda_{\widetilde E})
\leq
\phi_+(\Lambda_{\widetilde E})
<
\phi+\varepsilon .
\]
Here the phases are normalized using the graded microlocal coefficient
system $\cV_{\widetilde E}$ on
$\Lambda_{\widetilde E}^{\mathrm{sm}}$.

The collection of subcategories $\{\cC_\phi\}_{\phi\in\bR}$ is the slicing
of a Bridgeland stability condition with central charge $Z_\Omega$.
Equivalently, together with $Z_\Omega$ it uniquely determines this
stability condition.
\end{conjecture}

\subsubsection{Extension to Fukaya category with  coefficients}\label{subsubsec:coeff}

Let us briefly explain how the preceding picture should extend to Fukaya
categories with coefficients.  Recall the heuristic from
Section~\ref{subsec:origins}: if a symplectic manifold is fibered over
$X$ and the symplectic volumes of the fibers become small, its Fukaya
category should be describable as the Fukaya category of $X$ with
coefficients in the categories associated with the fibers.  Geometrically,
Lagrangian submanifolds in the total space should degenerate to
Lagrangian subsets of $X$ labeled by objects of these coefficient
categories.

We first consider the simplest case of \emph{constant coefficients}.
Fix a $\bC$-linear enhanced triangulated category $\cC$ with
finite-dimensional morphism spaces, together with a stability condition
$\sigma_{\cC}$ on $\cC$.  We also fix, as above, a closed complex-valued
$n$-form $\Omega$ on $X$, not necessarily translation-invariant, satisfying
the non-vanishing condition of Section~\ref{subsec:complex-forms}.

There is a simple $Nov_+$-linear enlargement of $\cC$.  Morphism spaces
are extended from $\bC$ to $Nov_+$, and one allows formal deformations of
objects governed by Maurer--Cartan equations.  Thus, for example, an
infinitesimal deformation of an object $A\in\cC$ is described by an element
\[
\delta\in\operatorname{Hom}^1_{\cC}(A,A)\otimes_{\bC}\mathfrak m
\]
satisfying the $A_\infty$ Maurer--Cartan equation.  More generally one
allows finite twisted complexes of such objects; this is the same
formalism which will be recalled in
Section~\ref{subsec:defAinf}.

In the present toric situation this leads to a concrete analogue of
$\mathsf C_B$.  On a rational polyhedron $P\subset B$ consider finite
twisted complexes whose underlying terms are of the form
\[
\bigoplus_\alpha L_\alpha\otimes A_\alpha,
\qquad A_\alpha\in\cC,
\]
where the $L_\alpha$ are metrized line bundles of the type considered
above.  The matrix coefficients of the twisting cochain are analytic
sections with values in
$\operatorname{Hom}_{\cC}(A_\alpha,A_\beta)$; they are required to be
non-expanding and to satisfy the $A_\infty$ Maurer--Cartan equation.
Using the same finite polyhedral descent and almost quasi-isomorphisms on
overlaps as above, one obtains a category which we denote provisionally by
\[
\mathsf C_B(\cC).
\]
For $\cC=\operatorname{Perf}(\bC)$ this reduces to the construction of
$\mathsf C_B$.

This elementary construction may be regarded as a notion of an analytic
family of objects of $\cC$ sufficient for our purposes: locally it is
generated by constant objects of $\cC$, while all analytic dependence is
carried by the metrized line bundles and by the coefficients of the
Maurer--Cartan equation.

The Rees construction and reduced singular support extend without change.
For $E\in\mathsf C_B(\cC)$ we again obtain a rational piecewise-linear
Lagrangian subset
\[
\Lambda_E\subset X
\]
finite over $B$.  The essential difference is that the microlocal
coefficient system on its smooth locus is now a locally constant family
of objects of $\cC$,
\[
\cV_E\in
\operatorname{Loc}\bigl(\Lambda_E^{\mathrm{sm}},\cC\bigr),
\]
rather than a local system of perfect complexes over the residue field.

The stability condition on $\cC$ supplies an additional contribution to
the phase.  In the simplest situation, if a smooth sheet of $\Lambda_E$
is labeled by a semistable object $A\in\cC$ of phase $\phi_{\cC}(A)$,
its total phase is
\begin{equation}\label{eq:totphase}
\phi_{\mathrm{tot}}
=
\pi^{-1}\arg^*\Omega|_{\Lambda_E}
+\phi_{\cC}(A).
\end{equation}
Correspondingly, the expected central charge is obtained by weighting
$\Omega$ by the central charge of the microlocal coefficient system,
schematically
\[
Z(E)=
\int_{\Lambda_E}
Z_{\cC}\bigl([\cV_E]\bigr)\,\Omega .
\]
Thus the approximate special Lagrangian condition discussed above has a
direct analogue with the geometric phase replaced by this total phase.

\ 

More generally, the coefficient data should consist of the following.
Let
\[
\pi_{\mathrm{LGr}}\colon \operatorname{LGr}^{\mathrm{or}}(TX)\longrightarrow X
\]
be the bundle of oriented Lagrangian subspaces.  We assume given a local
system $\mathscr C$ of enhanced $\bC$-linear triangulated categories on
$\operatorname{LGr}^{\mathrm{or}}(TX)$, with finite-dimensional morphism
spaces, whose stalk at $(x,L_x)$ will be denoted by
$\mathscr C_{x,L_x}$.  The monodromy around the positive generator of the
fundamental group of a fiber is required to be the shift $[2]$.  In
particular, the induced monodromy on
$K_0(\mathscr C_{x,L_x})$ is trivial.

We further assume given a local system $\Gamma$ of finitely generated
abelian groups on $X$, together with a morphism of local systems
\[
\operatorname{cl}\colon
K_0(\mathscr C)\longrightarrow \pi_{\mathrm{LGr}}^*\Gamma,
\]
and a morphism of sheaves of abelian groups
\[
\Omega\colon
\Gamma\longrightarrow \Omega^n_{X,\mathrm{cl}}\otimes_{\bR}\bC.
\]
Thus every locally constant section $\gamma$ of $\Gamma$ determines a
closed complex-valued $n$-form $\Omega(\gamma)$ on $X$.

For every $(x,L_x)\in\operatorname{LGr}^{\mathrm{or}}(TX)$ and every
positive element
\[
v\in\bigwedge^n L_x,
\]
we obtain a homomorphism
\[
Z_{x,L_x,v}\colon K_0(\mathscr C_{x,L_x})\longrightarrow\bC,
\qquad
Z_{x,L_x,v}(E)
=
\Omega(\operatorname{cl}(E))_x(v).
\]
Replacing $v$ by a positive multiple rescales $Z_{x,L_x,v}$ by a positive
real number.  We require these central charges to lift, locally
continuously in $(x,L_x)$, to stability conditions on
$\mathscr C_{x,L_x}$, compatibly with parallel transport.  Equivalently,
one may regard this as a continuously varying family of stability
conditions defined up to positive real rescaling.

Finally, we impose the following positivity condition.  Identify
\[
T_{L_x}\operatorname{LGr}^{\mathrm{or}}(T_xX)
\simeq \operatorname{Sym}^2(L_x^\vee).
\]
Let $L_s$ be a curve of oriented Lagrangian subspaces with $L_0=L_x$ and
with $\dot L_0$ positive definite.  If
$0\neq E\in\mathscr C_{x,L_x}$ is semistable, and $E_s$ denotes its
parallel transport to $\mathscr C_{x,L_s}$, then
\[
\left.
\frac{d}{ds}
\arg Z_{x,L_s}(E_s)
\right|_{s=0}>0.
\]

 After choosing a section of $\pi_{\mathrm{LGr}}$, the local system
$\mathscr C$ restricts to a local system of $\bC$-linear categories on
$B\simeq\bT^n$.  Equivalently, after choosing a base point, this amounts
to a single category together with $n$ commuting autoequivalences.  To pass
to the non-Archimedean picture one should in addition assume that it makes
sense to speak about analytic families of objects of these categories over
non-Archimedean analytic spaces.  This is a genuine extra structure, but
it is quite concrete in examples; one may keep in mind, for instance,
categories of finite-dimensional representations of a finite quiver with
relations, as in Chapter~7.

Under this assumption the construction above should have a version with
coefficients in $\mathscr C$, producing a global $Nov$-linear category.
Its objects have Lagrangian supports equipped, on their smooth loci, with
objects of the corresponding fiber categories.  Using the family of
stability conditions described above, one can then define approximate special Lagrangian representatives exactly in
the same spirit as in the preceding subsection: at every smooth point of
the Lagrangian support, the corresponding object of the fiber category is
required to be semistable for the stability condition associated with the
oriented tangent Lagrangian plane, and the resulting phase is required to
be constant, or arbitrarily close to constant, respectively.  The
resulting notion should determine the expected stability condition on the
global $Nov$-linear category.

\ 

We have restricted the discussion to powers of the Tate elliptic curve
only in order to have a particularly explicit model in which all the
constructions can be written down concretely.  The same picture should
extend much further, both to more general torus fibrations and to
non-constant systems of coefficient categories.  We leave the formulation
of these generalizations to the reader.

A natural next challenge is to extend this picture to bases $B$ with
codimension-$2$ singularities, as arise in non-Archimedean SYZ fibrations
of the type considered in \cite{KS_affine}.
\section{A hypothetical framework for categorical Kähler geometry}
\label{section:axiomatic}

In this section we attempt to outline a framework which unifies the various examples discussed throughout the paper, to arrive at a provisional definition of ``categorical K\"ahler structure''.
In Section~\ref{subsec:generalframework} we list those structures present in both Archimedean and non-Archimedean settings. These include a notion of metric on an object, mass measures, complexified K\"ahler potential, and flow.
In the case when the base field is $\bC$, much of the structure can be encoded in terms of $*$-algebras with a functional, as proposed in Section~\ref{subsec:frameworkarch}. 
In Section~\ref{subsec:frameworknonarch} we highlight that over non-Archimedean fields, but also more general settings such as the one discussed in Section~\ref{section:comonads}, metrized objects form a category whose defining feature is a rescaling action.
In the final Section~\ref{subsec:lozenge}, we show that the \textit{Lozenge algebra} formalism from our earlier work~\cite{HKKPitlogs2} and the framework of Bhattacharya and the third named author~\cite{BK21} are essentially equivalent. 
We expect that they provide a close-to-complete picture of the local structure of categorical K\"ahler geometry.

\subsection{General features}\label{subsec:generalframework}


We now list the structures expected on a stable $\infty$-category $\mathcal C$ with stability condition which together encode a categorical K\"ahler metric.
The following axioms should be read more as a collection of desiderata rather than as a final definition. 
Our guiding principle is to unify examples coming from both sides of homological mirror symmetry, as well as others discussed in later sections of this paper.

\begin{itemize}
\item \textbf{Metrized objects.}
The first layer of a categorical K\"ahler structure is a notion of ``metric'' on any object in $\cC$.
We encode this by a functor $\mathrm{Met}\colon \mathbf{Iso}(\cC)\to(\bR_{>0},\cdot)\text{-}\mathbf{Set}$ from the groupoid of isomorphisms in $\cC$, up to homotopy, to the category of sets with an action of $(\bR_{>0},\cdot)$ by ``rescaling''. (As a generalization, one could also replace $(\bR_{>0},\cdot)$ by any of its sub-monoids.)
Moreover, there should be natural isomorphisms $\mathrm{Met}(X)\cong\mathrm{Met}(X[1])$, as well as natural maps $\mathrm{Met}(X)\times\mathrm{Met}(Y)\to \mathrm{Met}(X\oplus Y)$.

We refer to a pair $(X,h)$, $X\in\cC$, $h\in\mathrm{Met}(X)$, as a \textit{metrized object}.
By the above assumption, any metrized object has a shift $(X[1],h[1])$ and we can form direct sums of metrized objects.

In examples, $\mathrm{Met}(X)$ often carries additional structures, such as a topology or metric.
We also see that it is often possible to enlarge (``partially compactify'') the sets/spaces $\mathrm{Met}(X)$ so that we get a map $\mathrm{Met}(X)\times\mathrm{Met}(Z)\to\mathrm{Met}(Y)$ for any exact triangle $X\to Y\to Z\to X[1]$, not just direct sums.
Roughly, this amounts to considering limiting metrics which are filtrations of the underlying object together with a (non-degenerate) metric on each component.

\item \textbf{Mass measures.}
For each metrized object $(X,h)$ we have a finite measure $\mu_{X,h}$ on $\bR$ with bounded support.
The role of this measure is to record, for a particular choice of metric on an object, the distribution of mass across various phases.
The total measure of $\bR$, $m(X,h)\coloneqq \mu_{X,h}(\bR)$, is by definition the \textbf{mass} of $(X,h)$\footnote{This mass will be, in general, greater than the mass assigned to $X$ in terms of the HN filtration coming from the stability condition. Crucially, it really depends on the metric $h$, not just $X$.}.
We require:
\begin{enumerate}[(1)]
    \item $m(X,h)=0\implies X=0$.
    \item $\mu_{Y,\mathrm{Met}(f)(h)}=\mu_{X,h}$ for any isomorphism $f\colon X\to Y$ in $\cC$.
    \item $\mu_{X,t\cdot h}=\mu_{X,h}$, i.e.\ mass measures are invariant under rescaling.
    \item $\mu_{X[1],h[1]}=(\phi\mapsto\phi+1)_*\mu_{X,h}$, i.e.\ equivariance under shift,
    \item $\mu_{X\oplus X',h\oplus h'}=\mu_{X,h}+\mu_{X',h'}$,
    \item The central charge of the stability condition on $\cC$ is recovered from the mass measures by
    \[ Z(X)=\int_{\bR}e^{\pi\sqrt{-1}\phi}d\mu_{X,h}(\phi). \]
    As a consequence, we get (tautologically) the \textit{BPS inequality} $|Z(X)|\leq m(X,h)$.
    \item \textbf{Categorical DUY principle:} Define a metric $h$ to be \textit{harmonic of phase $\phi$} if $\mathrm{supp}(\mu_{X,h})=\{\phi\}$. A necessary condition for the existence of such a metric should be that $X$ is semistable of phase $\phi$.
    A sufficient condition for the existence of such a metric should be that $X\neq 0$ is polystable of phase $\phi$.  
\end{enumerate}

When $\mathrm{supp}(\mu_{X,h})\neq\emptyset$ we define
\[
\phi_-(X,h)\coloneqq\min\left(\mathrm{supp}(\mu_{X,h})\right),\qquad \phi_+(X,h)\coloneqq\max\left(\mathrm{supp}(\mu_{X,h})\right).
\]
Thus, $h$ is harmonic iff $\phi_-(X,h)=\phi_+(X,h)$.

\item \textbf{Complexified K\"ahler potential.}
Suppose $\cC$ is linear over a normed field $\mathbf k$ and that objects of $\cC$ are parametrized by a moduli stack $\mathrm{Ob}(\cC)$.
We have a pair of oriented real line bundles $|\cL_{\mathrm{Re}}|,|\cL_{\mathrm{Im}}|$ over $\mathrm{Ob}(\cC)$.
These correspond to the central charge $Z$ in the following sense:
On the fiber of $|\cL_{\mathrm{Re}}|$ (resp.\ $|\cL_{\mathrm{Im}}|$) over $X\in\cC$, $\bG_{m,\mathbf k}\cong\mathbf k^{\times}\cdot 1_X\subseteq\mathrm{Aut}(X)$ acts as $t\cdot v=|t|^{\operatorname{Re}(Z(X))}v$ (resp.\ $t\cdot v=|t|^{\operatorname{Im}(Z(X))}v$).
A \textit{complexified K\"ahler potential} is given by metrizations of $|\cL_{\mathrm{Re}}|$ and $|\cL_{\mathrm{Im}}|$.
Given a phase $\phi\in\bR$, the metrized real oriented line bundle
\[
|\cL_{Z,\phi}|\coloneqq |\cL_{\mathrm{Re}}|^{\sin \pi\phi}\otimes |\cL_{\mathrm{Im}}|^{-\cos \pi\phi}
\]
should induce local K\"ahler potentials (in the Archimedean or non-Archimedean sense) on the (smooth part of the) moduli space $\cC_{\phi,\mathrm{ps}}$ of polystable objects of phase $\phi$.
See also Section~\ref{subsec:effective} for a related discussion and some justification for this axiom.

\item \textbf{Minimizing flow.}
This is an action of $(\bR_{\geq 0},+)$, or in some examples just $(\bZ_{\geq 0},+)$, on each $\mathrm{Met}(X)$.
We expect the following quantities to vary monotonically under this action: $m(X,h)$ decreases, $\phi_-(X,h)$ increases, and $\phi_+(X,h)$ decreases.
Moreover, the fixed points of the action on $\mathrm{Met}(X)/\text{rescaling}$ are precisely the harmonic metrics.
For an object $X\in\cC$ not admitting a harmonic metric, the asymptotics of the flow will be controlled by a refinement of the Harder--Narasimhan filtration of $X$, such as the iterated balanced filtration defined in~\cite{HKKPitlogs1}, see also~\cite{HKKPitlogs2,ibaneznunez_refined,stlam}.
\end{itemize}

From this point on our discussion splits into the ``Archimedean'' and ``non-Archimedean'' cases.
A unified treatment of the two cases remains an open problem.

\subsection{The Archimedean case}\label{subsec:frameworkarch}

The general framework of the previous subsection treats metrized objects, mass measures,
potentials, and flows as separate pieces of structure. 
In the Archimedean setting we expect these pieces to be governed locally by a single object:
a $*$-algebra $\mathcal A_{X,h}$ together with a complex-valued linear functional $\Omega$. 
The algebra should be regarded as a noncommutative replacement for the complexified tangent space 
to the space of metrics at $h$, while $\Omega$ is the first variation of the complexified Kähler potential.
We will now discuss these structures in more detail.

We assume that for each pair $(X,h)$, $X\in\cC$, $h\in\mathrm{Met}(X)$, we are given a $*$-algebra $\cA=\cA_{X,h}$ over $\bC$.
Some basic requirements are:
\begin{enumerate}[(1)]
    \item isomorphisms $\cA_{X,h}\cong \cA_{Y,\mathrm{Met}(f)(h)}$ for any isomorphism $f\colon X\to Y$ in $\cC$,
    \item isomorphisms $\cA_{X,h}\cong \cA_{X[1],h[1]}$,
    \item isomorphisms $\cA_{X,h}\cong \cA_{X,t\cdot h}$ for $t>0$,
    \item injections $\cA_{X,h}\oplus \cA_{X',h'}\to \cA_{X\oplus X',h\oplus h'}$,
    \item The above maps are part of the data and satisfy various compatibility requirements.
\end{enumerate}
The subspace of self-adjoint elements
\[
\cA_{\mathrm{real}}\coloneqq \{a\in\cA\mid a=a^*\}
\]
is the naive tangent space to $\mathrm{Met}(X)$.
When $\cA$ is infinite-dimensional, it should come with a topology and a well-defined exponential map $\exp\colon\cA\to\cA$.\footnote{We would like to be more precise about what kind of algebra $\cA$ should be in the infinite-dimensional case, but leave this as an open problem.}

The algebra $\cA$ is further equipped with a linear functional $\Omega\colon \cA\to \bC$ which comes with a choice of argument in the following sense: we have $\arg\Omega\in \cA_{\mathrm{real}}$ such that 
\[
\Omega\left(e^{-\sqrt{-1}\arg\Omega}a^*a\right)\geq 0,\qquad a\in\cA.
\]
The functional $\Omega$ and its choice of argument $\arg\Omega$ should be invariant under isomorphisms of metrized objects and rescaling. Under shift, $\Omega$ changes sign: $\Omega_{X[1],h[1]}=-\Omega_{X,h}$ and $\arg \Omega_{X[1],h[1]}=\arg\Omega_{X,h}+\pi$.

The interpretation of $\Omega$ in relation to the framework discussed in Section~\ref{subsec:generalframework}, is that it is the 1-form $dS_{\bC}$, the derivative of the complexified K\"ahler potential $S_{\bC}$ along $\mathrm{Met}(X)$.
Moreover, the central charge, mass (measure), and flow are determined by $\Omega$ as follows:
\begin{itemize}
\item 
The central charge of $X$ should be $Z(X)=\Omega(1)$, where $1\in\cA$ is the unit.
\item We define the mass of $(X,h)$ as $m(X,h)\coloneqq \Omega\left(e^{-\sqrt{-1}\arg\Omega}\right)$ and, more generally, the mass measure by
\[
\int_{\bR}fd\mu_{X,h}=\Omega\left(e^{-\sqrt{-1}\arg\Omega}f\!\left(\frac{\arg\Omega}{\pi}\right)\right)
\]
for (at least) polynomial functions $f$.
\item The proposed minimizing flow on $\mathrm{Met}(X)$ is given by $-\arg\Omega\in \cA_{\mathrm{real}}$.
\end{itemize}

\

To illustrate and motivate all the above, we spell out what the structures are in two basic cases.

\begin{example}
The simplest example is the following. 
Take $\cC=\mathrm{Perf}(\bC)$, the category of finite-dimensional chain complexes over $\bC$.
For $X\in\mathrm{Perf}(\bC)$, an element of $\mathrm{Met}(X)$ is represented by a chain complex $(E^\bullet,d)$ where each $E^n$ is a Hermitian vector space, together with an isomorphism ($X$-framing) $\alpha\colon X\to (E^\bullet,d)$. 
Two such $X$-framed Hermitian chain complexes define the same element of $\mathrm{Met}(X)$ if there is an isometric isomorphism of complexes, compatible with the $X$-framings, between them.
The $*$-algebra $\cA$ attached to a metrized complex $(E^\bullet,d)$ is the algebra of degree-preserving endomorphisms of the underlying graded vector space $E^\bullet$, and $*$ is the usual Hermitian adjoint. It is thus simply a product of matrix algebras.

Fix $z=Z(\bC)\in\mathbb R_{>0}e^{\sqrt{-1}(0,\pi]}$, i.e.\ $\arg z\in(0,\pi]$, then
\[
\Omega(a)=(-1)^nz \operatorname{tr}\left(e^{-dd^*-d^*d}a\right),\qquad (\arg\Omega) a=(\arg{z}-\pi n)a
\]
on the summand $\mathrm{End}(E^n)$ of $\cA=\mathrm{End}^0(E^\bullet)$ and extended by linearity.
The flow is thus just rescaling (by different scalars) the metric on each graded component, under which the singular values of $d$ increase to $+\infty$, which we interpret as converging to the harmonic subcomplex $\mathrm{Ker}(d)\cap\mathrm{Ker}(d^*)\subseteq E^\bullet$.
This example is generalized to complexes of representations of quivers in Section~\ref{subsec:quivercomplexes}.
\end{example}

\begin{example}
To a compact Lagrangian submanifold $L\subset M$ (which we think of, heuristically, as a metric on the corresponding object in the Fukaya category) we attach the $*$-algebra $\cA_{L}=C^\infty(L,\bC)$ of smooth, complex-valued functions on $L$. Given a suitable middle-degree complex-valued form $\Omega$ on $M$ as in Section~\ref{subsec:complex-forms}, the corresponding functional on $\cA_L$ is $f\mapsto \int_L f\Omega$.\footnote{To avoid possible confusion, we note that the base field throughout most of Section~\ref{section:heuristics} is the non-Archimedean field $Nov$, as opposed to $\bC$.}
\end{example}

\subsection{Metrized objects in the non-Archimedean case} \label{subsec:frameworknonarch}

In non-Archimedean examples, there is a \textit{category of metrized objects}, $\widetilde{\cC}$, which comes equipped with a rescaling action by $(\bR,+)$ (or some submonoid thereof). Taking the quotient $\cC\coloneqq \widetilde{\cC}/\mathrm{rescaling}$ yields the category on which we are considering a categorical K\"ahler structure and stability condition.
The functor $\mathrm{Met}$ then arises by taking fibers of the localization functor $\widetilde{\cC}\to\cC$. 
Thus, in the non-Archimedean setting, the set $\mathrm{Met}(X)$ is secondary.
The primary object is the category $\widetilde{\mathcal C}$ of metrized or integral representatives.
Unlike in the Archimedean case, however, we will not relate the category of metrized objects to mass measures, K\"ahler potentials, or flow.
For simplicity, we work at the level of homotopy categories, i.e.\ with triangulated categories.

\begin{definition}
A \textbf{category of metrized objects} for a triangulated category $\cC$ is a triangulated category $\widetilde{\cC}$ with an exact functor $L\colon\widetilde{\cC}\to\cC$. 
Furthermore, $\widetilde{\cC}$ comes with an action of the monoidal poset category $\cP_\Lambda$ of a submonoid $\Lambda\subseteq (\bR,+)$ with the usual partial order $\leq$ induced from $\bR$. 
More precisely, we are given a monoidal functor $R\colon \cP_\Lambda\to \mathrm{Fun}^{\mathrm{ex}}(\widetilde{\cC},\widetilde{\cC})$, $R(t)\eqqcolon R_t$, $R(s\to t)\eqqcolon \lambda_{s,t}$.
Let $\cN\subseteq\widetilde{\cC}$ be the thick triangulated subcategory generated by cones of morphisms 
\[
(\lambda_{s,t})_X\in\mathrm{Hom}_{\widetilde{\cC}}(R_s(X),R_t(X)), \qquad s\leq t\in\Lambda, X\in\widetilde{\cC}.
\]
Then $L$ should induce an equivalence of triangulated categories $\widetilde{\cC}/\cN\to \cC$, where $\widetilde{\cC}/\cN$ is the Verdier quotient.
\end{definition}

Suppose $\widetilde{\cC}$ is a category of metrized objects as in the above definition.
The corresponding functor $\mathrm{Met}$ is constructed as follows.
An element of $\mathrm{Met}(X)$ is represented by a pair
\[
(\widetilde X,x),\qquad \widetilde X\in\widetilde{\mathcal C},\quad x\colon X\xrightarrow{\simeq}L(\widetilde X)
\]
where $x$ is an isomorphism.
Two pairs $(\widetilde X,x)$, $(\widetilde Y,y)$ are identified in $\mathrm{Met}(X)$ if there is an isomorphism $f\colon\widetilde X\to\widetilde Y$ with $L(f)\circ x=y$.
In other words, $\mathrm{Met}(X)$ is the set of isomorphism classes of objects of the 2-fiber of the functor $L$ over $X$.
Given an isomorphism $g\colon X\to Y$, let $\mathrm{Met}(g)(\widetilde{X},x)\coloneqq (\widetilde{X},x\circ g^{-1})$.

The action of $\cP_\Lambda$ on $\widetilde{\cC}$ induces an action of $\Lambda$ on each $\mathrm{Met}(X)$: 
\[
t\cdot (\widetilde{X},x)\coloneqq \begin{cases} (R_t\widetilde{X},L(\lambda_{0,t})\circ x)& \text{if } t\geq 0,\\
(R_t\widetilde{X},L(\lambda_{t,0})^{-1}\circ x) & \text{if } t\leq 0, \end{cases}
\]
where $t\in\Lambda$, $(\widetilde{X},x)\in\mathrm{Met}(X)$.
Here, we should identify $\Lambda\subseteq (\bR,+)$ with $\exp(\Lambda)\subseteq (\bR_{>0},\cdot)$ to match the multiplicative rescaling action as in Section~\ref{subsec:generalframework}.

\begin{example}
Let $K$ be a non-Archimedean field with value group $\Lambda\subseteq \bR$ and valuation ring $\cO$ (see Section~\ref{subsec:prelimnorm}).
Assume that we are given a group homomorphism $\sigma\colon\Lambda\hookrightarrow K^{\times}$ which splits the valuation $\nu\colon K^\times\twoheadrightarrow\Lambda$.
Suppose that $\widetilde{\cC}$ is linear over $\cO$.
Then there is a natural action, $R$, of $\cP_\Lambda$ on $\widetilde{\cC}$ where each $R_t$ is the identity functor and $(\lambda_{a,b})_X=\sigma(b-a)1_X$.
In this case, $\widetilde{\cC}/\cN\cong \cC$ is just the general fiber $\widetilde{\cC}\otimes_{\cO}K$.

In the particular case where $K=Nov$ as in Section~\ref{section:heuristics}, we have $\Lambda=\bR$, $\sigma(a)=T^a$, $\cO=Nov_+$.
The hypothetical $Nov_+$-linear Fukaya-type category $\cF^\mathrm{global}_{L,+}(X,\omega)$ outlined in Section~\ref{sec:singular} is an important motivating example.
\end{example}

\begin{example}
Given a \textit{triangulated persistence category} $\cC$ in the sense of~\cite{BCZ_TPFC}, its 0-level category $\cC_0$ comes with an action of $\cP_\bR$ by \textit{shift functors}, and can thus be considered as a category of metrized objects for its localization, the $\infty$-level category $\cC_\infty$.
\end{example}

\subsection{Local structure in the Archimedean case}\label{subsec:lozenge}

In this subsection we recall the definition of \textit{Lozenge algebras} from our paper~\cite{HKKPitlogs2}.
The original motivation for this notion was to unify (1) representations of quivers over $\bC$ and (2) holomorphic bundles over Riemann surfaces.
It turns out that there are many more examples which fit into this formalism, and we show this by comparing Lozenge algebras with the framework from~\cite{BK21} by Bhattacharya and the third named author, who discuss a much wider range of examples, including Nekrasov's noncommutative instantons.
Our expectation is that this framework captures, at least approximately, the local picture of ``categorical K\"ahler geometry'' over $\bC$.

Lozenge algebras have the following list of ingredients:
\begin{enumerate}[(L1)]
    \item A graded algebra $A=A^0\oplus A^1\oplus A^2$ over $\bC$ (unital, associative), concentrated in degrees $0,1,2$.
    \item A derivation $d\colon A^\bullet\to A^{\bullet+1}$, i.e.\ a linear map satisfying the graded Leibniz rule. (We do not require $d^2=0$.)
    \item An element $F\in A^2$, the \textit{curvature}, such that
    \[ d^2=[F,\_], \]
    i.e.\ $A$ is a \textit{curved dg-algebra}. (In~\cite{HKKPitlogs2}, the curvature was denoted by ``$\theta$'', however we use ``$F$'' here instead to avoid confusion with $\theta$ in the context of $\theta$-stability.) 
    \item A functional $\tau\colon A^2\to\bC$ such that
    \[ \tau([a,b])=0,\qquad \tau(da)=0\]
    and such that the induced pairings $A^n\otimes A^{2-n}\to\bC$, $a\otimes b\mapsto\tau(ab)$ are non-degenerate.
    Here and throughout, $[a,b]=ab-(-1)^{|a||b|}ba$ is the supercommutator.
    \item A conjugate-linear, degree-preserving\footnote{This goes against our convention that $|a^*|=-|a|$, as stated in the introduction. The resolution is that $A$ is really bigraded by (L6) and $*$ preserves the total degree $p+q$ while inverting $p-q$.} map $*\colon A\to A$ such that
    \[
    a^{**}=a,\qquad (ab)^*=(-1)^{|a||b|}b^*a^*,\quad da^*=(da)^*,\qquad F^*=-F,\qquad \tau(a^*)=\overline{\tau(a)}.
    \]
    \item A splitting $A^1=A^{1,0}\oplus A^{0,1}$ as an $A^0$-bimodule with
    \[
    A^{0,1}=\left(A^{1,0}\right)^*,\qquad \left(A^{1,0}\right)^2=0
    \]
    and such that the Hermitian pairing
    \[
    \overline{A^{0,1}}\otimes A^{0,1}\to\bC,\qquad a\otimes b\mapsto \sqrt{-1}\tau(a^*b)
    \]
    is positive definite, i.e.\ $\sqrt{-1}\tau(a^*a)\geq 0$ and $>0$ for $a\neq 0$.
    \item An element $\omega\in A^2$ such that the map $L\colon A^0\to A^2$, $a\mapsto\omega a$ is an isomorphism and such that the pairing 
    \[
    \overline{A^{0}}\otimes A^{0}\to\bC,\qquad a\otimes b\mapsto \tau(\omega a^*b)
    \]
    is positive definite. We write the inverse of $L$ as $\Lambda\colon A^2\to A^0$.
\end{enumerate}
Some comments on this list are in order:
\begin{itemize}
\item 
Given (L1-L7) we can impose the \textit{Yang--Mills (YM) condition}:
\[
F\text{ and }\omega\text{ are central, }\qquad d\Lambda F=0.
\]
A \textbf{Lozenge algebra}, as defined in~\cite{HKKPitlogs2}, is the structure (L1-L7) satisfying the YM condition and so that $H^\bullet A$ is finite-dimensional.
Note that we can only define $H^\bullet A$ if $F$ is central and thus $d^2=0$.

\item
Given (L1-L3), i.e.\ a curved dg-algebra we can impose \textit{harmonicity}, meaning $F=0$, which clearly implies the YM condition, assuming $\omega$ is central.

\item
Comparing this to the classical notion of a K\"ahler manifold, (L1-L4) are the analogue of a complex symplectic structure, (L5) defines a real slice, thus a real symplectic structure, and (L6-L7) define the complex structure.

\item
The relation to the structures proposed in Section~\ref{subsec:frameworkarch} is as follows: The (ungraded) $*$-algebra is $\cA\coloneqq A^0$ and the linear functional $\Omega\colon\cA\to\bC$ is $\Omega(a)=\sqrt{-1}\tau((F+\omega)a)$.
\end{itemize}

\begin{remark}\label{rem:A2dual}
Suppose we are given (L1-L6). 
By assumption, we have an injection $A^2\hookrightarrow(A^0)^\vee$, $a\mapsto\tau(a\_)$, into the $\bC$-dual of $A^0$.
It is not difficult to see that one can replace $A^2$ by $(A^0)^\vee$ and extend all the structure (L1-L6) in a canonical way.
For example, the extension of $\tau$ is given by $\mathrm{ev}_1$, evaluation at the unit $1\in A^0$, left action of $A^0$ on $(A^0)^\vee$ is the dual of right action on $A^0$, and vice versa.
Of course, (L7) will not extend unless $\dim A^0<\infty$.
\end{remark}

\begin{example}
Suppose $X$ is a compact Riemann surface with K\"ahler form $\omega$ and $E$ is a holomorphic bundle over $X$ with Hermitian metric.
This fits into the above framework as follows.
\begin{enumerate}[(L1)]
\item $A=\mathcal A^\bullet(X,\mathrm{End}(E))$ is the algebra of (smooth) forms on $X$ with values in the bundle $\mathrm{End}(E)$. The product is the usual combination of the wedge product of forms and the composition of endomorphisms.
\item $d$ comes from the Chern connection $\nabla$ on $E$, the unique connection compatible with the holomorphic and metric structures on $E$.
\item $F=F_\nabla$ is the curvature of $\nabla$. 
\item $\tau(a)=\int_X\operatorname{tr}(a)$.
\item $*$ is the pointwise Hermitian conjugate with respect to the metric on $E$.
\item $A^{p,q}=\mathcal A^{p,q}(X,\mathrm{End}(E))$ is the decomposition into forms of type $(1,0)$ and type $(0,1)$.
\item $\omega\in A^2$ is the K\"ahler form.
\end{enumerate}
For the above data, the YM condition becomes the \textit{constant central curvature} condition on the metric on $E$, or rather $F_\nabla$.
By a celebrated theorem of Narasimhan--Seshadri, a holomorphic bundle $E$ has a metric satisfying the YM condition if and only if $E$ is a sum of stable bundles of equal slopes.
\end{example}

\begin{example}
A more elementary example comes from quiver representations. 
As argued in~\cite[Section 3.2]{HKKPitlogs2}, this exhausts all finite-dimensional examples of Lozenge algebras, up to isomorphism.
We use here definitions and notation for quivers and their representations from Section~\ref{subsec:quivers}.

Suppose $Q$ is a quiver, $(E,\rho)$ a metrized representation of $Q$ over $\mathbb C$, $z_i\in\bC$, $i\in Q_0$ with $\operatorname{Im}(z_i)>0$, and $\phi\in (0,1)$.
From this we construct (L1-L7) as follows.
\begin{enumerate}[(L1)]
\item $A^0=A^2=\bigoplus_{i\in Q_0}\operatorname{End}(E_i)$, $A^{0,1}=\bigoplus_{a\colon i\to j}\Hom(E_i,E_j)$, $A^{1,0}=\bigoplus_{a\colon i\to j}\Hom(E_j,E_i)$.
\item $d=([\rho(a),\_],[\rho(a)^*,\_])\colon A^0\to A^1$ (where $\rho(a)^*=\sqrt{-1}\rho(a)^\dagger$, see (L5) below).
\item $F=\sum_{a\in Q_1}[\rho(a)^*,\rho(a)]+\sqrt{-1}\left(-\operatorname{Re}(z)+\cot\pi\phi\operatorname{Im}(z)\right)$, where $z=\sum_{i\in Q_0}\rho(e_i)z_i\in A^2$.
\item $\tau(a)=\operatorname{tr}(a)$.
\item If $\dagger$ is the usual Hermitian adjoint with respect to the Hermitian metrics on the $E_i$, then $a^*=a^\dagger$ for $|a|$ even, and $a^*=\sqrt{-1}a^\dagger$ for $|a|$ odd, cf.\ \eqref{eq:superadjoint} below.
\item See (L1).
\item $\omega=\operatorname{Im}(z)\in A^2$.
\end{enumerate}
The relation between the harmonicity condition $F=0$ and polystability of $(E,\rho)$ goes back to A.~King and is discussed extensively in Section~\ref{sec:quivers_arch}.
\end{example}

Next, we discuss three constructions which have as input and output the data (L1-L7). Together, they are the analogue of passing from the trivial line bundle with its canonical connection to arbitrary bundles with connection. The YM and harmonicity conditions are typically not preserved.

\begin{itemize}
    \item \textbf{Matrix algebras.} Fix $n\geq 1$.
Here we pass from $A$ to $\widetilde{A}=\mathrm{Mat}(n\times n,A)=\mathrm{End}_A(A^n)$.
All the structure (L1-L7) extends in the obvious way, for example $\tilde{\tau}(a)=\tau(\operatorname{tr}(a))$, and if $a$ has coefficients $a_{ij}$ then $a^*$ has coefficients $a_{ji}^*$.

\item \textbf{Twisting.} Here we choose $\delta\in A^1$ with $\delta^*=-\delta$.
\begin{enumerate}[(L1)]
    \item $\widetilde{A}=A$ as graded associative algebra
    \item $\tilde{d}=d+[\delta,\_]$
    \item $\tilde{F}=F+d\delta+\delta^2$, then
    \begin{align*}
        \tilde{d}^2a &=\tilde{d}(da+[\delta,a]) \\
        &=d^2a+[\delta,da]+[d\delta,a]-[\delta,da]+[\delta,[\delta,a]] \\
        &=[F,a]+[d\delta,a]+[\delta^2,a]=[\tilde{F},a]
    \end{align*}
    \item $\tilde{\tau}=\tau$, we have $\tilde{\tau}(\tilde{d}a)=\tau(da+[\delta,a])=0$
    \item $\tilde{*}=*$, we have $\tilde{F}^*=F^*+d\delta^*-(\delta^*)^2=-F-d\delta-\delta^2=-\tilde{F}$
    \item $\widetilde A^{p,q}=A^{p,q}$
    \item $\widetilde{\omega}=\omega$
\end{enumerate}
The harmonicity condition $\tilde{F}=0$ is also known as the \textit{Maurer--Cartan equation} for $\delta$.

\item \textbf{Idempotents.} Suppose $P\in A^0$ with $P^2=P$, $P^*=P$, and $[P,\omega]=0$.
\begin{enumerate}[(L1)]
\item $\widetilde{A}=P AP\subseteq A$ as graded associative subalgebra with unit $P$
\item $\tilde{d}a=P daP$, then
\begin{align*}
\tilde{d}(ab) &= P dabP+(-1)^{|a|}P adbP=P daP b+(-1)^{|a|}aP dbP=\tilde{d}ab+(-1)^{|a|}a\tilde{d}b
\end{align*}
and in particular $\tilde{d}P=0$
\item $\tilde{F}=P\left(F+(dP)^2\right)P$. We need to check $\tilde{d}^2=[\tilde{F},\_]$.
First note that
\[ P dP da P=(dP)^2a,\qquad P dadPP=(-1)^{|a|}a(dP)^2\]
thus
\[ \tilde{d}^2a=P d(P da P)P=P dP da P+P d^2a P-(-1)^{|a|}P dadPP=[\tilde{F},a].\]
\item $\tilde{\tau}=\tau|_{\widetilde{A}}$. To check $\tilde{\tau}(\tilde{d}a)=0$ we note that $dP=P dP+dPP$ and thus for $a\in\widetilde{A}$:
\[ \tau(dP a)=\tau(P dP a)+\tau(dPP a)=2\tau(dP a)\]
hence $\tau(dP a)=0$ and 
\[ \tilde{\tau}(\tilde{d}a)=\tau(P daP)=\tau(P da)=-\tau(dP a)=0. \]
\item $\tilde{*}=*|_{\widetilde{A}}$
\item $\widetilde A^{p,q}=P A^{p,q}P$
\item $\widetilde{\omega}=P\omega P$
\end{enumerate}

\end{itemize}

\subsubsection*{Comparison with the formalism of Bhattacharya--Kontsevich}

We recall the setup from~\cite{BK21}, changing some notation to avoid collisions.
The first (``algebra'') part consists of:
\begin{enumerate}
\item[(A1)] $\cA$, an associative unital $*$-algebra over $\bC$,
\item[(A2)] $\Omega^1$, an $\cA$-bimodule,
\item[(A3)] $D\colon \cA\to\Omega^1$, a derivation,
\item[(A4)] $h\colon \Omega^1\otimes_{\bC}\overline{\Omega^1}\to\bC$, a bilinear form with
\[ h(x,\overline{y})=\overline{h(y,\overline{x})},\qquad h(axb,\overline{y})=h(x,\overline{a^*yb^*})\]
and which is positive definite,
\item[(A5)] $\eta\colon \cA\to\bC$, a linear functional with $\eta(a^*)=-\overline{\eta(a)}$ and
\begin{equation}\label{eq:BKeta}
\eta([a,b])=\frac{\sqrt{-1}}{2}\left(h(Da,\overline{Db^*})-h(Db,\overline{Da^*})\right).
\end{equation} 
\end{enumerate}
The second (``module'') part of the setup is:
\begin{enumerate}[(M1)]
\item $E$, a finitely generated projective $\cA$-module,
\item $\nabla\colon E\to\Omega^1\otimes_{\cA}E$, a linear map satisfying the Leibniz rule
\[ \nabla(ax)=Da\otimes x+a\otimes\nabla x,\]
\item $\mathsf{H}\colon E\otimes_{\bC}\overline{E}\to\cA$, a bilinear form such that
\[ \mathsf H(ax,\overline{by})=a\mathsf H(x,\overline{y})b^* \]
and so that the induced map
\[ \overline{E}\to\Hom_{\cA}(E,\cA),\qquad \overline{y}\mapsto\mathsf H(\_,\overline{y})\]
is an isomorphism.
Furthermore, $\mathsf H$ should be positive definite in the sense that $\mathsf H(x,\overline{x})$ can be written as a finite sum of elements of the form $aa^*$, $a\in \cA$.
\end{enumerate}

We remark that given (A1-A5), there is a default choice for (M1-M3) which is $E=\cA$, $\nabla=D$, and $\mathsf H(a,\overline{b})=ab^*$.
Conversely, given arbitrary (M1-M3) one can reduce to this basic rank 1 case by the analogue of the operations (matrix algebra, twisting, idempotent) in this formalism.

\begin{proposition}\label{prop:LtoBK}
Given the structure (L1-L6) one obtains (A1-A5) via the following:
\begin{enumerate}[(A1)]
\item $\cA=A^0$ as $*$-algebra
\item[(A2)] $\Omega^1=A^{0,1}$ as $\cA$-bimodule
\item[(A3)] $D=\mathrm{pr}_{0,1}\circ d$, where $\mathrm{pr}_{0,1}\colon A^1\to A^{0,1}$ is the projection.
\item[(A4)] $h(x,\overline{y})=\sqrt{-1}\tau(xy^*)$
\item[(A5)] $\eta(a)=-\frac{1}{2}\tau(F a)$
\end{enumerate}
\end{proposition}

\begin{proof}
The only non-trivial item to check is the identity~\eqref{eq:BKeta} for $\eta$.
We calculate:
\begin{align*}
-2\eta([a,b]) &= \tau(F[a,b]) \\
&=\tau([F,a]b) \\
&=\tau(d^2ab)\\
&=\tau(dadb)\\
&=\tau\left((Da+(Da^*)^*)(Db+(Db^*)^*)\right) \\
&=\tau(Da(Db^*)^*-Db(Da^*)^*) \\
&=-\sqrt{-1}\left(h(Da,\overline{Db^*})-h(Db,\overline{Da^*})\right)
\end{align*}
where we have used $da=Da+(Da^*)^*$ which is a consequence of $da^*=(da)^*$.
\end{proof}

\begin{proposition}\label{prop:BKtoL}
Given the structure (A1-A5) one obtains (L1-L6) via the following:
\begin{enumerate}[(L1)]
\item $A^0=\cA$, $A^{0,1}=\Omega^1$, $A^{1,0}=\overline{\Omega^1}$, $A^2=\cA^\vee$.
The product $A^1\otimes A^1\to A^2$ is defined by
\begin{gather*}
A^{0,1}\otimes A^{1,0}\to A^2, \qquad x\otimes \overline{y}\mapsto -\sqrt{-1}h(\_ x,\overline{y}),     \\
A^{1,0}\otimes A^{0,1}\to A^2, \qquad \overline{x}\otimes y\mapsto \sqrt{-1}h(\_ y,\overline{x}).     
\end{gather*}
\item The differential $d\colon A^0\to A^1$ is $da=Da+\overline{Da^*}$ and the differential $d\colon A^1=A^{0,1}\oplus A^{1,0}\to A^2$ is 
\[
d(x,\overline{y})=\left(a\mapsto-\sqrt{-1}h(x,\overline{Da^*})+\sqrt{-1}h(Da,\overline{y})\right)
\]
where $x,y\in\Omega^1$ and $a\in\cA$.
\item $F=-2\eta\in\cA^\vee$.
\item $\tau=\operatorname{ev}_1\colon \cA^\vee\to\bC$.
\item The $*$-structure on $A^0$ and $A^2$ comes from the $*$-structure on $\cA$, and the $*$-structure on $A^1$ is given by $x^*=\overline{x}$ for $x\in A^{0,1}$.
\item See (L1).
\end{enumerate}
\end{proposition}

\begin{proof}
There is a long list of things to check, but the verification is completely straightforward in each case.
For example, the identity $d^2=[F,\_]$ is seen as follows:
\begin{align*}
(d^2a)(b)=-\sqrt{-1}h(Da,\overline{Db^*})+\sqrt{-1}h(Db,\overline{Da^*})=-2\eta([a,b])=([F,a])(b)
\end{align*}
Here we have used~\eqref{eq:BKeta}.
\end{proof}

The two constructions from Proposition~\ref{prop:LtoBK} and Proposition~\ref{prop:BKtoL} are almost inverse to each other, up to isomorphism, except for the following:
Going from (L1-L6) to (A1-A5) and then back to (L1-L6) replaces $A^2$ by $(A^0)^\vee$. See also Remark~\ref{rem:A2dual}.

\begin{remark}
The lozenge algebra formalism reveals the strange identity~\eqref{eq:BKeta} to be essentially the much more familiar $d^2=[F,\_]$ in disguise. This is at the expense of introducing the additional summands $A^{1,0}$ and $A^2$.
\end{remark}

Given (A1-A5) and (M1-M3), a certain moment map is constructed in~\cite{BK21}, the vanishing of which is a ``generalized King's equation'' whose solutions are the \textit{harmonic} connections.
The moment map is written there in two distinct ways, first in terms of a presentation of $E$ as a summand of a finitely generated free module~\cite[Proposition 2]{BK21}, and second in terms of a triple tensor product~\cite[Definition 5.2]{BK21}.
In the language of Lozenge algebras, the moment map is just $u\mapsto -\frac{1}{2}\tau(Fu)$, where $u\in A^0$ with $u=-u^*$, and so its vanishing is equivalent to the vanishing of the curvature $F$.
We will check this in the simplest case, where the bimodule $E$ is free of rank one with arbitrary connection $\nabla=D+a$, $a\in\Omega^1$.
The moment map is then (see~\cite[Equation (5.16)]{BK21})
\begin{align*}
H_u(a)&=\eta(u)+\frac{1}{2\sqrt{-1}}\left(-h(a,\overline{Du})+h(Du,\overline{a})-h([a,u],\overline{a})\right) \\
&=-\frac{1}{2}\left(\tau(Fu)+\tau(a(Du)^*)-\tau(Dua^*)-\tau((a^*a+aa^*)u)\right) \\
&=-\frac{1}{2}\tau\left(Fu+(a^*-a)(Du+(Du^*)^*)+(a^*-a)^2u\right) \\
&=-\frac{1}{2}\tau\left(Fu+d\delta u+\delta^2u\right)=-\frac{1}{2}\tau\left(\tilde{F}u\right)
\end{align*}
where $\delta=a^*-a$ and we used $\tau(d\delta u)=\tau(\delta du)$.
In particular, the value of the moment map at $a\in\Omega^1$, the functional $u\mapsto H_u(a)$, vanishes iff $\tilde{F}=0$, so the two notions of harmonicity agree.

\section{Representations of quivers: Archimedean case}
\label{sec:quivers_arch}

In this section we discuss aspects of categorical K\"ahler geometry in  the case of categories of representations of quivers over the complex numbers.
This setting serves as a toy-model for differential-geometric examples involving challenging analysis and infinite-dimensional spaces.
In Section~\ref{subsec:quivers} we discuss the relation with the general framework proposed in Section~\ref{section:axiomatic}. 
This relies on a theorem of A.~King, which is analogous to the DUY theorem mentioned in the introduction, and provides a K\"ahler metric on moduli spaces of polystable representations. 
In Section~\ref{subsec:staralg} we reformulate that result in terms of $*$-algebras and universal C$^*$-algebras. 
The invariance of these constructions under Bernstein--Gel'fand--Ponomarev reflection functors --- basic examples of functors inducing derived Morita equivalences --- is shown in Section~\ref{subsec:BGP}. 
In Section~\ref{subsec:quivercomplexes} we propose an extension of mass and flow to complexes of representations which is inspired by the Chern--Weil theory of superconnections.
Finally, in Section~\ref{subsec:rampfunctions}, we show that a choice of increasing convex function for each arrow of the quiver provides a more general source of K\"ahler metrics on spaces of polystable representations.

\subsection{Relation to general framework}\label{subsec:quivers}

In this subsection we explain how categories of quiver representations provide examples of the general framework discussed in Section~\ref{section:axiomatic}. Here and throughout most of Section~\ref{sec:quivers_arch}, we focus on the abelian category of representations, postponing a discussion of the more tentative extension to complexes to Section~\ref{subsec:quivercomplexes}.

We fix some notation.
A \textbf{quiver}, $Q$, consists of a set $Q_0$ of \textit{vertices}, a set $Q_1$ of \textit{arrows}, and maps $s,t\colon Q_1\to Q_0$ mapping an arrow to its source and target, respectively.
We often write $a\colon i\to j$ to mean that $a\in Q_1$, $i=s(a)$, $j=t(a)$.
For simplicity, we only consider quivers with $Q_0$ and $Q_1$ finite.
A \textit{path} in $Q$ is either a sequence of arrows $a_n\cdots a_2 a_1$ with $s(a_{i+1})=t(a_i)$ or a \textit{constant path}, $e_i$, at a vertex $i\in Q_0$.
Fixing a coefficient field $\mathbf k$, the set of paths forms a basis of the \textit{path algebra} $\mathbf k Q$ of $Q$, with product defined by concatenation of paths.
With the above conventions, left $\mathbf k Q$-modules, $(E,\rho\colon \mathbf k Q\to\mathrm{End}(E))$, are the same thing as representations of $Q$ given by vector spaces $E_i=\rho(e_i)E$ and linear maps $\rho(a)\colon E_{s(a)}\to E_{t(a)}$.
We let $\mathbf{Rep}(Q,\mathbf k)$ denote the abelian category of finite-dimensional representations of $Q$ over $\mathbf k$. 
For the rest of this section, $\mathbf k=\mathbb C$, and $\mathbf{Rep}(Q)=\mathbf{Rep}(Q,\mathbb C)$.

Given choices of numbers $z_i\in\bC$, $i\in Q_0$, with $\operatorname{Im}(z_i)>0$ or $z_i\in(-\infty,0)$, we define a central charge 
\[
Z\colon K_0(\mathbf{Rep}(Q))\to\bC,\qquad Z(E)=\sum_{i\in Q_0}\dim(E_i)z_i.
\]
This provides each $0\neq E\in \mathbf{Rep}(Q)$ with a phase $\phi(E)\coloneqq\arg(Z(E))/\pi\in (0,1]$ and a notion of (slope) stability (see Section~\ref{subsec:stabab}).
Any such central charge automatically satisfies the Harder--Narasimhan and support properties, since $\mathbf{Rep}(Q)$ is finite-length with finitely many simples.

Fixing vector spaces $E_i$, $i\in Q_0$, the global quotient stack
\begin{equation}\label{eq:repquotstack}
\left[\bigoplus_{a\colon i\to j}\Hom(E_i,E_j)\bigg/\prod_{i\in Q_0}GL(E_i)\right]\eqqcolon [X/G]    
\end{equation}
is the moduli stack of representations with dimension vector $(\dim E_i)_{i\in Q_0}\in \bZ_{\geq 0}^{Q_0}$.
Since $X$ is an affine space, line bundles on $[X/G]$, equivalently $G$-linearized line bundles on $X$, are given by characters of $G$, which are of the form
\[
G\ni (g_i)_{i\in Q_0}\mapsto \prod_{i\in Q_0} \det(g_i)^{\theta_i}\in\bC^\times
\]
where $\theta_i\in\bZ$, $i\in Q_0$. 
More generally, $(\theta_i)_{i\in Q_0}\in\bR^{Q_0}$ corresponds to a virtual character/line bundle $L_\theta\in \operatorname{Pic}([X/G])\otimes_\mathbb Z\mathbb R$.
In particular, the central charge $Z$ corresponds to a pair of virtual line bundles $L_{\mathrm{Re}(Z)}$, $L_{\mathrm{Im}(Z)}$ on the ind-stack of objects of $\mathbf{Rep}(Q)$.

A \textbf{metric} on a representation $E=(E,\rho)$ is simply a choice of Hermitian inner product, $h_i$, on each vector space $E_i$, $i\in Q_0$.
Fixing a reference metric on $E$, the space $\mathrm{Met}(E)$ of metrics on $E$ is thus identified with the symmetric space
\[
\prod_{i\in Q_0}GL(E_i)/U(E_i).
\]
The tangent space at the point $h$ is the subspace of self-adjoint elements in the $*$-algebra $\cA_{E,h}\coloneqq\prod_{i\in Q_0}\mathrm{End}(E_i)$, where $a^*=a^{*_h}=h^{-1}\bar{a}^Th$ is the Hermitian conjugate of $a$ with respect to $h$.
The linear functional on $\cA_{E,h}$ which represents the first variation of the complexified K\"ahler potential is $a\mapsto\mathrm{tr}(\Omega a)$, where
\begin{equation}\label{eq:omega_qrep}
\Omega\coloneqq \sum_{a\in Q_1}\left[\rho(a)^*,\rho(a)\right]+\sum_{i\in Q_0}z_i\rho(e_i).
\end{equation}
At this point we should also require $\operatorname{Im}(z_i)>0$ to ensure $\Omega$ is non-singular.
Note that $\Omega$ is normal as the sum of a self-adjoint and a central operator, thus expressions like $\arg\Omega$ are well-defined by functional calculus.
The complexified K\"ahler potential itself is
\[
S_{\bC}(h)\coloneqq -\sum_{a\in Q_1}\tr\left(\rho(a)^{*_h}\rho(a)\right)+\sum_{i\in Q_0}z_i\log \det h_i
\]
and depends on a choice of volume form on each $E_i$ to make sense of $\det h_i$.

Fix a phase $\phi\in (0,1)$, then the harmonicity condition $\mathrm{arg}(\Omega)=\pi\phi$, which we can also write as $\operatorname{Re}(\Omega)=\cot\pi\phi\operatorname{Im}(\Omega)$, becomes
\begin{equation}\label{kings_equation}
\sum_{a\in Q_1}\left[\rho(a),\rho(a)^{*_h}\right]=\sum_{i\in Q_0}\theta_i\rho(e_i),
\end{equation}
where
\[
\theta_i\coloneqq \operatorname{Re}z_i-\cot\pi\phi\operatorname{Im}z_i.
\]
The existence of solutions to~\eqref{kings_equation} is equivalent to polystability of the representation $E$ by the following.

\begin{theorem}[{\cite[Proposition 6.5]{king_reps}}]\label{thm:King}
Let $Q$ be a quiver and $\theta_i\in\bR$, $i\in Q_0$.
A representation $(E,\rho)\in\mathbf{Rep}(Q)$ is $\theta$-polystable (i.e.\ a direct sum of $\theta$-stable representations) if and only if equation~\eqref{kings_equation} holds for some choice of Hermitian metrics on the $E_i$.
\end{theorem}

\begin{proof}[Proof (sketch):]
Fix finite-dimensional vector spaces $E_i$, $i\in Q_0$ and consider the quotient stack $[X/G]$ as in~\eqref{eq:repquotstack} with the polarization (line bundle) given by $\theta$ as before.
We may also consider the geometric invariant theory (GIT) quotient $X\sslash_\theta G$, which is more generally defined for reductive groups acting on projective-over-affine varieties.
Using the Hilbert--Mumford criterion for stability, together with a variant of Lemma~\ref{lem:thetastablefilt} for $\bZ$-filtrations, one shows that a $G$-orbit in $X$ is polystable in the sense of GIT iff the corresponding representation is $\theta$-polystable, i.e.\ $X\sslash_\theta G$ parametrizes $\theta$-polystable representations (with chosen $\dim E_i$).
On the other hand, by the Kempf--Ness theorem, $X\sslash_\theta G=\mu^{-1}(\theta)/U$ where we have chosen Hermitian metrics on the $E_i$, $U\coloneqq\prod_{i\in Q_0}U(E_i)$ is a maximal compact subgroup of $G$, and $\mu\colon X\to \mathfrak{u}^*$ is the moment map of the $U$ action on $X$, where $X$ is endowed with the translation invariant K\"ahler metric.
The equation $\mu(\rho)=\theta$ is precisely~\eqref{kings_equation}.
\end{proof}

\begin{example}
The special case of the above theorem where $Q$ is the one-loop quiver is familiar from linear algebra: An operator $T$ on $\bC^n$ is \textit{normal}, i.e.\ satisfies $[T,T^*]=0$, with respect to some choice of Hermitian inner product on $\bC^n$, iff it is diagonalizable (semisimple as representation). 
\end{example}

Keep the notation as in the above proof. Since $U$ preserves the (constant) K\"ahler metric on $X$, the quotient $\mathcal{M}(Q,d,\theta)=\mu^{-1}(\theta)/U$, which is identified with the moduli of $\theta$-polystable representations with given dimension vector $d$, inherits a K\"ahler metric, at least on the smooth locus.
More explicitly, the K\"ahler potential on the total space of the polarizing line bundle, minus the zero section, $X\times\bC^\times$, is given by
\begin{equation*}
S(\rho,w)=\sum_{a\in Q_1}\tr\left(\rho(a)^*\rho(a)\right)+\log |w|.
\end{equation*}
Fixing a point $\rho\in X$, we can restrict $S$ to the orbit $G\cdot \rho$ and consider this as a function on $G/U$, which is identified with the space of Hermitian metrics $(h_i)_i$ on the $(E_i)_i$:
\begin{equation}\label{eq:quiverSh}
S(h)=\sum_{a\colon i\to j}\tr\left(\rho(a)^{*_h}\rho(a)\right)-\sum_{i\in Q_0}\theta_i\log \det h_i.
\end{equation}
This is related to the complexified K\"ahler potential $S_{\bC}$ by
\[
S=(\csc\pi\phi)\operatorname{Im}\left(e^{-\pi\sqrt{-1}\phi}S_{\bC}\right).
\]
The critical points of $S$ are then precisely the solutions to~\eqref{kings_equation}. These are global minima on $G/U$ by the general theory.
To find solutions to \eqref{kings_equation} numerically on the computer, one thus follows the negative gradient flow of $S$, given by
\begin{equation}\label{eq:quiverflow}
m_ih_i^{-1}\frac{d h_i}{dt}=\sum_{a\colon i\to j}\rho(a)^{*_h}\rho(a)-\sum_{a\colon j\to i}\rho(a)\rho(a)^{*_h}+\theta_i
\end{equation}
where $i\in Q_0$ and $m_i>0$ are parameters of the Riemannian metric on $G/U$, cf.~\cite{HKKPitlogs1}. (Warning: $\theta$ here and in~\cite{king_reps} corresponds to $-\theta$ in~\cite{HKKPitlogs1}.)
This is a simplification and approximation of the general flow $h^{-1}dh=-\arg\Omega$, where $m_i=|z_i|$.

\subsection{\texorpdfstring{$*$}{*}-algebras and their completions}\label{subsec:staralg}

Theorem~\ref{thm:King} can be reformulated in terms of $*$-algebras.
Here, by \textit{$*$-algebra} we mean an associative algebra over $\bC$ with a conjugate-linear involution $a\mapsto a^*$ such that $(ab)^*=b^*a^*$.
A \textit{$*$-representation} of a $*$-algebra $A$ is a Hilbert space $H$ together with a $*$-homomorphism $\rho\colon A\to B(H)$ to the C$^*$-algebra $B(H)$ of bounded operators on $H$.

Let $B(Q,\theta)$ be the $*$-algebra generated by the path algebra $A=\bC Q$, subject to the relations $e_i^*=e_i$ and
\begin{equation}\label{kingsRelation}
\sum_{a\in Q_1}\left[a,a^*\right]=\sum_{i\in Q_0}\theta_ie_i.  
\end{equation}
Thus, $B(Q,\theta)$ is the quotient of the path algebra of the quiver obtained from $Q$ by adding for each arrow $a$ an arrow $a^*$ with the reverse orientation, modulo the two-sided ideal generated by~\eqref{kingsRelation}.
The underlying algebra is a \textit{deformed preprojective algebra} of Crawley-Boevey--Holland~\cite{CBH_defpre}.

\begin{theorem}[Reformulation of Theorem~\ref{thm:King}]
A representation $E\in\mathbf{Rep}(Q)$ is $\theta$-polystable iff it is the restriction of a finite-dimensional $*$-representation of $B(Q,\theta)$.
\end{theorem}

We want to complete $B(Q,\theta)$ to a $\mathrm C^*$-algebra, playing an analogous role to the algebra of continuous functions on the topological space $X(\bC)$ of $\bC$-points of a projective variety $X$.
The problem is to find a suitable norm on $B(Q,\theta)$.
This is a special case of the problem of defining $\mathrm C^*$-algebras via generators $S$ and relations $R$.
In some cases, a solution is given by the notion of a universal $\mathrm C^*$-algebra in the sense of Blackadar~\cite{blackadar85}.
The idea is to consider all $*$-representations $(E,\rho)$, where $E$ is a Hilbert space, of the free $*$-algebra generated by $S$, which satisfy the relations $R$. The norm of an element $b$ in the free $*$-algebra generated by $S$ is then defined as 
\[
\|b\|\coloneqq\sup_{(E,\rho)}\|\rho(b)\|
\]
where $\|\rho(b)\|$ is the operator norm. 
If $\|b\|=+\infty$, which happens for instance in the case of a single generator and no relations, then this approach fails and the universal $\mathrm C^*$-algebra is undefined.
Assuming $\|b\|<+\infty$ for all $b$, we can take the quotient by the two-sided ideal of elements with vanishing norm and then the completion, which will be a $\mathrm C^*$-algebra.

\begin{proposition}\label{prop:univcstar}
Suppose $Q$ is acyclic, then the universal $\mathrm C^*$-algebra with generators and relations as in $B(Q,\theta)$, denoted $C(Q,\theta)$, exists.
\end{proposition}

\begin{proof}
Let $\rho\colon B(Q,\theta)\to B(E)$ be some $*$-representation on a Hilbert space $E$.
We need to show that for $a\in Q_1$ there is a universal bound on $\|\rho(a)\|$, independent of $E$ and $\rho$.
The relation \eqref{kingsRelation} can be written as a system of equations
\begin{equation}\label{kingsRelVert}
\sum_{a\colon j\to i}a a^* - \sum_{a\colon i\to j}a^* a=\theta_ie_i,\qquad i\in Q_0,
\end{equation}
which under $\rho$ becomes
\[
\sum_{a\colon j\to i}\rho(a) \rho(a)^* - \sum_{a\colon i\to j}\rho(a)^* \rho(a)=\theta_i\rho(e_i),\qquad i\in Q_0,
\]
in $B(E)$.
For each $a\colon i\to j$ we thus have
\[
\rho(a)^* \rho(a) \leq \sum_{b\colon k\to i}\rho(b) \rho(b)^*-\theta_i\rho(e_i)
\]
with respect to the Loewner partial order on $B(E)$, which implies the upper bound
\[
\|\rho(a)\|^2 \leq \sum_{b\colon k\to i}\|\rho(b)\|^2 +|\theta_i|.
\]
In particular, we get the bound $\|\rho(a)\|\leq \sqrt{|\theta_i|}$ if $i$ is a source, and by induction and using acyclicity of $Q$ we get a bound for all $a\in Q_1$.
\end{proof}

\begin{example}
Let $Q$ be the quiver with $n$ vertices $Q_0=\{1,2,\ldots,n\}$ and $n-1$ arrows $i\to i+1$, $1\leq i<n$.
Assume $\theta_1<\theta_2<\ldots<\theta_n$ and $\sum\theta_i=0$.
Then $B(Q,\theta)=C(Q,\theta)\cong M_n(\bC)$ is the algebra of $n\times n$-matrices.
\end{example}

\begin{example}
Let $Q$ be the \textit{Kronecker quiver} with two vertices $Q_0=\{1,2\}$ and two arrows, both from $1$ to $2$.
The path algebra $\bC Q$ arises as endomorphisms of the rank 2 bundle $\cO\oplus \cO(1)$ in $D^b(\mathrm{Coh}(\bC\bP^1))$.
Suppose $-\theta_1=\theta_2>0$, then $C(Q,\theta)$ is isomorphic to the algebra $C(\bC\bP^1,\mathrm{End}(\cO\oplus \cO(1)))$ of continuous sections of the bundle of endomorphisms of the rank 2 bundle.
In particular, $C(Q,\theta)$ is (strongly) Morita equivalent to the commutative $\mathrm C^*$-algebra $C(\bC\bP^1)$.
\end{example}

If $Q$ has oriented cycles, then $C(Q,\theta)$ still exists as a complete Hausdorff topological $*$-algebra, in fact a \textit{pro-$\mathrm C^*$-algebra} in the sense of~\cite{phillips88}.
It is constructed by taking the Hausdorff completion of $B(Q,\theta)$ with respect to the coarsest topology which makes the functions $B(Q,\theta)\to\bR$, $a\mapsto \|\rho(a)\|$, where $(E,\rho)$ ranges over all $*$-representation on Hilbert spaces $E$, continuous, see \cite[Example 1.3(6)]{phillips88}.

Another extension is to the case of \textit{quivers with relations}, where the path algebra $\bC Q$ is replaced by a quotient $\bC Q/I$ by a two-sided ideal $I\subset\bC Q$ which is usually required to be contained in the ideal generated by the arrows of $Q$.
We then impose the same relations in the $*$-algebra $B(Q,\theta)$.

\begin{example}
Let $X\subseteq \bC^n$ be an affine variety over $\bC$. The algebra $A(X)$ of regular (polynomial) functions on $X$ is of the form $\bC Q/I$ where $Q$ has a single vertex and $n$ arrows and $I$ is generated by the commutators and regular functions vanishing on $X$. 
The unital $*$-algebra $B(Q,0)$ is generated by $z_1,\ldots,z_n$ with relations coming from the defining equations of $X$ as well as the single additional relation
\[
\sum_{i=1}^n[z_i,z_i^*]=0.
\]
The abelianization of $B(Q,0)$ is then the $*$-closure of $A(X)$ in the algebra $C(X)$ of continuous functions on $X$ and dense with respect to the topology of uniform convergence on compact subsets by the Stone--Weierstrass theorem (complex version).
\end{example}

\begin{conjecture}
Let $A=\mathbb CQ/I$ be the path algebra of a quiver $Q$ modulo the ideal of relations $I$. Fix $\theta\colon Q_0\to\bR$, which can be more abstractly viewed as an additive map $K_0(\mathbf{Mod}^{\mathrm{fd}}(A))\to\bR$.
Then the (pro-)$\mathrm C^*$-algebra $C(Q,I,\theta)$ depends only on $A$ and $\theta$, not the particular presentation in terms of a quiver with relations.
\end{conjecture}

\begin{remark}
In~\cite{Frohlich1998Supersymmetric}, Fr{\"o}hlich--Grandjean--Recknagel give a definition of \textit{K\"ahler spectral data} based on $N=(2,2)$ supersymmetry.
In a different direction, \'O Buachalla studied a notion of noncommutative K\"ahler structures on quantum homogeneous spaces~\cite{obuachalla}, see also subsequent work by \'O Buachalla and collaborators.
It would be interesting to investigate if there are connections between those proposals and the (Archimedean) framework and examples discussed here.
\end{remark}

\subsection{Invariance under BGP reflection}\label{subsec:BGP}

In this subsection we provide some evidence for the belief that the K\"ahler metrics and $\mathrm C^*$-algebras discussed above do not depend on the choice of heart of a bounded t-structure $\mathbf{Rep}(Q)\subset D^b(\mathbf{Rep}(Q))$. 
We consider the special case where two different hearts are related by mutation of $Q$ at a sink or source.

Suppose $Q$ is a quiver and $i^-\in Q_0$ is a sink (no outgoing arrows), then we obtain the \textit{mutated} quiver, $s_{i^-}Q$, by replacing $i^-$ by a vertex $i^+$ and each arrow $a\colon j\to i^-$ of $Q$ by an arrow $a^{op}\colon i^+\to j$ of $s_{i^-}Q$. 
Thus, $i^+$ is a source of $s_{i^-}Q$, and the inverse construction takes a quiver $Q$ with source $i^+$ to a quiver $s_{i^+}Q$ with sink $i^-$.

In their seminal paper~\cite{BGP}, Bernstein--Gel'fand--Ponomarev introduced \textit{reflection functors} $S_{i^\pm}\colon \mathbf{Rep}(Q)\to \mathbf{Rep}(s_{i^\pm}Q)$ defined as follows.
If $i^-$ is a sink of $Q$ and $(E,\rho)\in\mathbf{Rep}(Q)$, then $S_{i^-}(E,\rho)=(E',\rho')$ with $E'_j\coloneqq E_j$ for $j\neq i^-$,
\begin{equation}\label{eq:sinkBGPformula}
E'_{i^+}\coloneqq \mathrm{Ker}\left(\bigoplus_{a\colon j\to i^-}E_j\xrightarrow{\left(\rho(a)\right)_{a\colon j\to i^-}}E_{i^-}\right)    
\end{equation}
and for an arrow $a^{op}\colon i^+\to j$ the map $\rho'(a^{op})\colon E'_{i^+}\to E'_j=E_j$ is the component of the inclusion map.
This defines $S_{i^-}$ on objects and extends to morphisms of representations  in the obvious way.
The definition of $S_{i^+}$ is dual.

From a more conceptual point of view, the appropriate left/right derived functor of $S_{i^\pm}$ provides an equivalence of derived categories $D^b\left(\mathbf{Rep}(Q)\right)\to D^b\left(\mathbf{Rep}(s_{i^\pm}Q)\right)$, see~\cite{happel87}.

Under the assumption that the map in~\eqref{eq:sinkBGPformula} is surjective, the dimension vector $d'\in\bZ^{(s_{i^-}Q)_0}_{\geq 0}$ of $(E',\rho')$ is related to the dimension vector $d\in\bZ^{Q_0}_{\geq 0}$ of $(E,\rho)$ by
\[
d_{i^+}'=-d_{i^-}+\sum_{a\colon j\to i^-}d_j,\qquad d_j' = d_j,\quad j\neq i^{\pm}
\]
which describes, more generally, the induced map $K_0(RS_{i^-})$ on Grothendieck groups.
This is compatible, in the sense that $\langle\theta',d'\rangle=\langle\theta,d\rangle$, with the following definition of $\theta'\in\bR^{s_{i^-}Q_0}$:
\begin{equation}
    \theta'_{i^+}=-\theta_{i^-}, \qquad \theta'_j=\theta_j+n_{i^-,j}\theta_{i^-},\qquad j\neq i^\pm,
\end{equation}
where $n_{k,j}$ denotes the number of arrows from $j$ to $k$.
We will always assume that $\theta_{i^-}\geq 0$.

In the following it will be convenient to consider the quiver $\widetilde Q$ which is the union of $Q$ and $s_{i^-}Q$. Thus, $\widetilde Q$ is obtained from $Q$ by adding the single vertex $i^+$ and an arrow $a^{op}\colon i^+\to j$ for every arrow $a\colon j\to i^-$.

\begin{theorem}\label{thm:mutation-invariance}
Let $Q$ be a quiver with sink $i^-$, $d\in\bZ_{\geq 0}^{Q_0}$ a dimension vector, $\theta\in\bR^{Q_0}$ with $\theta_{i^-}> 0$, and the mutated data $Q'\coloneqq s_{i^-}Q$, $d'$, and $\theta'$ as above.
Then the moduli spaces $\cM(Q,d,\theta)$ and $\cM(Q',d',\theta')$ are isomorphic as K\"ahler quotients.
\end{theorem}

Technically, we define $\cM(Q',d',\theta')$ to be empty whenever $d'_{i^+}<0$ for the statement of the theorem to hold in that degenerate case. 

\begin{proof} 
Fix Hermitian vector spaces $E_i$ with $\dim E_i=d_i$, $i\in Q_0$, and $E_{i^+}$ with $\dim E_{i^+}=d'_{i^+}$.
It will be convenient to write
\[
V \coloneqq \bigoplus_{a\colon j\to i^-} E_j
\]
and let
\[ 
\mathrm{inc}_a\colon E_j\longrightarrow V, \qquad \mathrm{pr}_a\colon V\longrightarrow E_j 
\] 
be the inclusion and projection associated with the summand of $V$ corresponding to the arrow $a\colon j\to i^-$. 

Let
\[
X^-\coloneqq \bigoplus_{a\in Q_1}\Hom(E_{s(a)},E_{t(a)}), \qquad X^+\coloneqq \bigoplus_{a\in Q'_1}\Hom(E_{s(a)},E_{t(a)})
\]
and  
\[
U^-\coloneqq\prod_{i\in Q_0}U(E_i),\qquad U^+\coloneqq\prod_{i\in Q'_0}U(E_i),
\]
then $\mathcal{M}(Q,d,\theta)$ (resp. $\mathcal{M}(Q',d',\theta')$) is by definition the K\"ahler quotient of $X^-$ by $U^-$ (resp. $X^+$ by $U^+$).
The idea of the proof is that the quotient of $X^-$ by the smaller group $U(E_{i^-})\subseteq U^-$ already coincides with the quotient of $X^+$ by $U(E_{i^+})\subseteq U^+$.

To calculate the quotient of $X^-$ by $U(E_{i^-})$, we assign to a given $\rho\in X^-$ the linear map 
\[
A\colon V\longrightarrow E_{i^-},\qquad A\coloneqq (\rho(a))_{a:j\to i^-}
\]
which collects maps assigned to the arrows to the sink $i^-$.
Then the moment map equation for $U(E_{i^-})$ is $AA^*=\lambda$, where $\lambda \coloneqq \theta_{i^-}>0$. 
If we let
\[
\mathrm{Gr}(d_{i^+},V)\coloneqq \left\{A\colon V\longrightarrow E_{i^-}\mid AA^*=\lambda\right\}/U(E_{i^-}),
\]
which is isomorphic to the Grassmannian of $d_{i^+}$-dimensional subspaces in $V$ via $A\mapsto\operatorname{Ker}A$, then the K\"ahler quotient of $X^-$ by $U(E_{i^-})$ is thus
\begin{equation}\label{eq:quot1}
\mathrm{Gr}(d_{i^+},V)\times \bigoplus_{a\in Q_1,t(a)\neq i^-}\Hom(E_{s(a)},E_{t(a)}).    
\end{equation}

Similarly, we assign to a given $\rho'\in X^+$ the linear map 
\[
B\colon E_{i^+}\longrightarrow V,\qquad B\coloneqq (\rho(a^{\mathrm{op}}))_{a:j\to i^-}
\]
which collects maps assigned to the arrows from the source $i^+$.
Then the moment map equation for $U(E_{i^+})$ is $-B^*B = -\lambda= \theta'_{i^+}$. 
If we let
\[
\mathrm{Gr}'(d_{i^+},V)\coloneqq \left\{B\colon E_{i^+}\longrightarrow V\mid B^*B=\lambda\right\}/U(E_{i^+}),
\]
which is isomorphic to the Grassmannian of $d_{i^+}$-dimensional subspaces in $V$ via $B\mapsto\operatorname{Im}B$, then the K\"ahler quotient of $X^+$ by $U(E_{i^+})$ is thus
\begin{equation}\label{eq:quot2}
\mathrm{Gr}'(d_{i^+},V)\times \bigoplus_{a\in Q'_1,s(a)\neq i^+}\Hom(E_{s(a)},E_{t(a)}).
\end{equation}

It is clear that the two quotients~\eqref{eq:quot1} and \eqref{eq:quot2} are naturally identified, even as K\"ahler spaces, since the quotient yields the standard homogeneous K\"ahler form scaled by $\lambda$ on both Grassmannians. 
The remaining group 
\[
U_0\coloneqq \prod_{i\in Q_0,i\neq i^{-}}U(E_i)=U^-/U(E_{i^-})=U^+/U(E_{i^+})
\]
acts on both in a compatible way.
We claim that the two moment map equations cut out corresponding subsets.

Suppose $\rho\in X^-$ satisfies the moment map equation for $U^-$.
A corresponding $\rho'\in X^+$ (unique up to $U(E_{i^+})$) is obtained by choosing a map $B$ fitting into a short exact sequence
\begin{equation}\label{mutationSES}
0 \to E_{i^+}\xrightarrow{B}V\xrightarrow{A}E_{i^-}\to 0
\end{equation}
and satisfying $B^*B=\lambda$.
By definition of $A$, $\rho(a)=A\circ\mathrm{inc}_a\colon E_j\longrightarrow E_{i^-}$, and we define $\rho'$ by $\rho'(a^{\mathrm{op}})\coloneqq \mathrm{pr}_a\circ B\colon E_{i^+}\longrightarrow E_j$.
Also, $\rho'(a)\coloneqq \rho(a)$ for arrows not incident to $i^{\pm}$. 
This is the explicit description of the identification between \eqref{eq:quot1} and \eqref{eq:quot2}.

Note that 
\[
A^*A+BB^*=\lambda. 
\] 
Indeed, $\lambda^{-1}A^*A$ and $\lambda^{-1}BB^*$ are the orthogonal projections onto $\operatorname{Im}(A^*)=(\operatorname{Ker} A)^\perp$ and $\operatorname{Im}(B)=\operatorname{Ker}(A)$, respectively. 
The rearranged $BB^*=\lambda -A^*A$ implies
\[ 
\rho'(a^{\mathrm{op}}) \bigl(\rho'(a^{\mathrm{op}})\bigr)^* = \mathrm{pr}_a\,BB^*\,\mathrm{inc}_a = \lambda - \mathrm{pr}_a\,A^*A\,\mathrm{inc}_a = \lambda -\rho(a)^*\rho(a)
\] 
for $a\colon j\to i^-$.
Summing over all arrows from $j$ to $i^-$, we obtain 
\[ 
\sum_{a\colon j\to i^-} \rho'(a^{\mathrm{op}}) \bigl(\rho'(a^{\mathrm{op}})\bigr)^* = n_{i^-,j}\lambda - \sum_{a\colon j\to i^-}\rho(a)^*\rho(a). 
\] 
In the moment map equation at $j$, an arrow $a\colon j\to i^-$ contributes $-\rho(a)^*\rho(a)$, while after reflection the arrow 
$a^{\mathrm{op}}\colon i^+\to j$ contributes $\rho'(a^{\mathrm{op}}) \bigl(\rho'(a^{\mathrm{op}})\bigr)^*$.  
Consequently, the moment map parameter at $j$ changes from $\theta_j$ to 
\[ 
\theta_j+n_{i^-,j}\lambda = \theta_j+n_{i^-,j}\theta_{i^-} = \theta'_j
\] 
and the claim follows.

The above moment map calculation shows that the K\"ahler quotient of \eqref{eq:quot1} by $U_0$, which is just $\mathcal{M}(Q,d,\theta)$ by reduction in stages, is naturally identified with the K\"ahler quotient of \eqref{eq:quot2} by $U_0$, which is $\mathcal{M}(Q',d',\theta')$.
\end{proof}

By the same mechanism as in the proof of the above theorem, we obtain a strong Morita equivalence\footnote{In the modern operator algebra literature, one typically just refers to these as \textit{Morita equivalences}. We retain the adjective here to avoid potential confusion, as much of the rest of paper works in a purely algebraic setting.} of $\mathrm C^*$-algebras.

\begin{theorem}\label{thm:C-star}
Let $Q$ be an acyclic quiver with sink ${i^-}$, $\theta\in\bR^{Q_0}$ with $\theta_{i^-}> 0$, and $s_{i^-}Q$, $\theta'$ as above.
Then the $\mathrm C^*$-algebras $C(Q,\theta)$ and $C(s_{i^-}Q,\theta')$ are strongly Morita equivalent.
\end{theorem}

The proof will make use of the following elementary fact about strong Morita equivalence of $\mathrm C^*$-algebras, see for example~\cite[II.7.6.5(iii)]{blackadar_oa}.

\begin{lemma}\label{lem:mecorner}
Let $B$ be a $\mathrm C^*$-algebra, $p\in B$ a projection such that
\[
\overline{\operatorname{Span}(BpB)}=B. 
\]
Then $A\coloneqq pBp$ is strongly Morita equivalent to $B$.
\end{lemma}

\begin{proof}[Proof of Theorem~\ref{thm:C-star}]
We may assume that $\theta_{i^-}=1$ after rescaling the generators corresponding to arrows $a\colon j\to i^-$.

Let $\widetilde B$ be the $*$-algebra generated by the quiver $\widetilde Q$ with the relations of $B(Q,\theta)$, $B(s_{i^-}Q,\theta')$, and the additional relations
\begin{equation}\label{BtildeRels}
\sum_{a\colon j\to i^{-}}aa^{op}=0,\qquad a^*b+a^{op}(b^{op})^*=\begin{cases} e_{s(a)} & a=b \\ 0 & \text{else} \end{cases}
\end{equation}
for $a,b\in Q_1$ with $t(a)=t(b)=i^-$.
Let $a_1,\ldots,a_n$ be the arrows to $i^-$, then in terms of 
\[
A\coloneqq (a_1,\ldots,a_n),\qquad B\coloneqq (a_1^{op},\ldots,a_n^{op})^{\mathrm{T}}
\]
the relations become $AB=0$ and 
\[
A^*A+BB^*=\mathrm{diag}\left(e_{s(a_1)},\ldots,e_{s(a_n)}\right).
\]
We also have $AA^*=e_{i^-}$ and $B^*B=e_{i^+}$ from~\eqref{kingsRelation}. 
Thus, for any $*$-representation of $\widetilde B$ (in Hilbert spaces) the sequence~\eqref{mutationSES} is exact with the first map, $\rho(B)$, an isometric inclusion and the second map, $\rho(A)$, an orthogonal projection.
This means we can reconstruct the entire sequence, up to isometry, from either one of the maps, i.e.\ any $*$-representation of $B(Q,\theta)$ (or $B(s_{i^-}Q,\theta')$) extends uniquely up to unitary isomorphism to $\widetilde B$.
Here we have also used that the relation from $B(Q,\theta)$ at a vertex $j\neq i^-$, \eqref{kingsRelVert}, together with the second set of relations in \eqref{BtildeRels}, imply the corresponding relation from $B(s_{i^-}Q,\theta')$ at $j$:
\begin{align*}
\sum_{\substack{a\in Q_1 \\ t(a)=j}}a a^* &- \sum_{\substack{a\in Q_1 \\ s(a)=j}}a^* a-\theta_je_j=\sum_{\substack{a\in Q_1 \\ t(a)=j}}a a^* -\sum_{a\colon j\to i^-}a^* a- \sum_{\substack{a\in Q_1 \\ s(a)=j,t(a)\neq i^-}}a^* a-\theta_je_j \\
&=\sum_{\substack{a\in s_{i^-}Q_1 \\ t(a)=j,s(a)\neq i^+}}a a^* + \sum_{a\colon j\to i^-}a^{op}(a^{op})^* - n_{i^-,j}e_j - \sum_{\substack{a\in s_{i^-}Q_1 \\ s(a)=j}}a^* a-\theta_je_j \\
&=\sum_{\substack{a\in s_{i^-}Q_1 \\ t(a)=j}}a a^* - \sum_{\substack{a\in s_{i^-}Q_1 \\ s(a)=j}}a^* a-\theta_j'e_j.
\end{align*}

We conclude that $C(Q,\theta)$ is a sub-$\mathrm C^*$-algebra of the universal $\mathrm C^*$-algebra, $\widetilde{C}$, completing $\widetilde B$.
Let $e\coloneqq\sum_{j\in Q_0}e_j$ be the unit of $C(Q,\theta)$.
We claim that $C(Q,\theta)=e\widetilde{C}e$.
It suffices to show $B(Q,\theta)=e\widetilde{B}e$. 
To see this, note that if a path starting and ending in $Q\subset \widetilde Q$ contains an arrow starting or ending at $i^+$, then that is part of a term of the form $a^{op}(b^{op})^*$ as in \eqref{BtildeRels}, and so we can use that relation to remove such a pair. By induction we obtain a linear combination of paths containing only arrows of the type $a$ or $a^*$ with $a\in Q_1$.

By the above, to show that $C(Q,\theta)$ is strongly Morita equivalent to $\widetilde C$, it suffices to show that $e\in\widetilde{C}$ satisfies the condition of Lemma~\ref{lem:mecorner}. 
To see this, note that the relation $B^*B=e_{i^+}$ implies that $e_{i^+}\in\operatorname{Span}(\widetilde{C}e\widetilde{C})$ and thus $1\in\operatorname{Span}(\widetilde{C}e\widetilde{C})$.

Similarly, $C(s_{i^-}Q,\theta')$ is also strongly Morita equivalent to $\widetilde C$, so the theorem follows.
\end{proof}

\subsection{Extension of mass and flow to complexes}\label{subsec:quivercomplexes}

In this subsection we extend the notions of \textit{mass} and \textit{minimizing flow} from metrized quiver representations to complexes of such. 
Our construction is formally analogous to Quillen's extension of Chern--Weil theory to (Chern) superconnections~\cite{quillen_superconnection,BGS_torsion}.
The rough idea is that when we rescale the differential on the complex by a large positive constant, the mass should concentrate more and more on the harmonic part.

First, given a quiver $Q$, we take a certain double, $\cA(Q)$, which is a graded $*$-algebra and roughly Koszul dual to the $*$-algebra $B(Q,\theta)$ considered in Section~\ref{subsec:staralg}. 
Our requirements for a graded $*$-algebra are $|a^*|=-|a|$ and the Koszul sign rule $(ab)^*=(-1)^{|a||b|}b^*a^*$.

\begin{definition}
Given a quiver $Q$, define $\cA(Q)$ to be the graded $*$-algebra with the following basis:
\begin{itemize}
    \item $e_i$, $f_i$ for each vertex $i\in Q_0$, $|e_i|=|f_i|=0$, $e_i^*=e_i$, $f_i^*=f_i$,
    \item $a$, $a^*$ for each arrow $a\in Q_1$, $|a|=1$, $|a^*|=-1$. 
\end{itemize}
The non-zero products are:
\begin{gather*}
    e_i^2=e_i, \qquad e_if_i=f_ie_i=f_i, \qquad
    e_ja=ae_i=a,\qquad e_ia^*=a^*e_j=a^*, \\
    aa^*=\sqrt{-1}f_j, \qquad a^*a=-\sqrt{-1}f_i,
\end{gather*}
where $a\colon i\to j$ is any arrow.
\end{definition}

\begin{remark}
The algebra $\cA(Q)$ is formally analogous to the algebra of forms on a Riemann surface, with the $e_i$ corresponding to $0$-forms, the $a$'s corresponding to $(1,0)$-forms, the $a^*$'s corresponding to $(0,1)$-forms, and the $f_i$'s corresponding to $2$-forms. However, the grading on $\cA(Q)$ does not correspond to form degree, but rather to the \textit{exotic degree} considered in~\cite{QiangThesis}.
\end{remark}

Given a $Q_0\times \mathbb Z$-graded vector space $E=\bigoplus_{i\in Q_0,n\in\bZ}E_i^n$ with $\dim E<\infty$ and a Hermitian inner product on each homogeneous component, we consider the tensor product algebra
\[
\cA(Q,E)\coloneqq \cA(Q)\otimes_{\bC} \mathrm{End}(E)
\]
where $\mathrm{End}(E)$ is graded by cohomological degree and the $*$-involution is given by
\begin{equation}\label{eq:superadjoint}
T^*\coloneqq \begin{cases} T^\dagger & \text{for }T\text{ even} \\ \sqrt{-1}T^\dagger & \text{for }T\text{ odd}\end{cases}
\end{equation}
where $T^\dagger$ is the usual (ungraded) Hermitian adjoint of a linear map (cf.\ \cite[(4.4.6)]{deligne_morgan_susy}).

Suppose further that $E$ has the structure of a metrized complex of representations of $Q$.
Concretely, each $E^\bullet_i\coloneqq \bigoplus_{n\in\bZ}E_i^n$, $i\in Q_0$, is a cochain complex of Hermitian vector spaces with differential $d_i$ of degree $+1$ and for each $a\colon i\to j$ we have a chain map
$\rho(a)\colon E_i^\bullet\to E_j^\bullet$ of degree 0.
This data is encoded in terms of the element
\[
A\coloneqq \sum_{i\in Q_0}e_i\otimes d_i+\sum_{a\in Q_1}a\otimes \rho(a)\qquad \in \cA^1(Q,E).
\]
Then $A^2=0$ is equivalent to the identities $d^2=0$ and $[d,\rho(a)]=0$ for a complex of representations.

\begin{remark}
The suitable choice of $*$-algebra $\cA_{E,h}$, as in our general framework in Section~\ref{subsec:frameworkarch}, is $\bigoplus_{i,n}\mathrm{End}(E_i^n)\subseteq\cA^0(Q,E)$.
An element of $\mathrm{Met}(E)$ is however not just given by a Hermitian metric on each $E_i^n$, since this would not be invariant under isomorphism in $D^b\mathbf{Rep}(Q)$. Instead, an element of $\mathrm{Met}(E)$ is represented by any metrized complex of representations $E'$ which is isomorphic to $E$ as an object of $D^b\mathbf{Rep}(Q)$.
\end{remark}

Assume that $\operatorname{Im}(z_i)>0$ for all $i\in Q_0$. 
Set
\[
F_0\coloneqq (A-A^*)^2=-AA^*-A^*A,\qquad F_z\coloneqq \sqrt{-1}\sum_{i\in Q_0}f_i\otimes z_i,\qquad F\coloneqq F_0+F_z
\]
then $F_0$ is skew-Hermitian ($F_0^*=-F_0$), $F$ is normal ($F^*F=FF^*$) and of degree 0.
More explicitly, we compute
\begin{align*}
-\sqrt{-1}F=&-\sum_{i\in Q_0}e_i\otimes \left(d_id_i^\dagger+d_i^\dagger d_i\right) \\
&+\sum_{a\colon i\to j} a\otimes\left(d_j^\dagger \rho(a)-\rho(a)d_i^\dagger\right)+\sqrt{-1}\sum_{a\colon i\to j} a^*\otimes\left(\rho(a)^\dagger d_j-d_i\rho(a)^\dagger\right) \\
&+\sum_{a\colon i\to j}\left(f_i\otimes\rho(a)^\dagger\rho(a)-f_j\otimes\rho(a)\rho(a)^\dagger\right)+\sum_{i\in Q_0}f_i\otimes z_i.
\end{align*}
Using the exponential function $\exp\colon \cA(Q,E)\to\cA(Q,E)$ we define
\[
\Omega\coloneqq (-1)^\mathrm{deg}\exp(-\sqrt{-1}F)_f\qquad \in\mathrm{End}^0(E)
\]
where for $X\in \cA(Q,E)$, $X_f$ denotes the endomorphism of $E$ obtained by projecting $X$ to the summand
\[
\bigoplus_{i\in Q_0} f_i\otimes \mathrm{End}(E_i^\bullet)
\]
and then identifying $f_i\otimes T$ with $T$ on $E_i^\bullet$.
Note that $\Omega$ specializes to~\eqref{eq:omega_qrep} for complexes concentrated in degree zero.

\begin{lemma}\label{lem:omega_halfplane}
$\Omega$ is normal and $(-1)^{\mathrm{deg}}\Omega=\exp(-\sqrt{-1}F)_f$ has spectrum in the upper half-plane.
\end{lemma}

\begin{proof}
Normality of $\Omega$ follows directly from normality of $F$.
Since $F_0$ commutes with $F_z$ we get
\begin{align*}
(-1)^{\mathrm{deg}}\Omega &= \left(\exp(-\sqrt{-1}F_0)\exp(-\sqrt{-1}F_z)\right)_f\\
&=\left(\exp(-\sqrt{-1}F_0)\left(1+\sum_{i\in Q_0}f_i\otimes z_i\right)\right)_f\\
&=\exp(-\sqrt{-1}F_0)_f+\sum_{i\in Q_0}z_i\exp\left(-d_id_i^\dagger-d_i^\dagger d_i\right)
\end{align*}
thus
\[
\frac{(-1)^{\mathrm{deg}}}{2\sqrt{-1}}(\Omega-\Omega^*)=\sum_{i\in Q_0}\operatorname{Im}(z_i)\exp\left(-d_id_i^\dagger-d_i^\dagger d_i\right)
\]
which is positive definite.
\end{proof}

\begin{lemma}\label{lem:omega_trace}
$\operatorname{tr}\Omega=Z(E)$.
\end{lemma}

\begin{proof}
Consider
\[
Z_t\coloneqq \operatorname{tr}\left((-1)^{\mathrm{deg}}\exp\left(-\sqrt{-1}tF_0+F_z\right)_f\right)
\]
for $t\in [0,1]$, then $Z_0=Z(E)$ and $Z_1=\operatorname{tr}\Omega$, so it suffices to show that $Z_t$ is constant in $t$.
We compute
\begin{align*}
\frac{d}{dt}\exp\left(-\sqrt{-1}tF_0+F_z\right)&=\underbrace{\exp\left(-\sqrt{-1}tF_0+F_z\right)}_{\eqqcolon E}(-\sqrt{-1}F_0) \\
&= -\sqrt{-1} E(A-A^*)^2 \\
&= -\frac{\sqrt{-1}}{2}[E(A-A^*),A-A^*]_s
\end{align*}
using the fact that $[E,A-A^*]=0$. Here, $[a,b]_s$ denotes the supercommutator.
But the functional $a\mapsto \operatorname{tr}((-1)^{\mathrm{deg}}a_f)$ on $\cA(Q,E)$ vanishes on supercommutators, showing $\frac{d}{dt}Z_t=0$.
\end{proof}

The mass and flow on the space of metrics are then given by
\begin{equation}\label{eq:massflowcomplexes}
m\coloneqq \operatorname{tr}|\Omega|, \qquad h^{-1}\frac{dh}{dt}=-\arg(\Omega)    
\end{equation}
where $h$ is the Hermitian metric on $E$ and on the summand $E^n$ the branch of $\arg$ is chosen in $(-\pi n,\pi(1-n))$.
Lemma~\ref{lem:omega_trace} gives the ``BPS inequality'' $|Z(E)|=|\operatorname{tr}\Omega|\leq m$.

\begin{question}
Does the mass decrease monotonically under the flow~\eqref{eq:massflowcomplexes}?
\end{question}

Our next goal is to provide more explicit formulas for the mass and flow in the simplest examples where both $d$ and $\rho$ are non-trivial.

\subsubsection*{The rank 1 case}

\begin{proposition}\label{prop_omega_rk1}
Suppose that $E$ is a complex of ``rank 1'' representations of $Q$, i.e.\ $\dim E_i^n\leq 1$ for all $i\in Q_0,n\in\bZ$.
Then the component of $\Omega$ on $E_i^n$ (complex scalar) is given by
\begin{align*}
\Omega_i^n= \,&(-1)^n  e^{-\beta_i^n-\beta_i^{n+1}}\Bigg(z_i \\
&+\sum_{a\colon i\to j}\left(\alpha_a^n +(\alpha_a^n-\alpha_a^{n-1})(\beta_i^n-\beta_j^n)\varphi_2(\beta_i^n-\beta_j^n+\beta_i^{n+1}-\beta_j^{n-1})\right) \\
&+\sum_{a\colon k\to i}\left(-\alpha_a^n+(\alpha_a^{n+1}-\alpha_a^{n})(\beta_i^{n+1}-\beta_k^{n+1})\varphi_2(\beta_i^{n+1}-\beta_k^{n+1}+\beta_i^{n}-\beta_k^{n+2})\right)\Bigg)    
\end{align*}
where
\[
\alpha_a^n\coloneqq\rho^n(a)\rho^n(a)^\dagger,\qquad \beta_i^n\coloneqq d^n_i(d^n_i)^\dagger,\qquad \varphi_2(s)\coloneqq\frac{e^s-s-1}{s^2}.
\]
Here, $d_i^n\colon E_i^{n-1}\to E_i^n$ and $\rho^n(a)\colon E_i^n\to E_j^n$ for $a\colon i\to j$ and $\alpha_a^n$, $\beta_i^n$ are considered as non-negative real scalars.
\end{proposition}

\begin{proof}
We calculate $\exp(-\sqrt{-1}F)_f$.
Given the structure of $\cA(Q,E)$, this will be a sum of terms coming from each of the $f_i$-terms of $F$ and each pair of $a$- and $a^*$-terms of $F$.
Since we are in the rank 1 case, the $e_i$-terms commute with the $f_i$-terms. 
Thus, the contribution of the $f_i$-terms of $F$ to $(-1)^n\Omega_i^n$ is
\[
e^{-\beta_i^n-\beta_i^{n+1}}\left(z_i+\sum_{a\colon i\to j}\alpha_a^n-\sum_{a\colon k\to i}\alpha_a^n\right).
\]

Fix an arrow $a\colon i\to j$ and $n\in\bZ$ such that $E_i^n$ and $E_j^{n-1}$ are non-zero, thus 1-dimensional. 
Let $p\colon E_i^n\to E_i^n$ and $q\colon E_j^{n-1}\to E_j^{n-1}$ be the respective identity maps and choose linear maps $x\colon E_i^n\to E_j^{n-1}$ and $y\colon E_j^{n-1}\to E_i^n$ which are inverse isomorphisms.
The contribution of the $a$- and $a^*$-terms in $F$ to $\Omega$ only involves the 6-dimensional subalgebra $\cA_{a,n}\subseteq\cA(Q,E)$ with basis
\begin{gather*}
e_i\otimes p,\quad e_j\otimes q,\quad a\otimes x,\quad -\sqrt{-1}a^*\otimes y,\quad f_i\otimes p,\quad f_j\otimes q,
\end{gather*}
in which
\begin{gather*}
(a\otimes x)(-\sqrt{-1}a^*\otimes y)=\sqrt{-1}aa^*\otimes xy=-f_j\otimes q, \\ (-\sqrt{-1}a^*\otimes y)(a\otimes x)=\sqrt{-1}a^*a\otimes yx=f_i\otimes p.    
\end{gather*}
The action of $\cA_{a,n}$ on the left ideal with basis $e_i\otimes p$, $a\otimes x$, and $f_i\otimes p$ gives the (non-faithful) representation
\begin{align*}
\lambda_1e_i\otimes p+\lambda_2 e_j\otimes q+\lambda_3a\otimes x&-\lambda_4\sqrt{-1}a^*\otimes y+\lambda_5f_i\otimes p+\lambda_6f_j\otimes q \\
&\mapsto\quad\begin{pmatrix}
    \lambda_1 & 0 & 0 \\ 
    \lambda_3 & \lambda_2 & 0 \\
    \lambda_5 & \lambda_4 & \lambda_1
\end{pmatrix}
\end{align*}
under which the part of $-\sqrt{-1}F$ in $\cA_{a,n}$,
\begin{equation}\label{eq:Fapart}
\begin{split}
&e_i\otimes \left(-d^n_i(d^n_i)^\dagger-(d^{n+1}_i)^\dagger d^{n+1}_i\right)+e_j\otimes \left(-d^{n-1}_j(d^{n-1}_j)^\dagger-(d^{n}_j)^\dagger d^{n}_j\right) \\&+a\otimes \left((d_j^n)^\dagger\rho^n(a)-\rho^{n-1}(a)(d_i^n)^\dagger\right)+\sqrt{-1}a^*\otimes \left(\rho^n(a)^\dagger d_j^n-d_i^n\rho^{n-1}(a)^{\dagger}\right),    
\end{split}
\end{equation}
is mapped to
\[
\begin{pmatrix}
    -\beta_i^n-\beta_i^{n+1} & 0 & 0 \\ 
    \lambda_3 & -\beta_j^{n-1}-\beta_j^{n} & 0 \\
    0 & \lambda_4 & -\beta_i^n-\beta_i^{n+1}
\end{pmatrix}
\]
with
\begin{align*}
\lambda_3\lambda_4&=-\left((d_j^n)^\dagger\rho^n(a)-\rho^{n-1}(a)(d_i^n)^\dagger\right)\left(\rho^n(a)^\dagger d_j^n-d_i^n\rho^{n-1}(a)^{\dagger}\right) \\
&=-(d_j^n)^\dagger\rho^n(a)\rho^n(a)^\dagger d_j^n+(d_j^n)^\dagger\rho^n(a)d_i^n\rho^{n-1}(a)^{\dagger} \\
&\quad\,\, +\rho^{n-1}(a)(d_i^n)^\dagger \rho^n(a)^\dagger d_j^n-\rho^{n-1}(a)(d_i^n)^\dagger d_i^n\rho^{n-1}(a)^{\dagger}\\
&=(\alpha_a^n-\alpha_a^{n-1})(\beta_i^n-\beta_j^n).
\end{align*}
The lower-left entry of the matrix exponential of a matrix of this type is given by
\[
\left(\exp\begin{pmatrix}
    \lambda_1 & 0 & 0 \\ 
    \lambda_3 & \lambda_2 & 0 \\
    0 & \lambda_4 & \lambda_1
\end{pmatrix}\right)_{31}=e^{\lambda_1}\lambda_3\lambda_4\underbrace{\frac{e^{\lambda_2-\lambda_1}-(\lambda_2-\lambda_1)-1}{(\lambda_2-\lambda_1)^2}}_{\varphi_2(\lambda_2-\lambda_1)}
\]
hence the contribution to $(-1)^n\Omega_i^n$ is
\[
e^{-\beta_i^n-\beta_i^{n+1}}(\alpha_a^n-\alpha_a^{n-1})(\beta_i^n-\beta_j^n)\varphi_2(\beta_i^n-\beta_j^n+\beta_i^{n+1}-\beta_j^{n-1}).
\]

To calculate the contribution to $(-1)^{n-1}\Omega_j^{n-1}$ we consider the left ideal of $\cA_{a,n}$ with basis $e_j\otimes q$, $-\sqrt{-1}a^*\otimes y$, and $f_j\otimes q$ which yields the representation
\begin{align*}
\lambda_1e_i\otimes p+\lambda_2 e_j\otimes q+\lambda_3a\otimes x&-\lambda_4\sqrt{-1}a^*\otimes y+\lambda_5f_i\otimes p+\lambda_6f_j\otimes q \\
&\mapsto\quad\begin{pmatrix}
    \lambda_2 & 0 & 0 \\ 
    \lambda_4 & \lambda_1 & 0 \\
    \lambda_6 & -\lambda_3 & \lambda_2
\end{pmatrix}
\end{align*}
under which~\eqref{eq:Fapart} is mapped to
\[
\begin{pmatrix}
    -\beta_j^{n-1}-\beta_j^{n} & 0 & 0 \\ 
    \lambda_4 & -\beta_i^n-\beta_i^{n+1} & 0 \\
    0 & -\lambda_3 & -\beta_j^{n-1}-\beta_j^{n}
\end{pmatrix}.
\]
The contribution to $(-1)^{n-1}\Omega_j^{n-1}$ is then
\[
e^{-\beta_j^{n-1}-\beta_j^{n}}(\alpha_a^n-\alpha_a^{n-1})(\beta_j^n-\beta_i^n)\varphi_2(\beta_j^n-\beta_i^n+\beta_j^{n-1}-\beta_i^{n+1})
\]
which upon the substitution $n\mapsto n+1$, $i\mapsto k$, $j\mapsto i$ becomes
\[
e^{-\beta_i^{n}-\beta_i^{n+1}}(\alpha_a^{n+1}-\alpha_a^{n})(\beta_i^{n+1}-\beta_k^{n+1})\varphi_2(\beta_i^{n+1}-\beta_k^{n+1}+\beta_i^{n}-\beta_k^{n+2})
\]
which is the final term in the claimed formula.
\end{proof}

\subsubsection*{A minimal example}

We consider the following complex of representations of the $A_2$ quiver, concentrated in degrees $0$ and $1$.
\[
\begin{tikzcd}
    (\bC,e^{x_2}) \arrow[r] & 0 \\
    (\bC,e^{x_0}) \arrow[r,swap,"\rho(a)=1"] \arrow[u,"d=1"] & (\bC,e^{x_1}) \arrow[u] 
\end{tikzcd}
\]
Here, the horizontal map is the action of the unique arrow $a\colon 0\to 1$, the vertical map is the differential, and $x_i$, $i=0,1,2$, are coordinates on the space $(0,\infty)^3\cong\bR^3$ of Hermitian metrics on the complex of representations. 
Let 
\[
x\coloneqq x_1-x_0,\qquad y\coloneqq x_2-x_0,\qquad \alpha\coloneqq e^x=\alpha_a^0,\qquad \beta\coloneqq e^y=\beta_0^1.
\]
Fix central charge parameters $z_0,z_1$ in the upper half-plane.
Then from Proposition~\ref{prop_omega_rk1} we get
\begin{align*}
\Omega_0 &\coloneqq\Omega_0^0=e^{-\beta}(z_0+\alpha), \\
\Omega_1 &\coloneqq\Omega_1^0=z_1-\alpha+(-\alpha)(-\beta)\varphi_2(-\beta)=z_1-\alpha\varphi_1(-\beta), \\
\Omega_2 &\coloneqq\Omega_0^1=-e^{-\beta}(z_0+(-\alpha)\beta\varphi_2(\beta))=-e^{-\beta}(z_0+\alpha)+\alpha\varphi_1(-\beta)
\end{align*}
where $\varphi_1(s)\coloneqq(e^{s}-1)/s$.
The flow is 
\[
\frac{dx}{dt}=\arg(\Omega_0)-\arg(\Omega_1),\qquad \frac{dy}{dt}=\arg(\Omega_0)-\arg(\Omega_2)
\]
where $\arg(\Omega_0),\arg(\Omega_1)\in (0,\pi)$ and $\arg(\Omega_2)\in (-\pi,0)$.
The dynamics of the system depend on the order of $\arg(z_0)$ and $\arg(z_1)$. Especially when the angle between $z_0$ and $z_1$ is large, the vector field decomposes the $(x,y)$-plane into regions separated by curves along which the direction of the flow changes abruptly, see Figures~\ref{fig:2dsys_stable} and~\ref{fig:2dsys_unstable}.

\begin{figure}[ht]
    \centering
    \includegraphics[width=0.7\linewidth]{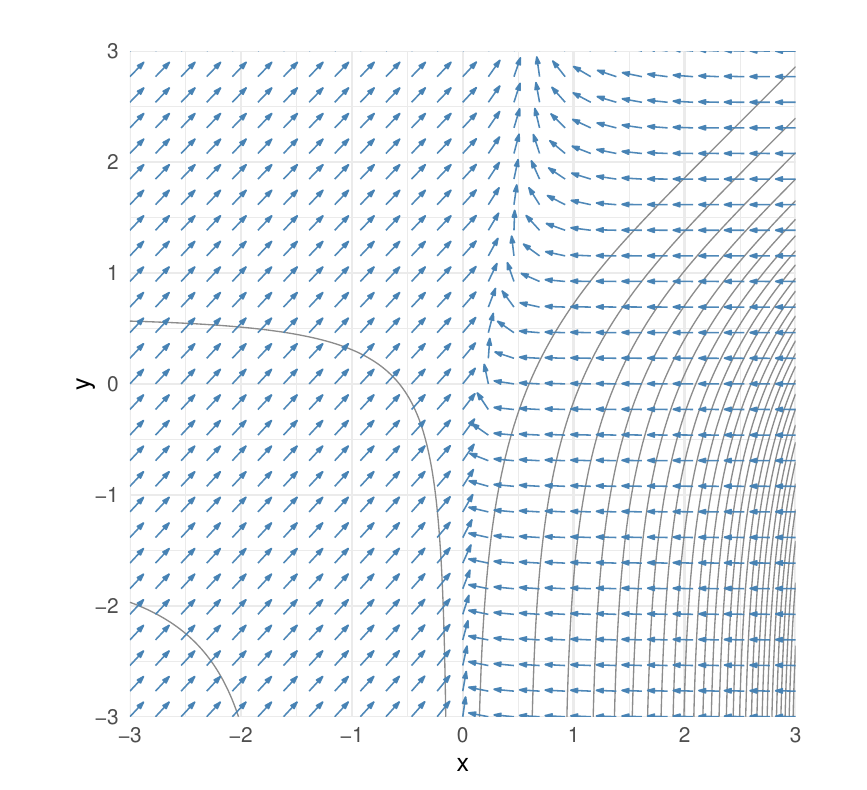}
    \caption{Vector field for $z_0=-1+0.1i$, $z_1=1+0.1i$ (stable case). Contour lines show mass $M$.}
    \label{fig:2dsys_stable}
\end{figure}

\begin{figure}[ht]
    \centering
    \includegraphics[width=0.7\linewidth]{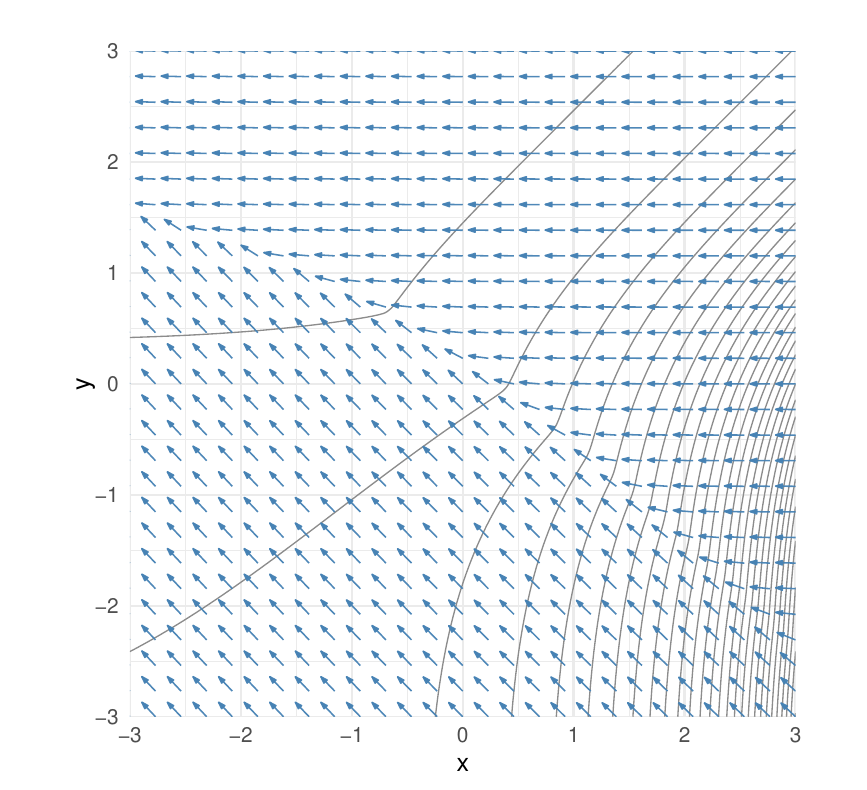}
    \caption{Vector field for $z_0=1+0.1i$, $z_1=-1+0.1i$ (unstable case). Contour lines show mass $M$.}
    \label{fig:2dsys_unstable}
\end{figure}

\subsection{K\"ahler metrics from increasing convex functions}\label{subsec:rampfunctions}

Fix a quiver $Q=(Q_0,Q_1,s,t)$ and for each vertex $i\in Q_0$ a stability parameter $\theta_i\in\bR$ and a Hermitian vector space $E_i$, $\dim E_i<\infty$.
In Section~\ref{subsec:quivers} we considered the K\"ahler quotient of the vector space $X=\bigoplus_{a\colon i\to j}\Hom(E_i,E_j)$, equipped with the standard translation-invariant (flat) metric, by the compact group $U=\prod_{i\in Q_0}U(E_i)$ of isometries. 
This quotient depends on $\theta\in \bR^{Q_0}\subseteq \mathfrak u^\vee$ and is identified with the moduli space of $\theta$-polystable representations of $Q$ with fixed dimension vector $(\dim(E_i))_{i\in Q_0}$ by Theorem~\ref{thm:King}.

Instead of starting with the flat metric on $X$, we could, more generally, take any $U$-invariant K\"ahler metric. 
According to our philosophy, the various K\"ahler metrics for different dimension vectors should be determined by a single ``noncommutative K\"ahler metric'' for $Q$.
Here, we propose examples of such metrics which depend on a choice of function $f_a\colon [-\infty,+\infty)\to \bR$ for each arrow $a\in Q_1$, satisfying certain properties. 
Given such functions, we replace the terms $\tr(\rho(a)^*\rho(a))$ in the K\"ahler potential by $\tr f_a\left(\log \rho(a)^*\rho(a)\right)$.
The following definition provides sufficient conditions on the $f_a$ to ensure positivity.

\begin{definition}
A \textbf{ramp function} is a function $f\colon [-\infty,+\infty)\to \bR$ which satisfies the following conditions:
\begin{enumerate}[(1)]
    \item $f\circ \log\colon [0,+\infty)\to\bR$ is of class $C^2$,
    \item $f'(\lambda)>0$ and $f''(\lambda)>0$ for $\lambda\in\bR$,
    \item $\lim_{x\to 0^+}\frac{f''(\log x)}{x}=\frac{d}{dx}\Big|_{x=0}f\circ\log>0$,
    \item $\lim_{\lambda\to+\infty}f'(\lambda)=+\infty$.
\end{enumerate}
\end{definition}

It follows from differentiability of $f\circ \log\colon [0,+\infty)\to\bR$ at $0$ that $\lim_{\lambda\to -\infty}f'(\lambda)=0$. 
Thus, $f'\colon \bR\to(0,+\infty)$ is a $C^1$ diffeomorphism.

\begin{example}
The function $f(\lambda)=Ae^{\lambda}+B$ is a ramp function for $A,B\in\bR$, $A>0$.
\end{example}

A $C^2$ function $\varphi\colon \bC^n\supseteq U\to\bR$ is \textit{strongly plurisubharmonic} if $\sqrt{-1}\partial\overline{\partial}\varphi$ is a K\"ahler form. More explicitly, this means that the matrix with entries $\frac{\partial^2\varphi}{\partial z_i\partial\bar z_j}$ is positive definite.

\begin{lemma}\label{lem:fmetric}
Suppose $f\colon [-\infty,+\infty)\to \bR$ is a ramp function, then the function
\[
\Hom(\bC^n,\bC^m)\to\bR,\qquad T\mapsto \tr f\left(\log TT^*\right)
\]
is strongly plurisubharmonic.\footnote{Using $T^*T$ instead of $TT^*$ changes the trace only by a constant depending on $m,n$, hence does not affect the associated Kähler form.}
\end{lemma}

In fact, the conclusion of the lemma requires only the first three conditions of a ramp function, not the fourth.

\begin{proof}
Suppose first that $f\colon [-\infty,+\infty)\to\bR$ is some function such that $g\coloneqq f\circ \log\colon [0,+\infty)\to\bR$ is $C^2$, but not necessarily satisfying the other conditions in the lemma.
Define $\mathfrak Hf\colon [0,+\infty)^2\to\bR$ by
\[
(\mathfrak Hf)(x,y)\coloneqq \begin{cases} \frac{f'(\log x)-f'(\log y)}{x-y} & \text{ for }x\neq y \\ \frac{f''(\log x)}{x} & \text{ for }x=y\end{cases}.
\]
To see that this is well-defined when $x=0$ or $y=0$, or both, write $\mathfrak Hf$ in terms of $g$:
\[
(\mathfrak Hf)(x,y)\coloneqq \begin{cases} \frac{g'(x)x-g'(y)y}{x-y} & \text{ for }x\neq y \\ g''(x)x+g'(x) & \text{ for }x=y\end{cases}.
\]
We claim that:
\begin{equation}\label{eq:magic}
\partial\overline{\partial}\tr f\left(\log TT^*\right)=\sum_{p,q,\bar p,\bar q}\left((\mathfrak Hf)(T^*T\otimes I_m,I_n\otimes TT^*)\right)_{q\bar p,\bar qp}dz_{pq}\wedge d\bar z_{\bar p\bar q}
\end{equation}
where $z_{pq}\colon \Hom(\bC^n,\bC^m)\to\bC$ is the $(p,q)$ matrix coefficient.
Given the above identity, the lemma follows, since $\mathfrak Hf$ is everywhere positive by our assumptions on $f$ and thus $(\mathfrak Hf)(T^*T\otimes I_m,I_n\otimes TT^*)$ is positive definite.

It remains to verify~\eqref{eq:magic}.
Note first that both sides of the equation depend linearly and continuously (in the $C^2$ topology) on $f$. 
By polynomial approximation, it suffices to check the special case where $g(x)=f(\log x)=x^k$, i.e.\ $f(\lambda)=e^{k\lambda}$, for some integer $k\geq 1$.
For $\mathfrak Hf$ we then get:
\[
(\mathfrak Hf)(x,y)=k\frac{x^k-y^k}{x-y}=k\sum_{r=0}^{k-1}x^ry^{k-1-r}
\]
so the right-hand side of~\eqref{eq:magic} is
\[
k\sum_{r=0}^{k-1}\sum_{p,q,\bar p,\bar q}\left((T^*T)^r\right)_{q\bar q}\left((TT^*)^{k-1-r}\right)_{\bar pp}dz_{pq}\wedge d\bar z_{\bar p\bar q}
\]
using the fact that $(A\otimes B)_{q\bar p,\bar qp}=A_{q\bar q}B_{\bar pp}$.
The left-hand side of~\eqref{eq:magic} is
\begin{align*}
\partial\overline{\partial}\tr(TT^*)^k &= k\partial\tr \left((TT^*)^{k-1}T\overline{\partial}T^*\right) \\
&= k\sum_{r=0}^{k-1}\tr \left(\partial T(T^*T)^r\wedge\overline{\partial}T^*(TT^*)^{k-1-r}\right) \\
&=k\sum_{r=0}^{k-1}\sum_{p,q,\bar p,\bar q}dz_{pq}\left((T^*T)^r\right)_{q\bar q}\wedge d\bar z_{\bar p\bar q}\left((TT^*)^{k-1-r}\right)_{\bar pp}
\end{align*}
which proves the claim.
\end{proof}

Let us fix for each $a\in Q_1$ a ramp function $f_a$.
Given vector spaces $E_i$, $i\in Q_0$ as before, we define $S\colon X\times\bC^\times\to\bR$ by
\begin{equation}\label{eq:quiverS_gen}
    S(\rho,w)\coloneqq \sum_{a\in Q_1}\tr f_a\left(\log\left(\rho(a)^*\rho(a)\right)\right)+\log|w|
\end{equation}
which is plurisubharmonic and strongly plurisubharmonic in $\rho$ by the previous lemma.
In particular, we obtain a K\"ahler metric on $X$ and on the K\"ahler quotient of $X$ by $U$.

The critical points of $S$ are given by the solutions to 
\[
\sum_{a\in Q_1}f_a'\left(\log\rho(a)\rho(a)^{*_h}\right)-f_a'\left(\log\rho(a)^{*_h}\rho(a)\right)=\sum_{i\in Q_0}\theta_i\rho(e_i)
\]
and this again has a solution for some choice of metric iff the representation is $\theta$-polystable.
This relies on a generalization of the Kempf--Ness theorem from the norm-squared function to suitable strongly plurisubharmonic functions~\cite{azadloeb93}.

In the next section we will consider the non-Archimedean limits of the K\"ahler metrics coming from the $f_a$.

\section{Representations of quivers: non-Archimedean case}
\label{sec:quivers_nonarch}

In this section we develop non-Archimedean analogues of some of the definitions and results which appeared in the previous section.
This is based on the well-known analogy between symmetric spaces and affine buildings.
Our main motivation for considering this case is that Fukaya categories are typically defined over a non-Archimedean (Novikov) field, instead of $\bC$.

In Section~\ref{subsec:nonarchlimit} we illustrate the passage from the Archimedean setting to the non-Archimedean one by considering families of metrics depending on a small parameter.
Section~\ref{subsec:prelimnorm} reviews definitions and results around norms on finite-dimensional vector spaces over a non-Archimdean field $K$, including the structure of the affine building of such a vector space.
Finally, in Section~\ref{subsect:non-arch-quiver}, we introduce a notion of \textit{harmonic norm} in this setting and prove an analogue of Theorem~\ref{thm:King} on the existence of such metrics.

\subsection{Passing to the non-Archimedean limit}\label{subsec:nonarchlimit}

We describe the passage from Archimedean to non-Archimedean K\"ahler geometry in the case of representations of quivers.
This motivates our constructions in Section~\ref{subsect:non-arch-quiver}, but is logically independent.

Throughout, we fix a quiver $Q$, for each vertex $i\in Q_0$ a stability parameter $\theta_i\in\bR$, and for each arrow $a\in Q_1$ a ramp function $f_a\colon [-\infty,+\infty)\to\bR$ as defined in Section~\ref{subsec:rampfunctions}.
Given a representation $(E,\rho)$ of $Q$ over $\bC$ with a reference Hermitian metric on each $E_i$ we consider the function 
\begin{equation}\label{eq:quiverSh_gen}
    S(h)\coloneqq \sum_{a\in Q_1}\tr f_a\left(\log\left(\rho(a)\rho(a)^{*_h}\right)\right)-\sum_{i\in Q_0}\theta_i\log\det h_i
\end{equation}
on the space of Hermitian automorphisms of $E$ which are a direct sum of Hermitian operators $h_i$ on $E_i$, $i\in Q_0$.
(This is the symmetric space $\prod_{i\in Q_0}GL(E_i)/U(E_i)$.)
Here, $\rho(a)^{*_h}=h_{s(a)}^{-1}\rho(a)^{*}h_{t(a)}$ is the adjoint with respect to the Hermitian metric determined by $h$.
Note that~\eqref{eq:quiverSh_gen} specializes to~\eqref{eq:quiverSh} when $f_a(\lambda)=e^\lambda$ for all $a\in Q_1$.

We let
\[
P\coloneqq -dS=\sum_{a\in Q_1}f_a'\left(\log\rho(a)^{*_h}\rho(a)\right)-f_a'\left(\log\rho(a)\rho(a)^{*_h}\right)+\sum_{i\in Q_0}\theta_i\rho(e_i)
\]
i.e.\ $-dS(h^{-1}\delta h)=\tr(h^{-1}\delta hP)$, where $dS$ is considered as a 1-form and $h^{-1}\delta h$ a tangent vector at $h$.

To find the non-Archimedean analogue of $S$ and $P$, we introduce a parameter $\epsilon>0$ and vary the metric by setting 
\[
f^{(\epsilon)}_a(\lambda)\coloneqq \epsilon^{-1}f_a(\epsilon \lambda).
\]
Note that $f^{(\epsilon)}_a$ will generally no longer be a ramp function, and the corresponding K\"ahler metric may be singular.
In spite of this, we will still consider
\begin{align*}
    S^{(\epsilon)} &= \epsilon^{-1}\sum_{a\in Q_1}\tr  f_a\left(\epsilon\log\left(\rho(a)\rho(a)^{*_h}\right)\right)-\sum_{i\in Q_0}\theta_i\log\det h_i, \\
    P^{(\epsilon)} &=\sum_{a\in Q_1}f_a'\left(\epsilon \log\rho(a)^{*_h}\rho(a)\right)-f_a'\left(\epsilon\log\rho(a)\rho(a)^{*_h}\right)+\sum_{i\in Q_0}\theta_i\rho(e_i).
\end{align*}

We want to make sense of the above expressions when $t=e^{-1/\epsilon}$ is a formal parameter.
To that end, consider as in Section~\ref{section:heuristics} the (Novikov) field $K=Nov$ of formal series of the form
\[
\sum_{\alpha\in\bR}c_\alpha t^{\alpha}
\]
where $c_\alpha\in\bC$ and the set of $\alpha\in\bR$ with $c_\alpha\neq 0$ is discrete and bounded below.
The field $K$ is non-Archimedean and complete with respect to the absolute value 
\[
\left|\sum_{\alpha\in\bR}c_\alpha t^\alpha\right|\coloneqq\exp\left(-\min\{\alpha\mid c_\alpha\neq 0\}\right)\in [0,+\infty).
\]
There is an operation of (coefficient-wise) complex conjugation which fixes the totally ordered subfield $K_{\bR}\subset K$ of series with coefficients in $\bR$.
The positive cone in $K_{\bR}$ consists of series whose leading coefficient (i.e.\ coefficient of the smallest power of $t$) is positive.
The order and absolute value are compatible in the sense that $0\leq x\leq y$ implies $|x|\leq|y|$.\footnote{$K$ is an \textit{ordered $*$-field} in the terminology of  Holland~\cite{holland80}, and the valuation $\nu(x)\coloneqq -\log|x|$ is its order valuation in the sense of~\cite[Theorem 4.4]{holland80}.}

\begin{definition}
A \textbf{Hermitian inner product} on a $K$-vector space $V$ is a $K$-linear map $\overline{V}\otimes_K V\to K$, $(v,w)\mapsto\langle v,w\rangle$ with $\overline{\langle v,w\rangle}=\langle w,v\rangle$ and $\langle v,v\rangle> 0$ for $v\neq 0$.
\end{definition}

The simplest example is $V=K$ with $\langle v,w\rangle=\bar vw$.
Intuitively, a $K$-vector space $V$ with Hermitian inner product is a family of Hermitian $\bC$-vector spaces depending on an infinitesimal parameter $t$. The limit of this family as $t\to 0$ is the $\bC$-vector space $$V_0\coloneqq V_{\leq 0}/V_{<0}\coloneqq \{v\in V\mid \|v\|\leq 1\}/\{v\in V\mid \|v\|<1\}$$ with the induced inner product.

\begin{definition}
A \textbf{refined norm} on a finite-dimensional $K$-vector space $V$ is given by:
\begin{enumerate}[(1)]
    \item a non-Archimedean lattice norm $\|\cdot\|$ on $V$ (see Section~\ref{subsec:prelimnorm} below),
    \item a Hermitian inner product $\langle\_,\_\rangle_0$ on the $\bC$-vector space $V_0$.
\end{enumerate}
\end{definition}

\begin{lemma}
A Hermitian inner product $\langle\_,\_\rangle$ on a $K$-vector space $V$ gives rise to a refined norm via $\|v\|\coloneqq\sqrt{|\langle v,v\rangle|}$, $\langle v,w\rangle_0\coloneqq \langle v,w\rangle\mod t^{>0}$.
\end{lemma}

\begin{proof}
The proof is standard.
Let us show the ultrametric triangle inequality 
\[
\|x+y\|\leq \max(\|x\|,\|y\|), \qquad x,y\in V.
\]
This follows from
\begin{align*}
    \|x+y\|^2 &= |\langle x+y,x+y\rangle| \leq \max\left(\|x\|^2,\|y\|^2,|\langle x,y\rangle|\right) 
    =\max\left(\|x\|^2,\|y\|^2\right)
\end{align*}
where we have used the Cauchy--Schwarz inequality
\[
|\langle x,y\rangle|\leq \|x\|\|y\|
\]
which is shown in the usual way.

The other statements are immediate.
\end{proof}

Suppose $a\in K_{\bR}$, $a>0$, converges absolutely for $t>0$ sufficiently small.
Write $a=t^\alpha c=e^{-\alpha/\epsilon}c$, where $|c|=1$, then
\[
\lim_{\epsilon\to 0^+}a^\epsilon=e^{-\alpha}\lim_{\epsilon\to 0^+}c^\epsilon=e^{-\alpha}=|a|.
\]
More generally, it is reasonable to define $\lim_{\epsilon\to 0^+}a^\epsilon\coloneqq|a|$ whenever $a\geq 0$.
Moreover, this extends to linear operators $Q\in\mathrm{End}(V)$ with $Q^*=Q$, $Q\geq 0$, where $V$ is a finite-dimensional Hermitian vector space over $K$.
In an orthonormal basis of $V$ where $Q=\mathrm{diag}(c_1,\ldots,c_n)$ we have
\[
\lim_{\epsilon\to 0^+}Q^\epsilon\coloneqq \mathrm{diag}\left(|c_1|,\ldots,|c_n|\right)\in\mathrm{End}(V_0)
\]
with respect to the induced orthonormal basis of $V_0$.

\begin{lemma}
\label{lem:limsvd}
Let $T\colon V\to W$ be a linear map between finite-dimensional Hermitian vector spaces over $K$.
Then
\[
\lim_{\epsilon\to 0^+}(T^*T)^\epsilon\in\mathrm{End}(V_0),\qquad \lim_{\epsilon\to 0^+}(TT^*)^\epsilon\in\mathrm{End}(W_0)
\]
depend only on the refined norms on $V$ and $W$.
\end{lemma}

\begin{proof}    
By the singular value decomposition for matrices with coefficients in $K$ (see Section~\ref{subsec:prelimnorm} for a review) we can write 
\[
T=\sum_{i=1}^rt^{-\sigma_i}\langle e_i,\_\rangle f_i,\qquad T^*=\sum_{i=1}^rt^{-\sigma_i}\langle f_i,\_\rangle e_i
\]
where $e_1,\dots,e_n$ is an orthonormal basis of $V$, $f_1,\ldots,f_m$ is an orthonormal basis of $W$, and $\sigma_1\geq\ldots\geq\sigma_r$, $r=\operatorname{rk} T$.
We compute
\[
T^*T=\sum_{i=1}^rt^{-2\sigma_i}\langle e_i,\_\rangle e_i, \qquad TT^*=\sum_{i=1}^rt^{-2\sigma_i}\langle f_i,\_\rangle f_i
\]
and
\[
\lim_{\epsilon\to 0^+}(T^*T)^\epsilon=\sum_{i=1}^re^{2\sigma_i}\langle e_i,\_\rangle_0e_i,\qquad  \lim_{\epsilon\to 0^+}(TT^*)^\epsilon=\sum_{i=1}^re^{2\sigma_i}\langle f_i,\_\rangle_0f_i
\]
where $e_i$ also denotes the image of that basis vector in $V_0$, and similarly for $f_i$.
If we forget the inner products on $V$, $W$ and only remember the induced refined norms, then $T^*$ is no longer defined, but we claim that the above limits can still be recovered from $T$ alone.
It suffices to know the set of real numbers $\Sigma=\{\sigma_1,\ldots,\sigma_r\}$ and the subspaces
\[
V_\lambda\coloneqq\mathrm{span}\{e_i\mid\sigma_i=\lambda\}\subseteq V_0,\qquad
W_\lambda\coloneqq\mathrm{span}\{f_i\mid\sigma_i=\lambda\}\subseteq W_0.
\]
In fact, using the splitting coming from the inner products on $V_0$ and $W_0$ it suffices to know the $\bR$-filtrations
\begin{align*}    
V_{\leq\lambda}\coloneqq V_{-\infty}\oplus\bigoplus_{\mu\leq\lambda}V_\mu,\qquad
W_{\geq\lambda}\coloneqq\bigoplus_{\mu\geq\lambda}W_\mu,
\end{align*}
where
\[
\qquad V_{-\infty}\coloneqq\left(\bigoplus_{\mu\in\bR}V_\mu\right)^\perp=(\Ker(T)\cap V_{\leq 0})/(\Ker(T)\cap V_{< 0}).
\]
But these can be defined alternatively using the non-Archimedean norms by
\begin{align*}
V_{\leq\lambda} = \cF^-(T)_{\lambda} &\coloneqq \left\{v\in V_0\mid \exists \text{ lift }\tilde{v}\in V\colon \|\tilde{v}\|\leq 1,\|T\tilde{v}\|\leq e^{\lambda}\right\}, \\
W_{\geq\lambda} = \cF^+(T)_{\lambda}&\coloneqq \left\{w\in W_0\mid\exists v\in V\colon \|Tv\|\leq 1,Tv\text{ lifts } w,\|v\|\leq e^{-\lambda}\right\}.
\end{align*}
Moreover, $\Sigma$ is precisely the set where $V_{\leq\lambda}$ (equivalently, $W_{\geq\lambda}$) jumps.
\end{proof}

The above discussion shows that we can make sense of $\lim_{\epsilon\to 0^+}\epsilon S^{(\epsilon)}$ and $\lim_{\epsilon\to 0^+}P^{(\epsilon)}$, and these only depend on the refined norms on the $K$-linear spaces $E_i$, $i\in Q_0$.
To write explicit formulas, let
\[
\rho(a)^{*_h}\rho(a)=\sum_{k=1}^{k_a}t^{-\lambda_{a,k}}p_{a,k}^-,\qquad
\rho(a)\rho(a)^{*_h}=\sum_{k=1}^{k_a}t^{-\lambda_{a,k}}p_{a,k}^+
\]
where $a\colon i\to j$ is an arrow, $\lambda_{a,1}<\ldots<\lambda_{a,k_a}$, and $p_{a,k}^\pm$ are orthogonal projectors (to eigenspaces).
We use the same notation for the induced projectors $p^-_{a,k}\colon (E_i)_0\to (E_i)_0$ and $p^+_{a,k}\colon (E_j)_0\to (E_j)_0$, $k=1,\ldots,k_a$.
Then
\[
\begin{aligned}
S^{(0)} &\coloneqq\lim_{\epsilon\to 0^+}\epsilon S^{(\epsilon)}=\sum_{a\colon i\to j}\sum_{k=1}^{k_a}f_a\left(\lambda_{a,k}\right)\tr\left(p^+_{a,k}\right)-\sum_{i\in Q_0}\theta_i\log|\det(h_i)|, \\
P^{(0)} &\coloneqq\lim_{\epsilon\to 0^+}P^{(\epsilon)}=\sum_{a\colon i\to j}\sum_{k=1}^{k_a}f_a'\left(\lambda_{a,k}\right)\left(p^-_{a,k}-p^+_{a,k}\right)+\sum_{i\in Q_0}\theta_i\mathrm{pr}_{(E_i)_0},
\end{aligned}
\]
where $|\det(h_i)|$ is the volume of the non-Archimedean norm on $E_i$, compared to some fixed reference norm.
We want to interpret the equation $P^{(0)}=0$ as an instance of King's equation~\eqref{kings_equation} for some quiver $Q^\star$.

\begin{definition}
Let $Q$ be a quiver with stability parameters $\theta_i\in\bR$, $i\in Q_0$, functions $f_a$, $a\in Q_1$ as before, and a representation $(E,\rho)$ over $K$ with a non-Archimedean norm on each $E_i$, $i\in Q_0$.
Remove any arrows $a\in Q_1$ with $\rho(a)=0$.
We define a quiver $Q^\star$ with $\theta_i^\star$, $i\in Q^\star_0$, and a representation $E^\star$ of $Q^\star$ over $\bC$ as follows.

The $\lambda_{a,k}$ are as before and, by Lemma~\ref{lem:limsvd} and its proof, are the well-defined values where the increasing filtration $\cF^-(\rho(a))_{\lambda}$ (equivalently: the decreasing filtration $\cF^+(\rho(a))_{\lambda}$) jumps.
The quiver $Q^\star$ has vertices of three types:
\begin{enumerate}[(1)]
    \item A vertex for each vertex $i\in Q_0$ with 
    \[ 
    \theta_i^\star \coloneqq \theta_i+\sum_{\substack{a\in Q_1: s(a)=i}} f_a'(\lambda_{a,k_a}),
    \]
    and $E^\star_i\coloneqq(E_i)_0$.
    \item A vertex $(a,k,-)$ for each pair $(a,k)$, $a\in Q_1$, $0\leq k\leq k_a-1$, with
    \[ 
    \theta_{a,k,-}^\star \coloneqq \begin{cases} -f_a'(\lambda_{a,1}), & k=0\\ f_a'(\lambda_{a,k})-f_a'(\lambda_{a,k+1}), & k=1,\ldots,k_a-1 \end{cases}
    \]
    and with an arrow to $s(a)$.
    We set 
    \[
    E^{\star}_{a,k,-}\coloneqq\begin{cases} \cF^-(\rho(a))_{-\infty}\coloneqq\cF^-(\rho(a))_{<\lambda_1}, & k=0 \\ \cF^-(\rho(a))_{\lambda_k}, & k=1,\ldots,k_a-1 \end{cases} 
    \]
    and assign the inclusion map to the arrow.
    \item A vertex $(a,k,+)$ for each pair $(a,k)$, $a\in Q_1$, $1\leq k\leq k_a$, with 
    \[
    \theta_{a,k,+}^\star \coloneqq \begin{cases} -f_a'(\lambda_{a,1}), & k=1 \\ f_a'(\lambda_{a,k-1})-f_a'(\lambda_{a,k}), & k=2,\ldots,k_a \end{cases}
    \]
    and with an arrow to $t(a)$.
    We set $E^{\star}_{a,k,+}\coloneqq\cF^+(\rho(a))_{\lambda_k}$ and assign the inclusion map to the arrow.
\end{enumerate}
\end{definition}

Note that our requirement that each $f_a$ is a ramp function implies that $\theta_{a,k,\pm}^\star<0$.

\begin{remark}
The above definition also makes sense if $\theta_i$ are replaced by central charges $z_i$ in the upper half-plane.
\end{remark}

\begin{proposition}
Suppose $Q$, $\theta_i$, $f_a$, $(E,\rho)$, and a non-Archimedean norm on each $E_i$ are given as above.
Then $P^{(0)}=0$ for some refinements of the norms if and only if $E^\star$ is $\theta^\star$-polystable.
\end{proposition}

\begin{proof}
Suppose we are given some refinement of the norms on each $E_i$, thus $p_{a,k}^\pm$ and $P^{(0)}$ are defined.
For an arrow $a\colon i\to j$ let
\[
p^-_{a,\leq k}\coloneqq \mathrm{pr}_{(E_i)_0}-p^-_{a,k+1}-\ldots-p^-_{a,k_a},\qquad p^+_{a,\geq k}\coloneqq p^+_{a,k}+\ldots+p^+_{a,k_a}
\]
then
\begin{align*}
\sum_{k=1}^{k_a}f_a'\left(\lambda_{a,k}\right)p^-_{a,k}&=\sum_{k=0}^{k_a-1}\theta_{a,k,-}^\star p^-_{a,\leq k}+f'(\lambda_{a,k_a})\mathrm{pr}_{(E_i)_0} \\
-\sum_{k=1}^{k_a}f_a'\left(\lambda_{a,k}\right)p^+_{a,k}&=\sum_{k=1}^{k_a}\theta_{a,k,+}^\star p^+_{a,\geq k}
\end{align*}
by the telescoping sum.
Thus, we can re-write $P^{(0)}$ as
\[
P^{(0)}= \sum_{i\in Q_0}\theta_i
^\star\mathrm{pr}_{(E_i)_0} + \sum_{a\colon i\to j}\left(\sum_{k=0}^{k_a-1}\theta_{a,k,-}^\star p^-_{a,\leq k} + \sum_{k=1}^{k_a}\theta_{a,k,+}^\star p^+_{a,\geq k}\right).
\]

Now, suppose $P^{(0)}=0$ for our refined norms. Then we give each $E_{a,k}^{\star,\pm}$ the Hermitian inner product so that the inclusion $I_{a,k}^\pm\colon E_{a,k}^{\star,\pm}\to E_i^\star$, which is part of the representation $E^\star$, is $\sqrt{-\theta_{a,k,\pm}^\star}$ times an isometric inclusion, i.e.\
\[
I_{a,k}^-\left(I_{a,k}^-\right)^*=-\theta_{a,k,-}^\star p^-_{a,\leq k},\qquad I_{a,k}^+\left(I_{a,k}^+\right)^*=-\theta_{a,k,+}^\star p^+_{a,\geq k},\qquad
\left(I_{a,k}^\pm\right)^*I_{a,k}^\pm=-\theta_{a,k,\pm}^\star.
\]
With this choice, \eqref{kings_equation} is satisfied and $E^\star$ is polystable by Theorem~\ref{thm:King}.

In the other direction, if $E^\star$ is polystable, then by the other half of Theorem~\ref{thm:King} we have a metric on $E^\star$ satisfying  \eqref{kings_equation} which means precisely that the inclusions assigned to arrows are isometric up to scalar factor as above and that $P^{(0)}=0$.
\end{proof}

The proposition suggests an analogy between the following ``harmonicity'' conditions on a metrized quiver representation $E$:
\begin{enumerate}[(1)]
    \item Over $\bC$: $P=0$.
    \item Over $K$: $E^\star$ is polystable.
\end{enumerate}
This analogy motivates the definitions in Section~\ref{subsect:non-arch-quiver}.

\subsection{Preliminaries on norms}\label{subsec:prelimnorm}

Throughout this subsection, $K$ is a complete non-Archimedean field, i.e.\ $K$ comes with an absolute value $|\cdot|\colon K\to [0,\infty)$ satisfying the ultrametric inequality $|a+b|\leq \max(|a|,|b|)$ and is complete with respect to the metric $d(a,b)=|a-b|$.
There is a corresponding valuation on $K$ given by $\nu\colon K^\times\to \bR$, $\nu(a)=-\log |a|$.
We denote by $\mathcal O\coloneqq \{a\in K: |a|\leq 1\}$ the valuation ring, $\mathfrak m\coloneqq\{a\in K: |a|<1\}$ the maximal ideal in $\cO$, and $\mathbf k\coloneqq\mathcal O/\mathfrak m$ the residue field.

\textbf{Conventions:} We will consistently index filtrations and gradings by the ``energy'' $-\nu=\log|\cdot|$. 
\textit{Vector space} always means \textit{finite-dimensional vector space} in this subsection. 

\begin{definition}
A \textbf{norm} on a vector space $V$ over $K$ is a function $\|\cdot\|\colon V\to[0,\infty)$ such that
\begin{enumerate}[(1)]
    \item $\|v\|=0$ iff $v=0$,
    \item $\|av\|=|a|\cdot\|v\|$ for all $a\in K$, $v\in V$,
    \item $\|v+w\|\leq \max(\|v\|,\|w\|)$.
\end{enumerate}
\end{definition}

Let $V$ be a vector space over $K$, $b_1,\ldots,b_n$ a basis of $V$, and $c_1,\ldots,c_n\in\bR$.
Then 
\begin{equation}\label{eq:diagnorm}
\|v_1b_1+\ldots +v_nb_n\|_{b,c}\coloneqq \max(e^{c_1}|v_1|,\ldots,e^{c_n}|v_n|)
\end{equation}
is a norm on $V$. 
Norms of this form are said to be \textbf{diagonalizable}.
For general $K$, not all norms are diagonalizable, see Proposition~\ref{prop:normdiag} below.
The norms of the form $\|\cdot\|_{b,0}$, i.e.\ with $c_1=\ldots=c_n=0$, are the \textbf{lattice norms} corresponding to the $\cO$-submodule (lattice) with basis $b_1,\ldots,b_n$, giving the closed unit ball in $V$.

\begin{proposition}\label{prop:normdiag}
The following are equivalent for a non-Archimedean field $K$:
\begin{enumerate}[(1)]
    \item All norms on all $K$-vector spaces are diagonalizable.
    \item $K$ is maximally complete, i.e.\ there is no complete extension $L/K$ with the same value group and residue field.
    \item $K$ is spherically complete, i.e.\ any nested family of closed balls has non-empty intersection.
\end{enumerate}
\end{proposition}

\begin{proof}
The equivalence between 1. and 2. is \cite[Lemma 1.12]{BE21}.
The equivalence between 2. and 3. is a theorem of I.~Kaplansky, cf.~\cite[Section 1.1]{BE21}.
\end{proof}

\begin{example}
(1) If $K$ is discretely valued, i.e.\ $\nu(K^\times)\subseteq\bR$ is discrete,  then $K$ is easily seen to be maximally complete.
In particular, this applies to any field with the trivial valuation ($\nu=0$), in which case norms on vector spaces are the same as complete $\bR$-filtrations.

(2) The Novikov field of formal power series $\sum_{\lambda\in\bR} c_\lambda t^\lambda$ where the set of $\lambda$ with $c_\lambda\neq 0$ is discrete and bounded below, as in Section~\ref{subsec:nonarchlimit}, is not maximally complete. Its maximal completion consists of those power series where the set of $\lambda$ with $c_\lambda\neq 0$ is only required to be well-ordered as a subset of $\bR$. This is an example of a Mal'cev--Neumann field, see~\cite{poonen93} for more details.
\end{example}

\subsubsection*{Non-expansive maps}

A linear map $f\colon V\to W$ between normed vector spaces is \textbf{non-expansive} if $\|f(v)\|\leq \|v\|$ for all $v\in V$.
Let $\mathbf{Norm}(K)$ denote the category whose objects are normed vector spaces over $K$ and whose morphisms are the non-expansive maps.
Thanks to the ultrametric inequality, $\mathbf{Norm}(K)$ is additive.
Furthermore, $\mathbf{Norm}(K)$ has kernels and cokernels: The kernel of a morphism $f\colon V\to W$ is the linear map kernel $\operatorname{Ker}(f)\hookrightarrow V$ together with the restriction of the norm on $V$.
The cokernel of a morphism $f\colon V\to W$ is the linear map cokernel $W\twoheadrightarrow\operatorname{Coker}(f)$ together with the quotient norm induced from $W$.

The category $\mathbf{Norm}(K)$ is moreover quasi-abelian: Kernels (strict monomorphisms) are stable under pushout and cokernels (strict epimorphisms) are stable under pullbacks~\cite{schneiders,benbassatkremnizer}.
The same is true for the full subcategory of diagonalizable norms, which is closed under subobjects and quotients by~\cite[Lemma 1.13]{BE21}, but not extensions (unless all norms are diagonalizable).

\subsubsection*{The space of norms}

If $V$ is a vector space over $K$, let $\cN(V)$ denote the set of norms on $V$.
The \textit{Goldman--Iwahori metric} on $\cN(V)$ is defined by
\[
d_\infty(\|\cdot\|,\|\cdot\|')\coloneqq \sup_{v\in V\setminus\{0\}}\left|\log\frac{\|v\|}{\|v\|'}\right|.
\]
It is easy to see that $(\cN(V),d_\infty)$ is complete, even without assuming completeness of $K$.

Let $\cN^{\mathrm{diag}}(V)\subseteq \cN(V)$ be the subspace of diagonalizable norms. This is a dense subset by~\cite[Theorem 1.19(ii)]{BE21} (and $\cN^{\mathrm{diag}}(V)= \cN(V)$ if $K$ is maximally complete, by the above).

Let $\Delta(V)\coloneqq \cN^{\mathrm{diag}}(V)/_\sim$ be the space of norms up to homothety (scaling).
This is an example of an affine (Bruhat--Tits) building~\cite{BT87,parreau00,rousseau09}. 
It is a simplicial building iff the valuation group $\nu(K^\times)\subseteq\bR$ is discrete.
See~\cite{RTW15} for a discussion on the relation with Berkovich spaces.
In the following, we will not assume any familiarity with the theory of affine buildings, however the more general perspective can be a useful conceptual device.

Since $\cN^{\mathrm{diag}}(V)$ is a trivial $\bR$-bundle over $\Delta(V)$, there is not much difference between these spaces, and we will routinely go back and forth between the two.

The non-positivity of the curvature of $\Delta(V)$ will play a key role in our considerations.
See \cite[Proposition 2.10]{parreau00}, \cite[Proposition 6.5]{rousseau09}.

\begin{proposition}[Bruhat--Tits]
Any affine building is a CAT(0) metric space with respect to the metric $d_2$ which restricts to the Euclidean metric on each apartment.  (This metric is bilipschitz equivalent to $d_\infty$.)
\end{proposition}

In the remainder of this subsection we will review some properties of $\cN_\mathrm{diag}(V)$ and $\Delta(V)$ using the terminology from the theory of buildings.

\subsubsection*{Apartments}

Apartments of $\cN_\mathrm{diag}(V)$ and $\Delta(V)$ correspond to splittings $V=L_1\oplus\ldots\oplus L_n$, where $\dim L_i=1$, $i=1,\ldots,n\coloneqq\dim V$, up to permutation of the $L_i$'s.
Choose $0\neq b_i\in L_i$, $i=1,\ldots,n$ then we have an isometric embedding
\[
\bR^n\hookrightarrow \cN_\mathrm{diag}(V), \qquad c\mapsto \|\cdot\|_{b,c}
\]
where the diagonalizable norm $\|\cdot\|_{b,c}$ is defined as in \eqref{eq:diagnorm}.
This induces a map $\bR^{n-1}=\bR^n/\bR\hookrightarrow\Delta(V)$ whose image is the apartment corresponding to the splitting.
Different choices of total order of the $L_i$'s and the basis give different parametrizations of the same apartment related by the action of the affine Weyl group $S_n\ltimes (\nu(K^\times))^{n-1}\subseteq S_n\ltimes \bR^{n-1}$.

\subsubsection*{Tangent cones and singular values}

Let $(V,\|\cdot\|)$ be a normed vector space over $K$. 
For $\alpha\in\bR$ consider the quotient of $\mathcal O$-modules
\[
V_\alpha\coloneqq\{v\in V:\|v\|\leq e^{\alpha}\}/\{v\in V:\|v\|<e^{\alpha}\}
\]
which is a vector space over the residue field $\mathbf k$.
Multiplication by $a\in K^\times$ defines an isomorphism $V_\alpha\to V_{\alpha-\nu(a)}$ which is well-defined up to a factor in $\mathbf k^\times$.
Let $\Gamma\coloneqq\nu(K^\times)\subseteq \bR$ be the valuation group and
\[
\operatorname{gr}V\coloneqq \bigoplus_{\alpha\in\bR/\Gamma}V_\alpha
\]
which is an $\bR/\Gamma$-graded vector space over $\mathbf k$, each summand being canonical up to rescaling by $\mathbf k^\times$.
We write $\operatorname{gr}(V,\|\cdot\|)$ if we wish to emphasize the dependence on the norm.
This construction is functorial in the following sense: If $f\colon (V,\|\cdot\|)\to (W,\|\cdot\|)$ is a non-expansive map, i.e.\ $\|f(v)\|\leq \|v\|$ for $v\in V$, then $f$ induces a grading-preserving map $\operatorname{gr}f\colon \operatorname{gr}V\to\operatorname{gr}W$ of $\bR/\Gamma$-graded vector spaces.

\begin{example}
    \begin{enumerate}[(1)]
        \item If $\Gamma=\{0\}$, then a norm is just a complete $\bR$-filtration on $V$, and $\operatorname{gr} V$ is the associated $\bR$-graded vector space.
        \item If $\|\cdot\|=\|\cdot\|_{b,0}$ is a lattice norm, then $\operatorname{gr} V=V_0$ is homogeneous.
    \end{enumerate}
\end{example}

While it is defined more generally, we will only consider $\operatorname{gr}(V,\|\cdot\|)$ when $\|\cdot\|$ is a diagonalizable norm.
The following proposition, whose proof we omit, explains the relevance of this definition.

\begin{proposition}\label{prop:tangentconefilt}
The tangent cone to $\cN^{\mathrm{diag}}(V)\subseteq \cN(V)$ at a diagonalizable norm $\|\cdot\|$ is canonically identified with the space of complete $\bR$-filtrations on $\operatorname{gr}(V,\|\cdot\|)$ as a graded vector space.
For $\Delta(V)$ the same is true if we identify $\bR$-filtrations under the action of $\bR$ on itself by translation.
\end{proposition}

In particular, the tangent cone to $\Delta(V)$ is the cone over the set of complete $\bR$-filtrations with at least two jumps on $\operatorname{gr}(V,\|\cdot\|)$, modulo the action of the orientation preserving affine group $\mathrm{Aff}^+(\bR)$ on $\bR$ by translation and scaling by positive constants.
If $\|\cdot\|$ is a lattice norm, this is nothing but the spherical building associated with $\operatorname{gr}V\simeq\mathbf k^n$.

The $\bR$-filtrations in Proposition~\ref{prop:tangentconefilt} are a special case of the following more general construction, which is the analogue of the singular value decomposition in the non-Archimedean setting.

\begin{definition}
    Let $T\colon (V,\|\cdot\|)\to (W,\|\cdot\|)$ be a linear map between vector spaces with diagonalizable norms.
    \begin{enumerate}[(1)]
        \item Define an increasing $\bR$-filtration, $\cF^-(T)$, on $\operatorname{gr}V$ whose restriction to $(\operatorname{gr}V)_\alpha$ is 
        \[
        \left(\cF^-(T)_{\leq\beta}\right)_\alpha\coloneqq\left\{v\in (\operatorname{gr}V)_\alpha\mid\exists\text{ lift }\tilde{v}\in V:\|T\tilde{v}\|\leq e^{\alpha+\beta}\right\}.
        \]
        (Note the inequality can be written, for $v\neq 0$, as $\log(\|T\tilde v\|/\|\tilde v\|)\leq \beta$.)
        
        \item Define a decreasing $\bR$-filtration, $\cF^+(T)$, on $\operatorname{gr}W$ whose restriction to $(\operatorname{gr}W)_\alpha$ is 
        \[
     \left(\cF^+(T)_{\geq\beta}\right)_\alpha\coloneqq\left\{w\in (\operatorname{gr}W)_\alpha\mid\exists v\in V:Tv \text{ lifts }w,\|v\|\leq e^{\alpha-\beta}\right\}.
        \]
        (Note the inequality can be written, for $w\neq 0$, as $\log(\|Tv\|/\|v\|)\geq \beta$.)
        
        \item The \textbf{singular values}\footnote{They should really be called \textit{logarithms of singular values}, in analogy with the Archimedean case.} of $T$ are the numbers $\sigma_1\geq\ldots\geq \sigma_r$, $r=\operatorname{rk}T$, where the filtration $\cF^-(T)$ (equivalently: $\cF^+(T)$) jumps, where the multiplicity of the jump $\sigma_i$ in the descending list is  $\dim_\mathbf k\cF^-(T)_{\leq\sigma_i}/\cF^-(T)_{<\sigma_i}$.
    \end{enumerate}
\end{definition}

Computationally, the filtrations $\cF^\pm(T)$ can be found by the algorithm that produces the Smith normal form, i.e.\ Gaussian elimination where the pivot is chosen to have maximal absolute value.

The filtrations $\cF^\pm(T)$ are functorial with respect to non-expansive maps, as made precise in the following lemma.

\begin{lemma}\label{lem:singfiltfunc}
Let 
\[
\begin{tikzcd}
    V \arrow[r,"f"]\arrow[d,swap,"T"] & V'\arrow[d,"T'"] \\
    W \arrow[r,swap,"g"] & W' 
\end{tikzcd}
\]
be a commutative diagram of normed vector spaces over $K$, where $f$ and $g$ are non-expansive. 
Then $\operatorname{gr}f$ and $\operatorname{gr}g$ are filtration preserving in the sense that
\[
(\operatorname{gr}f)(\cF^-(T)_{\leq\beta})\subseteq \cF^-(T')_{\leq\beta},\qquad (\operatorname{gr}g)(\cF^+(T)_{\geq\beta})\subseteq \cF^+(T')_{\geq\beta}.
\]
\end{lemma}

\begin{proof}
Let $v\in (\cF^-(T)_{\leq\beta})_\alpha$.
By definition, there is a lift $\tilde{v}\in V$ of $v$ such that $\|T\tilde v\|\leq e^{\alpha+\beta}$.
Take $\tilde v'\coloneqq f(\tilde v)\in V'$, then this is a lift of $(\operatorname{gr}f)_\alpha(v)$ and
\[
\|T'\tilde v'\|=\|T'f(\tilde v)\|=\|g(T\tilde v)\|\leq \|T\tilde v\|\leq e^{\alpha+\beta},
\]
thus $(\operatorname{gr}f)_\alpha(v)\in (\cF^-(T')_{\leq\beta})_\alpha$.

Similarly, suppose $w\in (\cF^+(T)_{\geq\beta})_\alpha$.
By definition, there is a $v\in V$ so that $Tv$ lifts $w$ and $\|v\|\leq e^{\alpha-\beta}$.
Take $v'\coloneqq f(v)$, then $T'v'=T'f(v)=g(Tv)$ lifts $(\operatorname{gr}g)_\alpha(w)$ and
\[
\|v'\|=\|f(v)\|\leq\|v\|\leq e^{\alpha-\beta}
\]
thus $(\operatorname{gr}g)_\alpha(w)\in(\cF^+(T')_{\geq\beta})_\alpha$.
\end{proof}

The convexity of certain sums of $\sigma_i$'s will be used in the next subsection.

\begin{lemma}\label{lem:singvalueconvex}
Let $T\colon V\to W$ be a linear map between vector spaces over $K$, $r=\operatorname{rk}T$.
Then the partial sums of singular values
\[
s_1\coloneqq\sigma_1,\quad s_2\coloneqq\sigma_1+\sigma_2,\quad\ldots,\quad s_r\coloneqq\sigma_1+\ldots+\sigma_r
\]
define convex functions on $\cN_{\mathrm{diag}}(V)\times \cN_{\mathrm{diag}}(W)$.
\end{lemma}

\begin{proof}
Fix some $i\in\{1,\ldots,r\}$.
Since any geodesic connecting two points in $\cN_{\mathrm{diag}}(V)\times \cN_{\mathrm{diag}}(W)$ is contained in an apartment, it suffices to prove that the restriction of $s_i$ to any apartment is convex.
Thus, we fix bases of $V$ and $W$, up to scalar factors depending on the coordinates in the apartment, and present $T$ as a matrix.
Then 
\[
s_i=\max\{\log|\det M|: M \text{ is an }i\times i\text{ minor of }T\},
\]
see for example~\cite[Chapter 3]{kedlaya_book}.
Hence, the restriction of $s_i$ to the apartment is the pointwise maximum of a collection of affine functions, thus convex.
\end{proof}

\subsubsection*{Spherical building at infinity}

Any apartment $\bR^n\hookrightarrow \cN_\mathrm{diag}(V)$ can be compactified by adding a sphere $S^{n-1}$ at infinity.
The points of this $S^{n-1}$ have the following modular interpretation:
Let $V=L_1\oplus\ldots\oplus L_n$ be the splitting of $V$ into lines corresponding to the given apartment.
Then to each $c\in\bR^n$ we can assign a complete $\bR$-filtration $\cF^c$ of $V$ given by
\[
\cF^c_\alpha\coloneqq \bigoplus_{i:c_i\leq\alpha}L_i.
\]
This induces a bijection between $S^{n-1}\cong (\bR^n\setminus\{0\})/\bR_{>0}$ and the set of complete $\bR$-filtrations $\cF$ of $V$ where each subspace $\cF_\alpha\subseteq V$ is a direct sum of $L_i$'s and which have a jump away from $0\in\bR$, up to rescaling.
The union of all these spheres is then the set of all complete $\bR$-filtrations $\cF$ of $V$ which have a jump away from $0\in\bR$, up to rescaling.

Similarly, each apartment $\bR^{n-1}\hookrightarrow \Delta(V)$ can be compactified by adding a sphere $S^{n-2}$ at infinity, whose points correspond to  complete $\bR$-filtrations $\cF$ of $V$ where each $\cF_\alpha$ is a direct sum of $L_i$'s and which have at least two jumps, up to action of $\mathrm{Aff}^+(\bR)$ on $\bR$ by scaling and translation.
The union of these spheres is thus the set of complete $\bR$-filtrations $\cF$ of $V$ which have at least two jumps, up to action of $\mathrm{Aff}^+(\bR)$.
This is the \textit{spherical building at infinity}, $\Delta(V)^\infty$, of $\Delta(V)$.
This is an example of the general construction of the \textit{boundary at infinity} of a CAT(0) space~\cite[Chapter II.8]{bridson_haefliger}.

\subsection{Harmonic norms}\label{subsect:non-arch-quiver}

Fix a non-Archimedean field $K$ with valuation ring $\mathcal O$ and residue field $\mathbf k$.
As before, $\Gamma\coloneqq \nu(K^\times)\subseteq\bR$ is the valuation group.
We assume that $K$ is spherically complete so that all norms on finite-dimensional vector spaces are diagonalizable (see Proposition~\ref{prop:normdiag}).

For the purposes of this subsection, a weaker notion of ramp function than the one defined in Section~\ref{subsec:rampfunctions} suffices:
\begin{definition}
A \textbf{weak ramp function} is a continuous function
$f\colon[-\infty,+\infty)\to\mathbb R$ such that:
\begin{enumerate}[(1)]
    \item $f|_{\mathbb R}$ is $C^1$,
    \item $f'(\lambda)>0$ for all $\lambda\in\mathbb R$,
    \item $f'$ is strictly increasing,
    \item $\lim_{\lambda\to-\infty} f'(\lambda)=0$ and
          $\lim_{\lambda\to+\infty} f'(\lambda)=+\infty$.
\end{enumerate}    
\end{definition}

Let $Q$ be a quiver and fix for each $i\in Q_0$ a stability parameter $\theta_i\in\mathbb R$ and for each arrow $a\in Q_1$ a weak ramp function $f_a\colon [-\infty,+\infty)\to \bR$ .
We begin by defining a suitable category $\cC(Q)$ of metrized representations, as well as a category $\cC_\star(Q)$ which receives the singular value data of such a representation.

\begin{definition}
Define $\cC(Q)$ to be the quasi-abelian category over $\mathcal O$ whose objects are representations $(E,\rho)$ of $Q$ over $K$ together with a norm on each $E_i$, and whose morphisms are intertwiners of representations which are furthermore required to be non-expansive. 
\end{definition}

\begin{definition}
Define $\cC_\star(Q)$ to be the quasi-abelian category over $\mathbf k$ whose objects are given by:
\begin{enumerate}[(1)]
    \item for each $i\in Q_0$ a finite-dimensional $\bR/\Gamma$-graded vector space $V_i$ over $\mathbf k$, 
    \item for each $a\in Q_1$ an increasing $\bR$-filtration $\cF^{a^-}$ on $V_{s(a)}$, compatible with the grading,
    \item for each $a\in Q_1$ a decreasing $\bR$-filtration $\cF^{a^+}$ on $V_{t(a)}$, compatible with the grading.
\end{enumerate}
Morphisms are collections of linear maps $V_i\to V_i'$ compatible with the gradings and filtrations.
\end{definition}

Define $\theta^\star\colon K_0(\cC_\star(Q))\to\bR$ by
\[
\theta^\star(V)\coloneqq \sum_{i\in Q_0}\theta_i\dim(V_i) + \sum_{\substack{a\in Q_1\\ \lambda\in\bR}} f_{a}'(\lambda)\left(\dim\left(\cF^{a^-}_{\leq\lambda}/\cF^{a^-}_{<\lambda}\right) - \dim\left(\cF^{a^+}_{\geq\lambda}/\cF^{a^+}_{>\lambda}\right)\right).
\]
There is a functor $\operatorname{Core}\colon \cC(Q)\to \cC_\star(Q)$ which sends a metrized representation $(E,\rho,\|\cdot\|)$ to the object $(V_i,\cF^{a^\pm})$ of $\cC_\star(Q)$ with $V_i=\operatorname{gr}(E_i,\|\cdot\|)$, $\cF^{a^\pm}=\cF^\pm(\rho(a))$.
A morphism $(f_i)$ in $\cC(Q)$ is sent to $(\operatorname{gr}(f_i))$ which preserves filtrations by Lemma~\ref{lem:singfiltfunc}.
There is also the obvious forgetful functor $\cC(Q)\to\mathbf{Rep}(Q,K)$.

\begin{definition}
A norm $\|\cdot\|$ on a representation $(E,\rho)$ of $Q$ over $K$, i.e.\ a lift to $\cC(Q)$, is \textbf{harmonic} if $\operatorname{Core}(E,\rho,\|\cdot\|)\in\cC_\star(Q)$ is $\theta^\star$-semistable.\footnote{We use semistability rather than polystability in the definition, since the limiting moment-map equation may degenerate to a semistability condition on the associated graded object.}
\end{definition}

To relate this to the quiver $Q^\star$ introduced in Section~\ref{subsec:nonarchlimit}, we note that the definition of $Q^\star$ generalizes to arbitrary non-Archimedean field $K$, with the only difference that all vector spaces over $\mathbf k$ are $\bR/\Gamma$-graded. Equivalently, one considers the quiver with a separate copy of the vertex $i$ (resp. $(a,k,\pm)$) for each homogeneous summand component of $E^\star_i$ (resp. $E^{\star,\pm}_{a,k}$).
Then $\operatorname{Core}(E,\rho,\|\cdot\|)$ is $\theta^\star$-(semi)stable iff $E^\star$ is $\theta^\star$-(semi)stable.

\begin{example}
For the simplest non-trivial example we take $Q=\bullet\to\bullet$ to be the $A_2$ quiver and the indecomposable representation $K\xrightarrow{1}K$.
This representation is $\theta$-stable for stability parameters $\theta_1,\theta_2$ attached to the first and second vertex, respectively, iff $\theta_1=-\theta_2<0$.
Assume this to be the case.
Norms on $K$ are classified by $\|1\|\in(0,+\infty)$, and up to overall rescaling we can assume that the norm on the first $K$ is the absolute value $|\_|$ and the norm on the second $K$ is $\|v\|=e^\lambda|v|$ for some $\lambda\in\bR$.
This norm is harmonic iff $\theta_2=f'(\lambda)$, where $f$ is the chosen weak ramp function for the single arrow. Existence and uniqueness of a solution $\lambda\in\bR$ to this equation are ensured by the conditions on $f$.
Note that in the limiting case where $\theta_1=\theta_2=0$ and the representation becomes strictly semistable, a harmonic norm does not exist.
\end{example}

\begin{example}
Let $Q$ be the $D_4$-type quiver
\[
\begin{tikzcd}
   \underset{-\frac{1}{2}}{\mathcircled{2}} \arrow[dr,"a_2"] & & \\
    & \underset{1}{\mathcircled{0}} & \underset{-1}{\mathcircled{1}} \arrow[l,"a_1"] \\
    \underset{-\frac{1}{2}}{\mathcircled{3}} \arrow[ur,swap,"a_3"] & & 
\end{tikzcd}
\]
where the numbers under the nodes indicate our choice of $\theta_i$'s.
Take a representation $(E,\rho)$ where $\dim E_0=2$, $\dim E_i=1$ for $i=1,2,3$, all $\rho(a_i)$ are inclusion maps, and the three lines $E_1,E_2,E_3\subset E_0$ are distinct. 
Under these conditions, $E$ is determined up to isomorphism and indecomposable.
Our choice of $\theta$ makes $E$ semistable, but not polystable, because of the $\theta$-stable subrepresentation given by $E_1$ (at both vertices 0 and 1).

Choose a norm $\|\cdot\|$ on $E$ and let $\sigma_i$ be the singular value of the rank 1 map $\rho(a_i)$.
The representation $E^\star$ of $Q^\star$ looks as follows:
\[
\begin{tikzcd}
  \underset{f'_2(\sigma_2)-\frac{1}{2}}{V_2} & & \underset{-f'_2(\sigma_2)}{V_2} \arrow[dr] & & & & \\
    & & & \underset{1}{V_0} & \underset{-f'_1(\sigma_1)}{V_1} \arrow[l] & & \underset{f'_1(\sigma_1)-1}{V_1} \\
  \underset{f'_3(\sigma_3)-\frac{1}{2}}{V_3}  & & \underset{-f'_3(\sigma_3)}{V_3} \arrow[ur] & & & &
\end{tikzcd}
\]
where $V_0\simeq\mathbf k^2$, possibly not $\bR/\Gamma$-homogeneous, $V_i\simeq \mathbf k$, $i=1,2,3$, maps assigned to arrows are injective, and 
where the numbers under the vector spaces indicate the value of $\theta^\star$ assigned to that vertex.
Thus, the metric is harmonic iff
\[
f'_1(\sigma_1)=1,\qquad f'_2(\sigma_2)=\frac{1}{2}=f'_3(\sigma_3)
\]
(which can be achieved by suitably rescaling the norms on the $E_i$, $i=1,2,3$) and the representation of the $D_4$-part of $Q^\star$ is $\theta^\star$-semistable.
This is true in two distinct cases:
\begin{enumerate}[(1)]
    \item $V_1,V_2,V_3\subset V_0$ are all distinct (identifying $V_i$ with its image in $V_0$). This essentially reproduces the situation we started with, just over $\mathbf k$ instead of $K$, so $E^\star$ is $\theta^\star$-semistable but not polystable.
    Note that this can only happen if $V_0$ is homogeneous with respect to the $\bR/\Gamma$-grading, i.e.\ the norm on $E_0$ is a lattice norm, up to rescaling.
    \item $V_1\neq V_2=V_3$, in which case $E^\star$ is $\theta^\star$-polystable.
\end{enumerate}
\end{example}
We conclude from the above example that the existence of a harmonic norm on a representation does not imply its $\theta$-polystability, unlike in the Archimedean case.
Additionally, $\theta$-semistability is not sufficient for existence of a harmonic norm as the $A_2$ quiver example shows.
Thus, both implications in the following theorem cannot be reversed, in general.

\begin{theorem}\label{thm:nonarch_harm}
Let $Q$ be a quiver, $\theta_i\in\bR$, $i\in Q_0$, $f_a$ weak ramp functions, $a\in Q_1$, and $E=((E_i)_i,\rho)$ a representation of $Q$ over $K$. 
Then there are implications:
\[
E\text{ polystable}\implies E\text{ admits a harmonic norm} \implies E\text{ semistable}.
\]
\end{theorem}

\begin{question}
Can the spherical completeness assumption on $K$ be omitted in Theorem~\ref{thm:nonarch_harm}?
\end{question}

The proof of Theorem~\ref{thm:nonarch_harm} will occupy the remainder of this subsection.
Fix $Q$, $\theta_i$, $f_a$, and a representation $E$ of $Q$ over $K$ as before.
Let
\[
\cN(E)\coloneqq \prod_{i\in Q_0}\cN(E_i)
\]
be the space of (diagonalizable, since $K$ is spherically complete) norms on $E$.
Motivated by the considerations in Section~\ref{subsec:nonarchlimit}, we introduce the potential function $S\colon \cN(E)\to\bR$ with
\[
S\coloneqq\sum_{a\in Q_1}\sum_{\lambda\in \bR}f_a(\lambda)\dim\left(\cF^-(\rho(a))_{\leq\lambda}/\cF^-(\rho(a))_{<\lambda}\right)-\sum_{i\in Q_0}\theta_i\log\operatorname{vol}(E_i)
\]
where we have fixed some norm on each $\operatorname{Det}(E_i)\cong K$ so that $\operatorname{vol}(E_i)$ is well-defined for a given choice of norm on $E_i$.

\begin{example}
Consider the ``rank 1'' case where $E_i=K$ for all $i\in Q_0$ and suppose, after possibly removing some arrows, that $\rho(a)\neq 0$ for all $a\in Q_1$.
Set $c_a\coloneqq \log|\rho(a)|$ and let $e^{x_i}|\cdot|$ be the norm on $E_i$, $i\in Q_0$, i.e.\ $x\in\bR^{Q_0}$ parametrizes the norms on the representation.
The potential is then 
\[
S(x)=\sum_{a\colon i\to j}f_a(x_j-x_i+c_a)-\sum_{i\in Q_0}\theta_ix_i
\]
and its critical points, the harmonic norms, given by solutions to the system of equations
\[
\theta_i=\sum_{a\colon j\to i}f_a'(x_i-x_j+c_a)-\sum_{a\colon i\to j}f_a'(x_j-x_i+c_a), \qquad i\in Q_0.
\]
In particular, there is no dependence on $K$, and in fact the equation is the same in both the Archimedean and the non-Archimedean case.
\end{example}

\begin{example}
Let $Q$ be the quiver with a single vertex and two arrows and the following simple representation:
\[
\begin{tikzcd}
    K^2 \arrow[loop left,"A"] \arrow[loop right,"B"]
\end{tikzcd},\qquad 
A=\begin{pmatrix} 0 & 1 \\ 0 & 0 \end{pmatrix},\qquad B=\begin{pmatrix} 0 & 0 \\ 1 & 0 \end{pmatrix}
\]
Assume, for concreteness, that $K=\bQ_2$ are the 2-adic numbers with the usual 2-adic norm such that $|2^n|=2^{-n}$ and choose $f_a(\lambda)=2^\lambda$.
Figure~\ref{fig:Sgraph} shows part of the graph of $S$.
\end{example}

\begin{figure}
    \centering
    \includegraphics[width=0.65\linewidth]{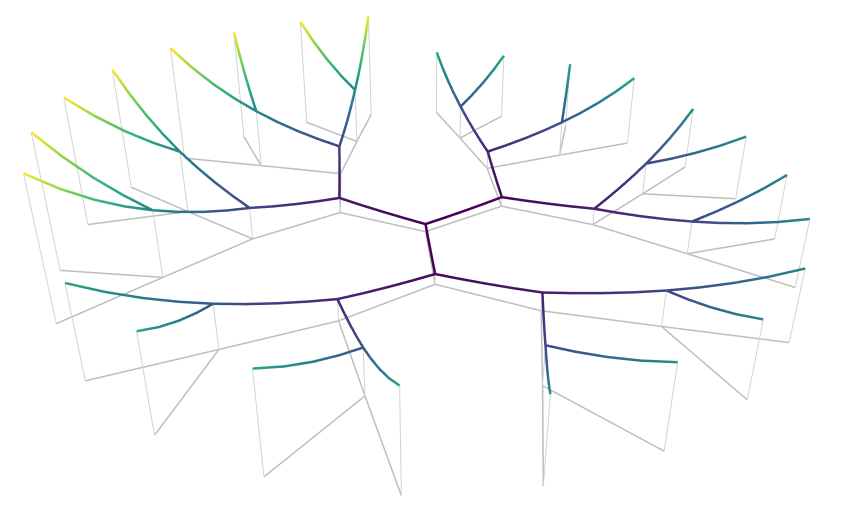}
    \caption{Graph of the function $S$ on the ball of radius 4 in the affine building $\cN(K^2)/_\sim$ with center the standard norm (the unique global minimum of $S$). $S$ increases slowest along the standard apartment $K^2=K\oplus K$.}
    \label{fig:Sgraph}
\end{figure}

\begin{lemma}\label{lem:Sconvex}
The function $S\colon \cN(E)\to\bR$ is convex.
\end{lemma}

\begin{proof}
By definition, convexity of $S$ means that for any geodesic $c\colon [0,L]\to\cN(E)$ the composition $S\circ c$ is convex.
Since any such geodesic is contained in an apartment, it suffices to show that the restriction of $S$ to any apartment is convex.
As the term
\[
-\sum_{i\in Q_0}\theta_i\log\operatorname{vol}(E_i)
\]
is an affine function when restricted to any apartment, the lemma is implied by the following claim:
If $T\colon V\to W$ is a linear map of rank $r$ between vector spaces $V,W$ over $K$ and $f$ a weak ramp function, then the function
\[
P\colon \cN(V)\times\cN(W)\to\bR,\qquad P\coloneqq \sum_{i=1}^rf(\sigma_i),
\]
where $\sigma_1\geq \sigma_2\geq\ldots\geq\sigma_r$ are the singular values of $T$, is convex.
In fact, we will only use that $f\colon \bR\to\bR$ is convex and increasing.

To prove this, let $c\colon [0,L]\to\cN(E)$ be a geodesic, $p\coloneqq c(0),q\coloneqq c(L)$, $t\in [0,1]$, $m\coloneqq c(tL)$,
\[
x_i\coloneqq t\sigma_i(p)+(1-t)\sigma_i(q),\qquad y_i\coloneqq \sigma_i(m),
\]
then by Lemma~\ref{lem:singvalueconvex} we have
\[
\sum_{i=1}^jx_i\geq \sum_{i=1}^jy_i, \qquad j=1,\ldots,r.
\]
By convexity of $f$ and Karamata's inequality (for non-decreasing convex functions $f$) we get 
\begin{align*}
tP(p)+(1-t)P(q)&=\sum_{i=1}^r\left(tf(\sigma_i(p))+(1-t)f(\sigma_i(q))\right) \\
&\geq \sum_{i=1}^rf(x_i)\geq \sum_{i=1}^rf(y_i)=P(m)
\end{align*}
as claimed.
\end{proof}

In the following, it will be convenient to introduce certain notations around $\bR$-filtrations.
First, if $\cF$ is an $\bR$-filtration and $\phi\colon \bR\to\bR$ an increasing homeomorphism, then $\phi(\cF)_\lambda=\cF_{\phi^{-1}(\lambda)}$.
Thus, if $\cF$ jumps at $\lambda_1<\ldots<\lambda_n$, then $\phi(\cF)$ jumps at $\phi(\lambda_1)<\ldots<\phi(\lambda_n)$ with the same subquotients.
Next, if $\cF$ is a finite $\bR$-filtration of some object in an abelian or exact category $\cA$ and $\theta\colon K_0(\cA)\to\bR$ is a linear functional, then as in Lemma~\ref{lem:thetastablefilt} we define
\[
\langle\theta,\cF\rangle\coloneqq \sum_{\lambda\in\bR}\lambda\, \theta(\cF_\lambda/\cF_{<\lambda}).
\]
Finally, suppose $\cF$ and $\cG$ are $\bR$-filtrations on a finite-dimensional vector space.
Then
\[
\langle\cF,\cG\rangle\coloneqq \sum_{\alpha,\beta\in\bR}\alpha\beta\dim\frac{\cF_\alpha\cap\cG_\beta}{\cF_\alpha\cap\cG_{<\beta}+\cF_{<\alpha}\cap\cG_\beta}.
\]

\begin{lemma}\label{lem:Smin}
A norm on $E$ is harmonic iff it minimizes $S$.
\end{lemma}

\begin{proof}
Fix a norm on $E$, i.e.\ a point in $\cN(E)$.
By Proposition~\ref{prop:tangentconefilt}, the tangent cone to that norm is given by the space of complete $\bR$-filtrations on $\operatorname{gr}E$.
Let $\cG$ be such an $\bR$-filtration, then we want an expression for the first derivative of $S$ in the direction $\cG$, $(\delta S)(\cG)$.
We have
\begin{align*}
(\delta S)(\cG) &= \sum_{a\in Q_1}\left(\langle f_a'(\cF^+(\rho(a))),\cG\rangle-\langle f_a'(\cF^-(\rho(a))),\cG\rangle\right)-\langle\theta,\cG\rangle\\
&= \sum_{a\in Q_1}\sum_{\alpha,\beta\in\bR}f'_a(\alpha)\beta\Bigg(\dim\frac{\cF^+(\rho(a))_\alpha\cap\cG_\beta}{\cF^+(\rho(a))_\alpha\cap\cG_{<\beta}+\cF^+(\rho(a))_{<\alpha}\cap\cG_\beta}\\
&\qquad\qquad\qquad\qquad\quad\quad -\dim\frac{\cF^-(\rho(a))_\alpha\cap\cG_\beta}{\cF^-(\rho(a))_\alpha\cap\cG_{<\beta}+\cF^-(\rho(a))_{<\alpha}\cap\cG_\beta}\Bigg)\\
&\quad-\sum_{i\in Q_0}\sum_{\alpha\in\bR}\alpha\,\theta_i\dim((\cG_{\alpha}\cap E_i)/(\cG_{<\alpha}\cap E_i))\\
&= -\langle\theta^\star,\cG\rangle.
\end{align*}
By Lemma~\ref{lem:thetastablefilt}, the norm is harmonic iff $-\langle\theta^\star,\cG\rangle\geq 0$ for all $\cG$.
Equivalently, the norm locally minimizes $S$. 
Since $S$ is convex by Lemma~\ref{lem:Sconvex}, the claim follows.
\end{proof}

For the statement of the following lemma, note that if $\theta(E)=0$, then $S$ is invariant under rescaling the norm and thus descends to $\cN(E)/_\sim$. In particular, this is true if $E$ is stable.

\begin{lemma}\label{lem:Sinfinity}
If $E$ is stable, then the sublevel sets $\{x\mid S(x)\leq C\}$ of $S\colon \cN(E)/_\sim\to\bR$ are bounded.
\end{lemma}

\begin{proof}
Fix a lattice norm $\|\cdot\|$ on $E$ and let $p_0\in \cN(E)/_\sim$ be the corresponding point in the building.
The statement of the lemma is equivalent to the claim that for any $C>0$ there is an $R>0$ such that if $x$ lies along a geodesic ray starting at $p_0$ and has distance at least $R$ from $p_0$, then $S(x)\geq C$. Importantly, $R$ should not depend on the choice of ray.

By the discussion of the spherical building at infinity in the previous subsection, a ray starting at $p_0$ corresponds to a complete $\bR$-filtration $\cF$ of $E$ as a $Q_0$-graded vector space (not necessarily compatible with the maps $\rho(a)$, $a\in Q_1$) which has at least two steps, modulo the action of $\mathrm{Aff}^+(\bR)$. 
We fix the $\mathrm{Aff}^+(\bR)$-ambiguity by requiring the first step of $\cF$ to be at $0\in\bR$ and the last step of $\cF$ to be at $1\in\bR$.

There are two cases to consider:
\begin{enumerate}[(1)]
    \item If $\cF$ is a filtration by subrepresentations, i.e.\ preserved by all $\rho(a)$, $a\in Q_1$, then by Lemma~\ref{lem:thetastablefilt} we have $\langle -\theta,\cF\rangle>0$.
    More precisely, \eqref{eq:thetafilt_pf} shows that, thanks to our normalization of $\cF$, there is a positive lower bound
    \[
    \langle -\theta,\cF\rangle\geq \frac{\theta_{\mathrm{min}}}{\dim E}> 0
    \]
    where $\theta_{\mathrm{min}}$ is the minimal value of $\theta(A)$ as $A$ ranges over non-trivial subrepresentations of $E$ (the dimension vectors of which form a finite set).
    Thus, the second term in $S$ increases at a rate of at least $\frac{\theta_{\mathrm{min}}}{\dim E}$. 
    On the other hand, the first term is bounded below on all of $\cN(E)/_\sim$  by
    \[
    \sum_{a\in Q_1}f_a(-\infty)\operatorname{rk}\rho_a.
    \]
    \item If $\cF$ is \textit{not} a filtration by subrepresentations, then there is some $a\in Q_1$ such that $\rho(a)$ does not preserve $\cF$. 
    Looking at the second term of $S$, its slope $\langle -\theta,\cF\rangle$ along the ray can in this case be negative, but has a lower bound 
    \[ 
    -\dim E\cdot \max \{|\theta_i| \mid i\in Q_0\} 
    \]
    which depends on $E$ only. 
    This is compensated by the first term: The largest singular value, $\sigma_1$, of $\rho(a)$ eventually (after an a priori bounded amount of time) grows at a rate
    \[
    \alpha \coloneqq \min\{\beta\mid \rho(a)(\cF_\bullet)\subseteq \cF_{\bullet+\beta}\}>0
    \]
    as one travels along the geodesic ray.
    Since $f_a'(\sigma_1)\to +\infty$ as $\sigma_1\to +\infty$ (the fourth condition on a weak ramp function), the first term eventually dominates and we get a linear lower bound on $S$ along the ray.
\end{enumerate}
\end{proof}

\begin{proof}[Proof of Theorem~\ref{thm:nonarch_harm}]
Let $\Delta\coloneqq\cN(E)/_\sim$ be the space of norms on the representation $E$, up to rescaling.
This is a CAT(0) space and metrically complete by our assumption on $K$, thus a Hadamard space.
If $E$ is stable, then $S$ has at least one minimum by Lemma~\ref{lem:Sconvex}, Lemma~\ref{lem:Sinfinity}, and the standard fact about Hadamard spaces that lower semicontinuous convex functions with bounded sublevel sets have minima, see \cite[Lemma 2.2.19]{bacak_convex}.
Lemma~\ref{lem:Smin} then implies existence of a harmonic norm on $E$. 
The general polystable case follows by taking direct sums of stable representations with harmonic norms.

For the second implication in the statement of the theorem, suppose that $E$ is not semistable, so there is a subrepresentation $0\neq A\subset E$ with $\theta(A)<0$. Then, for any choice of metric on $E$, there is a corresponding subobject $\operatorname{Core}A\subset \operatorname{Core}E$ and $\theta^\star(\operatorname{Core}A)=\theta(A)<0$, so $E$ cannot admit a harmonic norm. 
\end{proof}

\section{Spectral networks in the plane}
\label{section:spectral}

The goal of this section is to formulate a conjectural correspondence between spectral networks in the plane and quiver representations over a non-Archimedean field equipped with a pair of ascending-phase filtrations with metrized subquotients.
This is, by itself, not a full-fledged example of a categorical K\"ahler structure, but represents the local picture of hypothetical examples of such structures coming from Fukaya categories of surfaces with coefficients, cf. Section~\ref{subsubsec:coeff}.

In the preliminary Section~\ref{subsec:nonexpansive} we discuss the relation between semistability of a quiver representation over a non-Archimedean field $K$ and its residue field $\mathbf k$.
Here, We restrict to those representations which admit a norm for which all maps $\rho(a)$ are non-expansive, and the whole discussion can be seen as a much simplified variant of the one in Section~\ref{sec:quivers_nonarch}.
In Section~\ref{subsec:harmpaths} we define our notion of spectral networks in the plane in terms of certain curves in the space of norms on a quiver representation. We then formulate the conjectural correspondence with pairs of ascending-phase filtrations with metrized subquotients.
Finally, in Section~\ref{subsec:honeycombs}, we discuss a special case of the general conjecture which is related to Horn's problem and \textit{honeycomb diagrams}.

\subsection{Non-expansive representations}\label{subsec:nonexpansive}

Fix a complete non-Archimedean field $K$ with absolute value $|\cdot|\colon K\to [0,\infty)$, valuation ring $\mathcal O\coloneqq \{a\in K: |a|\leq 1\}$, maximal ideal $\mathfrak m\coloneqq\{a\in K: |a|<1\}$, and residue field $\mathbf k\coloneqq\mathcal O/\mathfrak m$.
Also fix a quiver $Q$ and central charges $z_i\in\bR_{>0}e^{\sqrt{-1}(0,\pi]}$, $i\in Q_0$.

In the following, \textit{norm} will always mean \textit{diagonalizable norm} (see Section~\ref{subsec:prelimnorm}). 

\begin{definition}
A norm on a representation $(E,\rho)$ of $Q$ over $K$ is \textbf{non-expansive} if each $\rho(a)$, $a\in Q_1$, is a non-expansive map with respect to this norm.
If a non-expansive norm exists on $E$, we say that $(E,\rho)$ itself is non-expansive.
Denote by $\mathbf{Rep}^{\mathrm{ne}}(Q,K)\subseteq\mathbf{Rep}(Q,K)$ the full subcategory of non-expansive representations.
This subcategory is closed under subobjects, quotients, and extensions, in particular itself abelian.
\end{definition}

\begin{remark}
$\mathbf{Rep}^{\mathrm{ne}}(Q,K)=\mathbf{Rep}(Q,K)$ iff $Q$ is acyclic (or $K$ is trivially valued).
To see this, note that if $c=a_na_{n-1}\cdots a_1$ is an oriented cycle of $Q$ and $\rho(c)$ has an eigenvalue $\mu$ with $|\mu|>1$, then $(E,\rho)$ cannot be in $\mathbf{Rep}^{\mathrm{ne}}(Q,K)$.
\end{remark}

Given a representation $(E,\rho)$ over $K$ equipped with a non-expansive norm we obtain a representation $(\operatorname{gr}E,\operatorname{gr}\rho(a))$ of $Q$ over $\mathbf k$, where the functor $\operatorname{gr}$ was defined in Section~\ref{subsec:prelimnorm} and is just $\_\otimes_{\cO}\mathbf k$ in the case of lattice norms.

The parameters $z_i$, $i\in Q_0$, give us central charges 
\[ 
Z\colon K_0(\mathbf{Rep}^{\mathrm{ne}}(Q,K))\to\bC, \qquad Z\colon K_0(\mathbf{Rep}(Q,\mathbf k))\to\bC
\]
satisfying the HN property and the support property.

\begin{definition}
A non-expansive norm on a representation $(E,\rho)$ is \textbf{harmonic} if $(\operatorname{gr}E,\operatorname{gr}\rho)$ is $Z$-semistable in $\mathbf{Rep}(Q,\mathbf k)$.
\end{definition}

\begin{proposition}\label{prop:neharm}
A non-expansive representation $(E,\rho)$ of $Q$ over $K$ admits a harmonic norm iff it is $Z$-semistable in $\mathbf{Rep}^{\mathrm{ne}}(Q,K)$.
\end{proposition}

\begin{proof}
One direction is immediate: If $(E,\rho)$ has a subrepresentation of strictly larger phase, then so does $(\operatorname{gr}E,\operatorname{gr}\rho)$ for any choice of non-expansive norm.

For the converse, let $E=(E,\rho)\in \mathbf{Rep}^{\mathrm{ne}}(Q,K)$. 
We must show that either there exists a harmonic norm on $E$, or that $E$ is not semistable.

Pick some non-expansive norm $\|\cdot\|$ on $E$.
If the norm is harmonic, we are done.
If not, choose $0\neq A_0\subset\operatorname{gr}E$ to be the maximal destabilizing subobject, i.e.\ the first HN factor of $\operatorname{gr}E$.
We can choose, by diagonalizability of $\|\cdot\|$, an orthogonal splitting $E=E'\oplus E''$ such that $\operatorname{gr}E'=A_0$.
There are two cases.
If $E'$ is a subrepresentation of $(E,\rho)$, i.e.\ preserved by all $\rho(a)$, $a\in Q_1$, then we have found a destabilizing subobject and are done.
Otherwise, there is an $a_*\in Q_1$ so that the block $\rho(a_*)_{21}\colon E'\to E''$ of $\rho(a_*)$ is non-zero. 
We can choose $a_*\in Q_1$ so that the operator norm $t\coloneqq \|\rho(a_*)_{21}\|\in (0,1)$ is maximal, and modify the norm on $E$ by multiplying its $E''$ component by $1/t>1$.
Denote this new norm by $\|\cdot\|_{mod}$.
The maximality of $t$ ensures that all 21-blocks remain non-expansive after this rescaling, while the 12-blocks become strictly contracting, hence $\rho$ is still non-expansive with respect to $\|\cdot\|_{mod}$.
With respect to the new norm, $\operatorname{gr}\rho(a_*)_{21}\neq 0$ and $\operatorname{gr}\rho(a)_{12}=0$ for all $a\in Q_1$.
This means there is a non-split short exact sequence in $\mathbf{Rep}(Q,\mathbf k)$ of the form
\begin{equation}\label{eq:reversedSES}
0\to B_0\to \operatorname{gr}(E,\|\cdot\|_{mod})\to A_0\to 0    
\end{equation}
i.e.\ with the roles of the subobject and quotient object reversed.
Applying Proposition~\ref{prop:mass} to~\eqref{eq:reversedSES} yields
\[
m(\operatorname{gr}(E,\|\cdot\|_{mod}))<m(A_0)+m(B_0)=m(\operatorname{gr}(E,\|\cdot\|)).
\] 
By discreteness of the set of masses (Proposition~\ref{prop:massdiscrete}), this process must terminate by either finding a harmonic norm or a destabilizing subobject of $E$.
\end{proof}

\begin{remark}
Inspection of the proof of Proposition~\ref{prop:neharm} shows that the same statement is true if we restrict everywhere to lattice norms, instead of diagonalizable norms.
\end{remark}

\begin{remark}
The notion of harmonicity considered here is reminiscent of the one considered in Section~\ref{subsect:non-arch-quiver}, and the statement of Proposition~\ref{prop:neharm} is somewhat similar to that of Theorem~\ref{thm:nonarch_harm}, although the proof is much simpler and does not require the spherical completeness assumption on $K$.

\begin{center}
\def\arraystretch{1.5}
\begin{tabular}{c|c}
    Proposition~\ref{prop:neharm} & Theorem~\ref{thm:nonarch_harm} \\
    $\mathbf{Rep}^{\mathrm{ne}}(Q,K)$ & $\mathbf{Rep}(Q,K)$ \\
    $\mathbf{Rep}(Q,\mathbf k)$ & $\cC_\star(Q)$ \\
    ? & $S\colon \cN_\mathrm{diag}(E)\to\bR$
\end{tabular}
\end{center}
\end{remark}

\begin{remark}
There is also an analogy between Proposition~\ref{prop:neharm} and Theorem~\ref{thm:comonads} below, and the idea of the proof of Proposition~\ref{prop:neharm} will be re-used in the proof of that theorem.
\end{remark}

\subsection{Harmonic paths}\label{subsec:harmpaths}

Fix a non-Archimedean field $K$ and a quiver $Q$ with central charge as before.
Additionally, we require that $\operatorname{Im}(z_i)>0$, which is generically satisfied.
We also fix a smooth area form $\omega=f(x,t)dx\wedge dt$ on $\mathbb R^2=\mathbb R^2_{xt}$ satisfying the following properties:
\begin{enumerate}[(1)]
    \item The area of the strip between any pair of parallel lines is finite.
    \item The area of any conical sector between intersecting lines is infinite.
\end{enumerate}
For instance, one can take $f(x,t)=1/(1+x^2+t^2)$.
This gives us another coordinate, $\lambda$, on $\mathbb R^2$ defined by
\[
\lambda(x,t)\coloneqq \int_0^{x}f(u,t)du.
\]

In the following we consider the closed subset $\cN_{\mathrm{ne}}(E)\subseteq \cN_\mathrm{diag}(E)$ of non-expansive norms on a representation $E=(E,\rho)$.
We have a notion of \textit{piecewise smooth curve} in $\cN_{\mathrm{ne}}(E)$: these are curves which, on each piece, are contained in a single apartment and are smooth there.

\begin{definition}
Suppose $c\colon \mathbb R\to \cN_{\mathrm{ne}}(E)$ is a continuous curve which is smooth away from a finite set $D\subset \mathbb R$ of the domain. 
Thus, for $t\notin D$, the velocity $\frac{dc}{dt}$ is a complete $\mathbb R$-filtration $\cF_\lambda^{t}$ of $\operatorname{gr}(E,c(t))$.
The \textbf{support} of $c$ is the closure, $\mathcal W(c)$, of the set of pairs $(x,t)\in\mathbb R^2$ such that $\operatorname{gr}_{\lambda(x,t)}\cF^t\neq 0$.
We restrict to curves $c$ where $\cW(c)$ is topologically a graph with finitely many vertices.
Each edge of the graph comes with a label, the object $\operatorname{gr}_{\lambda(x,t)}\cF^t\in\mathbf{Rep}(Q,\mathbf k)$.
The path is said to be \textbf{harmonic} if every edge of $\mathcal W(c)$ is labeled by a $Z$-semistable object in $\mathbf{Rep}(Q,\mathbf k)$, and the edge is a straight line with direction $\pm\overline{Z(A)}$. (See Figure~\ref{fig:specnetdeform}, right.)
\end{definition}

Note that, by definition, there can be no horizontal edges (constant $t$), which is why we impose the condition $\Imm(z_i)>0$ on the central charge.
The reason for using the complex conjugate $\overline{Z(A)}$ of the central charge, is that we want
\[
\text{total phase }=\text{ phase }(\text{=angle}/\pi)\text{ of edge } + \text{ phase of object on edge}
\]
to be constant across the graph, cf.~\cite[Section 4.2]{HKS} and~\eqref{eq:totphase}.

\begin{remark}
The above definition of spectral networks in the plane in terms of paths in the space of norms breaks the symmetry between $x$ and $t$.
An equivalent definition, more natural but more involved, can be made in terms of curved $A_\infty$-deformations and solutions to the $A_\infty$ Maurer--Cartan equation as in \cite{HKS}.
The supposed equivalence of the two definitions is established by a version of family Floer homology with respect to the family of horizontal lines in $\bR^2$.
This, minus the harmonicity condition, is similar in spirit to the equivalence between the augmentation category of a Legendrian link $\Lambda$ and the category of microlocal sheaves on $\bR^2$ with rank 1 singular support along the front of $\Lambda$~\cite{NRSSZ}.
\end{remark}

Suppose we are given a harmonic path $c\colon \bR\to\cN_{\mathrm{ne}}(E)$. We claim that we can assign to it a pair of ascending-phase filtrations on $E$ together with harmonic norms on the subquotients. These filtrations encode the asymptotic behavior of $c(t)$ as $t\to \pm\infty$.
First, suppose $t\gg 0$ so that vertices of $\mathcal W(c)$ are strictly below the horizontal line $\mathbb R\times \{t\}$.
Then the phases of the semistable objects $\operatorname{gr}_\lambda\cF^t$ increase with increasing $\lambda$ (going right, compare Figure~\ref{vertexdiagram}), so the filtration of $E$ by rate of growth of $c(t)$, as $t\to+\infty$, is an ascending-phase filtration $\cF_\phi^+$ of $E$.
Here, the area assumptions on $\omega$ ensure that the correspondence 
\[ 
\text{growth-rate of norm}\longleftrightarrow \text{phase }\phi
\]
is one-to-one.
The induced norms on the subquotients $\operatorname{gr}_\phi\cF^+$ are harmonic (because objects labeling edges are by assumption semistable), but not independent of $t$.
To fix this, we rescale these norms by $e^{-A(t)}$, where 
\[
A(t)\coloneqq\int_0^t\int_0^{\mu s}f(x,s)dxds, \qquad \mu\coloneqq-\cot\left(\pi\phi\right),
\]
and take the limit $t\to +\infty$.
Similarly, we get another ascending-phase filtration, $\cF_\phi^-$, from the limit $t\to-\infty$.
Indeed, suppose $t\ll 0$ so that vertices of $\mathcal W(c)$ are strictly above the horizontal line $\mathbb R\times \{t\}$.
Then the phases of the semistable objects $\operatorname{gr}_\lambda\cF^t$ increase with decreasing $\lambda$ (going left), so the filtration of $E$ by rate of growth of $c(t)$, as $t\to-\infty$, is an ascending-phase filtration $\cF_\phi^-$ of $E$.
The induced norms on the subquotients $\operatorname{gr}_\phi\cF^-$ are obtained by rescaling by $e^{-A(t)}$ and taking the limit $t\to -\infty$.

We are now ready to state our main conjecture in this section.

\begin{conjecture}\label{conj:bubbling}
Fix $K$, $Q$, $z_i$, $\omega$, and $E$ as before.
The map described above induces a one-to-one correspondence between:
\begin{enumerate}[(1)]
    \item harmonic paths in $\cN_{\mathrm{ne}}(E)$, and
    \item pairs of ascending-phase filtrations on $E$ together with harmonic norms on their subquotients.
\end{enumerate}
\end{conjecture}  

\begin{remark}
Following the proposal in Section~\ref{subsubsec:coeff}, we expect that the data of a quiver $Q$, or rather its category of representations, together with the $z_i$, can be replaced by a more general $\bC$-linear coefficient category $\cC$ with stability condition $\sigma_\cC$. In this case, $K=Nov$ is the Novikov field.
\end{remark}

\begin{example}
In the simplest case, $Q=\bullet$ is a single vertex with no arrows, so $E$ is just a vector space of some finite dimension $n=\dim E$ over $K$.
Let $z$ be the central charge assigned to the vertex and $\mu\coloneqq -\frac{\operatorname{Re}(z)}{\operatorname{Im}(z)}$.
Suppose first that $E=K$ and let $c(t)=e^{\alpha(t)}|\cdot|$ be a harmonic path.
The support $\cW(c)$ is given by the curve $x(t)$ with $\lambda(x(t),t)=\alpha'(t)$. 
By harmonicity, $x(t)=\mu t+x_0$ for some $x_0\in\bR$.
Hence,
\[
\alpha(t)=\int_0^t\lambda(\mu s+x_0,s)ds+C=\int_0^t\int_0^{\mu s+x_0}f(x,s)dxds+C
\]
for some constant $C\in\bR$.
The asymptotic norms on the subquotients of the (trivial) ascending-phase filtrations $\cF^\pm$ are $\|\cdot\|_\pm=e^{\lambda_\pm}|\cdot|$ where
\[
\lambda_\pm\coloneqq \int_{0}^{\pm\infty}\int_{\mu t}^{\mu t+x_0}f(x,t)dxdt+C
\]
where the integral is finite by the first condition on $f$. 
Moreover, given $\lambda_\pm$ we can find the corresponding $x_0$ and $C$ by the second condition on $f$, which ensures that
\[
\lim_{x_0\to\pm\infty}\int_{-\infty}^{+\infty}\int_{\mu t}^{\mu t+x_0}f(x,t)dxdt=\pm\infty.
\]

The case where $n=\dim E$ is arbitrary can be reduced to the $n=1$ case as follows. 
Suppose we are given a pair of diagonalizable norms $\|\cdot\|_\pm$ on $E$. There is a basis of $E$ which diagonalizes both of them.
Each basis vector yields a summand of the harmonic path, whose support is a union of $n$ parallel lines, counted with multiplicity, which is the $\mathbf k$-dimension of the vector space labeling them.
\end{example}

\begin{example}
\label{ex:harmpathA2}
Let $Q=\bullet\to\bullet$ be the $A_2$ quiver.
The three indecomposable representations, up to isomorphism, in  $\mathbf{Rep}(Q,K)$ are
\[
A\coloneqq(0\to K),\qquad B\coloneqq(K\xrightarrow{1} K),\qquad C\coloneqq(K\to 0)
\]
which fit into a short exact sequence $0\to A\to B\to C\to 0$.
Take the central charge with $Z(A)=e^{\pi \sqrt{-1}/6}$, $Z(C)=e^{5\pi \sqrt{-1}/6}$, $Z(B)=Z(A)+Z(C)=\sqrt{-1}$.
(The precise values are not so important in what follows, only the fact that $B$ is stable.)
Then any norm on $A$ or $C$ is harmonic, and a norm on $B$ is harmonic iff its restrictions to the two $K$'s coincide, i.e.\ $\rho(a)$ is norm-preserving.

Take $E=B^2=(K^2\xrightarrow{\mathrm{id}} K^2)$ and consider ascending-phase filtrations of the form $A\hookrightarrow E_1\hookrightarrow E_2\hookrightarrow E$, where $E_2/E_1\cong B$, $E/E_2\cong C$.
These are classified by a pair of lines in $K^2$: $\Imm (A\to E)$ and $\Ker(E\to E/E_2)$.
An example of a harmonic path with asymptotics of this type is shown in Figure~\ref{fig:hexagon}.
Other topologies are possible: The central polygon could have fewer vertices or collapse entirely, so that the graph is a tree.
The $\omega$-area of the central hexagon is essentially the valuation of the cross-ratio of the four lines in $K^2$, cf.\ \cite[Section 6.3]{HKS}.

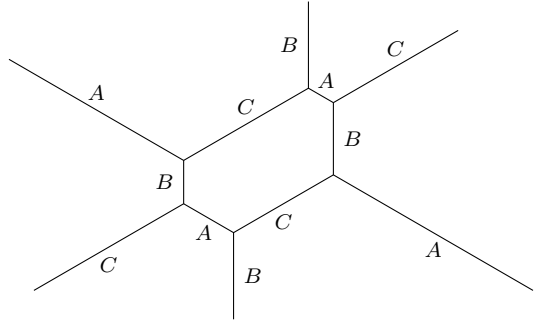
\begin{figure}[ht]
    \centering
    \[
    \begin{tikzpicture}[xscale=0.33,yscale=0.191]
        \draw (0,0) -- (4,4) -- (4,9) -- (3,10) -- (-2,5) -- (-2,2) -- cycle;
        \node[below] at (2,2) {$\scriptstyle{C}$};
        \node[right] at (4,6.5) {$\scriptstyle{B}$};
        \node[above] at (3.7,9.3) {$\scriptstyle{A}$};
        \node[above] at (0.5,7.5) {$\scriptstyle{C}$};
        \node[left] at (-2,3.5) {$\scriptstyle{B}$};
        \node[below] at (-1.2,1.2) {$\scriptstyle{A}$};
        \draw (0,0) -- (0,-6);
        \node[right] at (0,-3) {$\scriptstyle{B}$};
        \draw (4,4) -- (12,-4);
        \node[below] at (8,0) {$\scriptstyle{A}$};
        \draw (4,9) -- (9,14);
        \node[above] at (6.5,11.5) {$\scriptstyle{C}$};
        \draw (3,10) -- (3,16);
        \node[left] at (3,13) {$\scriptstyle{B}$};
        \draw (-2,5) -- (-9,12);
        \node[above] at (-5.5,8.5) {$\scriptstyle{A}$};
        \draw (-2,2) -- (-8,-4);
        \node[below] at (-5,-1) {$\scriptstyle{C}$};        
    \end{tikzpicture}
    \]
    \caption{Support of a harmonic path with central hexagon. The labels $A,B,C$ refer to the three indecomposable representations of the $A_2$ quiver over $\mathbf{k}$.}
    \label{fig:hexagon}
\end{figure}
\end{example}

A possible approach to proving existence of harmonic curves is to use calculus of variations and formulate the harmonicity condition as a critical point condition for an action functional
\[
S(c|_{[a,b]})=\int_a^bL(t,c(t),\dot{c}(t))dt
\]
for a suitable Lagrangian $L$ which is convex in $\dot{c}$.
There is an a priori bound on the number of vertices of the harmonic path coming from the Legendrian dual of the graph. 
The dual is a tiled polygon, where the sides of all tiles are central charges of semistable objects and thus have a lower bound on their lengths by the support property.
Moreover, the perimeter of the polygon is fixed by the boundary condition coming from the two ascending-phase filtrations.


\subsubsection*{Modular lattice variant}

Suppose $K$ is trivially valued, then a non-expansive norm on $(E,\rho)$ is simply an $\bR$-filtration on $(E,\rho)$, i.e.\ a collection of $\bR$-filtrations on each $E_i$, $i\in Q_0$, compatible with the maps $\rho(a)$, $a\in Q_1$.
This notion depends only on the lattice of subobjects of $E\in\mathbf{Rep}(Q,K)$, which is a bounded modular lattice.
Piecewise smoothness for paths in the space of filtrations means the following: Away from a finite set, the underlying finite sequence of subobjects does not change, only the points in $\bR$ where the filtration jumps.

Given any bounded modular lattice $\Lambda$, define a \textit{central charge} on it to be a map $Z\colon \Lambda\to \mathbb C$ satisfying 
\begin{enumerate}[(1)]
    \item $Z(a\vee b)+Z(a\wedge b)=Z(a)+Z(b)$, $a,b\in\Lambda$,
    \item $\operatorname{Im}(Z(a))>0$ for $a\neq 0$,
    \item $Z(0)=0$.
\end{enumerate}
This is required to satisfy the analogue of the HN and support properties.
See also~\cite[Section 4.2]{HKKPitlogs1}, where this is called a \textit{polarization}.
In this setting one can define, in the same way as before, a notion of harmonic path and the analogue of  Conjecture~\ref{conj:bubbling}.

\subsection{Honeycomb diagrams}\label{subsec:honeycombs}

We discuss a special case which is related to \textit{honeycomb diagrams}.
Assume, for simplicity, that $\nu(K^\times)=\bR$, e.g.\ $K$ is a Novikov field with real exponents.
Take $Q$ to be the $A_2$ quiver, $Z$ as in Example~\ref{ex:harmpathA2}, and $E\coloneqq(K^n\xrightarrow{\mathrm{id}}K^n)$.
Suppose $\cF^-$ is the trivial (one-step) filtration and $\cF^+$ is the (unique) filtration with subquotients $A^n$ and $C^n$.
Choices of harmonic norms on the subquotients of the filtrations $\cF^{\pm}$ amount to a triple of diagonalizable norms, $\|\cdot\|_k$, $k=1,2,3$, on $K^n$.
Such a triple gives three sequences of real numbers, $\alpha_1\geq\ldots\geq\alpha_n$, $\beta_1\geq\ldots\geq\beta_n$, $\gamma_1\geq\ldots\geq\gamma_n$, which are the singular values of the identity maps $(K^n,\|\cdot\|_i)\to (K^n,\|\cdot\|_j)$ for $(i,j)=(1,2),(2,3),(1,3)$, respectively.
By multiplicativity of the determinant,
\begin{equation}\label{eq:ABCequality}
    \alpha_1+\ldots+\alpha_n+\beta_1+\ldots+\beta_n=\gamma_1+\ldots+\gamma_n.
\end{equation}
Which triples $(\alpha,\beta,\gamma)$ satisfying~\eqref{eq:ABCequality} arise from a triple of norms?
For $n=1$ there are no additional constraints.
For $n=2$ it is an easy exercise that the inequalities
\[
\max(\alpha_1+\beta_2,\alpha_2+\beta_1)\leq\gamma_1\leq \alpha_1+\beta_1
\]
are necessary and sufficient.
For general $n$, the list of inequalities, discovered by A.~Horn, is recursively defined, see for example~\cite[Section 4.5]{kedlaya_book} which cites~\cite{fulton_horn}.

For harmonic paths $c$ giving rise to such $\cF^\pm$, the  support $\cW(c)$ is a \textit{honeycomb diagram} in the sense of Knutson--Tao~\cite{knutson_tao99,knutson_tao_notices}.
Fixing $\alpha,\beta,\gamma$ we can consider three sets:
\begin{enumerate}[(1)]
    \item The set $\cP$ of pairs of ascending-phase filtrations, equivalently (in this setup) triples of non-Archimedean norms on $K^n$ with prescribed singular values.
    \item The set $\cS$ of harmonic paths with prescribed asymptotics as $t\to\pm\infty$.
    \item The set $\cH$ of honeycomb diagrams with prescribed spacings of the infinite edges. (These depend on the triple $\alpha,\beta,\gamma$ and the area form $\omega$.)
\end{enumerate}
We have the following diagram of maps between these sets:
\[
\begin{tikzcd}
    & \cS \arrow[dl]\arrow[dr,"\cW"] \\
    \cP & & \cH
\end{tikzcd}
\]
The Knutson--Tao honeycomb theorem tells us that $\cP$ is non-empty iff $\cH$ is non-empty, equivalently iff the Horn inequalities for $(\alpha,\beta,\gamma)$ hold.
Conjecture~\ref{conj:bubbling} says that the left arrow is a bijection. Assuming this, we get a map $\cP\to\cH$ from triples of norms to honeycombs (depending on $\omega$).
From the perspective of \cite{HKS}, the fibers of the map $\cS\to\cH$ are sets of solutions of certain $A_\infty$ Maurer--Cartan equations.

\section{(Co)monads and stability}
\label{section:comonads}

In the previous sections we discussed examples of categorical K\"ahler structures arising from fairly concrete data: quivers, stability parameters, ramp functions, and area forms on the plane.
In this section we switch gears and consider examples related to a more abstract categorical setup: certain (co)monadic adjunctions of stable $\infty$-categories. 
These unify both curved $A_\infty$-categories and formal neighborhoods of effective divisors.

This section is organized as follows.
In Section~\ref{subsec:monads} we study monadic adjunctions between stable $\infty$-categories, reformulating monadicity in terms of filtered objects.
Section~\ref{subsec:transport} contains the proof of our main theorem of this section on transporting stability conditions from a stable $\infty$-category with a monad to the quotient of the Eilenberg--Moore category by the Kleisli category, under suitable conditions.
Sections~\ref{subsec:nbhdiv} and \ref{subsec:defAinf} discuss examples arising from divisors in smooth varieties and formal deformations of $A_\infty$-categories, respectively.

\subsection{Monads on stable \texorpdfstring{$\infty$}{infinity}-categories}\label{subsec:monads}

In classical category theory, a monad is given by an endofunctor $T\colon\cC\to\cC$ of a category $\cC$, together with a pair of natural transformations $\mu\colon T^2\to T$, $\eta\colon 1_{\cC}\to T$ satisfying certain compatibility conditions which say that $(T,\mu,\eta)$ is a monoid object in the category $\mathrm{Fun}(\cC,\cC)$ of endofunctors.
Any adjunction $F:\cC  \leftrightarrows \cD:G$ gives rise to a monad with $T\coloneqq GF$ on $\cC$, and conversely from any monad $(T,\mu,\eta)$ one constructs an adjunction $\cC \leftrightarrows \cC^T$ where $\cC^T$ is the Eilenberg--Moore category of $T$-algebras.
However, there are typically many adjunctions which give rise to the same monad in this way. 
An adjunction is \textit{monadic} if it is, up to equivalence, of the form $\cC \leftrightarrows \cC^T$ for some monad $T$.
Monadic adjunctions are characterized by the Barr--Beck monadicity theorem.

An $\infty$-categorical generalization of the theory of monads was proposed and developed by Lurie~\cite[Section 4.7]{LurieHA} and Riehl--Verity~\cite{riehl_verity_monads}.
The main challenge, compared with the classical case, is dealing with a rather large amount of higher coherence data when defining monads and their $\infty$-categories of algebras.
Nevertheless, it turns out that the monadicity theorem has a satisfactory generalization to this setting in the following form.
(The definitions of \textit{simplicial object} and \textit{split simplicial objects} will be recalled below.)

\begin{theorem}[$\infty$-categorical monadicity theorem, {\cite[Theorem 4.7.3.5]{LurieHA}}]\label{thm:BBL}
An adjunction $F:\cC \leftrightarrows \cD:G$ between $\infty$-categories  is monadic if and only if the right adjoint $G$ satisfies the following conditions:
\begin{enumerate}[(1)]
    \item $G$ is conservative, i.e.\ $f$ is an isomorphism iff $G(f)$ is.
    \item If $X_\bullet$ is a simplicial object in $\cD$ such that $G(X_\bullet)$ is split, then the colimit of $X_\bullet$ exists in $\cD$ and is preserved by $G$.
\end{enumerate}
\end{theorem}

For our purposes, we will take the above pair of conditions as the definition of a monadic adjunction.
Furthermore, we focus on adjunctions of stable $\infty$-categories, for which one can reformulate the second condition via the $\infty$-categorical Dold--Kan correspondence.

\begin{remark}
The definitions of monads and their categories of algebras have not been completely translated to the language of dg/$A_\infty$ categories.
Anno--Arkhipov--Logvinenko~\cite{AAL_monoidal} have given a definition of ``$A_\infty$-monad on a dg category'' and its Eilenberg--Moore category. In their upcoming work, they intend to prove a Barr--Beck-type theorem in this setting.
\end{remark}

We recall the definition of \textit{split simplicial object}.
Let $\Delta_{-\infty}$ be the category with objects the totally ordered sets 
\[
[n]\coloneqq\{-\infty<0<\cdots<n\},\qquad n\geq -1, 
\]
and morphisms the non-decreasing maps $f\colon[m]\to[n]$ such that $f(-\infty)=-\infty$.
The \textit{augmented simplex category} is the subcategory $\Delta_+\subset \Delta_{-\infty}$ with the same objects and morphisms satisfying $f^{-1}(\{-\infty\})=\{-\infty\}$. Thus, in $\Delta_+$ we might as well exclude the element $-\infty$ from each $[n]$. 
More precisely, $\Delta_+$ is isomorphic to the category whose objects are the totally ordered sets $\{0<\cdots<n\}$, $n\geq -1$, where $[-1]=\emptyset$, and morphisms are non-decreasing maps.
Finally, the \textit{simplex category} is the full subcategory $\Delta\subset\Delta_+$ which excludes the object $[-1]=\{-\infty\}$.
A \textit{simplicial object} in an $\infty$-category $\cC$ is a functor $X_\bullet\colon N(\Delta^\mathrm{op})\to \cC$.
An \textit{augmentation} (resp. \textit{splitting}) of a simplicial object is an extension of $X_\bullet$ to $N(\Delta^\mathrm{op}_+)$ (resp. $N(\Delta^\mathrm{op}_{-\infty})$). A simplicial object or augmented simplicial object is \textit{split} if it admits a splitting.

From now on we assume that $\cC$ is stable. The $\infty$-categorical version of the Dold--Kan correspondence~\cite[Theorem 1.2.4.1]{LurieHA} is a canonical equivalence of functor $\infty$-categories
\begin{equation}\label{eq:LurieDK}
\mathrm{Fun}(N(\Delta^\mathrm{op}),\cC)\to\mathrm{Fun}(N(\bZ_{\geq 0}),\cC)
\end{equation}
which sends a simplicial object $X_{\bullet}$ to the filtered object
\[
Y_0\to Y_1\to Y_2\to\ldots,\qquad Y_n\coloneqq\operatorname{colim} \mathrm{sk}_nX_\bullet
\]
where $\mathrm{sk}_nX_\bullet\in\mathrm{Fun}(N(\Delta^{\mathrm{op}}),\cC)$ is the $n$-skeleton of $X_\bullet$.
We remark at this point that the data of a functor $Y_\bullet\colon N(\bZ_{\geq 0})\to\cC$ includes objects $Y_n$, $n\geq 0$ and morphisms $f_n\colon Y_{n-1}\to Y_{n}$, $n\geq 1$, but also morphisms $Y_m\to Y_n$, $m\leq n$, as well as coherence data (values of $Y$ on simplices of dimension $>0$). 
However, all this additional data is determined, up to a contractible space of choices, by the sequence of $f_n$'s, and we can thus suppress the distinction.

There is also an augmented variant of~\eqref{eq:LurieDK}, a canonical equivalence
\[
\mathrm{Fun}(N(\Delta_+^\mathrm{op}),\cC)\to\mathrm{Fun}(N(\bZ_{\geq -1}),\cC)
\]
which can be constructed from~\eqref{eq:LurieDK} as follows.
We have $\Delta_+\cong[-1]\star\Delta\leftarrow \Delta^1\times\Delta$ which yields an embedding
\[
\mathrm{Fun}(N(\Delta_+^\mathrm{op}),\cC)\longrightarrow\mathrm{Fun}(\Delta^1\times N(\Delta^\mathrm{op}),\cC).
\]
On the level of objects, this sends an augmented simplicial object $X_\bullet$ to the corresponding morphism of simplicial objects $X_{\geq 0}\to \underline{X_{-1}}$, where $X_{\geq 0}$ is the restriction of $X_\bullet$ along the inclusion $\Delta^\mathrm{op}\subset \Delta_+^\mathrm{op}$ and $\underline{X_{-1}}$ is the constant simplicial object with value $X_{-1}\in\cC$. 
By~\eqref{eq:LurieDK}, there is a canonical equivalence
\[
\mathrm{Fun}(\Delta^1\times N(\Delta^\mathrm{op}),\cC)\to\mathrm{Fun}(\Delta^1\times N(\bZ_{\geq 0}),\cC)
\]
which sends the subcategory of morphisms to a constant simplicial object to the subcategory of morphisms to a constant (up to equivalence) filtered object.
Finally, we send the morphism of filtered objects $Y_\bullet\to \underline{X_{-1}}$, where $\underline{X_{-1}}\colon N(\bZ_{\geq 0})\to\cC$ is now, by abuse of notation, the constant filtered object with value $X_{-1}$, to a functor $Z_\bullet\colon N(\bZ_{\geq -1})\to\cC$ with $Z_{-1}\coloneqq X_{-1}[-1]$ and $Z_n\coloneqq\mathrm{fib}(Y_n\to \underline{X_{-1}})$, $n\geq 0$.
Moreover, this construction can be made functorial and is an equivalence, as it is fiber rotation in the stable $\infty$-category $\mathrm{Fun}(N(\mathbb Z_{\geq 0}),\cC)$.

To set up some terminology, we refer to $\bZ_{\geq -1}$-filtered objects as \textbf{extended filtered objects} and $\bZ_{\geq 0}$-filtered objects together with a morphism to a constant filtered object as \textbf{augmented filtered objects}.
By the discussion above, the two notions are equivalent.

\begin{lemma}
    An augmented simplicial object $X_\bullet$ in a stable $\infty$-category $\cC$ is split iff the corresponding extended filtered object
    \[
    Y_{-1}\xrightarrow{f_0}Y_0\xrightarrow{f_1}Y_1\xrightarrow{f_2}Y_2\to\ldots
    \]
    satisfies $f_n\simeq 0$ for $n\geq 0$. We call such an extended filtered object \textbf{split}.
\end{lemma}

\begin{proof}
By~\cite[Corollary 4.7.2.9]{LurieHA}, an augmented simplicial object $X_\bullet$ is split iff the canonical map $\mathrm{Dec}_+X_\bullet\to X_\bullet$ admits a right homotopy inverse. Here, $\mathrm{Dec}_+$, the \textit{augmented d\'ecalage}~\cite[\href{https://kerodon.net/tag/04SN}{Construction 04SN}]{kerodon}, is the endofunctor on the $\infty$-category of augmented simplicial objects in $\cC$ given by pullback along the functor $\Delta_+\to\Delta_+$, $[n]\mapsto [0]\star[n]\simeq[n+1]$.
Under the Dold--Kan correspondence, the augmented simplicial object $\mathrm{Dec}_+X_\bullet$ corresponds to an extended filtered object $W_n\coloneqq\fib(Y_n\to Y_{n+1})$, $n\geq -1$, which is split, and the morphism $\mathrm{Dec}_+X_\bullet\to X_\bullet$ corresponds to the canonical morphism of filtered objects $p\colon W_\bullet\to Y_\bullet$. 
If $p$ has a right homotopy inverse, then we conclude that $Y_\bullet$ is also split.
Conversely, if $Y_\bullet$ is split, then $W_n\simeq Y_n\oplus Y_{n+1}[-1]$, so the required right homotopy inverse comes from the inclusion of the first summand.
\end{proof}

\begin{corollary}\label{cor:monadic}
An adjunction $F:\cC \leftrightarrows \cD:G$ between stable $\infty$-categories is monadic if and only if the right adjoint $G$ satisfies the following conditions:
\begin{enumerate}[(1)]
    \item $G$ is conservative. (Equivalently: $G(A)\simeq 0\implies A\simeq 0$)
    \item If $X_\bullet\in\mathrm{Fun}(N(\mathbb Z_{\geq 0}),\cD)$ is a filtered object in $\cD$ and $G(X_\bullet)$ admits an augmentation so that the corresponding extended filtered object is split, then the colimit of $X_\bullet$ exists in $\cD$ and is preserved by $G$.
\end{enumerate}
\end{corollary}

The following lemma is useful for checking the second condition in the above corollary.

\begin{lemma}\label{lem:splitfilteredchar}
Let $\cC$ be a stable $\infty$-category and $X_\bullet\in\mathrm{Fun}(N(\mathbb Z_{\geq 0}),\cC)$ a filtered object in $\cC$.
Then the following are equivalent:
\begin{enumerate}[(1)]
    \item $X_\bullet$ admits an augmentation $X_\bullet\to\underline{X_{-1}}$ so that the corresponding extended filtered object $Y_\bullet$ (with $Y_{-1}=X_{-1}[-1]$, $Y_n=\mathrm{fib}(X_n\to X_{-1})$) is split.
    \item $X_\bullet$ is the direct sum of a constant filtered object $\underline{X_{-1}}$ and a filtered object $Y_0\xrightarrow{0}Y_1\xrightarrow{0}Y_2\xrightarrow{0}\ldots$ for which all maps are zero maps. We say in this case that $X_\bullet$ is \textbf{essentially constant}.
\end{enumerate}
\end{lemma}


\begin{proof}
The implication (2)$\implies$(1) is clear, since we can take as augmentation the projection to the direct summand $\underline{X_{-1}}$.
It remains to show (1)$\implies$(2).

So suppose $X_\bullet=X_0\xrightarrow{g_1}X_1\xrightarrow{g_2}X_2\to\cdots$ admits an augmentation $X_\bullet\to\underline{X_{-1}}$ so that the corresponding extended filtered object
\[
Y_{-1}\xrightarrow{f_0\simeq 0}Y_{0}\xrightarrow{f_1\simeq 0}Y_{1}\xrightarrow{f_2\simeq 0}Y_{2}\longrightarrow\cdots
\]
is split.
By definition, we have exact triangles
\[
Y_n\xrightarrow{\iota_n}X_n\xrightarrow{\rho_n}X_{-1}\xrightarrow{\delta_n}Y_n[1]
\]
in $\mathrm{h}\cC$, and $[f_n]=0$ implies $[\delta_n]= 0$, i.e.\ all these triangles are split.
We construct the splittings $s_n\colon X_{-1}\to X_n$ inductively so that $[s_{n}]=[g_n][s_{n-1}]$ in the homotopy category $\mathrm{h}\cC$ for $n\geq 1$ and $s_0$ is an arbitrary choice of splitting. 
The morphism between split exact triangles
\[
\begin{tikzcd}
Y_{n-1} \arrow[r,swap,"\iota_{n-1}"]\arrow[d,swap,"f_n\simeq 0"] & X_{n-1} \arrow[r,swap,"\rho_{n-1}"]\arrow[l,bend right,swap,dashed,"t_{n-1}"]\arrow[d,"g_n"] & X_{-1}\arrow[l,bend right,swap,dashed,"s_{n-1}"]\arrow[d,equal] \\
Y_n \arrow[r,"\iota_n"] & X_n \arrow[r,"\rho_n"]\arrow[l,bend left,dashed,"t_n"] & X_{-1}\arrow[l,bend left,dashed,"s_n"]
\end{tikzcd}
\]
then shows that $[g_n]$ is just projection to the summand $X_{-1}$, i.e.\ the only non-zero block of $[g_n]$ with respect to the decompositions $X_n=X_{-1}\oplus Y_n$ is the identity of $X_{-1}$. 
Indeed, identifying morphisms with their images in $\mathrm{h}\cC$ in the following calculation:
\[
\begin{pmatrix}
    \rho_ng_n s_{n-1} & \rho_ng_n\iota_{n-1} \\ t_ng_ns_{n-1} & t_ng_n\iota_{n-1}
\end{pmatrix}
=
\begin{pmatrix}
    \rho_{n-1}s_{n-1} & \rho_n\iota_nf_n \\ t_ns_n & t_n\iota_nf_n
\end{pmatrix}
=
\begin{pmatrix}
    1_{X_{-1}} & 0 \\ 0 & 0
\end{pmatrix}.
\]
Hence, $X_\bullet$ is a direct sum as a filtered object. 
\end{proof}

\begin{remark}
The second condition in Corollary~\ref{cor:monadic} almost makes sense if $\cC$ and $\cD$ are only triangulated, i.e.\ without assuming the existence of an enhancement. The problem is that homotopy colimits of filtered objects are, to our knowledge, not defined in general triangulated categories, but only in those with countable direct sums~\cite{bokstedt_neeman}.
The existence of countable direct sums is however incompatible with the axioms of a stability condition, so we will not impose it.
\end{remark}

\subsubsection*{Canonical resolutions}

For the rest of this subsection, fix a monadic adjunction between stable $\infty$-categories $F:\cC \leftrightarrows \cD:G$ with counit $\varepsilon\colon FG\to 1_{\cD}$ and unit $\eta\colon 1_{\cC}\to GF$.
Complete the counit to a cofiber sequence
\[
FG\xrightarrow{\varepsilon} 1_{\cD}\xrightarrow{\tau}S
\]
from which we build an extended filtered object $R_\bullet$ in $\mathrm{Fun}(\cD,\cD)$ given by
\begin{equation*}
\ldots\longrightarrow R_n\coloneqq S^{n+1}\circ[-1]\xrightarrow{\tau\,S^{n+1}[-1]}S^{n+2}\circ[-1]\eqqcolon R_{n+1}\longrightarrow\ldots,
\end{equation*}
thus beginning with $R_{-1}=[-1]\xrightarrow{\tau[-1]}R_0=S\circ [-1]\to\ldots$.

\begin{lemma}\label{lem:monadR}
The extended filtered object $GR_\bullet$ is split, i.e.\ the natural transformations $GR_n\to GR_{n+1}$ vanish, and moreover $\operatorname{colim}R_\bullet=0$.
\end{lemma}

\begin{proof}
By definition of $\tau$ there is a fiber sequence
\[
GFG\xrightarrow{G\varepsilon} G\xrightarrow{G\tau}GS,
\]
but $G\varepsilon$ has right homotopy inverse $\eta G$ by one of the triangle identities of the adjunction, so $G\tau= 0$ and the first claim follows. In particular, $\operatorname{colim} GR_\bullet=0$.
By the second condition in Corollary~\ref{cor:monadic} we get $G(\operatorname{colim} R_\bullet X)\simeq \operatorname{colim} GR_\bullet X=0$ for any $X\in \cD$. (View an extended filtered object as a filtered object by shifting the whole sequence to the right and take the augmentation by the zero object.)
Since $G$ is conservative by the first condition in Corollary~\ref{cor:monadic}, it follows that $\operatorname{colim} R_\bullet X=0$. Since $X$ was arbitrary, $\operatorname{colim}R_\bullet=0$.
\end{proof}

Let the augmented filtered object corresponding to $R_\bullet$ be $J_\bullet\to 1_{\cD}=R_{-1}[1]$. This is a coherent diagram
\[
\begin{tikzcd}
    FG=J_0\arrow[r]\arrow[d] & J_1 \arrow[r]\arrow[dl] & J_2 \arrow[r]\arrow[dll] & \ldots \\
    1_{\cD}
\end{tikzcd}
\]
where 
\[
J_n\coloneqq\mathrm{cofib}([-1]\to R_n)\simeq \mathrm{fib}(1_{\cD}\xrightarrow{\tau^{n+1}}S^{n+1}),\qquad \tau^n\coloneqq (\tau S^{n-1})\cdots(\tau S)\tau.
\]
By definition of $J_\bullet$ and the octahedral axiom, there are fiber sequences $J_n\to J_{n+1}\to FGS^{n+1}$, which shows by induction that the image of each $J_n$ is contained in the full triangulated subcategory of $\mathrm{h}\cD$ generated by the image of $F$. 
The corresponding full stable subcategory of $\cD$ is the \textit{stable Kleisli $\infty$-category} in the terminology of~\cite{christ_spherical}. 

\begin{proposition}\label{prop:monadJ}
The filtered object $GJ_\bullet$ is essentially constant and $\operatorname{colim}J_\bullet\simeq 1_{\cD}$.
\end{proposition}

\begin{proof}
The first claim follows immediately from Lemma~\ref{lem:splitfilteredchar}.
As a consequence of this and monadicity of $G$, there are natural isomorphisms $G(\operatorname{colim}J_\bullet)\simeq \operatorname{colim}GJ_\bullet\simeq G$, and thus $\operatorname{colim}J_\bullet\simeq 1_{\cD}$.
\end{proof}

The following lemmas will be used in the proofs in the next subsection.
Since they require no additional assumptions on the monadic adjunction, we will discuss them here.

\begin{lemma}\label{lem:GS_G}
Let $T\coloneqq \mathrm{cofib}(1_{\cC}\xrightarrow{\eta}GF)\circ [1]$, then there is a natural isomorphism of functors $GS\simeq TG$. (This is true for any adjunction of stable $\infty$-categories.)
\end{lemma}

\begin{proof}
This follows from the octahedral axiom applied to the triangle identity $1_G\simeq(G\varepsilon)\circ(\eta G)$:
\[
\begin{tikzcd}
 & 0 \arrow[dotted,ddl]\arrow[dr] & \\
TG[-1] \arrow[dotted,d]\arrow[ur] & & GS \arrow[dotted,ll,"\simeq"]\arrow[dotted,ddl] \\
G \arrow[dr,swap,"\eta G"]\arrow[rr,"1_G"] & & G \arrow[uul]\arrow[u,swap]\\
 & GFG \arrow[ur,swap,"G\varepsilon"]\arrow[uul] & 
\end{tikzcd}
\]
where the dotted arrows denote the degree-one connecting morphisms in the corresponding exact triangles.
\end{proof}

\begin{lemma}\label{lem:homvanish} 
Suppose $X,Y\in\cD$ are such that $
\operatorname{Map}_{\cC}(T^nGX,GY)$ is contractible for all $n\geq 0$, where $T$ is as in the previous lemma.
Then $\operatorname{Map}_{\cD}(X,Y)$ is contractible. 
In particular, $\operatorname{Hom}_{\mathrm h\cD}(X,Y)=0$. 
\end{lemma} 

\begin{proof} 
Recall from Proposition~\ref{prop:monadJ} that  $X\simeq\operatorname{colim}_{n\geq0}J_nX$, and that there are fiber sequences 
$J_nX\rightarrow J_{n+1}X\rightarrow FGS^{n+1}X$. 
We first show by induction that 
\[ 
\operatorname{Map}_{\cD}(J_nX,Y)\simeq * 
\] 
for every $n\geq 0$. 

For $n=0$, we have $J_0X\simeq FGX$, and hence, by adjunction, 
\[ 
\operatorname{Map}_{\cD}(J_0X,Y) \simeq \operatorname{Map}_{\cD}(FGX,Y) \simeq \operatorname{Map}_{\cC}(GX,GY) \simeq *. 
\] 
Now suppose that $\operatorname{Map}_{\cD}(J_nX,Y)\simeq *$. 
Applying $\operatorname{Map}_{\cD}(-,Y)$ to the fiber sequence 
$J_nX\rightarrow J_{n+1}X\rightarrow FGS^{n+1}X$ 
gives a fiber sequence of mapping spaces 
\[ 
\operatorname{Map}_{\cD}(FGS^{n+1}X,Y) \longrightarrow \operatorname{Map}_{\cD}(J_{n+1}X,Y) \longrightarrow \operatorname{Map}_{\cD}(J_nX,Y). 
\] 
By adjunction and Lemma~\ref{lem:GS_G}, 
\[ 
\begin{aligned} \operatorname{Map}_{\cD}(FGS^{n+1}X,Y) &\simeq \operatorname{Map}_{\cC}(GS^{n+1}X,GY)\\ &\simeq \operatorname{Map}_{\cC}(T^{n+1}GX,GY) \simeq *. 
\end{aligned} 
\] 
Thus the other two terms in the fiber sequence are contractible, and hence $\operatorname{Map}_{\cD}(J_{n+1}X,Y)$ is as well. 
This completes the induction. 

Finally, since mapping out of a colimit is a limit, 
\[ 
\operatorname{Map}_{\cD}(X,Y) \simeq \operatorname{Map}_{\cD} \left(\operatorname{colim}_{n\geq0}J_nX,Y\right) \simeq \lim_{n\geq0}\operatorname{Map}_{\cD}(J_nX,Y) \simeq *. 
\] 
\end{proof}

\begin{lemma}\label{lem:morlift}
Let $X\in\cD$, $\overline{Y}\in\cC$, and $\bar{f}\colon GX\to \overline{Y}$ be a morphism.
Suppose that there are:
\begin{enumerate}[(1)]
    \item objects $Y_n\in\cD$ and morphisms $f_n\colon GY_n\to\overline{Y}$, $n\geq 0$,
    \item bicartesian squares
        \[
        \begin{tikzcd}
            FGY_n\arrow[r,"F(f_n)"]\arrow[d,swap,"\varepsilon_{Y_n}"] & F\overline{Y} \arrow[d] \\
            Y_n\arrow[r] & Y_{n+1}
        \end{tikzcd}
        \]
        for $n\geq -1$ (set $Y_{-1}\coloneqq X$, $f_{-1}\coloneqq \bar{f}$),
\end{enumerate}
such that, for $n\geq 0$, $f_n$ is a left homotopy inverse to the morphism $\overline{Y}\xrightarrow{\eta_{\overline{Y}}}GF\overline Y\to GY_n$.
Then $\bar{f}$ lifts along $G$ in the sense that there is a $Y\in \cD$, $f\colon X\to Y$, and an isomorphism $GY\simeq \overline{Y}$ under which $G(f)$ corresponds to a morphism homotopic to $\bar{f}$.
\end{lemma}

\begin{proof}
The lower arrows in the given bicartesian squares together give a filtered object $Y_\bullet$.
We claim that $GY_\bullet$ is essentially constant.
To see this, consider the coherent diagram
\[
\begin{tikzcd}
    GY_n \arrow[r,"\eta_{GY_n}"]\arrow[d,swap,"f_n"] & GFGY_n\arrow[r,"G\varepsilon_{Y_n}"]\arrow[d,"GF(f_n)"] & GY_n\arrow[d] \\
    \overline{Y}\arrow[r,"\eta_{\overline{Y}}"] & GF\overline{Y}\arrow[r] &  GY_{n+1}
\end{tikzcd}
\]
which shows that the map $GY_n\to GY_{n+1}$ factors through the common summand $\overline{Y}$.
Here, the left square comes from naturality of $\eta$ while the right square is the image under $G$ of one of the given bicartesian squares.
By monadicity, the colimit $Y\coloneqq \operatorname{colim} Y_\bullet$ exists, and $GY\simeq \overline{Y}$.
By functoriality of the colimit, the morphisms $\underline{X}=\underline{Y_{-1}}\to Y_\bullet$ from the constant filtered object induce a morphism $f\colon X\to Y$ which is the required lift of $\bar{f}=f_{-1}$.
\end{proof}

\begin{remark}
One can show that for fixed $X\in\cD$ the following two types of data are equivalent:
\begin{enumerate}[(1)]
    \item A morphism $f\colon X\to Y$ in $\cD$.
    \item An object $\overline{Y}\in\cC$ and a sequence of objects $Y_0,Y_1,Y_2,\ldots\in\cD$ together with:
    \begin{itemize}
        \item morphisms $f_n\colon GY_n\to\overline{Y}$, where $n\geq -1$ and $Y_{-1}=X$, and
        \item bicartesian squares
        \[
        \begin{tikzcd}
            FGY_n\arrow[r,"F(f_n)"]\arrow[d,swap,"\varepsilon_{Y_n}"] & F\overline{Y} \arrow[d] \\
            Y_n\arrow[r] & Y_{n+1}
        \end{tikzcd}
        \]
        for $n\geq -1$, or equivalently, isomorphisms 
        \[ 
        g_n\colon Y_{n+1} \xrightarrow{\simeq} \mathrm{cofib}\left(\mathrm{fib}(F(f_n))\to Y_n\right),
        \]
    \end{itemize}
    so that, for $n\geq 0$, $f_n$ is a left homotopy inverse to the morphism $\overline{Y}\xrightarrow{\eta_{\overline{Y}}}GF\overline Y\to GY_n$.
\end{enumerate}
The previous lemma essentially recovers $f$ from the other data.
To go from (1) to (2) take $\overline{Y}\coloneqq GY$, $Z\coloneqq\mathrm{cofib}(f)$, $Y_n\coloneqq\mathrm{fib}\left(J_nZ\to Z\to X[1]\right)$, $f_{-1}\coloneqq G(f)$, $f_{n}\coloneqq G(Y_n\to Y)$.
Since we will only need Lemma~\ref{lem:morlift} in what follows, we omit the details.
\end{remark}

\subsection{Transport of stability conditions}\label{subsec:transport}

In this subsection we complete the proof of the following theorem using the results of the previous subsection and those of Section~\ref{subsec:mass}.

\begin{theorem}\label{thm:monadtransport}
Let $F:\cC\leftrightarrows\cD:G$ be a monadic adjunction of stable $\infty$-categories such that $T\coloneqq \mathrm{cofib}(1_{\cC}\xrightarrow{\eta}GF)\circ [1]$ is an autoequivalence of $\cC$, and let $\sigma\in\mathrm{Stab}(\cC,\Gamma,\mathrm{cl})$ be a stability condition preserved by $T$ and such that $\mathrm{cl}\circ K_0(T)=\mathrm{cl}$.
If we set $\cE\coloneqq \cD/F(\cC)$, then:
\begin{enumerate}[(1)]
    \item For any $X\in\cD$, the class $\mathrm{cl}([GX])\in \Gamma$ depends only on the isomorphism class of the image of $X$ in $\cE$. Moreover, this induces a well-defined map $\mathrm{cl}_{\cE}\colon K_0(\cE)\to \Gamma$.
    \item There is a stability condition $\sigma'$ on $\cE$ with respect to the map $\mathrm{cl}_{\cE}$ characterized by the following: The central charge is the same as for $\sigma$ and an isomorphism class of objects in $\cE$ is $\sigma'$-semistable of phase $\phi$ if and only if it has a representative $X\in\cD$ such that $GX$ is $\sigma$-semistable of phase $\phi$.
\end{enumerate}
\end{theorem}
In the above theorem, $\cD/F(\cC)$ is shorthand for the cofiber (in the $\infty$-category of small stable $\infty$-categories and exact functors) of the inclusion functor of the full stable subcategory of $\cD$ generated by the image of $F$. On the level of homotopy categories, this coincides with the Verdier quotient. In fact, the enhancement of $\mathrm{h}\cE$ will play no role in our arguments.

To say that $T$ preserves a stability condition $\sigma$, one really needs to specify a compatible action of $T$ on $\Gamma$ which preserves $Z$.
Since we already require $\mathrm{cl}\circ K_0(T)=\mathrm{cl}$, the additional condition is just that $T(\cC_\phi)=\cC_\phi$, i.e.\ $T$ preserves the property of being semistable of phase $\phi$ for all $\phi\in\bR$.

\begin{remark}
Suppose that $T$ fixes $\sigma$, but does not necessarily act trivially on $\Gamma$.
Then we can replace $\Gamma$ by $\Gamma'\coloneqq \Gamma/\operatorname{Ker}(Z)$.
To show the support property for $\Gamma'$, choose a norm on $\Gamma'\otimes_{\bZ}\bR$ so that the quotient map becomes non-expansive.
\end{remark}

\begin{remark}
If $T$ moreover induces the identity on $K_0(\cC)$, then there is in fact a well-defined map $K_0(\cE)\to K_0(\cC)$, as follows from inspection of the proof below.
\end{remark}

For the rest of this subsection we fix an adjunction  $F:\cC\leftrightarrows\cD:G$ and a stability condition $\sigma$ on $\cC$ as in the above theorem.
Denote by $L\colon\cD\to\cE$ the localization functor.
As in the previous subsection, let $S\coloneqq\mathrm{cofib}(FG\xrightarrow{\varepsilon} 1_{\cD})$ and $\tau\colon 1_{\cD}\to S$ the natural transformation.

\begin{remark}
We do not require that $S$ is an autoequivalence, i.e.\ that the adjunction is \textit{spherical} in the sense of~\cite{AL_spherical,DKSS}, but this will hold in all examples we consider. 
\end{remark}

\begin{lemma}\label{lem:loc_t}
The quotient category $\cE=\cD/F(\cC)$ coincides with the localization along the components of the natural transformation $\tau\colon 1_{\cD}\to S$. Furthermore, any morphism $f\in\Hom_{\mathrm{h}\cE}(X,Y)$ can be represented by a cospan in $\mathrm h\cD$ of the form
\[
X \xrightarrow{g} S^nY\xleftarrow{\tau^n_Y}Y, \qquad \tau^n_Y\coloneqq  \tau_{S^{n-1}Y}\cdots\tau_{SY}\tau_Y
\]
for some $n\geq 0$, in the sense that $f=(\tau_Y^n)^{-1}g$ in $\mathrm{h}\cE$.
\end{lemma}

\begin{proof}
Each component $\tau_X\colon X\to SX$ maps to an isomorphism under $L$ by definition of $\tau$.
Conversely, if $X\in\cC$, then $FX$ is a summand of $FGFX$ by the triangle identity of the adjunction, and $FGFX\simeq\mathrm{fib}(\tau_{FX})$, so $LFGFX=0$ and thus $LFX=0$. 
This shows the first claim.

For the second claim it suffices to show that morphisms which can be represented in this way are closed under composition.
This can be seen from the following diagram:
\[
\begin{tikzcd}
    X \arrow[dr,swap,"f"] & & Y \arrow[dl,swap,"\tau^m_Y"]\arrow[dr,"g"] & & Z \arrow[dl,"\tau^n_Z"] \\
    & S^mY \arrow[dr,swap,"S^m(g)"] & & S^nZ \arrow[dl,"\tau^m_{S^nZ}"] & \\
    & & S^{m+n}Z & & 
\end{tikzcd}
\]
The square commutes by naturality: $\tau_Y^m$ and $\tau_{S^nZ}^m$ are the components of the natural transformation 
\[
\tau^m\coloneqq (\tau S^{m-1})\circ(\tau S^{m-2})\circ \cdots\circ(\tau S)\circ \tau\colon 1_{\cD}\to S^m.
\]
Finally, $\tau^m_{S^nZ}\tau^n_Z=\tau^{m+n}_Z$.
\end{proof}

\begin{lemma}\label{lem:decphase}
Let $X,Y\in\cD$ such that $GX,GY\in\cC$ are non-zero $\sigma$-semistable with phases $\phi(GX)>\phi(GY)$.
Then $\Hom_{\mathrm h\cD}(X,Y)=0$ and $\Hom_{\mathrm h\cE}(LX,LY)=0$.
\end{lemma}

\begin{proof}
Since $T$ preserves $\sigma$, $T^nGX$ is semistable of phase $\phi(GX)$ for every $n\geq0$. 
Hence, for every $k\geq0$, 
\[ 
\operatorname{Hom}_{\mathrm h\cC}(T^nGX[k],GY)=0, 
\] 
because $\phi(T^nGX[k])=\phi(GX)+k>\phi(GY)$. 
In other words, $\operatorname{Map}_{\cC}(T^nGX,GY)$ is contractible for every $n\geq 0$. 
The vanishing of $\Hom_{\mathrm h\cD}(X,Y)$ then follows from Lemma~\ref{lem:homvanish}.

By Lemma~\ref{lem:loc_t} we can represent a morphism $LX\to LY$ by a cospan of the form $X \xrightarrow{f} S^nY\xleftarrow{\tau^n_Y}Y$ where
$f$ is a morphism in $\cD$ and $L(\tau^n_Y)$ is an isomorphism. 
By Lemma~\ref{lem:GS_G} we have $GS^nY\simeq T^nGY$, which, thanks to $T(\sigma)=\sigma$, is semistable of phase $\phi(GY)$. 
Hence the vanishing of $f$ and thus the corresponding morphism in $\cE$ follow from the first part applied to the pair $X,S^nY$.
\end{proof}

\begin{lemma}\label{lem:HNlift}
Let $X\in \cD$. Then either the HN filtration of $GX$ lifts to a filtration of $X$, or there is a morphism $\beta\colon X\to X'$ such that $L(\beta)$ is an isomorphism, $m(GX')<m(GX)$, and $\mathrm{cl}([GX'])=\mathrm{cl}([GX])$ in $\Gamma$.
\end{lemma}

\begin{proof}
If $X=0$, there is nothing to prove, so we assume $X\neq 0$ and thus also have $GX\neq 0$.
Let $\bar{f}\colon GX\to\overline{Y}$ be the map to the semistable factor $\overline{Y}$ of $GX$ of minimal phase $\phi\coloneqq\phi_{-}(GX)$, i.e.\ $\overline{W}\coloneqq\mathrm{fib}(\bar{f})$ is an extension of semistable objects of phase $>\phi$.
Set $Y_{-1}\coloneqq X$, $f_{-1}\coloneqq \bar{f}$ and assume by induction that we have a fiber sequence
\[
T^{n+1}\overline{W}\longrightarrow GY_{n}\xrightarrow{f_n}\overline{Y}.
\]
By adjunction we get a morphism $h_n\colon FT^{n+1}\overline{W}\to Y_n$ and we define $Y_{n+1}\coloneqq\mathrm{cofib}(h_n)$.
Consider the resulting octahedral diagram:
\[
\begin{tikzcd}
 & \overline{Y} \arrow[dotted,ddl]\arrow[dr] & \\
T^{n+2}\overline{W}[-1] \arrow[d,dotted]\arrow[ur,"\delta_n"] & & GY_{n+1} \arrow[dotted,ll]\arrow[dotted,ddl] \\
T^{n+1}\overline{W} \arrow[dr,swap,"\eta_{T^{n+1}\overline{W}}"]\arrow[rr] & & GY_n \arrow[uul,"f_n"]\arrow[u]\\
 & GFT^{n+1}\overline{W} \arrow[ur,swap,"G(h_n)"]\arrow[uul] & 
\end{tikzcd}
\]
There are two possibilities:
\begin{enumerate}[(1)]
    \item $\delta_n\neq 0$: Since $\phi=\phi(\overline Y)=\phi_+(\overline Y)<\phi_-(\overline{W})=\phi_-(T^{n+2}\overline{W})$, where we used $T(\sigma)=\sigma$, Proposition~\ref{prop:mass} gives 
    \[
    m(GY_{n+1})<m(\overline{Y})+m(\overline{W})=m(GX).
    \]
    Thus, we can set $X'\coloneqq Y_{n+1}$ and $\beta$ to the composition
    \[
    X=Y_{-1}\to Y_0 \to\ldots\to Y_n\to Y_{n+1}=X'
    \]
    where each $L(Y_i\to Y_{i+1})$ is an isomorphism by construction and 
    \[
    \mathrm{cl}([GY_{n+1}])=\mathrm{cl}([\overline Y])+\underbrace{\mathrm{cl}([T^{n+2}\overline W])}_{\mathrm{cl}([\overline W])}=\mathrm{cl}([GX])
    \]
    where we have used $\mathrm{cl}\circ K_0(T)=\mathrm{cl}$.
    This shows the claim and the induction ends.
    \item $\delta_n=0$: In this case we can choose $f_{n+1}\colon GY_{n+1}\to \overline{Y}$ which splits the top triangle in the octahedron.
    This yields a fiber sequence
    \[
    T^{n+2}\overline{W}\longrightarrow GY_{n+1}\xrightarrow{f_{n+1}}\overline{Y}
    \]
    and we continue the induction.
\end{enumerate}
If the induction does not terminate at any step (all $\delta_n=0$), then Lemma~\ref{lem:morlift} gives a morphism $f\colon X\to Y$ in $\cD$ which lifts $\bar{f}$.
Let $W\coloneqq\mathrm{fib}(f)$, so $GW\simeq\overline{W}$. By induction on $m(GX)$ (which ranges over a discrete subset of $\mathbb R_{\geq 0}$ by Proposition~\ref{prop:massdiscrete}) we may assume that the statement of the lemma is true for $W$. 

In the first case, the HN filtration of $GW$ lifts, so we combine it with the morphism $W\to X$ to obtain a lift of the HN filtration of $GX$.

In the second case, there is a $\beta_W\colon W\to W'$ with $L(\beta_W)$ an isomorphism, $m(GW')<m(GW)$, and $\mathrm{cl}([GW'])=\mathrm{cl}([GW])$.
Let $\beta_X\colon X\to X'$ be the pushout of $W\to W'$ along $W\to X$, i.e.\ $X'=\mathrm{cofib}(\mathrm{fib}(W\to W')\to X)$.
Then $L(\beta_X)$ is an isomorphism, $[GX']-[GX]=[GW']-[GW]$, and
\[
m(GX)=m(GY)+m(GW)>m(GY)+m(GW')\geq m(GX')
\]
using Proposition~\ref{prop:mass}.
This completes the proof of the lemma.
\end{proof}

\begin{proof}[Proof of Theorem~\ref{thm:monadtransport}]
\textit{Slicing}: Define $\cE_\phi$ to be the full additive subcategory of $\cE$ of those objects which have a lift $X\in\cD$ with $GX\in \cC_\phi$. 
To show existence of a HN filtration for $X\in\cE$ we choose a lift $Y\in\cD$, $LY\simeq X$, with $m(GY)$ minimal among all lifts. This exists by discreteness of the set of masses, cf.\ Proposition~\ref{prop:massdiscrete}.
Lemma~\ref{lem:HNlift} implies that the HN filtration of $GY$ lifts to a filtration of $Y$. The image of this filtration in $\cE$ is a HN filtration of $X$. 
Lemma~\ref{lem:decphase} tells us that there are no non-zero morphisms in strictly decreasing phase.

\textit{Map $\mathrm{cl}_{\cE}\colon K_0(\cE)\to \Gamma$}: 
Since $\cE=\cD/F(\cC)$, the sequence 
\[ 
K_0(\cC)\xrightarrow{K_0(F)} K_0(\cD)\xrightarrow{K_0(L)} K_0(\cE)\longrightarrow 0 
\] 
is exact.
It follows from the assumption $\mathrm{cl}\circ K_0(T)=\mathrm{cl}$, that $\mathrm{cl}\circ K_0(G)\circ K_0(F)=0$.
Thus, $\mathrm{cl}\circ K_0(G)$ lifts uniquely to a homomorphism 
$\mathrm{cl}_{\cE}\colon K_0(\cE)\longrightarrow\Gamma$.

\textit{Stability condition}: Compatibility of the slicing and the central charge on $\cE$ holds by definition. The support property is also clear, since the set of $\gamma\in\Gamma$ supporting a $\sigma'$-semistable object is contained in the set of classes supporting a $\sigma$-semistable object. 
\end{proof}

We state for convenience the dual version of our main result, which follows by passing to opposite categories and conjugate stability conditions (Proposition/Definition~\ref{propdef:conjstab}).

\begin{theorem}\label{thm:comonads}
Let $F:\cC\leftrightarrows\cD:G$ be a comonadic adjunction of stable $\infty$-categories such that $T\coloneqq\mathrm{fib}(FG\xrightarrow{\varepsilon}1_{\cD})\circ [-1]$ is an autoequivalence of $\cD$, and let $\sigma\in\mathrm{Stab}(\cD,\Gamma,\mathrm{cl})$ be a stability condition preserved by $T$ and such that $\mathrm{cl}\circ K_0(T)=\mathrm{cl}$.
If we set $\cE\coloneqq \cC/G(\cD)$, then:
\begin{enumerate}[(1)]
    \item For any $X\in\cC$, the class $\mathrm{cl}([FX])\in \Gamma$ depends only on the isomorphism class of the image of $X$ in $\cE$. Moreover, this induces a well-defined map $\mathrm{cl}_{\cE}\colon K_0(\cE)\to \Gamma$.
    \item There is a stability condition $\sigma'$ on $\cE$ with respect to the map $\mathrm{cl}_{\cE}$ characterized by the following: The central charge is the same as for $\sigma$ and an isomorphism class of objects in $\cE$ is $\sigma'$-semistable of phase $\phi$ if and only if it admits a representative $X\in\cC$ such that $FX$ is $\sigma$-semistable of phase $\phi$.
\end{enumerate}
\end{theorem}

\subsubsection*{Relation to general framework}

Let us explain how comonads with stability conditions as in Theorem~\ref{thm:comonads} provide an example of our general framework as described in Section~\ref{section:axiomatic}. (There is an almost identical discussion in the dual monadic case.)
We think of the adjunction and the stability condition $\sigma$ on $\cD$ as providing a geometric structure for the enhanced triangulated category $\cE$ with its stability condition $\sigma'$.

\textit{Metrized objects.} The category of metrized objects and non-expansive maps is $\cC$ with its localization functor  $L\colon\cC\to\cE$. Thus, a metric on an object $X\in\cE$ is an object $\widetilde{X}\in\cC$ together with an isomorphism $f\colon L\widetilde X\to X$.
The functor $S\coloneqq\mathrm{fib}(1_{\cC}\xrightarrow{\eta} GF)$ together with the natural transformation $\tau\colon S\to 1_{\cC}$ induce a $\bZ_{\geq 0}$-action on the set $\mathrm{Met}(X)$ of metrics on $X$ by sending $(\widetilde X,f)$ to $(S\widetilde{X},f\circ L(\tau_{\widetilde{X}}))$.
We can think of this action as rescaling the metric down (up in the monadic case).

\textit{Mass measures.} The mass measure $\mu_{\widetilde{X}}$ is determined by the HN filtration of $F\widetilde X$ as follows: If $A_1,\ldots,A_n\in \cD$ are the semistable factors of $F\widetilde X$, then
\[
\mu_{\widetilde{X}}\coloneqq\sum_{k=1}^n|Z(A_k)|\delta_{\phi(A_k)}
\]
where $\delta_{\phi}$ is the Dirac measure concentrated at $\phi\in\bR$.
In particular, $\widetilde{X}$ is harmonic iff $F\widetilde{X}$ is $\sigma$-semistable.

\textit{Flow.}
There is a kind of discretized minimizing flow implicit in the construction in the proof of Lemma~\ref{lem:HNlift}.
For a given object $X\in \cE$ with initial metric $\widetilde{X}_0\in\mathrm{Met}(X)$ it yields a sequence 
\[
\widetilde{X}_0\leftarrow \widetilde{X}_1\leftarrow \widetilde{X}_2\leftarrow \cdots\leftarrow \widetilde{X}_n=\widetilde{X}_\infty\leftarrow \widetilde{X}_\infty\leftarrow \widetilde{X}_\infty \leftarrow \cdots
\]
in $\cC$ and $\mathrm{Met}(X)$ so that $m(\widetilde{X}_k)$ decreases strictly until the sequence stabilizes at a metric $\widetilde{X}_\infty\coloneqq \widetilde{X}_n$, $n\gg 0$, so that the HN filtration of $F\widetilde{X}_\infty$ lifts to a HN filtration of $X$. In particular, $\widetilde{X}_\infty$ is harmonic iff $X$ is $\sigma'$-semistable.
See also Example~\ref{ex:discreteshortening} in Section~\ref{subsec:defAinf} below.

\subsection{Comonads from limits}\label{subsec:comonadslimits}

The following proposition will be useful in establishing comonadicity in examples.

\begin{proposition}\label{prop:comonadtower}
Let $F:\cC\leftrightarrows\cD:G$ be an adjunction of stable $\infty$-categories with unit $\eta$ completed to a cofiber sequence $S\xrightarrow{\tau}1_{\cC}\xrightarrow{\eta}GF$ and let $\tau^n\coloneqq \tau\circ(\tau S)\circ\cdots \circ  (\tau S^{n-1})$ as before (but dual).
Suppose there is a coherent diagram of stable $\infty$-categories and functors
\[
\begin{tikzcd}
\cC \arrow[d,swap,"F"]\arrow[dr,swap]\arrow[drr,swap]\arrow[drrr]\arrow[drrrrr,"P_n"] \\
\cD & \cC_1 \arrow[l,"Q_0"] & \cC_2 \arrow[l,"Q_1"] & \cC_3 \arrow[l,"Q_2"] & \ldots\arrow[l] & \arrow[l] \cC_n & \ldots\arrow[l]
\end{tikzcd}
\]
which is a limit cone with apex $\cC$, i.e.\ $\cC\simeq\lim \cC_n$.
Suppose further that $P_n(\tau^{n+1})=0$ and all $\cC_n$ are idempotent complete.
Then $F$ is comonadic.
\end{proposition}

\begin{remark}
One can weaken the assumption that the $\cC_n$ are idempotent complete to the hypothesis that those idempotents whose image in $\cD$ is split are also split in $\cC_n$.
\end{remark}

The proof of the proposition will require two basic lemmas on limits in $\infty$-categories.

\begin{lemma}\label{lem:cofilt_fibseq}
Let $X_\bullet\to Y_\bullet\to Z_\bullet$ be a fiber sequence in the $\infty$-category $\mathrm{Fun}(N(\mathbb Z_{\leq 0}),\cC)$ of cofiltered objects in a stable $\infty$-category $\cC$.
If $\lim Y_\bullet$ and $\lim Z_\bullet$ exist in $\cC$, then so does $\lim X_\bullet\simeq \operatorname{fib}(\lim Y_\bullet\to\lim Z_\bullet)$.
\end{lemma}

\begin{proof}
This is an instance of Fubini's theorem for limits in an $\infty$-category~\cite[\href{https://kerodon.net/tag/06A2}{Tag 06A2}]{kerodon} applied to the pair of simplicial sets $N(\mathbb Z_{\leq 0})$ and $N(\bullet\to\bullet\leftarrow \bullet)\cong\Lambda^2_2$.
\end{proof}

\begin{lemma}\label{lem:retract_lim}
Let $D$ be a simplicial set, $\cC$ an idempotent complete $\infty$-category, $X,Y\in\operatorname{Fun}(D,\cC)$ so that $Y$ is a retract of $X$.
If $\lim X\in \cC$ exists, then $\lim Y\in\cC$ exists and is a retract of $\lim X$.
\end{lemma}

\begin{proof}
Let $\mathrm{Idem}$ be the category with a single object $\bullet$ and a single non-identity morphism which is an idempotent.
Since retracts give idempotents~\cite[\href{https://kerodon.net/tag/03Z7}{Tag 03Z7}]{kerodon}, we get a functor 
$ X'\colon N(\mathrm{Idem})\times D\to\cC $
extending $X$.
By assumption, the functor $X=X'(\bullet,\_)\colon D\to\cC$ has a limit $A\coloneqq \lim X'(\bullet,\_)\in\cC$, thus, again using Fubini's theorem for limits, the functor $X'$ has a limit $A'\coloneqq \lim_iX'(\_,i)\in\mathrm{Fun}(N(\mathrm{Idem}),\cC)$.
By idempotent completeness of $\cC$, $B\coloneqq \lim A'\simeq\lim X'\simeq\lim Y$ exists and is the required retract of $A$.
\end{proof}

\begin{proof}[Proof of Proposition~\ref{prop:comonadtower}]
Consider, as in Section~\ref{subsec:monads}, the endofunctors $$J_n\coloneqq \mathrm{cofib}(S^{n+1}\xrightarrow{\tau^{n+1}}1_{\cC})$$ which can equivalently be defined inductively by
\begin{equation}\label{eq:comonadJinductive}
J_0=GF,\qquad GFS^{n+1}\to J_{n+1}\to J_n\qquad \text{is a fiber sequence}.
\end{equation}
Furthermore, the assumption $P_n(\tau^{n+1})=0$ implies 
\begin{equation}\label{eq:Jsplitting}
P_nJ_n\simeq P_n\oplus (P_nS^{n+1}[1]).    
\end{equation}

Suppose $X\in\cC$ with $FX=0$. 
From \eqref{eq:comonadJinductive}, together with $FS^{n+1}\simeq T^{n+1}F$ (see Lemma~\ref{lem:GS_G}), it follows that $J_nX=0$ for all $n\geq 0$. But then~\eqref{eq:Jsplitting} implies that $P_nX=0$ for all $n$, thus $X=0$, since $\cC\simeq\lim \cC_n$. This shows that $F$ is conservative.

Let $X_\bullet$ be a cofiltered object in $\cC$ such that $FX_\bullet$ is essentially constant.
Fix $n\geq 1$. 
We will show by induction on $i\geq 0$: $\lim_jP_nJ_iX_j$ exists in $\cC_n$ and is preserved by the functor $F_n\coloneqq Q_0Q_1\cdots Q_{n-1}\colon \cC_n\to\cD$.
For the base case $i=0$ we have that $P_nJ_0X_\bullet=P_nGFX_\bullet$ is essentially constant by assumption, so the limit exists and is preserved by any additive functor.
Suppose the statement is true for $i$.
We have a fiber sequence of cofiltered objects in $\cC_n$:
\[
P_nGFS^{i+1}X_\bullet\to P_nJ_{i+1}X_\bullet\to P_nJ_{i}X_\bullet
\]
and again using $FS^{i+1}\simeq T^{i+1}F$ we see that the first term is essentially constant.
By the induction hypothesis and Lemma~\ref{lem:cofilt_fibseq}, it follows that $\lim_jP_nJ_{i+1}X_j$ exists in $\cC_n$ and is preserved by $F_n$.
In particular, if we set $i=n$ we deduce from~\eqref{eq:Jsplitting} and Lemma~\ref{lem:retract_lim} that $X^{(n)}\coloneqq \lim_jP_nX_j$ exists and is preserved by $F_n$.
From the universal property we get isomorphisms $Q_nX^{(n+1)}\xrightarrow{\simeq} X^{(n)}$ which determine an object $X\in\cC$ which is a limit of $X_\bullet$ and 
\begin{align*}
FX &\simeq F_nX^{(n)} && \text{from }F\simeq F_nP_n\text{ and }P_nX\simeq X^{(n)},\\
&\simeq F_n \lim_jP_nX_j && \text{by definition of }X^{(n)},\\
&\simeq \lim_j F_n P_nX_j && \text{limit preserved by }F_n, \\
&\simeq \lim_j  FX_j && \text{using }F\simeq F_nP_n.
\end{align*}
This verifies the second condition of (the dual of) Corollary~\ref{cor:monadic} for $F$, so $F$ is comonadic.
\end{proof}

\begin{remark}
The proof of Proposition~\ref{prop:comonadtower} shows that the functors $F_n\colon \cC_n\to\cD$ are ``almost comonadic'' in the sense that they satisfy the two conditions of (the dual of) Corollary~\ref{cor:monadic}, but may not have a right adjoint, as is typically the case in the examples we consider. 
\end{remark}

\subsection{Formal neighborhoods of divisors}\label{subsec:nbhdiv}

Our goal in this subsection is to give examples from algebraic geometry where Theorem~\ref{thm:comonads} on transport of stability conditions along comonadic adjunctions can be applied.
As Proposition~\ref{prop:comonadtower} from the previous subsection already hints at, such examples can be found in the setting of formal schemes.
For simplicity, we restrict here to those formal schemes which arise as formal completions. 

Before continuing with examples, we briefly address the relation between the language of dg categories, which are common in derived algebraic geometry, and the language of $\infty$-categories (weak Kan complexes), which we have been using so far in this section.
The dg-nerve construction~\cite[Section 1.3.1]{LurieHA} produces an $\infty$-category $N^{\mathrm{dg}}(\cC)$ from a small dg category $\cC$, and $\cC$ and $N^{\mathrm{dg}}(\cC)$ have canonically isomorphic homotopy categories. 
If $\cC$ is moreover pretriangulated, then $N^{\mathrm{dg}}(\cC)$ is stable~\cite{cohn_dg,faonte}.
This correspondence is compatible with the formation of (homotopy) limits and colimits.
To obtain adjunctions between $\infty$-categories in the sense of~\cite{LurieHTT} we appeal to~\cite[Proposition 5.2.2.8]{LurieHTT} which constructs such an adjunction from the data of a unit or counit alone.
In what follows, we will assume without further mention that $N^{\mathrm{dg}}$ has been applied when appealing to results stated earlier in this section in the language of stable $\infty$-categories.

Let $X$ be a Noetherian separated scheme over a fixed ground field $\mathbf k$, and let $j\colon Z\hookrightarrow X$ be a closed subscheme cut out by the  corresponding ideal sheaf $\cI_Z\subseteq \cO_X$.
The \textit{$n$-th infinitesimal neighborhood}, $j_n\colon Z_n\hookrightarrow X$ of $Z$ has ideal sheaf $\cI_Z^{n+1}$.
If $j_{m,n}\colon Z_m\to Z_n$, $m<n$, are the inclusions, then the functors $j_{m,n}^*\colon \mathrm{Perf}(Z_n)\to\mathrm{Perf}(Z_m)$ form a projective system of dg categories.
Its (homotopy) limit $\lim_n \mathrm{Perf}(Z_n)\eqqcolon \mathrm{Perf}(\widehat{X}_Z)$ is, by definition, following~\cite[Definition 2.1]{efimov17}, the category of perfect complexes on the formal completion $\widehat{X}_Z$ of $X$ along $Z$.
By~\cite[Theorem 2.3]{efimov17}, $\mathrm{Perf}(\widehat{X}_Z)$ is quasi-equivalent to a certain full subcategory of the derived category of quasi-coherent sheaves on $X$, $D(\mathrm{QCoh}(X))$.

\begin{proposition}
Let $X$ be a Noetherian separated scheme over $\mathbf k$ and $j\colon Z\hookrightarrow X$ a closed subscheme with ideal sheaf $\cI_Z\subseteq \cO_X$.
Suppose that $X$ is regular and $\cI_Z$ is invertible, thus $Z$ an effective divisor.
Then there is a comonadic adjunction
\[
\begin{tikzcd}
     \mathrm{Perf}(\widehat{X}_Z)\arrow[r,shift left,"j^*"] & D^b\mathrm{Coh}(Z) \arrow[l,shift left,"j_*"]
\end{tikzcd}
\]
and the unit and counit fit into exact triangles
\begin{gather*}
\cI_Z\otimes\_\longrightarrow 1_{\mathrm{Perf}(\widehat{X}_Z)}\xrightarrow{\eta} j_*j^*\xrightarrow{+1} \\
j^*\cI_Z[1]\otimes\_\longrightarrow j^*j_*\xrightarrow{\varepsilon} 1_{D^b\mathrm{Coh}(Z)}\xrightarrow{+1}
\end{gather*}
where $j^*\cI_Z=\cN_{Z/X}^\vee$ is the conormal sheaf of $Z\subseteq X$.
\end{proposition}

\begin{proof}
The right adjoint $j_*$ is the composition
\[
D^b\mathrm{Coh}(Z)\longrightarrow D^b\mathrm{Coh}(X)\simeq \mathrm{Perf}(X)\longrightarrow \mathrm{Perf}(\widehat{X}_Z)
\]
where the first functor is induced by the (exact) direct image functor $\mathrm{Coh}(Z)\to\mathrm{Coh}(X)$, the quasi-equivalence is a consequence of regularity of $X$~\cite[\href{https://stacks.math.columbia.edu/tag/08E0}{Tag 08E0}]{stacks-project}, and the third functor arises from the universal property of the limit applied to the functors $j_n^*\colon \mathrm{Perf}(X)\to\mathrm{Perf}(Z_n)$.
The claimed adjunction follows from the usual adjunction between direct and inverse image, together with the realization of $\mathrm{Perf}(\widehat{X}_Z)$ as a full subcategory of $D(\mathrm{QCoh}(X))$.

The exact triangle for the unit comes from the short exact sequence of sheaves $0\to\cI_Z\to \cO_X\to j_*\cO_Z\to 0$.
The exact triangle for the counit is also standard and follows by a local calculation, see for example~\cite[Section 11.1]{huybrechts_fm}.
Comonadicity follows from Proposition~\ref{prop:comonadtower} applied to the sequence of categories $\mathrm{Perf}(Z_n)$.
The required vanishing of the morphism $j^*_n\cI_Z^{n+1}\to j^*_n\cO_X\cong \cO_{Z_n}$ follows by definition.
\end{proof}

Suppose we are in the setting of the above proposition.
The autoequivalence $T$ of $D^b\mathrm{Coh}(Z)$ as in Theorem~\ref{thm:comonads} is thus $\cN_{Z/X}^\vee\otimes\_$.
By that theorem, we can transport any stability condition preserved by $T$ (and with a suitable choice of lattice $\Gamma$) from $D^b\mathrm{Coh}(Z)$ to 
\[
\cE=\mathrm{Perf}(\widehat{X}_Z)/j_*D^b\mathrm{Coh}(Z)\eqqcolon \mathrm{Perf}(\widehat{X}_Z\setminus Z).
\] 
This can be thought of as the category of perfect complexes on the punctured formal neighborhood ``$\widehat{X}_Z\setminus Z$'' of $Z$.

\begin{example}
Assume that $\mathbf k$ is algebraically closed and $\operatorname{char}\mathbf k=0$.
Suppose $X$ is smooth over $\mathrm{Spec}({\mathbf k[[t]])}$ and $Z$ is the fiber over $0$.
Then $\mathrm{Perf}(\widehat{X}_Z\setminus Z)\cong D^b\mathrm{Coh}(X_{\mathbf k((t))})$ by~\cite{morimura}.
\end{example}

\subsection{Deformations of \texorpdfstring{$A_\infty$}{A-infinity}-categories}\label{subsec:defAinf}

Formal 1-parameter deformations of $A_\infty$-categories are described by curved $A_\infty$-categories over the formal power series ring $\mathbf k[[t]]$.
In this subsection we review this notion and show that it gives rise to a comonad of stable $\infty$-categories where $T=\mathrm{fib}(FG\xrightarrow{\varepsilon}1_{\cD})\circ [-1]$ is the identity functor.
For a more in-depth discussion of curved $A_\infty$-algebras and categories we refer the reader to~\cite{fukaya_def,positselski_book,dedeken_lowen}.

Fix a coefficient field $\mathbf k$.
A $\mathbf k[[t]]$-module $M$ is \textit{topologically free} if it is isomorphic to the completion (with respect to the $t$-adic filtration) of $\overline{M}\otimes_{\mathbf k}\mathbf k[[t]]$, where $\overline{M}\coloneqq M\otimes_{\mathbf k[[t]]}\mathbf k$.
In particular, if $\overline{M}$ is a finite-dimensional vector space over $\mathbf k$, then $\overline{M}\otimes_{\mathbf{k}}\mathbf k[[t]]$ is topologically free.

\begin{definition}
A \textbf{curved $A_\infty$-category} over $\mathbf k[[t]]$, $\cA$, is given by 
\begin{enumerate}[(1)]
    \item a collection of objects $\mathrm{Ob}(\cA)$,
    \item for each pair $X,Y\in \mathrm{Ob}(\cA)$ a $\bZ$-graded $\mathbf k[[t]]$-module $\Hom(X,Y)$ such that each $\Hom^i(X,Y)$ is topologically free,
    \item for each $n\geq 0$ and $X_0,\ldots,X_n\in \mathrm{Ob}(\cA)$ a map of $\mathbf k[[t]]$-modules
    \[
    \mk m_n\colon \Hom(X_{n-1},X_n)\otimes\cdots\otimes \Hom(X_0,X_1)\longrightarrow\Hom(X_0,X_n)
    \]
    of degree $2-n$ such that the \textit{curvature} $\mk m_0\in\Hom^2(X,X)$ has vanishing constant term and the (curved) $A_\infty$ equations
    \[
    \sum_{i+j+k=n}(-1)^{\|a_1\|+\ldots+\|a_i\|}\mk m_{i+1+j}(a_n,\ldots,\mk m_j(a_{i+j},\ldots,a_{i+1}),\ldots,a_1)=0
    \]
    hold, where $\|a\|\coloneqq |a|-1$.
    Moreover, we require the existence of strict units $1_X\in\Hom^0(X,X)$, $X\in\mathrm{Ob}(\cA)$.
\end{enumerate}
\end{definition}

For the definitions of $A_\infty$-functors and natural transformations in this setting we refer to~\cite[Section 2.1]{H3dCY}. 
These resemble their counterparts for the more standard (uncurved) $A_\infty$-categories over $\mathbf k$, except for additional terms coming from $\mk m_0$.

A curved $A_\infty$-category $\cA$ has a kind of closure under finite direct sums, shifts, extensions, and infinitesimal formal deformations of objects, the (uncurved) $A_\infty$-category of \textbf{twisted complexes}, $\mathrm{Tw}(\cA)$, whose objects are pairs consisting of:
\begin{enumerate}[(1)]
    \item A formal expression $X=\bigoplus_{i=1}^nV_i\otimes X_i$, where $V_i$ are finite-dimensional $\bZ$-graded vector spaces over $\mathbf k$ and $X_i\in \mathrm{Ob}(\cA)$.
    \item A \textit{twisting cochain} $\delta\in\Hom^1(X,X)$ whose constant term is strictly upper triangular. Thus, $\delta$ is given by matrix coefficients $\delta_{ij}\in\Hom^1(V_j\otimes X_j,V_i\otimes X_i)$ which vanish $\mod t$ for $i\geq j$. These need to satisfy the \textit{$A_\infty$ Maurer--Cartan equation}
    \[
    \sum_{n=0}^\infty\mk m_n(\delta,\ldots,\delta)=0
    \]
    where the series converges $t$-adically by the assumptions on $\delta$.
\end{enumerate}
In the case where $\cA$ is an (uncurved) $A_\infty$-category over $\mathbf k$, $\delta$ is required to be strictly upper triangular, and such twisted complexes are sometimes referred to as \textit{one-sided}.
See~\cite[Section 2.2]{H3dCY} for the definitions of the morphism spaces and structure maps of $\mathrm{Tw}(\cA)$.

If $\cA$ is a curved $A_\infty$-category over $\mathbf k[[t]]$, then the category $\cA_0\coloneqq\cA\otimes_{\mathbf k[[t]]}\mathbf k$ with the same objects and morphism spaces reduced modulo $t$ is an (uncurved) $A_\infty$-category over $\mathbf k$.
There is a canonical functor $F\colon \cA\to \cA_0$ which induces a functor $F\colon \mathrm{Tw}(\cA)\to \mathrm{Tw}(\cA_0)$ denoted by the same letter.

In a similar vein to the discussion in the previous subsection about the relation between $\infty$-categories and dg categories, we need to address the relation between (stable) $\infty$-categories and (pre-triangulated) $A_\infty$-categories, so that we can apply the main results of this section on comonads to the present setting.\footnote{Note that we will neither attempt nor need to relate \textit{curved} $A_\infty$-categories with $\infty$-categories.}
Starting from an $A_\infty$-category $\cA$, one constructs an $\infty$-category $N^{A_\infty}(\cA)$ via the $A_\infty$-nerve construction~\cite{faonte}, which generalizes the dg-nerve construction.
Alternatively, one can argue that any $A_\infty$-category is quasi-equivalent to a dg category, in fact canonically via the Yoneda embedding, and then apply the dg-nerve.
This construction extends to $A_\infty$-functors.
Further, if $\cA$ is pre-triangulated, in the sense that the embedding $\cA\to\mathrm{Tw}(\cA)$ is a quasi-equivalence, then $N^{A_\infty}(\cA)$ is stable. 

\begin{remark}
Additional difficulties arise if the base ring is not a field or if strict unitality is relaxed. We refer the reader to~\cite{COS24,tanaka24}.
\end{remark}

Our main goal in this subsection is to prove the following:

\begin{proposition}\label{prop:defadjunction}
Let $\cA$ be a curved $A_\infty$-category over $\mathbf k[[t]]$, then $F\colon \mathrm{Tw}(\cA)\to \mathrm{Tw}(\cA_0)$ yields, via the $A_\infty$-nerve construction, a comonadic functor between stable $\infty$-categories.
Furthermore, $T=\mathrm{fib}(FG\xrightarrow{\varepsilon}1_{\cD})\circ [-1]$, where $G$ denotes the right adjoint constructed below, is equivalent to the identity functor, and the base change functor induces an equivalence $\mathrm{Tw}(\cA)/G(\mathrm{Tw}(\cA_0))\to \mathrm{Tw}(\cA)\otimes_{\mathbf k[[t]]}\mathbf k((t))$.
\end{proposition}

Combining this with Theorem~\ref{thm:comonads}, we obtain, after fixing some $\mathrm{cl}\colon K_0(\mathrm{Tw}(\cA_0))\to\Gamma$, a canonical map
\[
\mathrm{Stab}(\mathrm{Tw}(\cA_0))\longrightarrow\mathrm{Stab}(\mathrm{Tw}(\cA)\otimes_{\mathbf k[[t]]}\mathbf k((t)))
\]
over $\mathrm{Hom}(\Gamma,\bC)$, which is thus a local homeomorphism, although in general neither injective nor surjective.

\begin{proof}
A right adjoint $G$ was already constructed in~\cite[Section 2.3]{H3dCY}. We recall its definition.
First, choose vector space splittings 
\[
\Hom_{\cA}(X,Y)=\overline{\Hom}_{\cA}(X,Y)\oplus \left(\Hom_{\cA}(X,Y)t\right)
\]
and write elements as $a=\bar{a}+\hat{a}t$ with $\bar{a}$ and $\hat{a}t$ in the first and second summand above, respectively.
We initially define $G$ on $\cA_0$ and then extend it to $\mathrm{Tw}(\cA_0)$ in the natural way. 
For $X\in\mathrm{Ob}(\cA_0)$ let
\begin{equation*}
G(X)\coloneqq\left(X\oplus X[1],\begin{pmatrix} 0 & t \\ \hat{\mk m}_0 & 0\end{pmatrix}\right)
\end{equation*}
and for a morphism $a$ in $\cA_0$: 
\begin{equation*}
G_1(a)\coloneqq\begin{pmatrix} a & 0 \\ (-1)^{\|a\|}\hat{\mk m}_1(a) & a \end{pmatrix}.
\end{equation*}
The higher order terms of $G$ are given by
\begin{equation*}
G_n(a_n,\ldots,a_1)\coloneqq\begin{pmatrix} 0 & 0 \\ (-1)^{\|a_1\|+\ldots+\|a_n\|}\hat{\mk m}_n(a_n,\ldots,a_1) & 0 \end{pmatrix}
\end{equation*}
for $n\geq 2$. It is shown in~\cite[Proposition 2.12]{H3dCY} that $G$ is an $A_\infty$-functor $\cA_0\to \mathrm{Tw}(\cA)$. We denote its extension to $\mathrm{Tw}(\cA_0)$ by the same letter.

The composite $FG$ is given on objects of $\cA_0$ by
\[
FGX=\left(X\oplus X[1],\begin{pmatrix} 0 & 0 \\ \bar{\hat{\mk m}}_0 & 0\end{pmatrix}\right)
\]
and there is a natural transformation $\varepsilon\colon FG\to 1_{\cA_0}$ whose components are just projection to the first factor. 
This is a counit in the sense that it induces quasi-isomorphisms
\begin{equation}\label{eq:catdefadjunction}
    \Hom_{\mathrm{Tw}(\cA)}(X,GY)\xrightarrow{\simeq}\Hom_{\mathrm{Tw}(\cA_0)}(FX,Y)
\end{equation}
as shown in~\cite[Proposition 2.13]{H3dCY}. 
The same proposition also tells us that $\mathrm{fib}(\varepsilon)\simeq [1]$, i.e. $T$ is the identity functor.

The composite $GF$ is given on objects of $\mathrm{Tw}(\cA)$ (which have vanishing curvature $\mathfrak m_0$, by definition) by
\[
GFX=\left(X\oplus X[1],\begin{pmatrix} 0 & t \\ 0 & 0\end{pmatrix}\right)\simeq\mathrm{cofib}(t\cdot1_X\colon X\to X).
\]
Furthermore, $S\coloneqq \mathrm{fib}(\eta)$ is also the identity functor and the natural transformation $\tau\colon S\to 1_{\mathrm{Tw}(\cA)}$ as in Proposition~\ref{prop:comonadtower} has components $\tau_X=t\cdot 1_X$.
In particular, the localization along $\tau$ is $\mathrm{Tw}(\cA)\otimes_{\mathbf k[[t]]}\mathbf k((t))$ and, together with the dual of Lemma~\ref{lem:loc_t}, this gives the claim about $\mathrm{Tw}(\cA)/G(\mathrm{Tw}(\cA_0))$.

Consider the sequence of $A_\infty$-categories
\[
\cA_n\coloneqq\mathrm{Tw}\left(\cA\otimes_{\mathbf k[[t]]}\mathbf k[t]/t^{n+1}\right)
\]
where we take twisted complexes which are two-sided, i.e.\ not necessarily strictly upper-triangular, and whose reduction mod $t$ is one-sided, just as for curved $A_\infty$-categories over $\mathbf k[[t]]$.
By construction, $\mathrm{Tw}(\cA)=\lim_n \cA_n$ and the image of $\tau^{n+1}$ in $\cA_n$ vanishes.
It follows from Proposition~\ref{prop:comonadtower} that $F$ is comonadic.
This completes the proof of the proposition.
An alternative proof that $F$ is conservative is found in~\cite[Proposition 2.1]{Hskein}.
\end{proof}

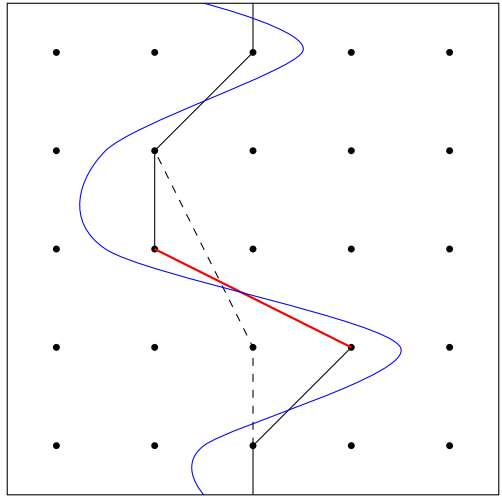
\begin{figure}[ht]
    \centering
\begin{tikzpicture}[scale=1.3]   
 \foreach \x in {0,1,2,3,4} {
  \foreach \y in {0,1,2,3,4} {
   \fill (\x,\y) circle (1pt);
  }
 }
 \draw (-0.5,-0.5) rectangle (4.5,4.5);
 \draw (2,-0.5) -- (2,0) -- (3,1) -- (1,2) -- (1,3) -- (2,4) -- (2,4.5);
 \draw[thick,red] (3,1) -- (1,2);
 \draw[dashed] (2,0) -- (2,1) -- (1,3);
 \draw[blue] plot [smooth] coordinates {(1.5,-0.5)  (1.5,0) (3.5,1) (0.5,2) (0.5,3) (2.5,4) (1.5,4.5)};
\end{tikzpicture}
    \caption{Discrete curve shortening on the torus $T^2$ punctured along the lattice of points $P$: Start with the object in $\mathcal F(T^2\setminus P)$ represented by the blue curve. Its Harder--Narasimhan filtration is represented by the solid piecewise-linear curve. The red segment is the HN component of maximal phase. A single step of the mass (length) minimizing procedure replaces the object by one whose HN filtration includes the dashed segments instead of the red segment and the two adjacent ones. Intuitively, we are allowing the curve to pass through the two endpoints of the red segment and pull it tight.}
    \label{fig:discreteshortening}
\end{figure}

\begin{example}\label{ex:discreteshortening}
Take $T^2=\bR^2/\bZ^2$ to be the flat torus, punctured along the $N^2$ points on the lattice $P\coloneqq \frac{1}{N}\bZ^2/\bZ^2\subset T^2$.
According to~\cite{hkk}, there is a corresponding stability condition on the wrapped Fukaya category $\cD\coloneqq \mathcal F(T^2\setminus P)$ (over some arbitrary fixed base field $\mathbf k$).
The category itself, up to equivalence, does not depend on the locations of the punctures, only their number, and is mirror dual to $D^b\mathrm{Coh}(Y)$, where $Y$ is a necklace of $N^2$ projective lines.
The stability condition does depend on the location of the punctures and its stable objects correspond to straight lines on $T^2$ which do not pass through any puncture and are either closed loops or are arcs whose endpoints lie in $P$.
The wrapped Fukaya category is the central fiber of a \textit{relative Fukaya category} $\cC\coloneqq \mathcal F(T^2,P)$ over $\mathbf k[[t]]$ where immersed disks in $T^2$ are counted with weight $t^n$, $n$ being the number of punctures inside the disk, cf.\ \cite{HKS}.
The general discrete minimizing procedure discussed at the end of Section~\ref{subsec:transport} is, in this example, a kind of curve-shortening procedure on piecewise straight paths on $T^2$ with vertices on $P$, see Figure~\ref{fig:discreteshortening}.
Of course, more general choices of surfaces and locations of punctures are possible.
\end{example}

\bibliographystyle{alpha}
\bibliography{catkah} 

\Addresses

\end{document}